\documentclass[12pt]{amsart}
\usepackage{amsfonts,amssymb,float,ulem,textcomp}
\RequirePackage{ifpdf}
\usepackage{amsmath}
\usepackage{color}
\usepackage{pstricks}
\usepackage{pstricks-add}
\usepackage{tikz-cd}

\def\appendix#1{
\addtocounter{section}{1} \setcounter{equation}{0}
\renewcommand{\thesection}{\Alph{section}}
\section*{Appendix \thesection\protect\indent\quad
#1}
}
\renewcommand{\theequation}{\thesection.\arabic{equation}}

\catcode`\@=11
\def\marginnote#1{}

\newcount\hour
\newcount\minute
\newtoks\amorpm
\hour=\time\divide\hour by60 \minute=\time{\multiply\hour by60
\global\advance\minute by-\hour}
\edef\standardtime{{\ifnum\hour<12 \global\amorpm={am}%
        \else\global\amorpm={pm}\advance\hour by-12 \fi
        \ifnum\hour=0 \hour=12 \fi
        \number\hour:\ifnum\minute<10 0\fi\number\minute\the\amorpm}}
\edef\militarytime{\number\hour:\ifnum\minute<100\fi\number\minute}

\newcommand{\tcr}{\textcolor{red}}
\newcommand{\tcm}{\textcolor{magenta}}
\newcommand{\tcb}{\textcolor{blue}}

\newcommand{\tcw}{\textcolor{white}}

\def\draftlabel#1{{\@bsphack\if@filesw {\let\thepage\relax
      \xdef\@gtempa{\write\@auxout{\string
          \newlabel{#1}{{\@currentlabel}{\thepage}}}}}\@gtempa \if@nobreak
    \ifvmode\nobreak\fi\fi\fi\@esphack} \gdef\@eqnlabel{#1}}
    \def\@eqnlabel{}
\def\@vacuum{}
\def\draftmarginnote#1{\marginpar{\raggedright\scriptsize\tt#1}}

\def\draft{
  \oddsidemargin -.5truein
  \def\@oddfoot{\footnotesize \sl preliminary draft \hfil
    \rm\thepage\hfil\sl\today\quad\militarytime}
  \let\@evenfoot\@oddfoot \overfullrule 3pt
    \let\label=\draftlabel
    \let\marginnote=\draftmarginnote
  \def\@eqnnum{(\theequation)\rlap{\kern\marginparsep\tt\@eqnlabel}%
    \global\let\@eqnlabel\@vacuum}

  }

\voffset= - 0.5in \hoffset= - 1.0in

\newcommand{\tr}{\,{\rm Tr}\,}

\def\be{\begin{equation}}
\def\ee{\end{equation}}
\def\bea{\begin{eqnarray}}
\def\eea{\end{eqnarray}}
\def\<{\langle}
\def\>{\rangle}
\def\nn{\nonumber}

\def\ZZ{{\Bbb Z}}
\def\A{\mathcal{A}}

\def\X{\mathcal{X}}

\def\bb{{\bf b}}
 \def\aa{{\bf a}}
 \def\rr{{\bf r}}
\def\M{{\mathcal M}}
\let\wtd=\widetilde
\def\Tr{{\rm Tr}}
\def\g{{\mathfrak g}}
\def\gl{{\mathfrak{gl}}}

\def\G{{\mathcal G}}
\def\h{\mathfrak h}
\def\D{{\mathfrak d}}
\def\U{{\mathcal U}}

\def\vv{{\mathbf v}}

\def\tr{\mathop{\rm{tr}}}

\def\ocomma{{\phantom{\Bigm|}^{\phantom {X}}_{\raise-1.5pt\hbox{,}}\!\!\!\!\!\!\otimes}}

\def\End{\operatorname{End}}
\def\Id{{\operatorname {Id}}}

\def\Poi{{\{\cdot,\cdot\}}}
\newcommand{\col}[1]{{\raise-2pt\hbox{\tiny$\bullet$}\hskip -4.5pt \raise4pt\hbox{\tiny$\bullet$}{{#1}} \raise-2pt\hbox{\tiny$\bullet$}\hskip -4.5pt \raise4pt\hbox{\tiny$\bullet$}}}

\newcommand{\sheet}[2]{{\stackrel{{#1}}{{#2}}}}

\def\dnabla{{\raisebox{2pt}{$\bigtriangledown$}}\negthinspace}

\newtheorem{theorem}{Theorem}[section]
\newtheorem{lemma}[theorem]{Lemma}

\theoremstyle{definition}
\newtheorem{definition}[theorem]{Definition}
\newtheorem{example}[theorem]{Example}
\newtheorem{remark}[theorem]{Remark}

\newtheorem{corollary}[theorem]{Corollary}

\allowdisplaybreaks

\begin{document}

\title[Darboux coordinates for symplectic groupoid and cluster algebras]
{Darboux coordinates for symplectic groupoid and cluster algebras}
\author{Leonid O. Chekhov$^{\ast}$}
\thanks{$^{\ast}$ Steklov Mathematical
Institute, Moscow, Russia, National Research University Higher School of Economics, Russia, and Michigan State University, East Lansing, USA. Email: chekhov@msu.edu.} 
\author{Michael Shapiro$^{\ast\ast}$} 
\thanks{$^{\ast\ast}$ Michigan State University, East Lansing, USA and National Research University Higher School of Economics, Russia . Email: mshapiro@msu.edu.} 


\begin{abstract}
Using Fock--Goncharov higher Teichm\"uller space variables we derive Darboux coordinate representation for entries of general symplectic leaves of the $\mathcal A_n$ groupoid of upper-triangular matrices and, in a more general setting, of higher-dimensional symplectic leaves for algebras governed by the reflection equation with the trigonometric $R$-matrix. The obtained results are in a perfect agreement with the previously obtained Poisson and quantum representations of groupoid variables for $\mathcal A_3$ and $\mathcal A_4$ in terms of geodesic functions for Riemann surfaces with holes. We realize braid-group transformations for $\mathcal A_n$ via sequences of cluster mutations in the special $\mathcal A_n$-quiver. We prove the groupoid relations for quantum transport matrices and, as a byproduct, obtain the Goldman bracket in the semiclassical limit. We prove the quantum algebraic relations of transport matrices for arbitrary (cyclic or acyclic) directed planar network.
\end{abstract}

\dedicatory{Dedicated to the memory of great mathematician and person Boris Dubrovin}

\maketitle

\section{Introduction}\label{s:intro}
\setcounter{equation}{0}

\subsection{Symplectic groupoid, induced Poisson structure on the unipotent upper triangular matrices}

Let $V$ denote an $n$-dimensional vector space, $\mathcal A$ be some subspace of bilinear forms on $V$. Fixing the basis in $V$, one can identify $\mathcal{A}$ with a subspace in the space of $n\times n$ matrices.
The matrix $B$ of a change of a basis  in $V$ takes a matrix of bilinear form  $\mathbb A\in \mathcal A$ to $B\mathbb AB^{\text{T}}$. 

Below we consider an important particular  case  when  
$\mathcal A$ is the space of unipotent forms identified with the space of the unipotent matrices. The basis change $B$ acts on $\mathcal A$ only if the product  $B\mathbb AB^{\text{T}}$ be unipotent itself. We thus introduce the space of \emph{morphisms} identified with admissible pairs of matrices $(B,\mathbb A)$ such that  
$$
\mathcal M=\bigl\{ (B,\mathbb A) \bigm | B\in GL(V),\ \mathbb A\in \mathcal A,\ B\mathbb AB^{\text{T}}\in \mathcal A  \bigr\}.
$$
We then have the standard set of maps:
$$
\begin{array}{llll}
\hbox{source}&s:&\mathcal M\to\mathcal A&(B,\mathbb A)\to \mathbb A,\\
\hbox{target}&t:&\mathcal M\to\mathcal A&(B,\mathbb A)\to B\mathbb AB^{\text{T}},\\
\hbox{injection}&e:&\mathcal A\to\mathcal M&\mathbb A\to (E, \mathbb A),\\
\hbox{inversion}&i:&\mathcal M\to\mathcal M&(B,\mathbb A)\to (B^{-1}, B\mathbb AB^{\text{T}}),\\
\hbox{multiplication}&m:&\mathcal M^{(2)}\to\mathcal M&\bigl((C, B\mathbb AB^{\text{T}}),(B,\mathbb A)\bigr)\to (CB, \mathbb A)
\end{array}
$$
such that the following diagram, where $p_1$ and $p_2$ are natural projections to the first and the second morphism in an admissible pair of morphisms, is commutative:
$$
\begin{tikzcd}
&\mathcal M \arrow[dr, "s"]&\\
\mathcal M^{(2)}\arrow[ur, "p_2"] \arrow[dr, "p_1"]&&\mathcal A\\
&\mathcal M \arrow[ur, "t"]&
\end{tikzcd}
$$
The crucial point of the construction is the existence of a \emph{symplectic structure}: a smooth groupoid endowed with a symplectic form $\omega\in \Omega^2\mathcal M$ on the morphism space $\mathcal M$ that satisfies the splitting (consistency) condition \cite{Karasev,Weinstein}
$$
m^\star \omega=p_1^\star \omega + p_2^\star \omega,
$$
which implies, in particular, that the source and target maps Poisson commute being respectively an automorphism and an anti-automorphism of the initial Poisson algebra. Since $p_1^\star \omega$ and $p_2^\star \omega$ are nondegenerate, they admit a (unique) Poisson structure, and because the immersion map $e$ is Lagrangian, this Poisson structure yields a Poisson structure on $\mathcal A$.

Identifying $\mathcal A$ with $\mathcal A_n$---the space of unipotent upper triangular matrices, in 2000, Bondal \cite{Bondal} obtained the Poisson structure on $\mathcal A_n$ using the algebroid construction; assuming $B=e^{\mathfrak g}$, we obtain the Bondal algebroid using the anchor map $D_{\mathbb A}$ to the tangent space $ T_{\mathbb A}\mathcal A_n$
\be
\label{gA}
\begin{array}{lccl}
D_{\mathbb A}:&\mathfrak g_{\mathbb A}&\to& T_{\mathbb A}\mathcal A_n\\
& g &\mapsto &{\mathbb A}g+g^{\text{T}}{\mathbb A},\quad \mathbb A\in \mathcal A_n,\\ \end{array}
\ee
where  $\mathfrak{g}_{\mathbb A}$ is the linear subspace
$$
\mathfrak{g}_{\mathbb A}:=\left\{ g\in\gl_{n}(\mathbb C),|\,
\mathbb A+ {\mathbb A}g+g^{\text{T}}{\mathbb A}\in\mathcal A _n\right\}
$$
of elements $g$ leaving $\mathbb A$ unipotent.

\begin{lemma}\label{lemma-g}\cite{Bondal}
The map
\be
\label{P_A}
\begin{array}{lccl}
P_{\mathbb A}:&T^\ast_{\mathbb A_n}\mathcal A&\to& \mathfrak g_{\mathbb A}\\
& w &\mapsto &P_{-,1/2}(w{\mathbb A})-P_{+,1/2}(w^{\text{T}}{\mathbb A}^{\text{T}}),\\ \end{array}
\ee
where $P_{\pm,1/2}$ are the projection operators:
\be\label{eq:pr}
P_{\pm,1/2}a_{i,j}:=\frac{1\pm {\rm sign}(j-i)}{2}a_{i,j}, \quad i, j=1,\dots,n,
\ee
and $w\in T^\ast\mathcal A_n$ is a strictly lower triangular matrix, 
defines an isomorphism between the Lie algebroid $(\mathfrak g,D_{\mathbb A})$ and the Lie algebroid
$\left(T^\ast \mathcal A_n,D_{\mathbb A}P_{\mathbb A}  \right)$.
\end{lemma}

The Poisson bi-vector $\Pi$ on $\mathcal A_n$ is then obtained by the anchor map on the Lie algebroid
$\left(T^\ast \mathcal A_n,D_{\mathbb A}P_{\mathbb A}  \right)$ (see Proposition 10.1.4 in \cite{kirill}) as:
\be
\label{eq:biv}
\begin{array}{lccl}
\Pi:&T^\ast_{\mathbb A}\mathcal A_n\times T^\ast_{\mathbb A}\mathcal A_n
&\mapsto& \mathcal C^\infty(\mathcal A_n)\\
&(\omega_1,\omega_2)&\to&\Tr\left(\omega_1 D_{\mathbb A}P_{\mathbb A}  (\omega_2)
\right)
\end{array}
\ee
It can be checked explicitly that the above bilinear form is in fact skew-symmetric and gives rise to the Poisson bracket
\be
\label{Poisson-bracket}
\{a_{i,k},a_{j,l}\}:=\frac{\partial}{\partial{\rm d}a_{i,k}}\wedge\frac{\partial}{\partial{\rm d}a_{j,l}}
\Tr\left( {\rm d}a_{i,k} D_{\mathbb A}P_{\mathbb A}  ({\rm d}a_{j,l})\right),
\ee
having the following form in components:
\begin{eqnarray}\label{eq:du}
&&
\left\{a_{i,k},a_{j,l}\right\}=0,\quad\hbox{for}\ i<k<j<l,\hbox{ and }  i<j<l<k,\nn
\\&&
\left\{a_{i,k},a_{j,l}\right\}=2 \left(a_{i,j}a_{k,l}-a_{i,l}a_{k,j}\right),\quad\hbox{for}\  i<j<k<l,
\\&&
\left\{a_{i,k},a_{k,l}\right\}=a_{i,k}a_{k,l}-2a_{i,l},\quad\hbox{for}\  i<k<l,\nn
\\&&
\left\{a_{i,k},a_{j,k}\right\}=-a_{i,k}a_{j,k}+2a_{i,j},\quad\hbox{for} \  i< j<k,\nn
\\&&
\left\{a_{i,k},a_{i,l}\right\}=-a_{i,k}a_{i,l}+a_{k,l},\quad\hbox{for} \  i<k<l.\nn
\end{eqnarray}
This bracket turned out to coincide with the bracket previously known in mathematical physics as Gavrilik--Klimyk--Nelson--Regge--Dubrovin--Ugaglia bracket \cite{GK91,NR,NRZ,Dub,Ugaglia} and it arises from skein relations satisfied by a special finite subset of {\it geodesic functions} (traces of monodromies of $SL_2$ Fuchsian systems, which are in 1-1 correspondence with closed geodesics on a Riemann surface $\Sigma_{g,s}$) described in \cite{ChF3}; a simple constant log-canonical (Darboux) bracket on the space of Thurston shear coordinates $z_\alpha$  on the Teichm\"uller space $T_{g,s}$ of Riemann surfaces $\Sigma_{g,s}$ of genus $g$ with $s=1,2$ holes was shown \cite{ChF2} to induce the above bracket on a special subset of geodesic functions identified with the matrix elements $a_{i,k}$. All such geodesic functions admit an explicit combinatorial description \cite{F97}, which immediately implies that they are Laurent polynomials with positive integer coefficients of $e^{z_\alpha/2}$. The Poisson bracket of $z_\alpha$ spanning the Teichm\"uller space $T_{g,s}$ has exactly $s$ Casimirs, which are linear combinations of shear coordinates incident to the holes, so the subspace of $z_{\alpha}$ orthogonal to the subspace of Casimirs parameterizes a symplectic leaf in the Teichm\"uller space which we call a {\it geometric symplectic leaf}. 

In\cite{ChF3}, the Poisson embedding of geometric symplectic leaf into $\mathcal A_n$ was constructed. Note however that the size $n$ of matrix $\mathbb A$ is related to the genus and the number of holes as $n=2g+s$ (with $s$ taking only two values, 1 and 2) and that the (real) dimension of $T_{g,s}$ is $6g-6+3s$ increasing linearly with $g$ whereas the total dimension of $\mathcal A_n$ is obviously $n(n-1)/2$ increasing quadratically with $n$; for $n=3$ and $n=4$ these two dimensions coincide and the geometric symplectic leaf having the dimension $6g-6+2s$ is of maximum dimension. 

For $n=5$, the dimension of the geometric symplectic leaf has still the maximum value $8$ of dimensions of symplectic leaves in $\mathcal A_5$, but we have just one central element in the corresponding Teichm\"uller space $T_{2,1}$ and two central elements in $\mathcal A_5$. For all larger $n$ the dimension of geometric symplectic leaf is strictly less than the maximal dimension of symplectic leaf in $\mathcal A_n$, so the geometric systems do not describe maximal symplectic leaves in the total Poisson space of $\mathcal A_n$. The Darboux coordinates in geometric situation are well known to be the above shear coordinates, but, as just mentioned, they can not help in constructing Darboux coordinates in $\mathcal A_n$ for $n\ge 5$.
	
The  {\bf first problem}  addressed in this publication is a construction of Darboux coordinates for a general symplectic leaf of $\mathcal A_n$ and explicit expression of  matrix elements $a_{i,j}$ in terms of these Darboux coordinates. It was expected for long, and we show below that these Darboux coordinates are related to cluster algebras, similar to the geometric cases $n=3,4$.

From the integrable models standpoint, algebras (\ref{eq:du}) (either with a unipotent $\mathbb A$ or with a general $\mathbb A_{\text{gen}}\in \gl_n$ are known under the name of {\it reflection equation algebras}. A task closely related to the first problem is to construct a Darboux coordinate representation for a general matrix $\mathbb A_{\text{gen}}$ enjoying the reflection equation.

\subsection{Standard Poisson-Lie group $\G$ and its dual
}
\label{double}

Another description of the Poisson structure on the space of triangular forms $\mathcal A_n$ as a push-forward of the standard Poisson bracket on the dual group $\G^*=SL_n^*$ to the set of fixed points of the natural involution was given in~\cite{Boalch}.

A reductive complex Lie group $\G$ equipped with a Poisson bracket $\Poi$ is called a {\em Poisson--Lie group\/}
if the multiplication map
$\G\times \G \ni (X,Y) \mapsto XY \in \G$
is Poisson. Denote by 
$\langle \ , \ \rangle$ an invariant nondegenerate form on
the corresponding Lie algebra $\g=Lie(\G)$, and by $\nabla^R$, $\nabla^L$ the right and
left gradients of functions on $\G$ with respect to this form defined by
\begin{equation*}
\left\langle \nabla^R f(X),\xi\right\rangle=\left.\frac d{dt}\right|_{t=0}f(Xe^{t\xi}),  \quad
\left\langle \nabla^L f(X),\xi\right\rangle=\left.\frac d{dt}\right|_{t=0}f(e^{t\xi}X)
\end{equation*}
for any $\xi\in\g$, $X\in\G$.

Let $\pi_{>0}, \pi_{<0}$ be projections of  
$\g$ onto subalgebras spanned by positive and negative roots, $\pi_0$ be the projection onto the Cartan 
subalgebra $\h$, and let $R=\pi_{>0} - \pi_{<0}$. 
The {\em standard Poisson-Lie bracket\/} $\Poi_r$ on $\G$  can be written as
\begin{equation}
\{f_1,f_2\}_r = \frac 1 2  \left( \left\langle R(\nabla^L f_1), \nabla^L f_2 \rangle - \langle R(\nabla^R f_1), \nabla^R f_2 \right\rangle \right).
\label{sklyabra}
\end{equation}

The standard Poisson--Lie structure is a particular case of Poisson--Lie structures corresponding to
quasitriangular Lie bialgebras. For a detailed exposition of these structures see, e.~g., 
\cite[Ch.~1]{CP}, \cite{r-sts} and \cite{Ya}.

Following \cite{r-sts}, let us recall the construction of {\em the Drinfeld double}. The double of $\g$ is 
$D(\g)=\g  \oplus \g$ equipped with an invariant nondegenerate bilinear form
$\langle\langle (\xi,\eta), (\xi',\eta')\rangle\rangle = \langle \xi, \xi'\rangle - \langle \eta, \eta'\rangle$. 
Define subalgebras $\D_\pm$ of $D(\g)$ by
$\D_+=\{( \xi,\xi) : \xi \in\g\}$ and $\D_-=\{ (R_+(\xi),R_-(\xi)) : \xi \in\g\}$,
where $R_\pm\in \End\g$ is given by $R_\pm=\frac{1}{2} ( R \pm \Id)$. 
The operator $R_D= \pi_{\D_+} - \pi_{\D_-}$ can be used to define 
a Poisson--Lie structure on $D(\G)=\G\times \G$, the double of the group $\G$, via
\begin{equation}
\{f_1,f_2\}_D = \frac{1}{2}\left (\left\langle\left\langle R_D(\dnabla^L f_1), \dnabla{^L} f_2 \right\rangle\right\rangle 
- \left\langle\left\langle R_D(\dnabla^R f_1), \dnabla^R f_2 \right\rangle\right\rangle \right),
\label{sklyadouble}
\end{equation}
where $\dnabla^R$ and $\dnabla^L$ are right and left gradients with respect to $\langle\langle \cdot ,\cdot \rangle\rangle$.
The diagonal subgroup $\{ (X,X)\ : \ X\in \G\}$ is a Poisson--Lie subgroup of $D(\G)$ (whose Lie algebra is $\D_+$) naturally isomorphic
to $(\G,\Poi_r)$.

The group $\G^*$  whose Lie algebra is $\D_-$ is a Poisson-Lie subgroup of $D(\G)$ called {\em the dual Poisson-Lie group of $\G$}.
The Poisson bracket $\{\cdot,\cdot\}_D$ induces the Poisson bracket on $\G^*$.
%


%
For $\G=SL_n$ 
the dual group $\G^*=\{(X_+,Y_-)\}\in B_+\times B_-$ satisfying the additional relation $\pi_0(X_+)\pi_0(Y_-)=\Id$ where $B_+(B_-)\subset SL_n$ are Borel subgroups of nondegenerate upper (lower) triangular matrices.

The involution $\iota_{\G^*}:\G^*\to \G^*$ takes $(X_+,Y_-)$ to $(Y_-^t,X_+^t)$.

The subgroup $\U_+$ of unipotent upper triangular matrices is embedded diagonally in $\G^*$. The embedding $\epsilon:\U_+\hookrightarrow\G^*$
maps $X\in \U_+$ to $(X,X)$.  The image $\epsilon(\U_+)$ is the set of fixed points of involution $\iota_{\G^*}$.

The image $\epsilon(\U_+)$ is not a Poisson subvariety of $\G^*$ however the Dirac reduction induces the Poisson bi-vector $\Pi$~(\ref{eq:biv}) on $\U_+$.

To remind the definition of Dirac reduction we consider a subvariety $X$ of a Poisson variety $(V,\{ \cdot,\cdot\}_{PB})$ defined by constrains $\phi_i=const$. The second class constrains are constrains $\tilde\phi_a$ whose  Poisson brackets with at least one other constraint do not vanish on the constraint surface.

Define matrix $U$ with entries
$U_{ab}=\{\tilde\phi_a,\tilde\phi_b\}_{PB}$. Note that $U$ is always invertible.

Then, Dirac bracket of functions $f$ and $g$ on $X$ is 
$$\{f, g\}_{DB} = \{f, g\}_{PB} - \sum_{a, b}\{f,\tilde{\phi}_a\}_{PB} U^{-1}_{ab}\{\tilde{\phi}_b,g\}_{PB},$$
see \cite{HenTeit} for details.


\subsection{Braid-group action on the unipotent matrices}\label{ss:braid}

The next important result concerning $\mathcal A_n$ is that this space admits the discrete braid-group action generated by morphisms $\beta_{i,i+1}:{\mathcal A_n}\to{\mathcal A_n}$, $i=1,\dots,n-1$, such that
\be
\label{beta}
\beta_{i,i+1}[{\mathbb A}]=B_{i,i+1}{\mathbb A}B_{i,i+1}^T\equiv \wtd{\mathbb A}\in \mathcal A_n,
\ee
where the matrix $B_{i,i+1}$ has the block form
\be
\label{Bii+1}
B_{i,i+1}=\begin{array}{c}\vdots\cr i\cr i+1\cr \vdots \end{array}\left[\begin{array}{cccccccc}1&&&&&&&\cr
&\ddots&&&&&&\cr&&1&&&&&\cr&&&a_{i,i+1}&-1&&&\cr
&&&1&0&&&\cr
&&&&&1&&\cr&&&&&&\ddots&\cr&&&&&&&1
\end{array}\right],
\ee
and this action is a Poisson morphism  \cite{Bondal}, \cite{Ugaglia}.
When acting on $\mathbb A$, $\beta_{i,i+1}$  satisfy the standard braid-group relations $\beta_{i,i+1} \beta_{i+1,i+2}\beta_{i,i+1}\mathbb A=\beta_{i+1,i+2}\beta_{i,i+1}\beta_{i+1,i+2}\mathbb A$ for $i=1,\dots,n-2$ together with the additional relation $\beta_{n-1,n}\beta_{n-2,n-1}\cdots \beta_{2,3}\beta_{1,2}\mathbb A=S_n \mathbb A$, where $S_n$ is an element of the group of permutations of matrix entries $a_{i,j}$ whose $n$th power is the identity transformation. Note that $\beta^2_{i,i+1}\mathbb A\ne \mathbb A$.

In  \cite{ChM}, the quantum version of the above transformations was constructed for a {\sl quantum} upper-triangular matrix 
\be\label{A-quantum}
\mathbb A^\hbar:=\left[ \begin{array}{ccccc}
q^{-1/2}  & a^\hbar_{1,2}  & a^\hbar_{1,3}  &\dots & a^\hbar_{1,n}   \\
 0 & q^{-1/2}  & a^\hbar_{2,3} & \dots  & a^\hbar_{2,n}  \\
 0 & 0 & q^{-1/2}  & \ddots &   \vdots \\
 \vdots & \vdots   & \ddots & \ddots &  a^\hbar_{n-1,n} \\
 0 & 0 & \dots & 0 & q^{-1/2} 
 \end{array}\right].
\ee
Here $a^\hbar_{i,j}$ are self-adjoint $\bigl( \bigl[a^\hbar_{i,j}\bigr]^\star=a^\hbar_{i,j}\bigr)$ operators enjoying quadratic--linear algebraic relations following from the  quantum reflection equation (see Theorem~\ref{th:A}) and coinciding with relations obtained for {\sl quantum geodesic functions} upon imposing quantum skein relations on the corresponding geodesics and $q=e^{-i\hbar/2}$. The analogous quantum braid-group action is $\mathbb A^\hbar\to B^\hbar_{i,i+1}\mathbb A^\hbar \bigl[ B^\hbar_{i,i+1} \bigr]^\dagger$ with
\be
\label{Bii+1-hbar}
B^\hbar_{i,i+1}=\begin{array}{c}\vdots\cr i\cr i+1\cr \vdots \end{array}\left[\begin{array}{cccccccc}1&&&&&&&\cr
&\ddots&&&&&&\cr&&1&&&&&\cr&&&q^{1/2}a^\hbar_{i,i+1}&-q&&&\cr
&&&1&0&&&\cr
&&&&&1&&\cr&&&&&&\ddots&\cr&&&&&&&1
\end{array}\right],
\ee

In the geometric cases, the above braid-group morphisms are related to modular transformations generated by (classical or quantum \cite{Kashaev-Dehn}) Dehn twists along geodesics corresponding to the geodesic functions $a_{i,i+1}$ (see \cite{ChF3}).  In the absence of geometric interpretation, the only possibility we may resort to is to address the {\bf second problem}: to find a sequence of cluster mutations in a quiver still to be constructed that produces the above braid-group transformation for a generic symplectic leaf of $\mathcal A_n$.

We solve the both formulated problems in this paper: we explicitly construct the quiver (called an $\mathcal A_n$-quiver) such that the entries $a_{i,j}$ of the unipotent matrix $\mathbb A$ are positive Laurent polynomials of the cluster quiver variables, construct a quantum version of this quiver thus realizing the representation (\ref{A-quantum})  and finding explicitly chains of mutations of the $\mathcal A_n$-quiver that produce the braid-group transformations.

The structure of the paper is as follows:

In Sec.~\ref{s:triangle}, we describe quantum algebras of transport matrices in the Fock--Goncharov $SL_n$-quiver (Theorem~\ref{th:MM}); this quantum algebra is based on a more general Lemma~\ref{lem:r-matrix} proven in Sec.~\ref{sec:QuantumMeasurements} for any planar (acyclic) directed network. We also prove the groupoid condition (Theorem~\ref{th:groupoid}) satisfied by quantum transport matrices in the $SL_n$-quiver. 

In Sec.~\ref{s:Goldman}, we briefly describe general algebraic relations enjoyed by quantum transport matrices for $SL_n$ character variety on a general triangulated Riemann surface $\Sigma_{g,s,p}$ with $p>0$ marked points on the hole boundaries; namely we demonstrate the satisfaction of quantum Goldman relations.

Section~\ref{s:reflection} contains the first main result: out of cluster variables of the $SL_n$-quiver we construct a unipotent $\mathbb A$ satisfying the quantum reflection equation (Theorem~\ref{th:A}). We generalize this construction to solutions of quantum reflection equation that are not necessarily unipotent (Theorem~\ref{th:A-gen}).

In Sec.~\ref{s:braid}, we associate the unipotent $\mathbb A$ constructed in the preceding section with a special $\mathcal A_n$-quiver and prove that special sequences of mutations at vertices of this quiver generate braid-group transformations of elements of $\mathbb A$ (Theorem~\ref{th:braid}).  

In Sec.~\ref{s:Casimirs}, we collect statements about Casimir elements of $SL_n$- and $\mathcal A_n$-quivers. 

In Sec.~\ref{sec:QuantumMeasurements} we consider quantum transport matrices for general acyclic planar directed networks, establish the relation to Postnikov's quantum Grassmannians and measurement maps, and prove the general $R$-matrix relation for the corresponding quantum transport matrices (Lemma~\ref{lem:r-matrix}).

In Sec.~\ref{s:cycles}, we generalize the results of Sec.~\ref{sec:QuantumMeasurements} to arbitrary planar directed network (relaxing the acyclicity condition) showing in Theorem~\ref{thm:qnetwork} that quantum transport elements in any such network satisfy the same closed algebraic relations as elements of an acyclic planar directed network.

Section~\ref{s:conclusion} is a brief conclusion.

\section{$sl_n$-Algebras for the triangle $\Sigma_{0,1,3}$}\label{s:triangle}
\setcounter{equation}{0}

Let $\Sigma_{g,s,p}$ denote a topological genus $g$ surface with $s$ boundary components and $p$ marked points.
In this section, we concentrate on the case of the disk with 3 marked points on the boundary $\Sigma_{0,1,3}$.
(To simplify notations, we use $\Sigma=\Sigma_{0,1,3}$.)
In this section we review the definition of quantized moduli space ${\X}_{SL_n,\Sigma}$  of \emph{framed}   $SL_n$-local systems on the disk with three marked points (\cite{FG1}).
We call disk with three marked points $A,B,C$ on its boundary \emph{triangle with vertices $A,B,C$} and use notation $\triangle ABC$ (see Fig~\ref{fi:triSigma}). 


\subsection{Quantum torus} Let lattice $\Lambda=\mathbb{Z}^m$ be equipped with a skew-symmetric integer form $\langle \cdot, \cdot \rangle$. Introduce the $q$-multiplication operation in the vector space 
$\Upsilon=\operatorname{Span}\{ Z_\lambda\}_{\lambda\in\Lambda}$ as follows 
$Z_\lambda Z_\mu=q^{\langle \lambda,\mu\rangle} Z_{\lambda+\mu}$.  The algebra $\Upsilon$ is called a \emph{quantum torus}. Fix a basis $\{e_i\}$ in $\Lambda$, we consider $\Upsilon$ as a non-commutative algebra of Laurent polynomials in variables $Z_i:=Z_{e_i}$, $i\in [1,N]$.
For any sequence ${\bf a}=(a_1,\dots,a_t)$, $a_i\in [1,m]$, let $\Pi_{\bf s}$ denote the monomial $\Pi_{\bf s}=Z_{a_1} Z_{a_2} \dots Z_{a_t}$. Let $\lambda_{\bf s}=\sum_{j=1}^t e_{a_j}$.
Element $Z_{\lambda_{\bf s}}$ is called in physical literature the Weyl form
of $\Pi_{\bf s}$ and we denote it by two-sided colons $\col{\Pi_{\bf s}}$  It is easy to see that 
$\col{\Pi_{\bf s}}=Z_{\lambda_{\bf s}}=q^{-\sum_{j<k}\langle e_j, e_k\rangle} \Pi_{\bf s}$. 

\begin{figure}[H]
\begin{pspicture}(-1.5,-1.5)(1.5,1.) {\psset{unit=0.8}
\psline(-0.9,-1)(0,0.5)
\psline(-0.9,-1)(0.9,-1)
\psline(0.9,-1)(0,0.5)
\put(0.2,0.6){\makebox(0,0)[cc]{$1$}}
\put(1.1,-0.7){\makebox(0,0)[cc]{$2$}}
\put(-1.2,-0.7){\makebox(0,0)[cc]{$3$}}
\psarc[linecolor=blue]{<-}(0,0.5){0.7}{220}{300} 
\psarc[linecolor=blue]{->}(0.9,-1){0.7}{120}{200} 
\put(-0.9,-0.2){\makebox(0,0)[cc]{$M_1$}}
\put(-0.2,-1.3){\makebox(0,0)[cc]{$M_2$}}
}
\end{pspicture} 
\caption{Triangle $\triangle 123$ with vertices $1,2,3$.}
\label{fi:triSigma}
\end{figure}
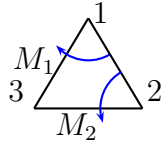

\subsection{Moduli space $\X_{SL_n,\Sigma}$ and transport matrices}

Following \cite{FG1}, recall that a moduli space $\X_{SL_n,\Sigma}$ parametrizes framed $SL_n$-local systems on $\Sigma$ that is isomorphic to the triple of flags in $\mathbb C^n$ in general position.  Any  framed $SL_n$-local system in the triangle $\triangle 123$ determines  \emph{transport} matrices. 

Two transport matrices $M_1$ and $M_2$ correspond respectively to directed paths 
going from one side of $\Sigma$ to the other side as on Fig~\ref{fi:triSigma}.
If $M$ is associated to a directed path then the inverse matrix $M^{-1}$ corresponds to the same path in the opposite  direction.

As mentioned above,  $\X_{SL_n,\Sigma}$ parametrizes the configurations of triples of complete flags in general position $(F_1)_\bullet,(F_2)_\bullet,(F_3)_\bullet$ in $\mathbb C^n$. 
Recall that a \emph{complete flag} $F_\bullet$ is a collection of consecutively embedded subspaces $\{0=F_0 \subset F_1\subset\dots\subset F_k\subset\dots\subset F_{n-1}\subset F_n={\mathbb C}^n\}$ where $F_k$ is a linear subspace of dimension $k$. Denote by $F^a=F_{n-a}$, $a=0,1,\dots,n$, the vector subspace of codimension $a$.

Consider the subtriangulation of $\triangle 123$ into $n\choose 2$ white upright triangles and $n-1\choose 2$ upside-down black triangles (see Fig.~\ref{fi:triSL3}).
Label all white upright triangles by triples $\{(a,b,c) | a,b,c\ge 0\, \&\, a+b+c=n-1\}$. Each white triangle $(a,b,c)$ corresponds to a line $\ell_{abc}=(F_1)^a\cap(F_2)^b\cap(F_3)^c$.
Similarly, label black upside-down triangles by triples $\{(a,b,c)|a,b,c\ge 0\,\&\,a+b+c=n-2\}$.
Each upside-down triangle $(a,b,c)$ is associated with the plane $P_{abc}=(F_1)^a\cap(F_2)^b\cap(F_3)^c$.
Note that every plane $P_{abc}$ of a black triangle contains all three lines $\ell_{(a+1)bc},\ell_{a(b+1)c},\ell_{ab(c+1)}$ of white triangles which are neighbors of the black one. For every such plane $P_{abc}$ choose three vectors $\vv_{(a+1)bc}\in\ell_{(a+1)bc}$, $\vv_{a(b+1)c}\in\ell_{a(b+1)c}$, $\vv_{ab(c+1)}\in\ell_{ab(c+1)}$ such that they satisfy condition $\vv_{(a+1)bc}+\vv_{a(b+1)c}=\vv_{ab(c+1)}$.
Hence, given a configuration of lines corresponding to triple of flags $((F_1)_\bullet,(F_2)_\bullet,(F_3)_\bullet)$, the choice of one vector $\vv_{abc}\in\ell_{abc}$ determines uniquely all other vectors in the lines $\ell_{a'b'c'}$ for all $(a'b'c')$ (see Fig.~\ref{fi:triSL3v1}).

Thus, the configuration of lines $\ell_{abc}$ determines \emph{projective collection of vectors} $\{\vv_{abc}\}$ modulo scalar scaling. Note that exactly two vectors at vertices of any gray triangle are independent. 

Define a \emph{snake} as an oriented path running from the top black triangle containing the line $\ell_{n00}$ downwards to the bottom black triangle containing
$\ell_{0ab}$ (for example, bold path in Fig~\ref{fi:triSL3v1}). 

	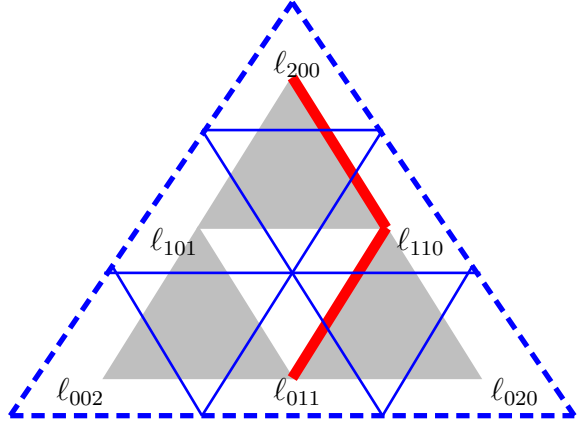
\begin{figure}[h]
\begin{pspicture}(-3,-3)(3,3.5){
%


\pspolygon[linecolor=lightgray,fillstyle=solid,fillcolor=lightgray](0.,2.5)(1.25,0.5)(-1.25,0.5)
\pspolygon[linecolor=lightgray,fillstyle=solid,fillcolor=lightgray](1.25,0.5)(2.5,-1.5)(0,-1.5)
\pspolygon[linecolor=lightgray,fillstyle=solid,fillcolor=lightgray](-1.25,0.5)(-2.5,-1.5)(0,-1.5)
\put(0,0){\psline[linecolor=red,linewidth=4pt]{-}(0,2.5)(1.25,0.5)}
\put(0,0){\psline[linecolor=red,linewidth=4pt]{-}(1.25,0.5)(0,-1.5)}

\put(0,0){\psline[linecolor=blue,linewidth=1pt]{-}(-2.5,-0.1)(2.5,-0.1)}
\put(0,0){\psline[linecolor=blue,linewidth=1pt]{-}(-1.2,-2.)(-2.4,0)}
\put(0,0){\psline[linecolor=blue,linewidth=1pt]{-}(-1.2,1.8)(1.2,1.8)}
\put(0,0){\psline[linecolor=blue,linewidth=1pt]{-}(-1.2,-2)(1.2,1.8)}
\put(0,0){\psline[linecolor=blue,linewidth=1pt]{-}(1.2,-2)(-1.2,1.8)}
\put(0,0){\psline[linecolor=blue,linewidth=1pt]{-}(1.2,-2)(2.4,0)}
\put(0,0){\psline[linecolor=blue,linewidth=2pt,linestyle=dashed]{-}(-3.75,-2.)(0.0,3.5)}
\put(0,0){\psline[linecolor=blue,linewidth=2pt,linestyle=dashed]{-}(0.0,3.5)(3.75,-2)}
\put(0,0){\psline[linecolor=blue,linewidth=2pt,linestyle=dashed]{-}(-3.75,-2.)(3.75,-2)}

\put(0,2.5){\makebox(0,0)[bc]{\hbox{{ $\ell_{200}$}}}}
\put(1.25,0.5){\makebox(0,0)[tl]{\hbox{{ $\ell_{110}$}}}}
\put(-1.25,0.5){\makebox(0,0)[tr]{\hbox{{$\ell_{101}$}}}}
\put(2.5,-1.5){\makebox(0,0)[tl]{\hbox{{ $\ell_{020}$}}}}
\put(0,-1.5){\makebox(0,0)[tc]{\hbox{{ $\ell_{011}$}}}}
\put(-2.5,-1.5){\makebox(0,0)[tr]{\hbox{{ $\ell_{002}$}}}}

}
\end{pspicture}
\caption{\small
Configuration of lines corresponding to triple of flags in ${\mathbb C}^3$. Black triangles are equipped with planes $P_{abc}$.
Plane $P_{100}$ contains lines $\ell_{200},\,\ell_{110},\,\ell_{101}$,  $P_{010}$ contains lines $\ell_{110},\,\ell_{020},\,\ell_{011}$,
$P_{001}$ contains lines $\ell_{101},\,\ell_{011},\,\ell_{002}$. Vectors $\vv_{abc}\in\ell_{abc}$  satisfy relations
$\vv_{101}=\vv_{200}+\vv_{110}, \vv_{020}=\vv_{101}+\vv_{011}, \vv_{011}=\vv_{110}+\vv_{002}$. \tcb{The bold broken line indicates a snake.}
}
\label{fi:triSL3v1}
\end{figure}

Any snake defines a projective basis $\vv_{\alpha_1},\dots,\vv_{\alpha_n}$ of ${\mathbb C}^n$. Note that choosing another corner of triangle as a top one leads to different choice of projective basis. In particular, if the basis defined by the only snake running from $\ell_{n00}$ to $\ell_{00n}$ is 
$\vv_{n00},\vv_{n-1\,00}\dots,\vv_{00n}$ then the basis defined by the only snake in the opposite direction from $\ell_{00n}$ to $\ell_{n00}$ is 
$\vv_{00n},-\vv_{10\,n-1}\dots,(-1)^{n-1}\vv_{n00}$.

Denote by $\bb_{\bf p}$ the basis defined by snake ${\bf p}$.
Let $\bb_{12}$ be the basis defined by the unique snake from $\ell_{n00}$ to $\ell_{0n0}$ in the triangle $\triangle 123$ and by $\bb_{31}$ 
the basis defined by the snake $\ell_{00n}$ to $\ell_{n00}$. The basis $\bb_{12}$ in Fig~\ref{fi:snakes} is $\bb_{12}=(\vv_{200}, \vv_{110},\vv_{020})$,
the basis $\bb_{31} =(\vv_{002},\vv_{101},\vv_{200})$. 

	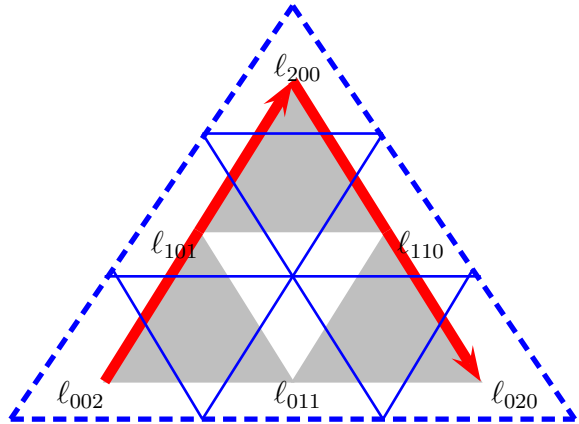
\begin{figure}[h]
\begin{pspicture}(-2,-2.5)(3,3.5){
%


\pspolygon[linecolor=lightgray,fillstyle=solid,fillcolor=lightgray](0.,2.5)(1.25,0.5)(-1.25,0.5)
\pspolygon[linecolor=lightgray,fillstyle=solid,fillcolor=lightgray](1.25,0.5)(2.5,-1.5)(0,-1.5)
\pspolygon[linecolor=lightgray,fillstyle=solid,fillcolor=lightgray](-1.25,0.5)(-2.5,-1.5)(0,-1.5)
\put(0,0){\psline[linecolor=red,linewidth=4pt]{-}(0,2.5)(1.25,0.5)}
\put(0,0){\psline[linecolor=red,linewidth=4pt]{->}(1.25,0.5)(2.5,-1.5)}

\put(0,0){\psline[linecolor=red,linewidth=4pt]{<-}(0,2.5)(-1.25,0.5)}
\put(0,0){\psline[linecolor=red,linewidth=4pt]{-}(-1.25,0.5)(-2.5,-1.5)}

\put(0,0){\psline[linecolor=blue,linewidth=1pt]{-}(-2.5,-0.1)(2.5,-0.1)}
\put(0,0){\psline[linecolor=blue,linewidth=1pt]{-}(-1.2,-2.)(-2.4,0)}
\put(0,0){\psline[linecolor=blue,linewidth=1pt]{-}(-1.2,1.8)(1.2,1.8)}
\put(0,0){\psline[linecolor=blue,linewidth=1pt]{-}(-1.2,-2)(1.2,1.8)}
\put(0,0){\psline[linecolor=blue,linewidth=1pt]{-}(1.2,-2)(-1.2,1.8)}
\put(0,0){\psline[linecolor=blue,linewidth=1pt]{-}(1.2,-2)(2.4,0)}
\put(0,0){\psline[linecolor=blue,linewidth=2pt,linestyle=dashed]{-}(-3.75,-2.)(0.0,3.5)}
\put(0,0){\psline[linecolor=blue,linewidth=2pt,linestyle=dashed]{-}(0.0,3.5)(3.75,-2)}
\put(0,0){\psline[linecolor=blue,linewidth=2pt,linestyle=dashed]{-}(-3.75,-2.)(3.75,-2)}

\put(0,2.5){\makebox(0,0)[bc]{\hbox{{ $\ell_{200}$}}}}
\put(1.25,0.5){\makebox(0,0)[tl]{\hbox{{ $\ell_{110}$}}}}
\put(-1.25,0.5){\makebox(0,0)[tr]{\hbox{{$\ell_{101}$}}}}
\put(2.5,-1.5){\makebox(0,0)[tl]{\hbox{{ $\ell_{020}$}}}}
\put(0,-1.5){\makebox(0,0)[tc]{\hbox{{ $\ell_{011}$}}}}
\put(-2.5,-1.5){\makebox(0,0)[tr]{\hbox{{ $\ell_{002}$}}}}

}
\end{pspicture}
\caption{\small
Configuration of lines corresponding to triple of flags in ${\mathbb C}^3$. Black triangles are equipped with planes $P_{abc}$.
Plane $P_{100}$ contains lines $\ell_{200},\,\ell_{110},\,\ell_{101}$,  $P_{010}$ contains lines $\ell_{110},\,\ell_{020},\,\ell_{011}$,
$P_{001}$ contains lines $\ell_{101},\,\ell_{011},\,\ell_{002}$. Vectors $\vv_{abc}\in\ell_{abc}$  satisfy relations
$\vv_{101}=\vv_{200}+\vv_{110}, \vv_{020}=\vv_{101}+\vv_{011}, \vv_{011}=\vv_{110}+\vv_{002}$.
}
\label{fi:snakes}
\end{figure}

 Define $T_1\in SL_n$ as the transformation matrix from basis $\bb_{31}$ to $\bb_{12}$, namely, $i$-th column of $T_1$ is $[(\bb_{31})_i]_{\bb_{12}}$, i.e. coordinate vector  $(\bb_{31})_i$ with respect to basis $\bb_{12}$ precomposed and postcomposed  with  multiplications by  diagonal matrices defined by the Fock-Goncharov coordinates on the sides $2-1$ and $1-3$ as in Formula~\ref{eq:T1}.  Since $\bb_{31}$ and $\bb_{12}$ are both defined up to multiplicative scalars
 $T_1$ is defined up to scalar too. Hence, condition $T_1\in SL_n$ fixes $T_1$ uniquely.
    
 Similarly, we define the transport matrix $T_2$ as transformation matrix from the side $1-2$ to the side $2-3$ and $T_3$ as transformation
 from $2-3$ to $3-1$.
 
 Finally, let $M_1=T_1$, $M_2=T_2^{-1}$ (see Fig.~\ref{fi:triSL3}).
  
 Matrix $M_1$ is an upper-anti-diagonal matrix and $M_2$ is a lower-anti-diagonal matrix (see Example~\ref{ex:toy}).

 Fock-Goncharov coordinates $Z_\alpha$ parametrize $\X_{SL_n,\Sigma}$. They are associated with vertices of triangular subdivision of $\Sigma$ except vertices $1,2,3$ of triangle and labelled  using barycentric indices $(i,j,k), i+j+k=n$ denoted often below by Greek letters (see Fig.~\ref{fi:triSL3} for $n=3$ and Fig.~\ref{fi:triangle} for $n=6$). 
 The expressions for classical transport matrices were first found in~\cite{FG3}, (see also \cite{Coman-Gabella-Teschner}  Appendix A.2).
 
  Initial diagonal prefactor can be vizualized as carrying the snake from outside to inside over the side $1-2$ of the triangle. The diagonal postfactor is vizualed as carrying over the snake from inside to outside across the side $1-3$.
 
 The middle factor of the transport matrix $T_1$ corresponding to the base change from $\bb_{31}$ to $\bb_{12}$ is factorizable in a product of elementary basis changes 
 corresponoding to the following sequence of snake transformations. 
 
 	\begin{figure}[H]
	\psscalebox{0.45}{
\begin{pspicture}(-12,-3)(13,3){
%

		\newcommand{\TRIANGLE}{%
			{\psset{unit=1}
				\pspolygon[linecolor=lightgray,fillstyle=solid,fillcolor=lightgray](0.,2.5)(1.25,0.5)(-1.25,0.5)
				\pspolygon[linecolor=lightgray,fillstyle=solid,fillcolor=lightgray](1.25,0.5)(2.5,-1.5)(0,-1.5)
				\pspolygon[linecolor=lightgray,fillstyle=solid,fillcolor=lightgray](-1.25,0.5)(-2.5,-1.5)(0,-1.5)

				\put(0,2.5){\makebox(0,0)[bc]{\hbox{{ $\ell_{200}$}}}}
				\put(1.25,0.5){\makebox(0,0)[tl]{\hbox{{ $\ell_{110}$}}}}
				\put(-1.25,0.5){\makebox(0,0)[tr]{\hbox{{$\ell_{101}$}}}}
				\put(2.5,-1.5){\makebox(0,0)[tl]{\hbox{{ $\ell_{020}$}}}}
				\put(0,-1.5){\makebox(0,0)[tc]{\hbox{{ $\ell_{011}$}}}}
				\put(-2.5,-1.5){\makebox(0,0)[tr]{\hbox{{ $\ell_{002}$}}}}
		}}
		\multiput(-15.,0.)(8,0.){5}{\TRIANGLE}
				\put(-15,0){\psline[linecolor=red,linewidth=4pt]{-}(0,2.5)(1.25,0.5)}
				\put(-15,0){\psline[linecolor=red,linewidth=4pt]{->}(1.25,0.5)(2.5,-1.5)}

				\put(-7,0){\psline[linecolor=red,linewidth=4pt]{-}(0,2.5)(1.25,0.5)}
				\put(-7,0){\psline[linecolor=red,linewidth=4pt]{->}(1.25,0.5)(0,-1.5)}

				\put(1,0){\psline[linecolor=red,linewidth=4pt]{-}(0,2.5)(-1.25,0.5)}
				\put(1,0){\psline[linecolor=red,linewidth=4pt]{->}(-1.25,0.5)(0,-1.5)}

				\put(9,0){\psline[linecolor=red,linewidth=4pt]{-}(0,2.5)(-1.25,0.5)}
				\put(9,0){\psline[linecolor=red,linewidth=4pt]{->}(-1.25,0.5)(-2.5,-1.5)}

				\put(17,0){\psline[linecolor=red,linewidth=4pt]{<-}(0,2.5)(-1.25,0.5)}
				\put(17,0){\psline[linecolor=red,linewidth=4pt]{-}(-1.25,0.5)(-2.5,-1.5)}
}
\end{pspicture}
}
\caption{\small
Sequence of snakes factorizing transport matrix $T_1$
}
\label{fi:snakes2}
\end{figure}
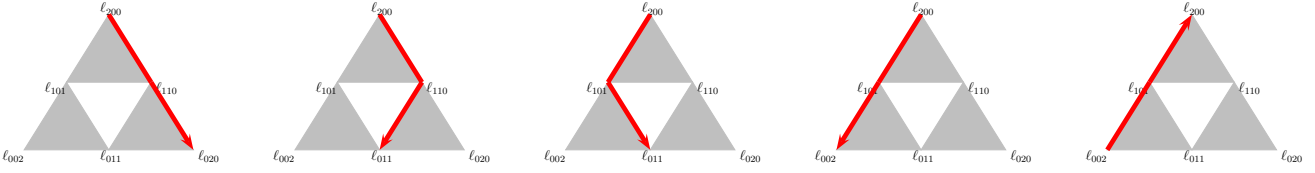

	\begin{figure}[h]
\psscalebox{0.8}{
\begin{pspicture}(-3,-3)(3,3.5){
%


\pspolygon[linecolor=lightgray,fillstyle=solid,fillcolor=lightgray](-1.25,2)(1.25,2)(0,0)
\pspolygon[linecolor=lightgray,fillstyle=solid,fillcolor=lightgray](1.25,-2)(2.5,0)(0,0)
\pspolygon[linecolor=lightgray,fillstyle=solid,fillcolor=lightgray](-1.25,-2)(-2.5,0)(0,0)

\put(0,0){\psline[linecolor=blue,linewidth=1pt]{-}(0.5,0)(2.,0)}
\put(0,0){\psline[linecolor=blue,linewidth=1pt]{-}(0.25,-0.4)(1.,-1.6)}
\put(0,0){\psline[linecolor=blue,linewidth=1pt]{-}(1.5,-1.6)(2.3,-0.4)}
\put(0,0){\psline[linecolor=blue,linewidth=1pt]{-}(-1.5,-1.6)(-2.3,-0.4)}
\put(0,0){\psline[linecolor=blue,linewidth=1pt]{-}(-0.25,-0.4)(-1.,-1.6)}
\put(0,0){\psline[linecolor=blue,linewidth=1pt]{-}(-0.5,0)(-2.,0)}
\put(0,0){\psline[linecolor=blue,linewidth=2pt,linestyle=dashed]{-}(-3.75,-2.)(0.0,3.5)}
\put(0,0){\psline[linecolor=blue,linewidth=2pt,linestyle=dashed]{-}(0.0,3.5)(3.75,-2)}
\put(0,0){\psline[linecolor=blue,linewidth=2pt,linestyle=dashed]{-}(-3.75,-2.)(3.75,-2)}
\put(0.1,3.0){\makebox(0,0)[br]{\hbox{{${\bf 1}$}}}}
\put(3.4,-1.8){\makebox(0,0)[br]{\hbox{{$\bf 2$}}}}
\put(-3.1,-1.8){\makebox(0,0)[br]{\hbox{{$\bf 3$}}}}
\put(0,0){\pscircle[linecolor=black,fillstyle=solid,fillcolor=white]{.5}}
\put(0.3,-0.1){\makebox(0,0)[br]{\hbox{{\tiny $Z_{111}$}}}}
\put(0,0){\psline[linecolor=blue,linewidth=1pt]{-}(0.25,0.4)(1.,1.6)}
\put(1.25,2){\pscircle[linecolor=black,fillstyle=solid,fillcolor=white]{.5}}
\put(1.55,1.9){\makebox(0,0)[br]{\hbox{{\tiny $Z_{012}$}}}}
\put(0,0){\psline[linecolor=blue,linewidth=1pt]{-}(-0.25,0.4)(-1.,1.6)}
\put(-1.25,2){\pscircle[linecolor=black,fillstyle=solid,fillcolor=white]{.5}}
\put(-0.95,1.9){\makebox(0,0)[br]{\hbox{{\tiny $Z_{102}$}}}}
\put(0,0){\psline[linecolor=blue,linewidth=1pt]{-}(0.75,2)(-.75,2)}
\put(2.5,0){\pscircle[linecolor=black,fillstyle=solid,fillcolor=white]{.5}}
\put(2.8,-.1){\makebox(0,0)[br]{\hbox{{\tiny $Z_{021}$}}}}
\put(-2.5,0){\pscircle[linecolor=black,fillstyle=solid,fillcolor=white]{.5}}
\put(-2.2,-.1){\makebox(0,0)[br]{\hbox{{\tiny $Z_{201}$}}}}
\put(1.25,-2){\pscircle[linecolor=black,fillstyle=solid,fillcolor=white]{.5}}
\put(1.55,-2.1){\makebox(0,0)[br]{\hbox{{\tiny $Z_{120}$}}}}
\put(-1.25,-2){\pscircle[linecolor=black,fillstyle=solid,fillcolor=white]{.5}}
\put(-.95,-2.1){\makebox(0,0)[br]{\hbox{{\tiny $Z_{210}$}}}}

\psarc[linewidth=6pt,linecolor=red]{<-}(-3.9,-2.0){3}{-20}{70}
\psarc[linewidth=6pt,linecolor=red]{<-}(0,3.7){3}{220}{320}
\psarc[linewidth=6pt,linecolor=red]{<-}(3.9,-2.0){3}{105}{200}
\put(-1.3,-1.2){\makebox(0,0)[br]{\hbox{{\Large $T_3$}}}}
\put(1.3,-1.2){\makebox(0,0)[bl]{\hbox{{\Large $T_2$}}}}
\put(0,1.){\makebox(0,0)[bc]{\hbox{{\Large $T_1$}}}}

}
\end{pspicture}
}
\caption{\small
Fock-Goncharov parameters for $\X_{SL_3,\Sigma_{0,1,3}}$. Arrows shows the direction of transport matrices.
$M_1=T_1$, $M_2=T_2^{-1}$.
}
\label{fi:triSL3}
\end{figure}
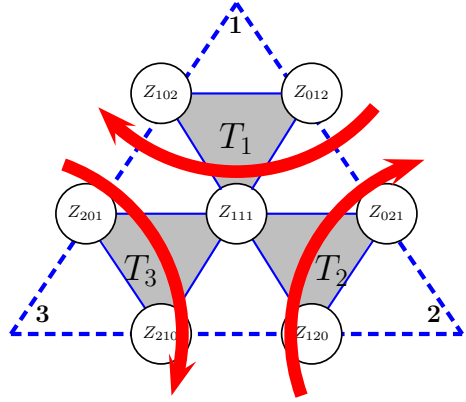
 
 
 Let $\chi(k)$ denote the integer step function, $\chi(k)=\{ 0, \  k<0,\ 1,\  k\ge 0\}$.
 Define $n\times n$ matrices $H_k(t)$, $L_k$, $S$ as follows $H_k(t)=t^{-\frac{n-k}{n}}{\operatorname{diag}}(t^{\chi(1-k-1)},t^{\chi(2-k-1)}, \dots, t^{\chi(n-k-1)}),$ 
 $L_k=\operatorname{Id}_n+E_{k}$ for $k\in [1,n-1]$ where $\operatorname{Id}_n$ is the identity $n\times n $ matrix, $(E_k)_{i,j}=\delta_{k+1,i}\cdot\delta_{k,j}$ is the matrix whose only nonzero element is $1$ at the position $(k+1,k)$, $(S)_{ij}=(-1)^{n-i}\delta_{i,n+1-j}$.
Then 
\begin{equation}\label{eq:T1}
\begin{split}
T_1 =&
S \Bigl[ \prod_{j=1}^{n-1} H_{n-j}(Z_{n-j,0,j})\Bigr]\\
&\cdot L_{n-1}\prod_{p=1}^{n-2} 
\Bigl[ \prod_{q=1}^{p} L_{n-q-1}H_{n-q}(Z_{n-1-p,n-q,p+q+1-n})\Bigr]
 L_{n-1}\\
& \cdot \Bigl[ \prod_{j=1}^{n-1} H_{j}(Z_{0,j,n-j})\Bigr].
\end{split}
 \end{equation}

Let ${\mathbb I}=\{(a,b,c)| a,b,c\in {\mathbb Z}_+, a+b+c=n\}$ be the set of barycentric indices in the triangle with side $n$,  $\tau:{\mathbb I}\to{\mathbb I}$ be the  clockwise rotation by $2\pi/3$, $\tau$ acts naturally on the sequences of barycentric parameters and hence on sequences of Fock-Goncharov parameters: for 
${\bf Z}=\left(Z_{\alpha_1},\dots,Z_{\alpha_k}\right)$ the sequence 
$\tau{\bf Z}=(Z_{\tau(\alpha_1)},\dots,Z_{\tau(\alpha_k)})$,  
if $O({\bf Z})$ is an object depending on the collection ${\bf Z}=(Z_{\alpha_i})_{i=1}^k$ of Fock-Goncharov parameters then $\tau O=O(\tau{\bf Z})$.

Note that $T_2=\tau T_1$, $T_3=\tau^2 T_1$ (see Fig.~\ref{fi:triSL3}).
The transport matrix $M_2=\left(\tau M_1\right)^{-1}$.

\begin{example}  For $n=3$, we have
 	$H_1(t)=\begin{pmatrix}
	t^{-2/3} & 0 & 0 \\
	0 & t^{1/3} & 0 \\
	0 & 0 & t^{1/3}
	\end{pmatrix}
	$,
	$H_2(t)=\begin{pmatrix}
	t^{-1/3} & 0 & 0 \\
	0 & t^{-1/3} & 0 \\
	0 & 0 & t^{2/3}
	\end{pmatrix}
	$,
	$L_1=\begin{pmatrix}
	1 & 0 & 0 \\
	1 & 1 & 0 \\
	0 & 0 & 1
	\end{pmatrix}
	$, 
	$L_2=\begin{pmatrix}
	1 & 0 & 0 \\
	0 & 1 & 0 \\
	0 & 1 & 1
	\end{pmatrix}
	$, $S=\begin{pmatrix}
	0 & 0 & 1 \\
	0 & -1 & 0 \\
	1 & 0 & 0
	\end{pmatrix}
	$.
	
Transport matrices $T_1$ from side $1-2$ to side $1-3$, $T_2$ from side $2-3$ to side $2-1$ and $T_3$ from side $3-1$ to side $3-2$ (see Fig.~\ref{fi:triSL3}) have the following form
\begin{eqnarray*}
M_1=T_1 &=& S H_2(Z_{201})H_1(Z_{102})L_2 L_1 H_2(Z_{111}) L_2 H_1(Z_{012})H_2(Z_{021})\\
          T_2&=& S H_2(Z_{012})H_1(Z_{021})L_2 L_1 H_2(Z_{111}) L_2 H_1(Z_{120})H_2(Z_{210})\\ 
          T_3&=& S H_2(Z_{120})H_1(Z_{210})L_2 L_1 H_2(Z_{111})L_2 H_1(Z_{201})H_2(Z_{102}).	
\end{eqnarray*}

	$M_1={\textstyle\begin{pmatrix}
	Z_{021}^{-1/3}Z_{102}^{1/3} Z_{111}^{-1/3}Z_{012}^{-2/3}Z_{201}^{2/3} & Z_{021}^{-1/3}Z_{102}^{1/3}(Z_{111}^{-1/3}+Z_{111}^{2/3}) Z_{102}^{1/3}  Z_{201}^{2/3} & Z_{021}^{2/3} Z_{102}^{1/3}Z_{111}^{2/3} Z_{012}^{1/3} Z_{201}^{2/3} \\
	Z_{021}^{-1/3}Z_{102}^{-1/3} Z_{111}^{-1/3}Z_{012}^{-2/3}Z_{201}^{-1/3} & Z_{021}^{-1/3}Z_{102}^{-1/3} Z_{111}^{-1/3}Z_{012}^{1/3}Z_{201}^{-1/3} & 0 \\
	Z_{021}^{-1/3}Z_{102}^{-2/3} Z_{111}^{-1/3}Z_{012}^{-2/3}Z_{201}^{-1/3} & 0 & 0
	\end{pmatrix}}
	$,
	\smallskip

Finally, $M_2=T_2^{-1}$. We can easily factorize $M_2$ in the product of elementary matrices noting that $S^{-1}=(-1)^{n-1}S$, 
$H_k(t)^{-1}=H_k(t^{-1})=SH_{n-k}(t)$S, $L_k^{-1}=\operatorname{Id}_n-E_k=SL_{n-k}^{\text{T}} S$, where $L_j^{\text{T}}$ is the transpose of matrix $L_j$.
Then, 
\begin{eqnarray*}
M_2 &=& H_2(Z_{210})^{-1}  H_1(Z_{120})^{-1} L_2^{-1}  H_2(Z_{111})^{-1}  L_1^{-1}  L_2 ^{-1} H_1(Z_{021})^{-1}  H_2(Z_{012})^{-1}   S ^{-1}\\
 &=& S H_1(Z_{210})S S H_2(Z_{120}) S  S L_1^{\text{T}} S S H_1(Z_{111}) S  S L_2^{\text{T}} S S L_1^{\text{T}} S S H_2(Z_{021}) S S H_1(Z_{012}) S S (-1)^{n-1}\\
 &=& (-1)^{n-1} S H_1(Z_{210}) H_2(Z_{120})  L_1^{\text{T}}  H_1(Z_{111})  L_2^{\text{T}}  L_1^{\text{T}}  H_2(Z_{021})  H_1(Z_{012}).
 \end{eqnarray*}

	 $M_2={\textstyle\begin{pmatrix}
	0 & 0 &  Z_{210}^{1/3} Z_{111}^{1/3} Z_{012}^{1/3} Z_{120}^{2/3} Z_{021}^{2/3}\\
	0 &  Z_{210}^{1/3} Z_{111}^{1/3} Z_{012}^{1/3}  Z_{120}^{-1/3} Z_{021}^{-1/3}&   
	Z_{210}^{1/3} Z_{111}^{1/3} Z_{012}^{1/3}  Z_{120}^{-1/3} Z_{021}^{2/3}\\
	Z_{210}^{-2/3} Z_{111}^{-2/3} Z_{012}^{-2/3}  Z_{120}^{-1/3} Z_{021}^{-1/3}&  
	Z_{210}^{-2/3} (Z_{111}^{-2/3}+Z_{111}^{1/3}) Z_{012}^{1/3}  Z_{120}^{-1/3} Z_{021}^{-1/3} &
	Z_{210}^{-2/3} Z_{111}^{1/3} Z_{012}^{1/3}  Z_{120}^{-1/3} Z_{021}^{2/3}
	\end{pmatrix}}
	$.

\end{example}

 To obtain quantum transport matrices we expand all entries of classical transport matrix $M_i$ in the sum of monomials $m_j(Z_\alpha)$ 
and replace all $m_j$ by the corresponding Weyl form $\col{m_j}$ \,. 
For instance, the $(1,2)$-entry of quantum $M_1$ becomes
$$\left(M_1\right)_{12}=\col{Z_{021}^{-1/3}Z_{102}^{1/3}Z_{111}^{-1/3}Z_{102}^{1/3}  Z_{201}^{2/3}} +
\col{Z_{021}^{-1/3}Z_{102}^{1/3}Z_{111}^{2/3} Z_{102}^{1/3}  Z_{201}^{2/3}}$$

In Section~\ref{sec:QFG} we generalize this construction to non-normalized quantum transport matrices defined for more general class of planar quivers.
 

\begin{example} \label{ex:toy}
A toy example is the one in which all $Z_{\tcb \alpha}$ are the units. Matrix entries then just count numbers of monomials entering the corresponding matrix elements $a_{i,j}\in (-1)^{i+1}\mathbb Z_{+} [[ Z_\alpha^{\pm 1}]]$. 
Then, for the $M_1$ matrix, we have the following representation:
\be
M_1=
\left(
  \begin{array}{rrr}
    1 & 2 & 1 \\
    -1 & -1 & 0 \\
    1 & 0 & 0 
  \end{array}
\right), \quad
\left(
  \begin{array}{rrrr}
    1 &3&3&1 \\
    -1 & -2 & -1&0 \\
    1 & 1 & 0&0 \\
    -1 & 0 & 0 &0
  \end{array}
\right), \hbox{etc},
\ee
that is, $[M_1]_{i,j}=(-1)^{i+1}\Bigl[ {n-i\atop j}\Bigr]$ for $SL_n$. We introduce the antidiagonal unit matrix $|S|=\delta_{i,n+1-i}$
(to distinguish it from $S=(-1)^{i+1}\delta_{i,n+1-i}$).

For $M_2$ we have
\be
M_2=M_1^2=
\left(
  \begin{array}{rrr}
     0 & 0 & 1\\
    0 & -1 & -1 \\
    1 & 2 & 1
  \end{array}
\right), \quad
\left(
  \begin{array}{rrrr}
    0 & 0 & 0 &1\\
    0 & 0 & -1&-1 \\
    0 & 1 & 2& 1 \\
    -1 &-3&-3&-1
   \end{array}
\right), \hbox{etc}
\ee
A riddle-thirsty reader can check the following relations between these matrices:
$$
M_1^2=M_2=(-1)^{n+1}|S|\cdot M_1\cdot |S|,\quad M_1^3=(-1)^{n+1}[|S|\cdot M_1]^2=(-1)^{n+1}I
$$
\be\label{eq:An}
M_1^{\text{T}}M_2=\mathbb A= \left(
  \begin{array}{rrr}
     1 & 3 & 3\\
    0 & 1 & 3 \\
    0 & 0 & 1
  \end{array}
\right), \quad
\left(
  \begin{array}{rrrr}
    1 & 4 & 6 &4\\
    0 & 1 & 4& 6 \\
    0 & 0 & 1& 4 \\
    0 & 0& 0& 1
   \end{array}
\right), \quad\hbox{etc}
\ee
that is $[\mathbb A]_{i,j}=\Bigl[ {n\atop j-i}\Bigr]$.
\end{example}

\subsection{Quantum transport matrices and Fock-Goncharov coordinates}\label{sec:QFG} We describe now how quantized transport matrices are expressed in terms of quantized Fock-Goncharov parameters (see also~
\cite{Coman-Gabella-Teschner} and~\cite{FG2}).

In the quantization of $\X_{SL_n,\Sigma}$ the quantized Fock-Goncharov variables form a quantum torus $\Upsilon$ with commutation relation described by the quiver shown on Fig.~\ref{fi:triangle}. Vertices of the quiver label quantum Fock-Goncharov coordinates $Z_\alpha$ (we use Greek letters to indicate barycentric labels) while 
the arrows  encode commutation relations:  if there are $m$ arrows from vertex $\alpha$ to $\beta$ then $Z_{\beta}Z_{\alpha}=q^{-2m} Z_{\alpha}Z_{\beta}$.
Dashed arrow counts as $m=1/2$.
In particular, a solid arrow from $Z_{\alpha}$ to $Z_{\beta}$ implies $Z_{\beta}Z_{\alpha}=q^{-2} Z_{\alpha}Z_{\beta}$,  a dashed arrow from $Z_{\alpha}$ to $Z_{\beta}$ implies $Z_{\beta}Z_{\alpha}=q^{-1}Z_{\alpha}Z_{\beta}$, and, for the future use, a double arrow from $Z_\alpha$ to $Z_\beta$ means 
$Z_{\beta}Z_{\alpha}=q^{-4}Z_{\alpha}Z_{\beta}$.  Vertices not connected by an arrow commute.

 \begin{figure}[H]
 \psscalebox{1.0}{
\begin{pspicture}(-3,-2)(3,3){
\newcommand{\PATGEN}{%
{\psset{unit=1}
\rput(0,0){\psline[linecolor=blue,linewidth=2pt]{->}(0,0)(.45,.765)}
\rput(0,0){\psline[linecolor=blue,linewidth=2pt]{->}(1,0)(0.1,0)}
\rput(0,0){\psline[linecolor=blue,linewidth=2pt]{->}(0,0)(.45,-.765)}
\put(0,0){\pscircle[fillstyle=solid,fillcolor=lightgray]{.1}}
}}
\newcommand{\PATLEFT}{%
{\psset{unit=1}
\rput(0,0){\psline[linecolor=blue,linewidth=2pt,linestyle=dashed]{->}(0,0)(.45,.765)}
\rput(0,0){\psline[linecolor=blue,linewidth=2pt]{->}(1,0)(0.1,0)}
\rput(0,0){\psline[linecolor=blue,linewidth=2pt]{->}(0,0)(.45,-.765)}
\put(0,0){\pscircle[fillstyle=solid,fillcolor=lightgray]{.1}}
}}
\newcommand{\PATRIGHT}{%
{\psset{unit=1}
\rput(0,0){\psline[linecolor=blue,linewidth=2pt,linestyle=dashed]{->}(0,0)(.45,-.765)}
\put(0,0){\pscircle[fillstyle=solid,fillcolor=lightgray]{.1}}
}}
\newcommand{\PATBOTTOM}{%
{\psset{unit=1}
\rput(0,0){\psline[linecolor=blue,linewidth=2pt]{->}(0,0)(.45,.765)}
\rput(0,0){\psline[linecolor=blue,linewidth=2pt,linestyle=dashed]{->}(1,0)(0.1,0)}
\put(0,0){\pscircle[fillstyle=solid,fillcolor=lightgray]{.1}}
}}
\newcommand{\PATTOP}{%
{\psset{unit=1}
\rput(0,0){\psline[linecolor=blue,linewidth=2pt]{->}(1,0)(0.1,0)}
\rput(0,0){\psline[linecolor=blue,linewidth=2pt]{->}(0,0)(.45,-.765)}
\put(0,0){\pscircle[fillstyle=solid,fillcolor=lightgray]{.1}}
}}
\newcommand{\PATBOTRIGHT}{%
{\psset{unit=1}
\rput(0,0){\psline[linecolor=blue,linewidth=2pt]{->}(0,0)(.45,.765)}
\put(0,0){\pscircle[fillstyle=solid,fillcolor=lightgray]{.1}}
\put(.5,0.85){\pscircle[fillstyle=solid,fillcolor=lightgray]{.1}}
}}
\multiput(-2.5,-0.85)(0.5,0.85){4}{\PATLEFT}
\multiput(-2,-1.7)(1,0){4}{\PATBOTTOM}
\put(-0.5,2.55){\PATTOP}
\multiput(-1.5,-0.85)(1,0){4}{\PATGEN}
\multiput(-1,0)(1,0){3}{\PATGEN}
\multiput(-.5,0.85)(1,0){2}{\PATGEN}
\put(0,1.7){\PATGEN}
\multiput(-1.5,-0.85)(1,0){4}{\PATGEN}
\multiput(0.5,2.55)(0.5,-0.85){4}{\PATRIGHT}
\put(2,-1.7){\PATBOTRIGHT}
\put(-2.7,-1.0){\makebox(0,0)[br]{\hbox{{\tiny $(5,0,1)$}}}}
\put(-2.3,-.15){\makebox(0,0)[br]{\hbox{{\tiny (4,0,2)$$}}}}
\put(-.7,2.4){\makebox(0,0)[br]{\hbox{{\tiny $(1,0,5)$}}}}
\put(-2.,-2.1){\makebox(0,0)[br]{\hbox{{\tiny (5,1,0)$$}}}}
\put(-0.6,-2.1){\makebox(0,0)[br]{\hbox{{\tiny $(4,2,0)$}}}}
\put(2.4,-2.1){\makebox(0,0)[br]{\hbox{{\tiny $(1,4,0)$}}}}
\put(-0.6,-2.1){\makebox(0,0)[br]{\hbox{{\tiny $(4,2,0)$}}}}
\put(2.4,-2.1){\makebox(0,0)[br]{\hbox{{\tiny $(1,4,0)$}}}}
}
\end{pspicture}
}
\bigskip
\caption{\small
The quiver of Fock-Goncharov parameters in the triangle $\Sigma_{0,1,3}$  parametrizing $\A_{SL_6,\Sigma}$; note that we vertices 
$(6,0,0)$, $(0,6,0)$, and $(0,0,6)$ are excluded.
}
\label{fi:triangle}
\end{figure}
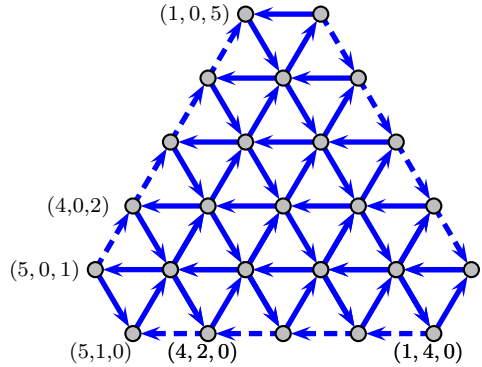


Consider the following planar oriented graph in the disk dual to the quiver above. Label vertices on the left, on the right and on the bottom sides from $1$ to $n$ as shown on Figure~\ref{fi:plab}. Now barycentric indices label the vertices of the quiver which correspond to the faces of the dual oriented graph. Vertices of the new dual graph are colored black and white depending on whether there are two or one incoming arrows.
Faces of $G$ are equipped with $q$-commuting weights $Z_\alpha$.
We add also three face weights $Z_{n,0,0}$, $Z_{0,n,0}$, and $Z_{0,0,n}$ with similar commutation relations. 

Any maximal oriented path in the dual graph connects a vertex on the right side 1--2 of the triangle either with a vertex of the left side 1--3 or with a vertex on the bottom side 2--3.  
We assign to every oriented path $P:j\leadsto i'$ from the right side to the left side or to the bottom side $P:j\leadsto i''$ the \emph{quantum weight} 
$\displaystyle{w(P)=\col{\mathop{\prod_{\text{ face }\alpha\text{ lies to the right}}}_{\text{ of the path } P} Z_\alpha}}$

\begin{figure}[h]
\psscalebox{0.9}{
	\begin{pspicture}(-3,-3)(4,4){
		\newcommand{\LEFTDOWNARROW}{%
			{\psset{unit=1}
				\rput(0,0){\psline[linecolor=black,linewidth=2pt]{<-}(0,0)(.765,.45)}
		}}
		\newcommand{\DOWNARROW}{%
	{\psset{unit=1}
					\rput(0,0){\psline[linecolor=black,linewidth=2pt]{->}(0,0.1)(0,-0.566)}
		\put(0,0){\pscircle[fillstyle=solid,fillcolor=lightgray]{.15}}
}}
		\newcommand{\LEFTUPARROW}{%
	{\psset{unit=1}
		\rput(0,0){\psline[linecolor=black,linewidth=2pt]{->}(0,0)(-.765,.45)}
}}
	\newcommand{\STARUP}{
			{\psset{unit=1}
	\rput(0,0){\psline[linecolor=black,linewidth=2pt]{<-}(0,0)(.5,-.26)}
	\rput(0,0){\psline[linecolor=black,linewidth=2pt]{<-}(0,0.1)(0,.466)}
	\rput(0,0){\psline[linecolor=black,linewidth=2pt]{->}(0,0)(-.5,-.26)}
	\put(0,0){\pscircle[fillstyle=solid,fillcolor=black]{.15}}
	\put(0,.566){\pscircle[fillstyle=solid,fillcolor=lightgray]{.15}}
}}
		\newcommand{\PATGEN}{%
			{\psset{unit=1}
				\rput(0,0){\psline[linecolor=blue,linewidth=2pt]{->}(0,0)(.45,.765)}
				\rput(0,0){\psline[linecolor=blue,linewidth=2pt]{->}(1,0)(0.1,0)}
				\rput(0,0){\psline[linecolor=blue,linewidth=2pt]{->}(0,0)(.45,-.765)}
				\put(0,0){\pscircle[fillstyle=solid,fillcolor=lightgray]{.1}}
		}}
		\newcommand{\PATLEFT}{%
			{\psset{unit=1}
				\rput(0,0){\psline[linecolor=blue,linewidth=2pt,linestyle=dashed]{->}(0,0)(.45,.765)}
				\rput(0,0){\psline[linecolor=blue,linewidth=2pt]{->}(1,0)(0.1,0)}
				\rput(0,0){\psline[linecolor=blue,linewidth=2pt]{->}(0,0)(.45,-.765)}
				\put(0,0){\pscircle[fillstyle=solid,fillcolor=lightgray]{.1}}
		}}
		\newcommand{\PATRIGHT}{%
			{\psset{unit=1}
				\rput(0,0){\psline[linecolor=blue,linewidth=2pt,linestyle=dashed]{->}(0,0)(.45,-.765)}
				\put(0,0){\pscircle[fillstyle=solid,fillcolor=lightgray]{.1}}
		}}
		\newcommand{\PATBOTTOM}{%
			{\psset{unit=1}
				\rput(0,0){\psline[linecolor=blue,linewidth=2pt]{->}(0,0)(.45,.765)}
				\rput(0,0){\psline[linecolor=blue,linewidth=2pt,linestyle=dashed]{->}(1,0)(0.1,0)}
				\put(0,0){\pscircle[fillstyle=solid,fillcolor=lightgray]{.1}}
		}}
		\newcommand{\PATTOP}{%
			{\psset{unit=1}
				\rput(0,0){\psline[linecolor=blue,linewidth=2pt]{->}(1,0)(0.1,0)}
				\rput(0,0){\psline[linecolor=blue,linewidth=2pt]{->}(0,0)(.45,-.765)}
				\put(0,0){\pscircle[fillstyle=solid,fillcolor=lightgray]{.1}}
		}}
		\newcommand{\PATBOTRIGHT}{%
			{\psset{unit=1}
				\rput(0,0){\psline[linecolor=blue,linewidth=2pt]{->}(0,0)(.45,.765)}
				\put(0,0){\pscircle[fillstyle=solid,fillcolor=lightgray]{.1}}
				\put(.5,0.85){\pscircle[fillstyle=solid,fillcolor=lightgray]{.1}}
		}}
		\multiput(-2.5,-0.85)(0.5,0.85){4}{\PATLEFT}
		\multiput(-2,-1.7)(1,0){4}{\PATBOTTOM}
		\multiput(-2,-1.176)(1.0,0){5}{\STARUP}
		\put(-0.5,2.55){\PATTOP}
		\multiput(-1.5,-0.85)(1,0){4}{\PATGEN}
		\multiput(-1.5,-0.335)(1.0,0){4}{\STARUP}
		\multiput(-1,0)(1,0){3}{\PATGEN}
		\multiput(-1.0,0.5)(1.0,0){3}{\STARUP}
		\multiput(-.5,0.85)(1,0){2}{\PATGEN}
		\multiput(-.5,1.4)(1.0,0){2}{\STARUP}
		\put(0,1.7){\PATGEN}
		\put(0,2.3){\STARUP}
		\multiput(-1.5,-0.85)(1,0){4}{\PATGEN}
		\multiput(0.5,2.55)(0.5,-0.85){4}{\PATRIGHT}
		\put(2,-1.7){\PATBOTRIGHT}
		\multiput(2.6,-1.4)(-0.5,.85){6}{\LEFTDOWNARROW}
		\multiput(-2.6,-1.4)(0.5,.85){6}{\LEFTUPARROW}
		\multiput(-2.5,-1.5)(1.0,0){6}{\DOWNARROW}
		\put(1.2,3.2){\makebox(0,0)[br]{\hbox{{$1$}}}}
		\put(1.7,2.4){\makebox(0,0)[br]{\hbox{{$2$}}}}
		\put(2.2,1.6){\makebox(0,0)[br]{\hbox{{$3$}}}}
		\put(2.7,0.8){\makebox(0,0)[br]{\hbox{{$4$}}}}
		\put(3.2,-0.1){\makebox(0,0)[br]{\hbox{{$5$}}}}
		\put(3.7,-1.0){\makebox(0,0)[br]{\hbox{{$6$}}}}

		\put(-1.2,3.2){\makebox(0,0)[br]{\hbox{{$1'$}}}}
		\put(-1.7,2.4){\makebox(0,0)[br]{\hbox{{$2'$}}}}
		\put(-2.2,1.6){\makebox(0,0)[br]{\hbox{{$3'$}}}}
		\put(-2.7,0.8){\makebox(0,0)[br]{\hbox{{$4'$}}}}
		\put(-3.2,-0.1){\makebox(0,0)[br]{\hbox{{$5'$}}}}
		\put(-3.7,-1.0){\makebox(0,0)[br]{\hbox{{$6'$}}}}

		\put(-2.4,-2.6){\makebox(0,0)[br]{\hbox{{$1''$}}}}
		\put(-1.4,-2.6){\makebox(0,0)[br]{\hbox{{$2''$}}}}
		\put(-0.4,-2.6){\makebox(0,0)[br]{\hbox{{$3''$}}}}
		\put(0.6,-2.6){\makebox(0,0)[br]{\hbox{{$4''$}}}}
		\put(1.6,-2.6){\makebox(0,0)[br]{\hbox{{$5''$}}}}
		\put(2.6,-2.6){\makebox(0,0)[br]{\hbox{{$6''$}}}}
\put(-2.7,-1.8){\makebox(0,0)[br]{\hbox{{\tiny{\color{red} $Z_{6,0,0}$}}}}}
\put(2.7,-1.8){\makebox(0,0)[bl]{\hbox{{\tiny{\color{red} $Z_{0,6,0}$}}}}}
\put(0,3.2){\makebox(0,0)[bc]{\hbox{{\tiny{\color{red} $Z_{0,0,6}$}}}}}
	}
	\end{pspicture}
}
	\caption{\small
		The plabic graph $G$ dual to the quiver of Fock-Goncharov parameters for $\X_{SL_6,\Sigma_{0,1,3}}$. Face weights $Z_{6,0,0}, Z_{0,6,0}, Z_{0,0,6}$ are added.
	}
	\label{fi:plab}
\end{figure}
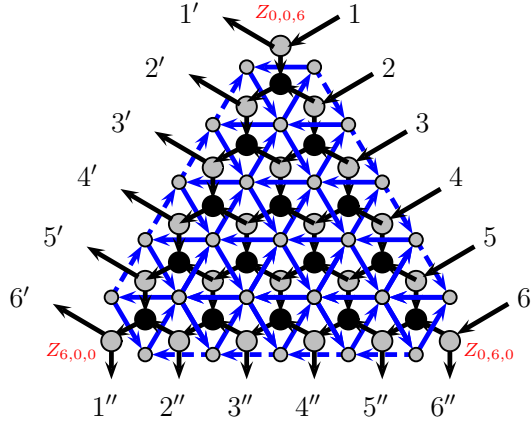


\begin{definition}
We define two $n\times n$ \emph{non-normalized quantum transition matrices} \newline
$\displaystyle{
(\mathcal M_1)_{i,j}=\mathop{\sum_{\text{directed path\,}{P: j\leadsto i'}}}_{\text{ from right to left}} w(P)}$\qquad and\qquad
$\displaystyle{
(\mathcal M_2)_{i,j}=\mathop{\sum_{\text{directed path\,}{P: j\leadsto i''}}}_{\text{ from right to bottom}} w(P).}$\newline 
Note that each ${\mathcal M}_1$ is a lower-triangular matrix and ${\mathcal M}_2$ is an upper-triangular matrix.
\end{definition}

In section \ref{sec:QuantumMeasurements} we generalize this definition. Let $\Gamma$ be a planar oriented graph
in the rectangle with no sources or sinks inside (see Fig~\ref{fig:network}),  $m$ univalent boundary sinks on the left labeled $1$ to $m$ top to bottom and  $n$ univalent boundary sources on the right labelled $1$ to $n$ top to bottom. All arcs of  $\Gamma$ are oriented right to left, in particular, $G$ has no oriented cycles. Note that this condition is in particular satisfied by the plabic graph $G$ (see Fig~\ref{fi:plab}) considered as a graph with $n$ sources and $2n$ sinks. Indeed, we can redraw $G$ in a rectangle such that the right side of the triangle becomes the right vertical side of the rectangle while union of the left and the bottom sides becomes the left side of the triangle. 

Faces of $\Gamma$ are equipped with $q$-commuting weights $Z_\alpha$ whose commutation relations are governed by the plabic graph (see Section~\ref{sec:QuantumMeasurements} for details).
  We define weight of the maximal oriented path $P$ from a source $a$ to a sink $b$ as 
  \begin{equation}\label{eq:pathweight}
  \displaystyle{w(P)\- =\- \col{\mathop{\prod_{\text{ face }\alpha\text{ lies to the right }}}_{\text{ of the path } P} Z_\alpha}}.
  \end{equation} 
 Then, entries of a
$m\times n$ \emph{non-normalized transport matrix} are  $\displaystyle{[\M]_{ij}=\mathop{\sum_{\text{directed path\,}{P: j\leadsto i}}} w(P).}$ 
Lemma~\ref{lem:r-matrix} implies that the matrix $\mathcal M$ satisfies the quantum $R$-matrix relation $\mathcal R_{m}(q) \sheet{1}{\mathcal M}\otimes \sheet{2}{\mathcal M}=\sheet{2}{\mathcal M}\otimes \sheet{1}{\mathcal M}\mathcal R_n(q)$, where $\mathcal R_{k}(q)$ is a $k^2\times k^2$ matrix
\be\label{R-matrix1}
\mathcal R_{k}(q)= \sum_{1\le i,j\le k} \sheet{1}e_{ii}\otimes \sheet{2} e_{jj}+(q-1)\sum_{1\le i\le k} \sheet{1}e_{ii}\otimes \sheet{2} e_{ii}
+(q-q^{-1})\sum_{1\le j<i \le k} \sheet{1}e_{ij}\otimes \sheet{2} e_{ji}
\ee
Note that $\mathcal R_k$ has the following properties:
\be\label{R-rel}
\mathcal R^{-1}_k(q)=\mathcal R_k(q^{-1}),\qquad \mathcal R_k(q)-\mathcal R^{\text{T}}_k(q^{-1})=(q-q^{-1})P_k,
\ee
where $P_k$ is the standard permutation matrix $P_k:=\sum_{1\le i,j\le k}\sheet{1}e_{ij}\otimes \sheet{2} e_{ji}$. Note that the total transposition of $\mathcal R_{k}(q)$ results in interchanging the space labels $1\leftrightarrow 2$. In Sec.~\ref{s:cycles} we show that this algebra remains valid also in the case of a planar directed network with loops.

In Fig.~\ref{fi:plab} we have an example of a directed network with $6$ sources
and $12$ sinks.
Adding face weights $Z_{600},\ Z_{060},\ Z_{006}$ supplied with ``natural'' commutation relations $Z_{600}Z_{501}=qZ_{501}Z_{600}$ and $Z_{510}Z_{600}=qZ_{600}Z_{510}$, etc., we obtain graph $\Gamma$ in the rectangle with $n=6$ sources and $2n=12$ sinks; then  $\mathcal M$ has a block matrix form $\mathcal M=\begin{pmatrix} \mathcal M_1 \\ \mathcal M_2\end{pmatrix}$ in which we let $\mathcal M_1$ is the upper $n\times n$ block  and $\mathcal M_2$ is the lower $n\times n$ block. We want to show that Lemma~\ref{lem:r-matrix} implies the following commutation relations for $\mathcal M_1$ and $\mathcal M_2$:
\be\label{R1}
\mathcal R_n(q)\sheet{1}{\mathcal M}_i \otimes \sheet{2}{\mathcal M}_i=\sheet{2}{\mathcal M}_i\otimes \sheet{1}{\mathcal M}_i \mathcal R_{n}(q),\quad i=1,2,
\ee
and
\be\label{R2}
\sheet{1}{\mathcal M}_1 \otimes \sheet{2}{\mathcal M}_2=\sheet{2}{\mathcal M}_2\otimes \sheet{1}{\mathcal M}_1 \mathcal R_{n}(q).
\ee
Let now indices $i,j$ run from $1$ to $n$. We rewrite the above matrix $\mathcal R_{2n}(q)$ as
\begin{align*}
\mathcal R_{2n}(q)&= \sum_{1\le i,j\le n} \sheet{1}e_{ii}\otimes \sheet{2} e_{jj}+(q-1)\sum_{1\le i\le n} \sheet{1}e_{ii}\otimes \sheet{2} e_{ii}
+(q-q^{-1})\sum_{1\le j<i\le n} \sheet{1}e_{ij}\otimes \sheet{2} e_{ji}\\
&+\sum_{1\le i,j\le n} \sheet{1}e_{n+i,n+i}\otimes \sheet{2} e_{n+j,n+j}+(q-1)\sum_{1\le i\le n} \sheet{1}e_{n+i,n+i}\otimes \sheet{2} e_{n+i,n+i}\\
& \hskip 8cm +(q-q^{-1})\sum_{1\le j<i \le n} \sheet{1}e_{n+i,n+j}\otimes \sheet{2} e_{n+j,n+i}\\
&+\sum_{1\le i,j\le n} \sheet{1}e_{i,i}\otimes \sheet{2} e_{n+j,n+j}\\
&+\sum_{1\le i,j\le n} \sheet{1}e_{n+i,n+i}\otimes \sheet{2} e_{j,j}
+(q-q^{-1})\sum_{1\le i,j\le n} \sheet{1}e_{n+i,j}\otimes \sheet{2} e_{j,n+i}
\end{align*}
In the first two lines we immediately recognize $\mathcal R_n(q)$ in two diagonal $n\times n$ blocks of $\mathcal R_{2n}(q)$: the relations for the pair of first indices $(i,j)$ and $(n+i,n+j)$ generate (\ref{R1}) for $\mathcal M_1$ and $\mathcal M_2$ respectively; setting $(i,n+j)$, which corresponds to the third line, we obtain just a unit matrix in the left-hand side thus producing relation (\ref{R2}), whereas in the case $(n+i,j)$ (the fourth line), we have the equation
$$
\sheet{1}{\mathcal M}_2 \otimes \sheet{2}{\mathcal M}_1+(q-q^{-1})P_n \sheet{1}{\mathcal M}_1 \otimes \sheet{2}{\mathcal M}_2=\sheet{2}{\mathcal M}_1 \otimes \sheet{1}{\mathcal M}_2\mathcal R_n(q).
$$
We first push the permutation operator $P_n$ through the $\mathcal M$-matrix product interchanging the labels of spaces in the direct product and then use
the identity $(q-q^{-1})P_n=\mathcal R_n(q)-\mathcal R^{\text{T}}_n(q^{-1})$ obtaining
$$
\sheet{1}{\mathcal M}_2 \otimes \sheet{2}{\mathcal M}_1+\sheet{2}{\mathcal M}_1 \otimes \sheet{1}{\mathcal M}_2(\mathcal R_n(q)-\mathcal R^{\text{T}}_n(q^{-1})) =\sheet{2}{\mathcal M}_1 \otimes \sheet{1}{\mathcal M}_2\mathcal R_n(q),
$$
or
$$
\sheet{1}{\mathcal M}_2 \otimes \sheet{2}{\mathcal M}_1=\sheet{2}{\mathcal M}_1 \otimes \sheet{1}{\mathcal M}_2 \mathcal R^{\text{T}}_n(q^{-1}),
$$
which is just another form of writing relation (\ref{R2}).
This accomplishes the proof of relations ~\ref{R1} and~\ref{R2}.
\qed

Now, in order to eliminate extra variables $Z_{n,0,0},\ Z_{0,n,0}$ and $Z_{0,0,n}$ we normalize the matrices $\M_i$, $i=1,2$ by multiplying them by corresponding special functions. Namely, we multiply the matrix ${\mathcal M}_1$ by $D_1^{-1}$, where $D_1$ is the function (\ref{D-Casimir}) commuting with all elements of ${\mathcal M}_1$. Clearly, $D_1^{-1} {\mathcal M}_1$ will be independent of $Z_{0,0,n}$. Similarly, we multiply ${\mathcal M}_2$ by $\col{D_1^{-1}D_2^{-1}}$ where $D_2=\tau^2D_1$ is the similar element that starts with the variable $Z_{0,n,0}$. These multiplications preserve the form of relations (\ref{R1}) and their only effect on  (\ref{R1}) is the appearance of the constant factor in the $R$-matrix in the right-hand side (this is because $D_1$ commutes with 
$\col{D_1^{-1}D_2^{-1}}\mathcal M_2$, and $\col{D_1^{-1}D_2^{-1}}$ commutes with $D_1^{-1}\mathcal M_2$; only $D_1$ and $D_2$ do not commute.

We now define the \emph{quantum transport} matrices of the \emph{standard $SL_n$ quiver}:

\begin{definition}\label{def:M}
Transport matrices for the  quantum space $\A_{SL_n,\Sigma}$ are defined as follows 
$$M_1=QS{\mathcal M}_1 D_1^{-1}\text{ and } M_2=QS{\mathcal M}_2 D_1^{-1}D_2^{-1}$$ where $Q=\hbox{diag}\{ q^{\frac{1}{2}-j}\}_{j=[1,n]}$.
\end{definition}

Note that 
$$
\sheet{1}Q\otimes \sheet{2}Q \mathcal R_n(q)=\mathcal R_n(q)\sheet{1}Q\otimes \sheet{2}Q\  \hbox{for any diagonal matrix $Q$} 
$$
and 
$$
\sheet{1}S\otimes \sheet{2}S \mathcal R_n(q)=\mathcal R^{\text{T}}_n(q)\sheet{1}S\otimes \sheet{2}S\  \hbox{for any antidiagonal matrix $S$}. 
$$


We have therefore proved the following theorem.

\begin{theorem}\label{th:MM}
The above $M_1$ and $M_2$ satisfy the relations
\begin{align*}
R^{\text{T}}_{n}(q)\sheet{1}M_i \otimes \sheet{2} M_i &=\sheet{2}M_i\otimes \sheet{1}M_i R_{n}(q),\quad i=1,2, \\
\sheet{1}M_1 \otimes \sheet{2} M_2&=\sheet{2}M_2\otimes \sheet{1}M_1 R_{n}(q)
\end{align*}
where
\be\label{R-matrix}
R_{n}(q)=q^{-1/n}\left[ \sum_{i,j} \sheet{1}e_{ii}\otimes \sheet{2} e_{jj}+(q-1)\sum_i \sheet{1}e_{ii}\otimes \sheet{2} e_{ii}
+(q-q^{-1})\sum_{i>j} \sheet{1}e_{ij}\otimes \sheet{2} e_{ji}\right]
\ee
is the quantum trigonometric $R$-matrix
\end{theorem}

\begin{theorem}\label{th:groupoid}
The quantum transport matrices $T_i$ satisfy the quantum groupoid relation
$$T_1 T_2 T_3=\operatorname{Id}.$$
\end{theorem}

\begin{remark}
Recalling $M_1=T_1$, $M_2=T_2^{-1}, M_3=T_3$ we have
$$
M_3 M_1=M_2.
$$
\end{remark}

{\bf Proof.} 
The product $T_3T_1$ is given by the following double sum over directed paths:
\be\label{MM}
[(QS)^{-1}T_3T_1]_{ij}=\sum_{k=1}^n (-1)^k q^{\frac{1}{2}-k} \sum_{\text{paths\,}k\to i}:\prod Z_\alpha: \sum_{\text{paths\,}j\to k}:\prod Z_\beta: .
\ee


We now consider the pattern in the figure below. We do not indicate arrows on edges recalling that all paths in 
$T_1$ 
go from right to left and from top to bottom whereas all paths in 
$T_3$ 
go from left to right and from top to bottom. Two paths: $j\to k$ from 
$T_1$ 
and $k\to i$ from 
$T_3$ 
share the common horizontal leg; if we remove this leg then the remaining part of the union of $j\to k$ and $k \to i$ is a path that first goes from right to left, then (in a general Case I) has the leftmost vertical edge, then goes from left to right. In a very special Case II, the path does not have the last part; this happens only for $k=1$ and only if the last horizontal part of the path $j\to k=1$ is strictly longer than the shared horizontal leg
$$
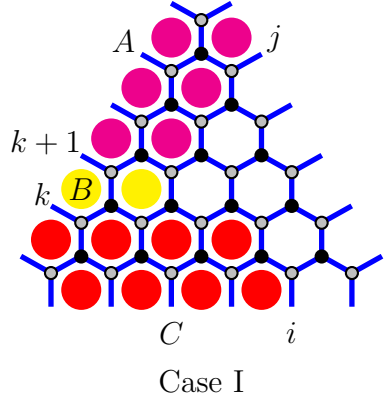
\begin{pspicture}(-2.7,-3)(2.7,3){\psset{unit=0.8}
\newcommand{\PATGEN}{%
{\psset{unit=1}
\rput(0,0){\psline[linecolor=blue,linewidth=2pt](0,0)(.5,0.28)}
\rput(0,0){\psline[linecolor=blue,linewidth=2pt](0,0)(-0.5,0.28)}
\rput(0,0){\psline[linecolor=blue,linewidth=2pt](0,0)(0,-0.56)}
\put(0,0){\pscircle[fillstyle=solid,fillcolor=lightgray]{.1}}
\put(0,-0.56){\pscircle[fillstyle=solid,fillcolor=black]{.1}}
}}
\newcommand{\PATBOT}{%
{\psset{unit=1}
\rput(0,0){\psline[linecolor=blue,linewidth=2pt](0,0)(.5,0.28)}
\rput(0,0){\psline[linecolor=blue,linewidth=2pt](0,0)(-0.5,0.28)}
\rput(0,0){\psline[linecolor=blue,linewidth=2pt](0,0)(0,-0.56)}
\put(0,0){\pscircle[fillstyle=solid,fillcolor=lightgray]{.1}}
}}
\multiput(-2.5,-1.68)(1,0){6}{\PATBOT}
\multiput(-2,-0.84)(1,0){5}{\PATGEN}
\multiput(-1.5,0)(1,0){4}{\PATGEN}
\multiput(-1,0.84)(1,0){3}{\PATGEN}
\multiput(-0.5,1.68)(1,0){2}{\PATGEN}
\put(0,2.53){\PATGEN}
\multiput(-0.5,2.24)(1,0){2}{\pscircle[fillstyle=solid,fillcolor=magenta,linecolor=white]{.35}}
\multiput(-1,1.4)(1,0){2}{\pscircle[fillstyle=solid,fillcolor=magenta,linecolor=white]{.35}}
\multiput(-1.5,0.56)(1,0){2}{\pscircle[fillstyle=solid,fillcolor=magenta,linecolor=white]{.35}}
\multiput(-2,-0.28)(1,0){2}{\pscircle[fillstyle=solid,fillcolor=yellow,linecolor=white]{.35}}
\multiput(-2.5,-1.12)(1,0){4}{\pscircle[fillstyle=solid,fillcolor=red,linecolor=white]{.35}}
\multiput(-2,-1.96)(1,0){4}{\pscircle[fillstyle=solid,fillcolor=red,linecolor=white]{.35}}
\put(-1.1,2){\makebox(0,0)[br]{\hbox{{$A$}}}}
\put(-2,-0.28){\makebox(0,0)[cc]{\hbox{{$B$}}}}
\put(-0.5,-2.5){\makebox(0,0)[tc]{\hbox{{$C$}}}}
\put(1.1,2){\makebox(0,0)[bl]{\hbox{{$j$}}}}
\put(-2,0.3){\makebox(0,0)[br]{\hbox{{$k+1$}}}}
\put(-2.5,-0.54){\makebox(0,0)[br]{\hbox{{$k$}}}}
\put(1.5,-2.5){\makebox(0,0)[tc]{\hbox{{$i$}}}}
\put(0,-3.3){\makebox(0,0)[tc]{\hbox{{Case I}}}}
}
\end{pspicture}
$$
In Case I, given a path $j\to k$ encompassing the regions $A$ and $B$ and a path $k\to i$ encompassing the region $C$ we have the corresponding path $j\to k+1$ encompassing the regions $A$  and the path $k+1\to i$ encompassing the regions $B$ and $C$. These pairs of paths are in bijection being the only two possible combinations of paths having the same union of domains $A\cup B\cup C$. Contributions from these two pairs of paths have opposite signs and since $\col{CB}\,\col{A}=q\col{C}\,\col{BA}\ $ they are mutually canceled in the sum (\ref{MM}).

The only pairs of paths ($j\to k$, $k\to i$)  that do not have counterparts are those for which the region $C$ is absent (Case II):
$$
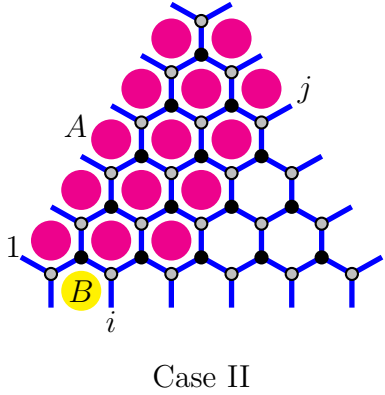
\begin{pspicture}(-2.7,-3)(2.7,3){\psset{unit=0.8}
\newcommand{\PATGEN}{%
{\psset{unit=1}
\rput(0,0){\psline[linecolor=blue,linewidth=2pt](0,0)(.5,0.28)}
\rput(0,0){\psline[linecolor=blue,linewidth=2pt](0,0)(-0.5,0.28)}
\rput(0,0){\psline[linecolor=blue,linewidth=2pt](0,0)(0,-0.56)}
\put(0,0){\pscircle[fillstyle=solid,fillcolor=lightgray]{.1}}
\put(0,-0.56){\pscircle[fillstyle=solid,fillcolor=black]{.1}}
}}
\newcommand{\PATBOT}{%
{\psset{unit=1}
\rput(0,0){\psline[linecolor=blue,linewidth=2pt](0,0)(.5,0.28)}
\rput(0,0){\psline[linecolor=blue,linewidth=2pt](0,0)(-0.5,0.28)}
\rput(0,0){\psline[linecolor=blue,linewidth=2pt](0,0)(0,-0.56)}
\put(0,0){\pscircle[fillstyle=solid,fillcolor=lightgray]{.1}}
}}
\multiput(-2.5,-1.68)(1,0){6}{\PATBOT}
\multiput(-2,-0.84)(1,0){5}{\PATGEN}
\multiput(-1.5,0)(1,0){4}{\PATGEN}
\multiput(-1,0.84)(1,0){3}{\PATGEN}
\multiput(-0.5,1.68)(1,0){2}{\PATGEN}
\put(0,2.53){\PATGEN}
\multiput(-0.5,2.24)(1,0){2}{\pscircle[fillstyle=solid,fillcolor=magenta,linecolor=white]{.35}}
\multiput(-1,1.4)(1,0){3}{\pscircle[fillstyle=solid,fillcolor=magenta,linecolor=white]{.35}}
\multiput(-1.5,0.56)(1,0){3}{\pscircle[fillstyle=solid,fillcolor=magenta,linecolor=white]{.35}}
\multiput(-2,-0.28)(1,0){3}{\pscircle[fillstyle=solid,fillcolor=magenta,linecolor=white]{.35}}
\multiput(-2.5,-1.12)(1,0){3}{\pscircle[fillstyle=solid,fillcolor=magenta,linecolor=white]{.35}}
\multiput(-2,-1.96)(1,0){1}{\pscircle[fillstyle=solid,fillcolor=yellow,,linecolor=white]{.35}}
\put(-1.9,0.6){\makebox(0,0)[br]{\hbox{{$A$}}}}
\put(-2,-1.96){\makebox(0,0)[cc]{\hbox{{$B$}}}}
\put(1.6,1.16){\makebox(0,0)[bl]{\hbox{{$j$}}}}
\put(-3,-1.4){\makebox(0,0)[br]{\hbox{{$1$}}}}
\put(-1.5,-2.3){\makebox(0,0)[tc]{\hbox{{$i$}}}}
\put(0,-3.2){\makebox(0,0)[tc]{\hbox{{Case II}}}}
}
\end{pspicture}
$$
In this case, $\col{B}\,\col{A}=\col{BA}$, $k$ is necessarily equal $1$, and after removing the common leg, all these pairs of paths are in bijection with single paths going from right to left and encompassing the regions $A$ and $B$; note that these paths are exactly paths constituting the matrix $M_2$! So the sum in (\ref{MM}) just gives the matrix $M_2$ (up to the factor $QS$, which we can now reconstruct). 
We have therefore proved that 
$T_3 T_1=M_2$ 
\ $\square$

Note that we shall present below the second proof of the groupoid property using quantum Grassmannian.

\begin{remark}\label{rem:MM}
The semiclassical limit of Theorem \ref{th:MM} statement reads
$$
\{\sheet{1}M_1 \underset{,}{\otimes} \sheet{2} M_2\}=\sheet{2}M_2\otimes \sheet{1}M_1\Bigl(-\frac 1n \sheet{1}I\otimes \sheet{2}I+ r_{n}\Bigr)
$$
where
\be\label{r-matrix}
r_{n}=\sum_i \sheet{1}e_{ii}\otimes \sheet{2} e_{ii}+2\sum_{i>j} \sheet{1}e_{ij}\otimes \sheet{2} e_{ji}
\ee
is the semiclassical $r$-matrix. Equivalently,
$$
\{[M_1]_{ab},[M_2]_{cd}\}=-\frac1n [M_1]_{ab}[M_2]_{cd}+[M_1]_{ad}[M_2]_{cb}\theta(b-d),\quad \theta(x)=\begin{cases}
2,& \text{ if }x>0,\\ 
1, &\text{ if }x=0,\\
0, &\text{ if }x<0.
\end{cases}
$$
\end{remark}

\begin{remark}\label{rem:M-M}
Note that for the trigonometric $R$-matrix (\ref{R-matrix})  the quantum relation
$$
R^T_{n}(q)\sheet{1}M_i \otimes \sheet{2} M_i=\sheet{2}M_i\otimes \sheet{1}M_i R_{n}(q)
$$
has an equivalent form of writing 
$$
\sheet{1}M_i \otimes \sheet{2} M_iR^T_{n}(q)=R_{n}(q)\sheet{2}M_i\otimes \sheet{1}M_i  \quad \hbox{for $i=1,2$.}
$$
Both these relations generate the same quantum algebra on elements of the matrices $M_1$ and $M_2$ and have the same semiclassical limit
$$
\{[M_i]_{ab},[M_i]_{cd}\}=[M_i]_{ad}[M_i]_{cb}(\theta(b-d)-\theta(a-c))\quad i=1,2,\ 
\theta(x)=\begin{cases}
2,& \text{ if }x>0,\\ 
1, &\text{ if }x=0,\\
0, &\text{ if }x<0.
\end{cases}
$$
\end{remark}

\begin{example}
	For $n=3$ direct computations show
	$$\sheet{1}{M_1}\otimes \sheet{2}{M_2}=\sheet{2}{M_2}\otimes \sheet{1}{M_1} R_{3}(q),$$ where

	$
	R_{3}(q)=q^{-1/3}\left(\begin{array}{ccc|ccc|ccc}
	q & 0 & 0 & 0 & 0 & 0 & 0 & 0 & 0 \\
	0 & 1 & 0 & 0 & 0 & 0 & 0 & 0 & 0 \\
	0 & 0 & 1 & 0 & 0 & 0 & 0 & 0 & 0 \\
	\hline
	0 & q-q^{-1} & 0 & 1 & 0 & 0 & 0 & 0 & 0 \\
	0 & 0 & 0 & 0 & q & 0 & 0 & 0 & 0 \\
	0 & 0 & 0 & 0 & 0 & 1 & 0 & 0 & 0 \\
	\hline
	0 & 0 & q-q^{-1} & 0 & 0 & 0 & 1 & 0 & 0 \\
	0 & 0 & 0 & 0 & 0 & q-q^{-1} & 0 & 1 & 0 \\
	0 & 0 & 0 & 0 & 0 & 0 & 0 & 0 & q \\
	\end{array}\right)
	$
	
	Similarly, for both $i=1,2$, we have
	$$ R_{3}^T(q) \sheet{1}{M_i}\otimes\sheet{2}{M_i}=\sheet{2}{M_i}\otimes\sheet{1}{M_i}R_{3}(q)$$
	
\end{example}

Again, by direct computation one can check $T_2 T_3 T_1=\operatorname{1}$. 



\section{Goldman brackets and commutation relations between transport matrices}\label{s:Goldman}
\setcounter{equation}{0}

To obtain a full-dimensional (without zero entries) form of transport matrices we define transport matrices along more general paths.

Namely, let ${\mathbb G}=(G,\mathbb{M})$ be a disk $G$ with four marked boundary points $\mathbb{M}=\{A,B,D,C\}$ in clockwise order, $(\Delta ABC,\Delta BCD)$ be a triangulation of $G$ and we also assume clockwise orientation of all triangle sides inside every triangle.  The space $\A_{SL_n,{\mathbb G}}$ coincides with the space of quadruple of complete flags; one flag assigned to each marked point.  A snake inside a triangle determines a projective basis as explained above. Each oriented side of triangulation determines such a snake. It  gives a projective basis in $\mathbb{C}^n$. The orientation of $BC$ determines the triangle associated with this side whose orientation is compatible with the orientation of the side. For instance, side $B\to C$ determines the unique snake from $B$ to $C$ inside triangle $\Delta ABC$ the side $C\to B$  is associated with the snake from $C$ to $B$ inside triangle $\Delta BCD$ with corresponding projective basis chosen in each case. Transition matrix from basis associated to $BC$ to the one associated with $CB$ is the product $T_{BC\leftarrow CB}=S\dot H$ where $H$ is the diagonal matrix defined by the Fock-Goncharov parameters on the side $BC$.   This allows to define transport matrix  for  any pair of oriented sides as transition matrix for the pair of corresponding bases; the matrix $S$ acts by changing  the orientation of the corresponding side. Let $T_{BC\leftarrow AB}$ be a transport in $\Delta ABC$ from side $AB$ to $BC$.
Pay attention that the sides are oriented and the order of endpoints in side notation matters.
Similarly, $T_{CB\leftarrow DC}$ is a transport matrix in $\Delta BCD$ from $DC$ to $CB$. Note that  $T_{DC\leftarrow CB}=T_{CB\leftarrow DC}^{-1}$. Then,
we define a transport $T_{DC \leftarrow AB}$ from $AB$ to $DC$ as $T_{DC \leftarrow AB}=T_{DC\leftarrow CB} T_{CB\leftarrow BC} T_{BC\leftarrow AB}=
T_{CB\leftarrow DC}^{-1} T_{CB\leftarrow BC} T_{BC\leftarrow AB}$.  Similarly, $T_{BD \leftarrow AB}=T_{BD\leftarrow CB} T_{CB\leftarrow BC} T_{BC\leftarrow AB}$.

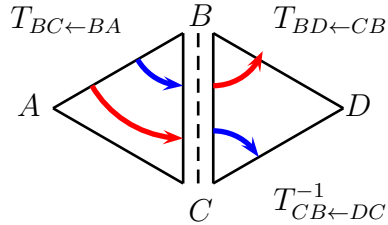
\begin{figure}[H]
{\psset{unit=1}
\begin{pspicture}(-2,-2)(2,2)
\newcommand{\TRI}{%
\psline[linewidth=1pt,linecolor=black](-1,0)(0,1.73)
\psline[linewidth=1pt,linecolor=black](-1,0)(1,0)
\psline[linewidth=1pt,linecolor=black](1,0)(0,1.73)
}
\rput{90}(-0.2,0){\TRI
\psarc[linewidth=2pt,linecolor=blue]{->}(1,0){0.7}{120}{180}
\psarc[linewidth=2pt,linecolor=red]{->}(1,0){1.4}{120}{180}
}
\psline[linewidth=1pt,linecolor=black,linestyle=dashed](0,1)(0,-1)
\rput{-90}(0.2,0){\TRI
\psarc[linewidth=2pt,linecolor=blue]{<-}(1,0){0.7}{120}{180}
\psarc[linewidth=2pt,linecolor=red]{->}(-1,0){0.7}{0}{70}
}
\put(-2.1,-0.1){\makebox(0,0)[rb]{\hbox{{$A$}}}}
\put(0.2,1.1){\makebox(0,0)[rb]{\hbox{{$B$}}}}
\put(0.2,-1.5){\makebox(0,0)[rb]{\hbox{{$C$}}}}
\put(2.3,-0.1){\makebox(0,0)[rb]{\hbox{{$D$}}}}

\put(-1,1){\makebox(0,0)[rb]{\hbox{{$T_{BC\leftarrow BA}$}}}}
\put(1,1){\makebox(0,0)[lb]{\hbox{{$T_{BD\leftarrow CB}$}}}}
\put(1,-1){\makebox(0,0)[lt]{\hbox{{$T_{CB\leftarrow DC}^{-1}$}}}}
\end{pspicture}
}
\caption{Disk with four marked points.}\label{fi:Disk4}
\end{figure}

To describe the quantum case, we split each quantum Fock-Goncharov parameter (or quantum cluster parameter) on the side $BC$ into a product of two, one inside triangle $\Delta ABC$ the other inside triangle $\Delta BCD$. The quantum parameters in triangle $\Delta ABC$ commute with those of triangle $\Delta BCD$, so the product of two Weyl-ordered monomials is itself a Weyl-ordered monomial, in which we perform an amalgamation of variables on the side $BC$. Due to the double action of the matrix $S$, the amalgamation of  boundary (frozen) variables in neighbor triangles respects the surface orientation, so, we amalgamate pairwise variables on the sides $BC$ of the two triangles ordered in the same direction, from $B$ to $C$. After the amalgamation, we unfreeze the obtained new variables. Therefore, in a network on a surface obtained as a union of several triangles, the Weyl ordering of weights of any path that does not go through any given triangle more than once is the  product of Weyl orderings of weights inside each triangle. It was explicitly demonstrated in \cite{Gus-Al} that the weights defined by such Weyl ordering are preserved by quantum mutations which now include mutations of amalgamated variables as well as mutations of variables in the interior of triangles.


We now show that the commutation relations from Theorem~\ref{th:MM} together with the groupoid condition (Theorem~\ref{th:groupoid}) imply the commutativity relations and Goldman brackets.

We begin with the pattern in the figure below. 

$$
{\psset{unit=1}
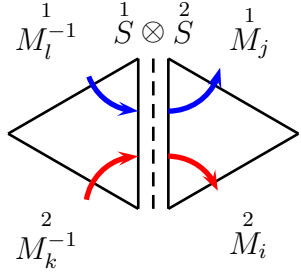
\begin{pspicture}(-2,-2)(2,2)
\newcommand{\TRI}{%
\psline[linewidth=1pt,linecolor=black](-1,0)(0,1.73)
\psline[linewidth=1pt,linecolor=black](-1,0)(1,0)
\psline[linewidth=1pt,linecolor=black](1,0)(0,1.73)
}
\rput{90}(-0.2,0){\TRI
\psarc[linewidth=2pt,linecolor=blue]{->}(1,0){0.7}{110}{180}
\psarc[linewidth=2pt,linecolor=red]{<-}(-1,0){0.7}{0}{80}
}
\psline[linewidth=1pt,linecolor=black,linestyle=dashed](0,1)(0,-1)
\rput{-90}(0.2,0){\TRI
\psarc[linewidth=2pt,linecolor=red]{<-}(1,0){0.7}{110}{180}
\psarc[linewidth=2pt,linecolor=blue]{->}(-1,0){0.7}{0}{80}
}
\put(-1,1){\makebox(0,0)[rb]{\hbox{{$\sheet{1}{M_l^{-1}}$}}}}
\put(-1,-1){\makebox(0,0)[rt]{\hbox{{$\sheet{2}{M_k^{-1}}$}}}}
\put(1,1){\makebox(0,0)[lb]{\hbox{{$\sheet{1}{M_j}$}}}}
\put(1,-1){\makebox(0,0)[lt]{\hbox{{$\sheet{2}{M_i}$}}}}
\put(0,1.2){\makebox(0,0)[cb]{\hbox{{$\sheet{1}{S}\otimes \sheet{2}{S}$}}}}
\end{pspicture}
}
$$
We use the identities 
\begin{align*}&\sheet{1}{M_l^{-1}}\otimes \sheet{2}{M_k^{-1}}=R_{n}(q) \sheet{2}{M_k^{-1}}\otimes \sheet{1}{M_l^{-1}}\\
&\sheet{1}{M_j}\otimes \sheet{2}{M_i}=\sheet{2}{M_i}\otimes \sheet{1}{M_j}R_{n}^{-\text{T}}(q)\\
&R_{n}(q) \sheet{1}{S}\otimes \sheet{2}{S}=\sheet{1}{S}\otimes \sheet{2}{S} R_{n}(q)^{\text{T}}, 
\end{align*}
where the last identity holds for any antidiagonal matrix $S$ whose elements commutes with all elements of quantum torus.

We then have
\begin{align*}
&[\sheet{1}{M_j}\sheet{1}{S}\sheet{1}{M_l^{-1}}]\otimes [\sheet{2}{M_i}\sheet{2}{S}\sheet{2}{M_k^{-1}}]
=\sheet{2}{M_i} \otimes \sheet{1}{M_j}  R_{n}^{-\text{T}}(q) \sheet{1}{S}\otimes \sheet{2}{S} R_{n}(q) \sheet{2}{M_k^{-1}}\otimes \sheet{1}{M_l^{-1}}\\
=&\sheet{2}{M_i} \otimes \sheet{1}{M_j} \sheet{1}{S}\otimes \sheet{2}{S}  R_{n}^{-1}(q) R_{n}(q) \sheet{2}{M_k^{-1}}\otimes \sheet{1}{M_l^{-1}}
=[\sheet{2}{M_i}\sheet{2}{S}\sheet{2}{M_k^{-1}}]\otimes [\sheet{1}{M_j}\sheet{1}{S}\sheet{1}{M_l^{-1}}],
\end{align*}
so two {transport matrices corresponding to nonintersecting paths}  commute. This is consistent with the quantum mapping class group transformations: flipping $BC$ edge separates the paths $AB\to BD$ and $AC\to CD$ into two adjacent triangles.

Consider now the case of two \emph{intersecting} paths (we consider a single intersection inside a quadrangle).
$$
{\psset{unit=1}
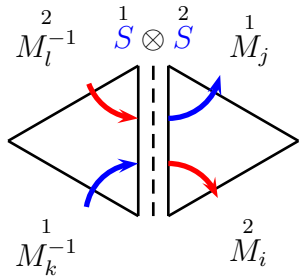
\begin{pspicture}(-2,-2)(2,2)
\newcommand{\TRI}{%
\psline[linewidth=1pt,linecolor=black](-1,0)(0,1.73)
\psline[linewidth=1pt,linecolor=black](-1,0)(1,0)
\psline[linewidth=1pt,linecolor=black](1,0)(0,1.73)
}
\rput{90}(-0.2,0){\TRI
\psarc[linewidth=2pt,linecolor=red]{->}(1,0){0.7}{110}{180}
\psarc[linewidth=2pt,linecolor=blue]{<-}(-1,0){0.7}{0}{80}
}
\psline[linewidth=1pt,linecolor=black,linestyle=dashed](0,1)(0,-1)
\rput{-90}(0.2,0){\TRI
\psarc[linewidth=2pt,linecolor=red]{<-}(1,0){0.7}{110}{180}
\psarc[linewidth=2pt,linecolor=blue]{->}(-1,0){0.7}{0}{80}
}
\put(-1,1){\makebox(0,0)[rb]{\hbox{{$\sheet{2}{M_l^{-1}}$}}}}
\put(-1,-1){\makebox(0,0)[rt]{\hbox{{$\sheet{1}{M_k^{-1}}$}}}}
\put(1,1){\makebox(0,0)[lb]{\hbox{{$\sheet{1}{M_j}$}}}}
\put(1,-1){\makebox(0,0)[lt]{\hbox{{$\sheet{2}{M_i}$}}}}
\put(0,1.2){\makebox(0,0)[cb]{\hbox{{$\sheet{1}{\tcb{S}}\otimes \sheet{2}{\tcb{S}}$}}}}
\end{pspicture}
}
$$
We then have
\begin{align*}
&q^{1/n}[\sheet{1}{M_j}\sheet{1}{S}\sheet{1}{M_k^{-1}}]\otimes [\sheet{2}{M_i}\sheet{2}{S}\sheet{2}{M_l^{-1}}]
-q^{-1/n} [\sheet{2}{M_i}\sheet{2}{S}\sheet{2}{M_l^{-1}}] \otimes [\sheet{1}{M_j}\sheet{1}{S}\sheet{1}{M_k^{-1}}] \\
=&\sheet{2}{M_i} \otimes \sheet{1}{M_j}  \sheet{1}{S}\otimes \sheet{2}{S}  [q^{1/n} R_{12}^{-1}(q)-q^{-1/n} R_{12}^{\text{T}}(q)] \sheet{1}{M_k^{-1}} \otimes \sheet{2}{M_l^{-1}} \\
=&(q^{-1}{-}q) \sheet{2}{M_i} \otimes \sheet{1}{M_j} \sheet{1}{S}\otimes \sheet{2}{S}  P_{12} \sheet{1}{M_l^{-1}} \otimes  \sheet{2}{M_k^{-1}}=
(q^{-1}{-}q) \sheet{2}{M_i} \sheet{2}{S} \sheet{2}{M_l^{-1}} \otimes   \sheet{1}{M_j} \sheet{1}{S}  \sheet{1}{M_k^{-1}} P_{n}  .
\end{align*}
So we have a {\sl quantum Goldman relation}
$$
\begin{pspicture}(-4,-1)(4,0.7){\psset{unit=0.7}
\rput(-5,0){
\psclip{\pscircle[linewidth=1.5pt, linestyle=dashed](0,0){1}}
\rput(0,0){\psline[linewidth=1.5pt,linecolor=red]{->}(-0.7,0.7)(0.7,-0.7)}
\rput(0,0){\psline[linewidth=3pt,linecolor=white](-0.7,-0.7)(0.7,0.7)}
\rput(0,0){\psline[linewidth=1.5pt,linecolor=blue]{->}(-0.7,-0.7)(0.7,0.7)}
\endpsclip
\rput(-1.2,0){\makebox(0,0)[rc]{$q^{1/n}$}}
}
\rput(-1,0){
\psclip{\pscircle[linewidth=1.5pt, linestyle=dashed](0,0){1}}
\rput(0,0){\psline[linewidth=1.5pt,linecolor=blue]{->}(-0.7,-0.7)(0.7,0.7)}
\rput(0,0){\psline[linewidth=3pt,linecolor=white](-0.7,0.7)(0.7,-0.7)}
\rput(0,0){\psline[linewidth=1.5pt,linecolor=red]{->}(-0.7,0.7)(0.7,-0.7)}
\endpsclip
\rput(-1.2,0){\makebox(0,0)[rc]{$-q^{-1/n}$}}
}
\rput(4,0){
\psclip{\pscircle[linewidth=1.5pt, linestyle=dashed](0,0){1}}
\rput(0,-1.4){\psarc[linewidth=1.5pt,linecolor=magenta, linestyle=dashed]{<-}(0,0){1}{45}{135}}
\rput(0,1.4){\psarc[linewidth=1.5pt,linecolor=orange, linestyle=dashed]{->}(0,0){1}{225}{315}}
\endpsclip
\rput(-1.2,0){\makebox(0,0)[rc]{$=(q^{-1}-q)$}}
\rput(1.2,0){\makebox(0,0)[lc]{$P_{n}$}}
}
}
\end{pspicture}
$$
with $P_n$ the permutation matrix. In the semiclassical limit with $q=e^{\hbar/2}$, the term linear in $\hbar$ gives rise to the {\sl Goldman bracket} for $SL_n$ \cite{Gold}.

Let a polygon be triangulated into collection of triangles containing triangle $\Delta ABC$, let $EF$ be another side of triangulation different from sides of $\Delta ABC$, let $\gamma$ be a path connecting $EF$ to $AB$ crossing any side of triangulation at most once and crossing  neither $AC$ nor $BC$
(see Figure~\ref{fig:full_size_matrices}).
Denote by $T_\gamma=T_{BA\leftarrow EF}$ the composition of transport matrices along the path $\gamma$, define transport matrices $M_1=T^{-1}_{AB\leftarrow CA} S T_\gamma$, $M_2=T_{BC\leftarrow AB} S T_\gamma$.

\begin{figure}[htbp]
\begin{center}
$$
{\psset{unit=1}
\begin{pspicture}(-2,-2)(2,2)
\newcommand{\TRI}{%
\psline[linewidth=1pt,linecolor=black](-1,0)(0,1.73)
\psline[linewidth=1pt,linecolor=black](-1,0)(1,0)
\psline[linewidth=1pt,linecolor=black](1,0)(0,1.73)
}
\newcommand{\POLI}{%
\psline[linewidth=1pt,linecolor=black](0,-1)(0,1)
\psline[linewidth=1pt,linecolor=black,linestyle=dashed](0,1)(0.5,1.2)
\psline[linewidth=1pt,linecolor=black,linestyle=dashed](0.5,1.2)(1,1)
\psline[linewidth=1pt,linecolor=black](1,1)(1,-1)
\psline[linewidth=1pt,linecolor=black,linestyle=dashed](0,-1)(0.5,-0.8)
\psline[linewidth=1pt,linecolor=black,linestyle=dashed](0.5,-0.8)(1,-1)
}
\rput{90}(-0.2,0){\TRI
\psarc[linewidth=2pt,linecolor=blue]{<-}(1,0){0.7}{120}{180}
\psarc[linewidth=2pt,linecolor=red]{->}(-1,0){0.7}{0}{70}
}
\psline[linewidth=1pt,linecolor=black,linestyle=dashed](0,1)(0,-1)
\rput{0}(0.2,0){\POLI
\psarc[linewidth=2pt,linecolor=red]{->}(0.5,-0.8){1.2}{65}{120}
\psarc[linewidth=2pt,linecolor=blue]{->}(0.5,-1.4){1.2}{65}{120}
}
\put(-2.1,-0.1){\makebox(0,0)[rb]{\hbox{{$C$}}}}
\put(0.2,1.1){\makebox(0,0)[rb]{\hbox{{$A$}}}}
\put(0.2,-1.5){\makebox(0,0)[rb]{\hbox{{$B$}}}}

\put(-0.4,1){\makebox(0,0)[rb]{\hbox{{$T^{-1}_{AB\leftarrow CA}$}}}}
\put(-0.4,-1.2){\makebox(0,0)[rb]{\hbox{{$T_{BC\leftarrow AB}$}}}}
\put(1,0.3){\makebox(0,0)[rt]{\hbox{{$T_\gamma$}}}}
\put(1.2,1){\makebox(0,0)[lb]{\hbox{{$E$}}}}
\put(1.2,-1){\makebox(0,0)[lt]{\hbox{{$F$}}}}
\end{pspicture}
}
$$
\caption{$M_1=T_{CA\leftarrow AB} S T_{BA\leftarrow EF}$, $M_2=T_{BC\leftarrow AB} S T_{BA\leftarrow EF}$.}
\label{fig:full_size_matrices}
\end{center}
\end{figure}
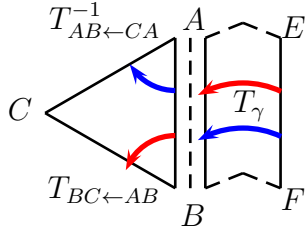


\begin{theorem}\label{th:MMsquare}
The transport matrices $M_1$ and $M_2$ in Fig.~\ref{fig:full_size_matrices} satisfy the commutation relations
$$
\sheet{1}M_1 \otimes \sheet{2} M_2=\sheet{2}M_2\otimes \sheet{1}M_1 R_{n}(q), 
$$
$$
R^T_{n}(q)\sheet{1}M_i \otimes \sheet{2} M_i=\sheet{2}M_i\otimes \sheet{1}M_i R_{n}(q) \quad \hbox{for $i=1,2$.}
$$
and
$$M_2^{-1} M_{BC\leftarrow CA} M_1=\operatorname{1}.$$
\end{theorem}

{\bf Proof}. Consider a transport matrix corresponding to a path not passing twice through the same triangle. It is given by the matrix product $M^{(m-1)}_{i_{m-1}}S\cdots S M^{(2)}_{i_2} S M^{(1)}_{i_1}=T_{BA\leftarrow EF}$ where $i_k=1,2$ and variables of $M^{(k)}_{i_k}$ and $M^{(p)}_{i_p}$ commute for distinct $k$ and $p$. Every such product satisfies the relation $R^{\text{T}}_{n}(q)\sheet{1}T \otimes \sheet{2} T=\sheet{2}T\otimes \sheet{1}T R_{n}(q)$. Then,
\begin{align*}
&\sheet{1}M{}^{(m)}_1 \sheet{1} S \sheet{1} T \otimes \sheet{2}M{}^{(m)}_2 \sheet{2} S \sheet{2} T= \sheet{1}M{}^{(m)}_1 \otimes \sheet{2}M{}^{(m)}_2  \sheet{1} S\otimes  \sheet{2} S \sheet{1} T \otimes \sheet{2} T =  \sheet{2}M{}^{(m)}_2 \otimes \sheet{1}M{}^{(m)}_1 R_{n}(q)\sheet{1} S\otimes  \sheet{2} S \sheet{1} T \otimes \sheet{2} T\\
=&\sheet{2}M{}^{(m)}_2 \sheet{2} S \otimes \sheet{1}M{}^{(m)}_1 \sheet{1} S R^{\text{T}}_{n}(q) \sheet{1} T \otimes \sheet{2} T
=\sheet{2}M{}^{(m)}_2 \sheet{2} S \sheet{2} T \otimes \sheet{1}M{}^{(m)}_1 \sheet{1} S \sheet{1} T  R_{n}(q).\quad \square
\end{align*}

General algebras of transport matrices in an ideal triangle decomposition of $\Sigma_{g,s,p}$---a genus $g$ Riemann surface with $s$ holes and $p>0$ marked points on the hole boundaries are governed by a quantum version of the Fock--Rosly Poisson algebra \cite{Fock-Rosly} also considered in \cite{CMR}. 


\section{Reflection equation and groupoid of upper triangular matrices}\label{s:reflection}
\setcounter{equation}{0}

\subsection{Representing an upper-triangular $\mathbb A$}

We assume again that both $M_1$ and $M_2$ are triangular matrices.

The main theorem concerns a special combination of $M_1$ and $M_2$:
\be\label{A}
\mathbb A:=M_1^{\text{T}}M_2.
\ee
Note that the transposition in the quantum case if formal: the quantum ordering is preserved, only matrix elements are permuted. Also, since $M_1$ and $M_1^{\text{T}}$ are upper-anti-diagonal matrices and $M_2$ is a lower-anti-diagonal matrix, the matrix $\mathbb A$ is automatically upper-triangular.

\begin{theorem}\label{th:A}
The matrix $\mathbb A=M_1^{\text{T}}M_2$ satisfies the quantum reflection equation
$$
R_{n}(q)\sheet{1}{\mathbb A} R_{n}^{\text{t}_1}(q)\sheet{2} {\mathbb A}=
\sheet{2}{\mathbb A} R_{n}^{\text{t}_1}(q)\sheet{1} {\mathbb A}R_{n}(q)
$$
with the trigonometric $R$-matrix (\ref{R-matrix}), where $R_{n}^{\text{t}_1}(q)$ is a partially transposed (w.r.t. the first space) $R$-matrix.
\end{theorem}

The proof is a short direct calculation that uses only $R$-matrix relations. Note that transposing with respect to  the first space the second relation in Theorem~\ref{th:MM} we obtain
$$
\sheet{1}M{}_1^{\text{T}} \otimes \sheet{2} M_2=\sheet{2}M_2  R_{n}^{\text{t}_1}(q) M{}_1^{\text{T}} 
$$
and the total transposition of the first relation gives
$$
R_{n}(q) \sheet{1}M{}_1^{\text{T}} \otimes \sheet{2}M{}_1^{\text{T}}= \sheet{2}M{}_1^{\text{T}} \otimes \sheet{1}M{}_1^{\text{T}}R_{n}^{\text{T}} (q),
$$
so
\begin{align}
&R_{n}(q) \sheet{1}M{}_1^{\text{T}}  \sheet{1}M_2 R_{n}^{\text{t}_1}(q) \sheet{2}M{}_1^{\text{T}}  \sheet{2}M_2=
R_{n}(q) \sheet{1}M{}_1^{\text{T}} \sheet{2}M{}_1^{\text{T}}  \sheet{1}M_2 \sheet{2}M_2
=\sheet{2}M{}_1^{\text{T}} \sheet{1}M{}_1^{\text{T}} R_{n}^{\text{T}}(q)  \sheet{1}M_2 \sheet{2}M_2\nonumber\\
=&\sheet{2}M{}_1^{\text{T}} \sheet{1}M{}_1^{\text{T}}   \sheet{2}M_2 \sheet{1}M_2 R_{n}(q)
=\sheet{2}M{}_1^{\text{T}}  \sheet{2}M_2 R_{n}^{\text{t}_1}(q)  \sheet{1}M{}_1^{\text{T}}  \sheet{1}M_2 R_{n}(q),\nonumber
\end{align}
which completes the proof.

We thus conclude that {\sl we have a Darboux coordinate representation} for operators satisfying the reflection equation. Moreover
{\it all matrix elements of $\mathbb A$ are Laurent polynomials with positive coefficients of $Z_\alpha$ and $q$}. In particular, positive integers in equation (\ref{eq:An}) count numbers of monomials in the corresponding Laurent polynomials. 

By construction of quantum transport matrices in Sec.~\ref{sec:QFG}, all matrix elements of $M_1$ and $M_2$ are Weyl-ordered. For $[\mathbb A]_{i,j}=\sum\limits_{k=i}^j [M_1]_{k,i}[M_2]_{k,j}$ we obtain that for $i<j$, $[M_1]_{k,i}$ commutes with  $[M_2]_{k,j}$ (all paths contributing to $[M_1]_{k,i}$ are disjoint from all paths contributing to $[M_2]_{k,j}$ for $i<j$), so the corresponding products are automatically Weyl-ordered,  
$$
[\mathbb A]_{i,j}=\sum\limits_{k=i}^j \col{[M_1]_{k,i}[M_2]_{k,j}}.
$$ 
For $i=j$, the corresponding two paths share the common starting edge, and then 
$$
[\mathbb A]_{i,i}=q^{-1/2} \col{[M_1]_{i,i}[M_2]_{i,i}}.
$$ 
This explains the appearance of $q^{-1/2}$ factors on the diagonal of the quantum matrix $\mathbb A^\hbar$ (see (\ref{A-quantum}), \cite{ChM}). Note that all Weyl-ordered products of $Z_\alpha$ are self-adjoint and we assume that all Casimirs are also self-adjoint. We have that
$$
 \col{[M_1]_{i,i}[M_2]_{i,i}}=\prod_{j=1}^i K_j,
$$
where $K_j$ are special Casimirs (\ref{Cas-out}) of the $\mathcal A_n$-quiver introduced in the next section. 

To obtain a full-dimensional (not upper-triangular) form of the matrix $\mathbb A_{\text{gen}}$ let us consider adjoint action by any transport matrix:
\begin{theorem}\label{th:A-gen}
Any matrix $\mathbb A_{\text{gen}}:=M_\gamma^{\text{T}} S^{\text{T}}\mathbb A S M_\gamma$, where $M_\gamma$ is a (transport) matrix satisfying commutation relations of Theorem \ref{th:MM} and such that its elements commute with those of $\mathbb A=M_1^{\text{T}} M_2$, satisfies the quantum reflection equation of Theorem \ref{th:A}.
\end{theorem}

The proof is again a direct computation; note also that if we represent $\mathbb A=M_1^{\text{T}} M_2$, then $M_1'=M_1SM_\gamma$ and $M_2'=M_2 SM_\gamma$ satisfy commutation relations from Theorem~\ref{th:MMsquare} and $\mathbb A_{\text{gen}}:={M_1'}^{\text{T}} M'_2$ then satisfy the quantum reflection equation.

{\section{The quiver for an upper-triangular $\mathbb A$ and the braid-group action}\label{s:braid}}
\setcounter{equation}{0}

In this section we first construct the quiver corresponding to the Darboux coordinates on the set ${\mathcal A}_n$ of upper-triangular matrices and, second, present the braid-group action on $\mathcal A_n$ via chains of mutations in the newly constructed quiver.
We would like to notice that  the quasi-cluster braid group action on the transport matrices and therefore on the framed moduli space ${\X}_{SL_n}$ was constructed in \cite{GS19}.
However, our construction seems to be different.

\subsection{$\mathcal A_n$-quiver}\label{ss:An-quiver}

Let us have a closer look on the structure of matrix entries of the product $[\mathbb A]_{i,j}:=\bigl[M_1^{\text{T}} M_2\bigr]_{i,j}=\sum_{k} (\mathcal M_1)_{k,i}(\mathcal M_2)_{k,j}$. All monomials contributing to $(\mathcal M_1)_{k,i}$ contain the same factor $\prod_{i=1}^k Z_{i,0,n-i}$ and all monomials contributing to  $(\mathcal M_2)_{k,j}$ contain  the same factor $\prod_{i=1}^k Z_{n-i,i,0}$, so the dependence of all elements of $[\mathbb A]_{i,j}$ on the frozen variables $Z_{i,0,n-i}$ and $Z_{n-i,i,0}$  is via their pairwise products, and we therefore amalgamate these variables pairwise thus obtaining  a single new variable 
$$
\bar Z_{i}:=\col{Z_{i,0,n-i}Z_{n-i,i,0}},\quad i=1,\dots,n-1.
$$ 
This results in a ``twisted'' pattern shown in Fig.~\ref{fi:amalgamation}. 

Another outcome of the above amalgamation procedure is that the resulting quiver admits $n-1$ new Casimirs depicted in Fig.~\ref{fi:diagonal-Casimirs}: 
Note that each of the new Casimirs $K_i$ is a monomial expression 
\be\label{Cas-out}
K_i:= \col{Z_{0,i,n-i}^2 \prod_{j=1}^{i-1} Z_{j,i-j,n-i} \bar Z_{i} \prod_{j=1}^{n-i-1} Z_{j,i,n-i-j}},
\ee
which contains exactly one square of one of the remaining frozen variable $Z_{0,i,n-i}$, and we have exactly one such Casimir per every remaining frozen variable $Z_{0,i,n-i}$. Let us change the system of coordinates in $\mathcal A_n$  replacing $Z_{0,i,n-i}$ by  $(K_i)^{1/2}$. Clearly, it is well defined nondegenerate change of coordinates. Fixing values of $K_i$'s leads to a Poisson submanifold $\mathcal U$. The remaining elements $Z_{a,b,c}$, $a>0$ still form log-canonical coordinate system on $\mathcal U$  whose Poisson bracket is still described by a quiver obtained by forgetting frozen variables $Z_{0,i,n-i}$.

Finally, we unfreeze the variables $\bar Z_{k}$, and then the connected part of the resulting quiver, which we call the {\sl $\mathcal A_n$-quiver}, contains only unfrozen variables. A  most symmetric way of vizualizing this quiver is to ``cut out'' the half-sized triangle located in the left-lower corner, then reflect this small triangle through the diagonal passing through its left-lower corner preserving the incidence relations for arrows in the both parts of the quiver, then glue pairwise the amalgamated variables  $Z_{k,0,n-k}$ and $Z_{n-k,k,0}$. Amalgamation operations become planar in this procedure schematically depicted below for $\mathcal A_5$.
$$
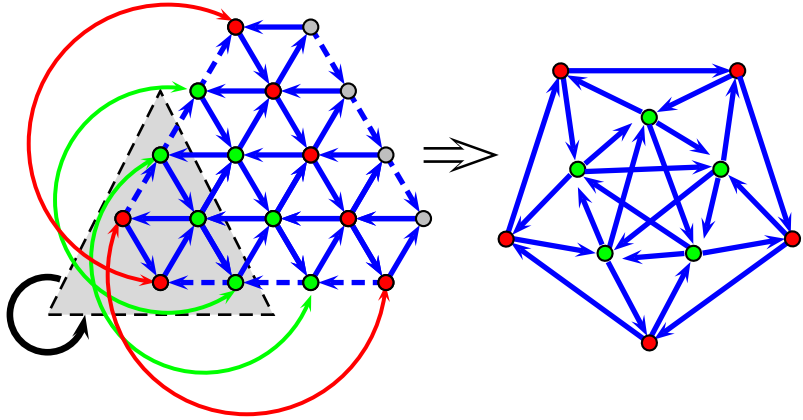
\begin{pspicture}(-3,-3.5)(2,2){
\newcommand{\PATGEN}{%
{\psset{unit=1}
\rput(0,0){\psline[linecolor=blue,linewidth=2pt]{->}(0,0)(.45,.765)}
\rput(0,0){\psline[linecolor=blue,linewidth=2pt]{->}(1,0)(0.1,0)}
\rput(0,0){\psline[linecolor=blue,linewidth=2pt]{->}(0,0)(.45,-.765)}
\put(0,0){\pscircle[fillstyle=solid,fillcolor=lightgray]{.1}}
}}
\newcommand{\PATLEFT}{%
{\psset{unit=1}
\rput(0,0){\psline[linecolor=blue,linewidth=2pt,linestyle=dashed]{->}(0,0)(.45,.765)}
\rput(0,0){\psline[linecolor=blue,linewidth=2pt]{->}(1,0)(0.1,0)}
\rput(0,0){\psline[linecolor=blue,linewidth=2pt]{->}(0,0)(.45,-.765)}
\put(0,0){\pscircle[fillstyle=solid,fillcolor=lightgray]{.1}}
}}
\newcommand{\PATRIGHT}{%
{\psset{unit=1}
\rput(0,0){\psline[linecolor=blue,linewidth=2pt,linestyle=dashed]{->}(0,0)(.45,-.765)}
\put(0,0){\pscircle[fillstyle=solid,fillcolor=lightgray]{.1}}
}}
\newcommand{\PATBOTTOM}{%
{\psset{unit=1}
\rput(0,0){\psline[linecolor=blue,linewidth=2pt]{->}(0,0)(.45,.765)}
\rput(0,0){\psline[linecolor=blue,linewidth=2pt,linestyle=dashed]{->}(1,0)(0.1,0)}
\put(0,0){\pscircle[fillstyle=solid,fillcolor=lightgray]{.1}}
}}
\newcommand{\PATTOP}{%
{\psset{unit=1}
\rput(0,0){\psline[linecolor=blue,linewidth=2pt]{->}(1,0)(0.1,0)}
\rput(0,0){\psline[linecolor=blue,linewidth=2pt]{->}(0,0)(.45,-.765)}
\put(0,0){\pscircle[fillstyle=solid,fillcolor=lightgray]{.1}}
}}
\newcommand{\PATBOTRIGHT}{%
{\psset{unit=1}
\rput(0,0){\psline[linecolor=blue,linewidth=2pt]{->}(0,0)(.45,.765)}
\put(0,0){\pscircle[fillstyle=solid,fillcolor=lightgray]{.1}}
\put(.5,0.85){\pscircle[fillstyle=solid,fillcolor=lightgray]{.1}}
}}
\pspolygon[linewidth=1pt,linecolor=black,linestyle=dashed, fillstyle=solid, fillcolor=gray, opacity=0.3](-0.5,-2.125)(-2,0.85)(-3.5,-2.125)
\psarc[linecolor=black, linewidth=2.5pt]{->}(-3.5,-2.125){0.5}{70}{360}
\multiput(-2.5,-0.85)(0.5,0.85){3}{\PATLEFT}
\multiput(-2,-1.7)(1,0){3}{\PATBOTTOM}
\put(-1,1.7){\PATTOP}
\multiput(-1.5,-0.85)(1,0){3}{\PATGEN}
\multiput(-1,0)(1,0){2}{\PATGEN}
\multiput(-.5,0.85)(1,0){1}{\PATGEN}
\multiput(-1.5,-0.85)(1,0){3}{\PATGEN}
\multiput(0,1.7)(0.5,-0.85){3}{\PATRIGHT}
\put(1,-1.7){\PATBOTRIGHT}
\psarc[linecolor=green, linewidth=1.5pt]{<->}(-1.9,-0.6){1.5}{80}{307}
\rput{27}(0,0){\psarc[linecolor=green, linewidth=1.5pt]{<->}(-1.9,-0.6){1.5}{87}{315}}
\psarc[linecolor=red, linewidth=1.5pt]{<->}(-1.9,0.15){1.85}{62}{265}
\psarc[linecolor=red, linewidth=1.5pt]{<->}(-0.85,-1.6){1.85}{157}{356}
\multiput(-1,1.7)(0.5,-0.85){5}{\pscircle[fillstyle=solid,fillcolor=red](0,0){.1}}
\multiput(-2.5,-0.85)(0.5,-0.85){2}{\pscircle[fillstyle=solid,fillcolor=red](0,0){.1}}
\multiput(-1.5,0.85)(0.5,-0.85){4}{\pscircle[fillstyle=solid,fillcolor=green](0,0){.1}}
\multiput(-2,0)(0.5,-0.85){3}{\pscircle[fillstyle=solid,fillcolor=green](0,0){.1}}
\psline[linecolor=black,linewidth=6pt]{->}(1.5,0)(2.5,0)
\psline[linecolor=white,linewidth=4pt]{->}(1.45,0)(2.4,0)
}
\end{pspicture}
\begin{pspicture}(-2.5,-3)(2,2)
\newcommand{\PATGEN}{%
{\psset{unit=1}
\psline[linecolor=blue,linewidth=2pt]{->}(-0.53,-0.73)(-0.07,0.9)
\psline[linecolor=blue,linewidth=2pt]{<-}(0.81,0.5)(0.1,0.93)
\psline[linecolor=blue,linewidth=2pt]{<-}(0.1,1.07)(1.1,1.56)
\psline[linecolor=blue,linewidth=2pt]{->}(-0.1,1.07)(-1.1,1.56)
\psline[linecolor=blue,linewidth=2pt]{<-}(1.1,1.62)(-1.1,1.62)
\pscircle[fillstyle=solid,fillcolor=green](0,1){.1}
\pscircle[fillstyle=solid,fillcolor=red](1.175,1.62){.1}
}}
\rput(0,0){\PATGEN}
\rput{72}(0,0){\PATGEN}
\rput{144}(0,0){\PATGEN}
\rput{216}(0,0){\PATGEN}
\rput{288}(0,0){\PATGEN}
\end{pspicture}
$$

Note each of $[n/2]$ original Casimirs of the $SL_n$-quiver gives rise to the corresponding Casimir element of the $\mathcal A_n$-quiver
\be\label{Cas-in}
C_k=\col{\bar Z_k \prod_{i=1}^{n-k-1}Z_{k,i,n-k-i}\bar Z_{n-k}  \prod_{j=1}^{k-1}Z_{n-k,j,k-j}},
\ee 
and in figures representing $\mathcal A_n$-quivers, cluster variables of sites of the same color contribute (all in power one) to the same Casimir. We present the $\mathcal A_n$-quivers for $n=3,4,5,6$ in Fig.~\ref{fi:An-quivers} where we indicate $[n/2]$ independent Casimir elements depicted in Fig.~\ref{fi:Casimirs-A}.

\begin{figure}[tb]
\begin{pspicture}(-1.2,-1)(1.2,1)
\newcommand{\PATGEN}{%
{\psset{unit=1}
\psline[linecolor=blue,linewidth=1.5pt]{->}(0.76,-0.57)(-0.76,-0.57)
\psline[linecolor=white,linewidth=2.5pt]{->}(0.76,-0.43)(-0.76,-0.43)
\psline[linecolor=blue,linewidth=1.5pt]{->}(0.76,-0.43)(-0.76,-0.43)
\pscircle[fillstyle=solid,fillcolor=red](0.866,-0.5){.1}
}}
\rput(0,0){\PATGEN}
\rput{120}(0,0){\PATGEN}
\rput{240}(0,0){\PATGEN}
\put(0,-0.8){\makebox(0,0)[tc]{\hbox{{$n=3$}}}}
\end{pspicture}
\begin{pspicture}(-1.3,-1.5)(1.3,1.5)
\newcommand{\PATGEN}{%
{\psset{unit=1}
\psline[linecolor=blue,linewidth=2pt]{->}(-0.9,1)(0.9,1)
\pscircle[fillstyle=solid,fillcolor=red](1,1){.1}
}}
\rput(0,0){\PATGEN}
\rput{90}(0,0){\PATGEN}
\rput{180}(0,0){\PATGEN}
\rput{270}(0,0){\PATGEN}
\psline[linecolor=white,linewidth=2.5pt]{->}(-0.91,0.94)(0.41,0.06)
\psline[linecolor=blue,linewidth=1.5pt]{->}(-0.91,0.94)(0.41,0.06)
\psline[linecolor=white,linewidth=2.5pt]{<-}(-0.958,0.916)(-0.542,0.084)
\psline[linecolor=blue,linewidth=1.5pt]{<-}(-0.958,0.916)(-0.542,0.084)
\psline[linecolor=white,linewidth=2.5pt]{->}(0.91,0.94)(-0.41,0.06)
\psline[linecolor=blue,linewidth=1.5pt]{->}(0.91,0.94)(-0.41,0.06)
\psline[linecolor=white,linewidth=2.5pt]{<-}(0.958,0.916)(0.542,0.084)
\psline[linecolor=blue,linewidth=1.5pt]{<-}(0.958,0.916)(0.542,0.084)
\psline[linecolor=white,linewidth=2.5pt]{<-}(-0.91,-0.94)(0.41,-0.06)
\psline[linecolor=blue,linewidth=1.5pt]{<-}(-0.91,-0.94)(0.41,-0.06)
\psline[linecolor=white,linewidth=2.5pt]{->}(-0.958,-0.916)(-0.542,-0.084)
\psline[linecolor=blue,linewidth=1.5pt]{->}(-0.958,-0.916)(-0.542,-0.084)
\psline[linecolor=white,linewidth=2.5pt]{<-}(0.91,-0.94)(-0.41,-0.06)
\psline[linecolor=blue,linewidth=1.5pt]{<-}(0.91,-0.94)(-0.41,-0.06)
\psline[linecolor=white,linewidth=2.5pt]{->}(0.958,-0.916)(0.542,-0.084)
\psline[linecolor=blue,linewidth=1.5pt]{->}(0.958,-0.916)(0.542,-0.084)
\pscircle[fillstyle=solid,fillcolor=green](-0.5,0){.1}
\pscircle[fillstyle=solid,fillcolor=green](0.5,0){.1}
\put(0,-1.2){\makebox(0,0)[tc]{\hbox{{$n=4$}}}}
\end{pspicture}
\begin{pspicture}(-2.5,-2.5)(2.5,2.5)
\newcommand{\PATGEN}{%
{\psset{unit=1}
\psline[linecolor=white,linewidth=2.5pt]{->}(-0.55,-0.7)(-0.003,0.9)
\psline[linecolor=blue,linewidth=1.5pt]{->}(-0.55,-0.7)(-0.003,0.9)
\psline[linecolor=blue,linewidth=2pt]{<-}(0.9,0.4)(0.1,0.93)
\psline[linecolor=blue,linewidth=2pt]{<-}(0.1,1.07)(1.1,1.56)
\psline[linecolor=blue,linewidth=2pt]{->}(-0.1,1.07)(-1.1,1.56)
\psline[linecolor=blue,linewidth=2pt]{<-}(1.1,1.62)(-1.1,1.62)
\pscircle[fillstyle=solid,fillcolor=green](0,1){.1}
\pscircle[fillstyle=solid,fillcolor=red](1.175,1.62){.1}
}}
\rput(0,0){\PATGEN}
\rput{72}(0,0){\PATGEN}
\rput{144}(0,0){\PATGEN}
\rput{216}(0,0){\PATGEN}
\rput{288}(0,0){\PATGEN}
\put(0,-2.2){\makebox(0,0)[tc]{\hbox{{$n=5$}}}}
\end{pspicture}
\begin{pspicture}(-3,-3.5)(3,3)
{\psset{unit=1}
\newcommand{\PATGEN}{%
{\psset{unit=1}
\psline[linecolor=blue,linewidth=2pt]{<-}(1,1.5)(2.5,1.5)
\psline[linecolor=blue,linewidth=2pt]{->}(-0.75,1.5)(0.75,1.5)
\psline[linecolor=blue,linewidth=2pt]{->}(-1,1.5)(-2.5,1.5)
\psline[linecolor=blue,linewidth=2pt]{->}(0,3)(2.5,1.56)
\pscircle[fillstyle=solid,fillcolor=green](0.866,1.5){.1}
\pscircle[fillstyle=solid,fillcolor=red](2.6,1.5){.1}
}}
\newcommand{\PATIN}{%
{\psset{unit=1}
\psline[linecolor=blue,linewidth=2pt]{<-}(-0.75,1.45)(-0.1,1.05)
\psline[linecolor=blue,linewidth=2pt]{->}(0.75,1.45)(0.1,1.05)
%
\psbezier[linecolor=white,linewidth=2.5pt]{->}(-0.84,-1.4)(-0.34,-0.8)(-0.1,0.5)(-0.04,0.9)
\psbezier[linecolor=blue,linewidth=1.5pt]{->}(-0.84,-1.4)(-0.34,-0.8)(-0.1,0.5)(-0.04,0.9)
\psbezier[linecolor=white,linewidth=2.5pt]{->}(-0.05,0.9)(-0.692,0.4)(-0.692,0.4)(-0.81,-0.42)
\psbezier[linecolor=blue,linewidth=1.5pt]{->}(-0.05,0.9)(-0.692,0.4)(-0.692,0.4)(-0.81,-0.42)
\psbezier[linecolor=white,linewidth=2.5pt]{<-}(0.84,-1.4)(0.34,-0.8)(0.1,0.5)(0.04,0.9)
\psbezier[linecolor=blue,linewidth=1.5pt]{<-}(0.84,-1.4)(0.34,-0.8)(0.1,0.5)(0.04,0.9)
\pscircle[fillstyle=solid,fillcolor=yellow](0,1){.1}
}}
\rput(0,0){\PATIN}
\rput{120}(0,0){\PATIN}
\rput{240}(0,0){\PATIN}
\rput(0,0){\PATGEN}
\rput{60}(0,0){\PATGEN}
\rput{120}(0,0){\PATGEN}
\rput{180}(0,0){\PATGEN}
\rput{240}(0,0){\PATGEN}
\rput{300}(0,0){\PATGEN}
\put(0,-3.3){\makebox(0,0)[tc]{\hbox{{$n=6$}}}}
}
\end{pspicture}
\caption{$\mathcal A_n$ -quivers for $n=3,4,5,6$.}
\label{fi:An-quivers}
\end{figure}
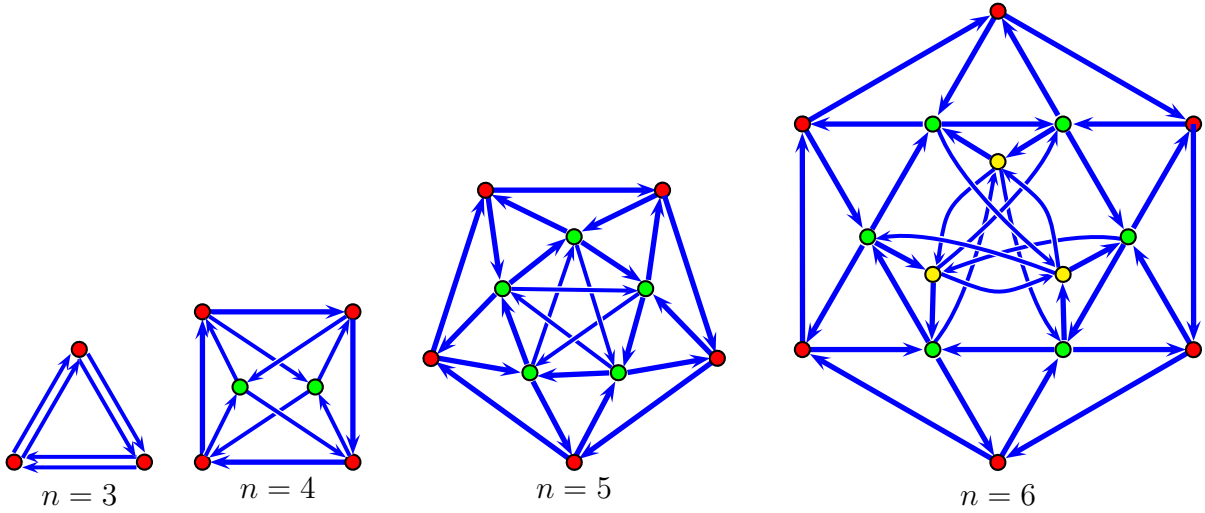

Since all Casimirs of $\mathbb A\in {\mathcal A}_n$ are generated by $\lambda$-power expansion terms for $\det(\mathbb A+\lambda \mathbb A^{\text{T}})$, we automatically obtain the following lemma
\begin{lemma}
$\det(\mathbb A+\lambda \mathbb A^{\text{T}})=P(C_1,\dots,C_{[n/2]})$, where $C_i$ are Casimirs of the $\mathcal A_n$-quiver.\footnote{The corresponding polynomials were found in \cite{ChShSh}.}
\end{lemma}

\begin{figure}[tb]
{\psset{unit=1}
\begin{pspicture}(-2,-2)(2,2)
\newcommand{\SIX}{%
\psbezier[linecolor=blue,linewidth=12pt](0,1.5)(-0.866,2)(-2,1.6)(-0.5,-1)
\psbezier[linecolor=white,linewidth=10pt](0,1.5)(-0.866,2)(-2,1.6)(-0.5,-1)
\psbezier[linecolor=blue,linewidth=12pt](0,1.5)(0.866,2)(2,1.6)(0.5,-1)
\psbezier[linecolor=white,linewidth=10pt](0,1.5)(0.866,2)(2,1.6)(0.5,-1)
\psbezier[linecolor=blue,linewidth=12pt](0,0.5)(-2.6,-1)(-1,-1.866)(-0.5,-1)
\psbezier[linecolor=white,linewidth=10pt](0,0.5)(-2.6,-1)(-1,-1.866)(-0.5,-1)
\psbezier[linecolor=blue,linewidth=12pt](0,0.5)(2.6,-1)(1,-1.866)(0.5,-1)
\psbezier[linecolor=white,linewidth=10pt](0,0.5)(2.6,-1)(1,-1.866)(0.5,-1)
\psline[linecolor=blue,linewidth=12pt](0,0.5)(0,1.5)
\psline[linecolor=white,linewidth=10pt](0,0.5)(0,1.5)
\psline[linecolor=blue,linewidth=12pt](-0.5,-1)(0.5,-1)
\psline[linecolor=white,linewidth=10pt](-0.5,-1)(0.5,-1)
\pscircle[fillstyle=solid,fillcolor=white,linecolor=white](0,1.5){.22}
\pscircle[fillstyle=solid,fillcolor=white,linecolor=white](0,0.5){.22}
\pscircle[fillstyle=solid,fillcolor=white,linecolor=white](-0.5,-1){.22}
\pscircle[fillstyle=solid,fillcolor=white,linecolor=white](0.5,-1){.22}
\rput(0,0.1){
\psbezier[linecolor=blue,linewidth=1.5pt,linestyle=dashed](0,1.5)(0.866,2)(2,1.6)(0.5,-1)
\psline[linecolor=blue,linewidth=1.5pt,linestyle=dashed](0,0.5)(0,1.5)
\psbezier[linecolor=blue,linewidth=1.5pt,linestyle=dashed](0,0.5)(-2.6,-1)(-1,-1.866)(-0.5,-1)
\psline[linecolor=blue,linewidth=1.5pt,linestyle=dashed](-0.5,-1)(0.5,-1)}
\psbezier[linecolor=white,linewidth=10pt](0,0.5)(2.6,-1)(1,-1.866)(0.5,-1)
\rput(0,-0.1){\psbezier[linecolor=red,linewidth=1.5pt,linestyle=dashed](0,0.5)(2.6,-1)(1,-1.866)(0.5,-1)
\psbezier[linecolor=red,linewidth=1.5pt,linestyle=dashed](0,0.5)(-2.6,-1)(-1,-1.866)(-0.5,-1)
\psline[linecolor=red,linewidth=1.5pt,linestyle=dashed](-0.5,-1)(0.5,-1)}
}
\newcommand{\TRI}{%
\psbezier[linecolor=blue,linewidth=12pt](-1,0)(-1,-1.5)(1,-1.5)(1,0)
\psbezier[linecolor=white,linewidth=10pt](-1,0)(-1,-1.5)(1,-1.5)(1,0)
\psbezier[linecolor=blue,linewidth=12pt](-1,0)(1.6,1.5)(1.866,0.5)(1,0)
\psbezier[linecolor=white,linewidth=10pt](-1,0)(1.6,1.5)(1.866,0.5)(1,0)
\psbezier[linecolor=blue,linewidth=12pt](-1,0)(-1.866,0.5)(-1.6,1.5)(1,0)
\psbezier[linecolor=white,linewidth=10pt](-1,0)(-1.866,0.5)(-1.6,1.5)(1,0)
\pscircle[fillstyle=solid,fillcolor=white,linecolor=white](-1,0){.22}
\pscircle[fillstyle=solid,fillcolor=white,linecolor=white](1,0){.22}
\psbezier[linecolor=red,linewidth=1.5pt,linestyle=dashed](-1,0)(-1.866,0.5)(-1.6,1.5)(1,0)
\psbezier[linecolor=red,linewidth=1.5pt,linestyle=dashed](-1,0)(-1,-1.5)(1,-1.5)(1,0)
}
\rput(-2,0){\TRI
\put(0,-1.5){\makebox(0,0)[tc]{\hbox{{$\Gamma_{1,1}$}}}}
\put(-1,1.1){\makebox(0,0)[bc]{\hbox{{\tcr{$a_{1,2}$}}}}}
}
\rput{0}(2,0){\SIX
\put(0,-1.5){\makebox(0,0)[tc]{\hbox{{$\Gamma_{1,2}$}}}}
\put(1.5,-1.5){\makebox(0,0)[cl]{\hbox{{\tcr{$a_{1,2}$}}}}}
\put(1.5,1.5){\makebox(0,0)[cl]{\hbox{{\tcb{$a_{1,3}$}}}}}
}
\end{pspicture}
}
\caption{Fat graphs $\Gamma_{1,1}$ and $\Gamma_{1,2}$ corresponding to the respective quivers for $\mathcal A_3$ and $\mathcal A_4$. We indicate closed paths that produce elements $a_{1,2}\in {\mathbb A}\in\mathcal A_3$ and $a_{1,2}$ and $a_{1,3}$ of ${\mathbb A}\in\mathcal A_4$. For example, setting all $Z_\alpha=1$ for simplicity and using formulas from \cite{ChF2} for the corresponding geodesic functions we obtain $a_{1,2}=\tr (LR)=3$ in $\mathcal A_3$ and $a_{1,2}=\tr (LLR)=4$ and $a_{1,3}=\tr (LLRR)=6$ in $\mathcal A_4$, where $L=\bigl[ {1\ 0\atop 1\ 1} \bigr]$ and $R=\bigl[ {1\ 1\atop 0\ 1} \bigr]$ 
are matrices of the respective left and right turns of paths occuring at three-valent vertices of a spine. It is easy to see that the above $a_{i,j}$ coincide with those in Example~\ref{ex:toy}.}
\label{fi:spines}
\end{figure}
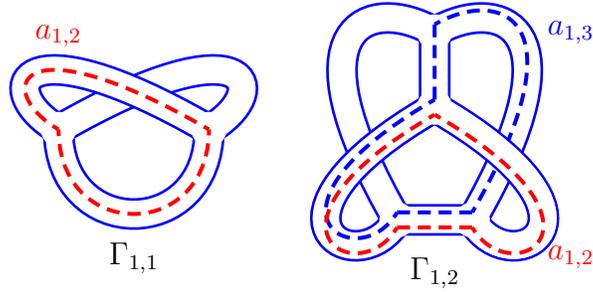

\begin{remark}
In the cases $n=3$ and $n=4$, the constructed quivers are those of geometric systems: these cases admit three-valent fat-graph representations in which Y-cluster variables are identified with (exponentiated) Thurston shear coordinates $z_\alpha$ enumerated by edges of the corresponding graphs, and nontrivial commutation relations are between variables on adjacent edges; for $n=3$ and $n=4$ these graphs are the relative spines of Riemann surfaces $\Sigma_{1,1}$ and $\Sigma_{1,2}$ depicted in Fig.~\ref{fi:spines}; the Laurent polynomials for entries of $\mathbb A$ coincide up to a linear change of log-canonical variables with the expressions obtained by identifying these entries with {\sl geodesic functions} corresponding to closed paths on these graphs; for more details and for the explicit construction of geodesic functions, see \cite{ChF2}, \cite{ChM}. Note here that, likewise all $a_{i,j}$ constructed in this paper, all geodesic functions for all surfaces $\Sigma_{g,s}$ are positive Laurent polynomials of $e^{z_\alpha/2}$.
\end{remark}

\subsection{Braid-group action through mutations}\label{ss:braid-mutation}

Our second major goal in this paper is to find a representation of the braid-group action from Sec.~\ref{ss:braid} in terms of cluster mutations of the  $\mathcal A_n$-quiver.  It is well-known that  for $\mathbb A$ belonging to a specific symplectic leaf in $\mathcal A_n$ its matrix elements $a_{i,j}$ are identified with the geodesics functions. In this leaf, braid-group transformations correspond to Dehn twists along geodesics corresponding to geodesic functions $a_{i,i+1}$  on $\Sigma_{g,s}$ ($n=2g+s$ and $s=1,2$). We know that every Dehn twist on a Riemann surface $\Sigma_{g,s}$ is a sequence of cluster mutations since it can be presented as a chain of mutations of shear variables on edges of the corresponding spine $\Gamma_{g,s}$. Whereas, for $n=3$ and $n=4$, generic symplectic leaves in $\mathcal A_n$ are geometric and the corresponding mutation sequences are identical, for larger $n$ the generic symplectic leaves become essentially different and we have to reinvent a braid group action. Knowing the answer for $n=3$ and $4$ helps in guessing the answer for a general $n$. We begin with the example of $\mathcal A_5$-quiver:

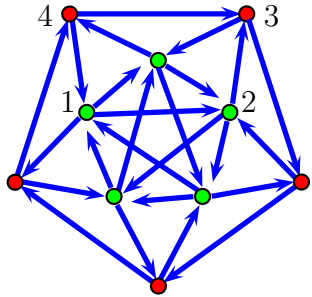
\begin{center}
\begin{pspicture}(-2.5,-2.5)(2.5,2.5)
\newcommand{\PATGEN}{%
{\psset{unit=1}
\psline[linecolor=blue,linewidth=2pt]{->}(-0.53,-0.73)(-0.07,0.9)
\psline[linecolor=blue,linewidth=2pt]{<-}(0.81,0.5)(0.1,0.93)
\psline[linecolor=blue,linewidth=2pt]{<-}(0.1,1.07)(1.1,1.56)
\psline[linecolor=blue,linewidth=2pt]{->}(-0.1,1.07)(-1.1,1.56)
\psline[linecolor=blue,linewidth=2pt]{<-}(1.1,1.62)(-1.1,1.62)
\pscircle[fillstyle=solid,fillcolor=green](0,1){.1}
\pscircle[fillstyle=solid,fillcolor=red](1.175,1.62){.1}
}}
\rput(0,0){\PATGEN}
\rput{72}(0,0){\PATGEN}
\rput{144}(0,0){\PATGEN}
\rput{216}(0,0){\PATGEN}
\rput{288}(0,0){\PATGEN}
\put(-1.1,0.58){\makebox(0,0)[tr]{\hbox{{$1$}}}}
\put(1.1,0.58){\makebox(0,0)[tl]{\hbox{{$2$}}}}
\put(-1.4,1.6){\makebox(0,0)[cr]{\hbox{{$4$}}}}
\put(1.4,1.6){\makebox(0,0)[cl]{\hbox{{$3$}}}}
\end{pspicture}
\end{center}

The following chain of  mutations $\beta_{3,4}=\mu_1\mu_2\mu_3\mu_2\mu_1=S_{3,4}$ preserves the form of the original quiver with the interchanged vertices $3$ and $4$, where $\mu_i$ is a mutation at vertex $i$.

A convenient way to represent a set of elementary braid-group transformations for a general $\mathcal A_n$ quiver is the process schematically depicted in Fig.~\ref{fi:covering} below: we take another copy of the triangle representing the quiver, reflect it and glue the resulting triangle to the original one along the bottom side of the latter in a way that amalgamated variables on the sides of two triangles match and the colored vertices representing Casimir elements are stretched along SE diagonals. The sequences of mutations corresponding to elementary generating elements $\beta_{i,i+1}$ of the braid groups are indicated in the figure.

\begin{figure}[tb]
{\psset{unit=1}
\begin{pspicture}(-5,-4)(5,4)
\newcommand{\PATGEN}{%
{\psset{unit=1}
\rput(0,0){\psline[linecolor=blue,linewidth=2pt]{->}(0,0)(.45,.765)}
\rput(0,0){\psline[linecolor=blue,linewidth=2pt]{->}(0.9,0)(0.1,0)}
\rput(0,0){\psline[linecolor=blue,linewidth=2pt]{->}(0,0)(.45,-.765)}
\put(0,0){\pscircle[fillstyle=solid,fillcolor=lightgray]{.1}}
}}
\newcommand{\PATLEFT}{%
{\psset{unit=1}
\rput(0,0){\psline[linecolor=blue,linewidth=2pt,linestyle=dashed]{->}(0,0)(.45,.765)}
\rput(0,0){\psline[linecolor=blue,linewidth=2pt]{->}(0.9,0)(0.1,0)}
\rput(0,0){\psline[linecolor=blue,linewidth=2pt]{->}(0,0)(.45,-.765)}
\put(0,0){\pscircle[fillstyle=solid,fillcolor=lightgray]{.1}}
}}
\newcommand{\PATRIGHT}{%
{\psset{unit=1}
\rput(0,0){\psline[linecolor=blue,linewidth=2pt]{->}(0,0)(.45,-.765)}
\put(0,0){\pscircle[fillstyle=solid,fillcolor=lightgray]{.1}}
}}
\newcommand{\PATBOTTOM}{%
{\psset{unit=1}
\rput(0,0){\psline[linecolor=blue,linewidth=2pt,linestyle=dashed]{->}(0,0)(.45,.765)}
\put(0,0){\pscircle[fillstyle=solid,fillcolor=lightgray]{.1}}
}}
\newcommand{\DASHEDBOX}{%
{\psset{unit=1}
\psline[linecolor=magenta,linewidth=2pt,linestyle=dashed](-1.3,0.3)(1.3,0.3)
\psline[linecolor=magenta,linewidth=2pt,linestyle=dashed](-1.3,-0.3)(1.3,-0.3)
\psline[linecolor=magenta,linewidth=2pt,linestyle=dashed](-1.3,0.3)(-1.3,-0.3)
\psline[linecolor=magenta,linewidth=2pt,linestyle=dashed](1.3,0.3)(1.3,-0.3)
}}
\newcommand{\DASHEDBOXLONG}{%
{\psset{unit=1}
\psline[linecolor=magenta,linewidth=2pt,linestyle=dashed](-1.8,0.3)(1.8,0.3)
\psline[linecolor=magenta,linewidth=2pt,linestyle=dashed](-1.8,-0.3)(1.8,-0.3)
\psline[linecolor=magenta,linewidth=2pt,linestyle=dashed](-1.8,0.3)(-1.8,-0.3)
\psline[linecolor=magenta,linewidth=2pt,linestyle=dashed](1.8,0.3)(1.8,-0.3)
}}
\rput(-3,0){
\multiput(-2,0.85)(0.5,0.85){3}{\PATLEFT}
\multiput(0.5,-3.4)(0.5,0.85){3}{\PATBOTTOM}
\multiput(-1.5,0)(0.5,0.85){3}{\PATGEN}
\multiput(-1,-0.85)(0.5,0.85){3}{\PATGEN}
\multiput(-0.5,-1.7)(0.5,0.85){3}{\PATGEN}
\multiput(0,-2.55)(0.5,0.85){3}{\PATGEN}
\multiput(-0.5,3.4)(0.5,-0.85){5}{\PATRIGHT}
\multiput(-0.5,3.4)(0.5,-0.85){6}{\pscircle[fillstyle=solid,fillcolor=red](0,0){.1}}
\multiput(-2,0.85)(0.5,-0.85){6}{\pscircle[fillstyle=solid,fillcolor=red](0,0){.1}}
\multiput(-1,2.55)(0.5,-0.85){6}{\pscircle[fillstyle=solid,fillcolor=green](0,0){.1}}
\multiput(-1.5,1.7)(0.5,-0.85){6}{\pscircle[fillstyle=solid,fillcolor=green](0,0){.1}}
\multiput(-0.5,1.7)(0.5,-0.85){4}{
\rput{60}(0,0){\DASHEDBOX}}
\put(0.3,3.05){\makebox(0,0)[bl]{\hbox{{$\beta_{1,2}$}}}}
\put(0.8,2.2){\makebox(0,0)[bl]{\hbox{{$\beta_{2,3}$}}}}
\put(1.3,1.35){\makebox(0,0)[bl]{\hbox{{$\beta_{3,4}$}}}}
\put(1.8,0.5){\makebox(0,0)[bl]{\hbox{{$\beta_{4,5}$}}}}
}
\rput(3.5,0){
\multiput(-2.5,0.85)(0.5,0.85){4}{\PATLEFT}
\multiput(0.5,-4.25)(0.5,0.85){4}{\PATBOTTOM}
\multiput(-2,0)(0.5,0.85){4}{\PATGEN}
\multiput(-1.5,-0.85)(0.5,0.85){4}{\PATGEN}
\multiput(-1,-1.7)(0.5,0.85){4}{\PATGEN}
\multiput(-0.5,-2.55)(0.5,0.85){4}{\PATGEN}
\multiput(0,-3.4)(0.5,0.85){4}{\PATGEN}
\multiput(-0.5,4.25)(0.5,-0.85){6}{\PATRIGHT}
\multiput(-0.5,4.25)(0.5,-0.85){7}{\pscircle[fillstyle=solid,fillcolor=red](0,0){.1}}
\multiput(-2.5,0.85)(0.5,-0.85){7}{\pscircle[fillstyle=solid,fillcolor=red](0,0){.1}}
\multiput(-1,3.4)(0.5,-0.85){7}{\pscircle[fillstyle=solid,fillcolor=green](0,0){.1}}
\multiput(-2,1.7)(0.5,-0.85){7}{\pscircle[fillstyle=solid,fillcolor=green](0,0){.1}}
\multiput(-1.5,2.55)(0.5,-0.85){7}{\pscircle[fillstyle=solid,fillcolor=yellow](0,0){.1}}
\multiput(-0.75,2.125)(0.5,-0.85){5}{
\rput{60}(0,0){\DASHEDBOXLONG}}
\put(0.3,3.9){\makebox(0,0)[bl]{\hbox{{$\beta_{1,2}$}}}}
\put(0.8,3.05){\makebox(0,0)[bl]{\hbox{{$\beta_{2,3}$}}}}
\put(1.3,2.2){\makebox(0,0)[bl]{\hbox{{$\beta_{3,4}$}}}}
\put(1.8,1.35){\makebox(0,0)[bl]{\hbox{{$\beta_{4,5}$}}}}
\put(2.3,0.5){\makebox(0,0)[bl]{\hbox{{$\beta_{5,6}$}}}}
}
\psline[linewidth=3pt,linestyle=dotted](-5.7,0)(-4.7,0)
\psline[linewidth=3pt,linestyle=dotted](-1,0)(1.2,0)
\psline[linewidth=3pt,linestyle=dotted](6,0)(6.8,0)
\end{pspicture}
}
\caption{The braid-group action represented on the union of two triangles, each of which is the copy of the $\mathcal A_n$-quiver, for $n=5$ and $6$. 
The dotted line indicates the line of the triangle gluing. 
Dashed blocks enclose cluster variables sequences of mutations at which produce elementary braid-group transformation $\beta_{i,i+1}$: every such sequence commences with mutating the lowest element inside a box (corresponding to a six-valent vertex), then its upper-right neighbor and so on until we reach the upper element, mutate it, and repeat mutations at all inner elements in the reverse order. So, every such braid-group transformation is produced by a sequence of $2n-5$ mutations in the corresponding $\mathcal A_n$-quiver.} 
\label{fi:covering}
\end{figure}
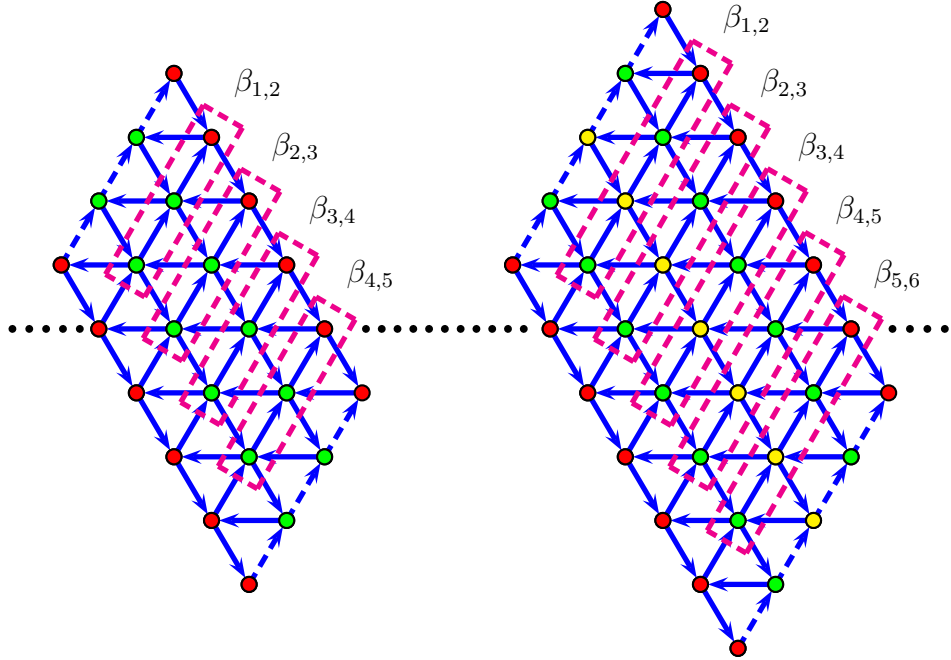


Before presenting the result of the braid-group transformation, we describe contributions of cluster variables located at sites of the  $\mathcal A_n$-quiver to a normalized element $a_{i,j}\in {\mathbb A}$ with $i<j$. Before normalization this element is homogeneous in frozen cluster variables $\rho_k\equiv Z_{0,k,n-k}$ and is proportional to the product $\prod_{k=0}^i \rho^2_k\prod_{l=i+1}^j \rho_l$. We normalize this element by dividing it by the product $\prod_{k=0}^i K_k\prod_{l=i+1}^j K^{1/2}_l$ of Casimirs (\ref{Cas-out}) thus  eliminating the dependence of ${\mathbb A}$ on $\rho_i$. This normalization changes the powers in which cluster variables enter sums over paths. We have eight domains in total in the leftmost part of Fig.~\ref{fi:gluing}: let us describe two of them. In the domain labeled ``$a$,'' nonnormalized variables enter with power $1$ and each of them enters exactly one Casimir $K_k$ with $k\le i$ and one Casimir $K_l$ with $i+1\le l\le j$, so the normalization decreases the power by $3/2$ and the total power with which these variables enter the normalized element is $-1/2$. In the uppermost domain labeled ``b,'' every element enters with power two into a nonnormalized sum over paths and it enters two Casimirs $K_k$ with $k\le i$, so the normalization add power $-2$ and the total power is zero.  We indicate powers by different hatchings, which overlap in the figure on the left; the resulting powers $-1/2$, $0$, and $1/2$ are indicated in the middle figure.

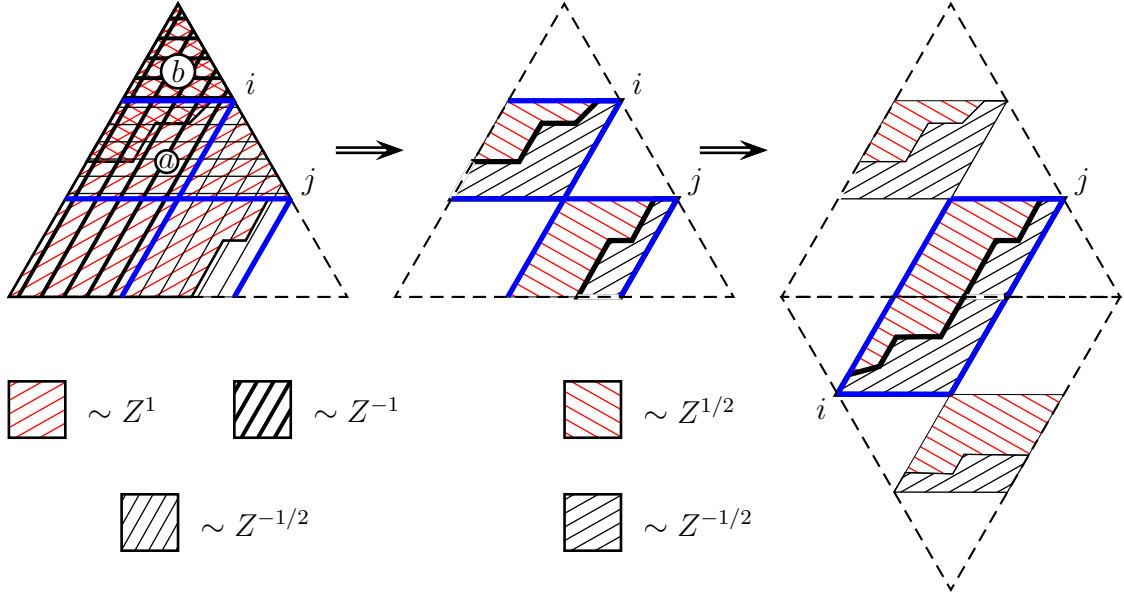
\begin{figure}[tb]
\begin{pspicture}(-2.5,-4)(2.5,4)
{\psset{unit=1.5}
\pspolygon[linestyle=dashed](-1.5,0)(0,2.6)(1.5,0)
\pspolygon[linestyle=solid, linewidth=1pt, fillstyle=vlines, hatchcolor=red, hatchwidth=0.5pt, hatchsep=4pt, hatchangle=60](0,2.6)(0.5,1.74)(0.3,1.74)(0.1,1.54)(-0.2,1.54)(-0.32,1.34)(-0.4,1.2)(-0.81,1.2)
\pspolygon[linestyle=solid, linewidth=1pt, fillstyle=vlines, hatchcolor=red, hatchwidth=0.5pt, hatchsep=4pt, hatchangle=-60](0,2.6)(1,0.87)(0.8,0.87)(0.6,0.5)(0.4,0.5)(0.12,0)(-1.5,0)
\pspolygon[linestyle=solid, linewidth=0pt, fillstyle=vlines, hatchcolor=black, hatchwidth=1.5pt, hatchsep=6pt, hatchangle=-30](0,2.6)(0.5,1.74)(-0.5,0)(-1.5,0)
\pspolygon[linestyle=solid, linewidth=0pt, fillstyle=vlines, hatchcolor=black, hatchwidth=1.5pt, hatchsep=6pt,hatchangle=90](0,2.6)(0.5,1.74)(-0.5,1.74)
\pspolygon[linestyle=solid, linewidth=0pt, fillstyle=vlines, hatchcolor=black, hatchwidth=0.5pt, hatchsep=6pt, hatchangle=-30](0.5,1.74)(1,0.87)(0.5,0)(-0.5,0)
\pspolygon[linestyle=solid, linewidth=0pt, fillstyle=vlines, hatchcolor=black, hatchwidth=0.5pt, hatchsep=6pt, hatchangle=90](0.5,1.74)(1,0.87)(-1,0.87)(-0.5,1.74)
\psline[linewidth=2pt,linestyle=solid,linecolor=blue](1,0.87)(-1,0.87)
\psline[linewidth=2pt,linestyle=solid,linecolor=blue](0.5,1.74)(-0.5,1.74)
\psline[linewidth=2pt,linestyle=solid,linecolor=blue](1,0.87)(0.5,0)
\psline[linewidth=2pt,linestyle=solid,linecolor=blue](0.5,1.74)(-0.5,0)
\put(1.7,1.3){\psline[doubleline=true,linewidth=1pt, doublesep=1pt, linecolor=black]{->}(-0.3,0)(0.3,0)}
\put(0.6,1.8){\makebox(0,0)[bl]{\hbox{{$i$}}}}
\put(1.1,.9){\makebox(0,0)[bl]{\hbox{{$j$}}}}
\pscircle[linestyle=solid, linewidth=1pt, fillstyle=solid, fillcolor=white](0,2){0.15}
\put(0,2){\makebox(0,0)[cc]{\hbox{{$b$}}}}
\pscircle[linestyle=solid, linewidth=1pt, fillstyle=solid, fillcolor=white](-0.1,1.2){0.1}
\put(-0.1,1.2){\makebox(0,0)[cc]{\hbox{{$a$}}}}
\put(-1,-1){
\pspolygon[linestyle=solid, linewidth=1pt, fillstyle=vlines, hatchcolor=red, hatchwidth=0.5pt, hatchsep=4pt, hatchangle=-60](-0.5,-0.25)(0,-0.25)(0,0.25)(-0.5,0.25)
\put(0.2,0){\makebox(0,0)[cl]{\hbox{{$\sim Z^{1}$}}}}
 }
\put(1,-1){
\pspolygon[linestyle=solid, linewidth=1pt, fillstyle=vlines, hatchcolor=black, hatchwidth=1.5pt, hatchsep=4pt, hatchangle=-30](-0.5,-0.25)(0,-0.25)(0,0.25)(-0.5,0.25)
\put(0.2,0){\makebox(0,0)[cl]{\hbox{{$\sim Z^{-1}$}}}}
 }
\put(0,-2){
\pspolygon[linestyle=solid, linewidth=1pt, fillstyle=vlines, hatchcolor=black, hatchwidth=0.5pt, hatchsep=4pt, hatchangle=-30](-0.5,-0.25)(0,-0.25)(0,0.25)(-0.5,0.25)
\put(0.2,0){\makebox(0,0)[cl]{\hbox{{$\sim Z^{-1/2}$}}}}
 }
}
\end{pspicture}
\begin{pspicture}(-2.5,-4)(2.5,4)
{\psset{unit=1.5}
\pspolygon[linestyle=solid, linewidth=0pt, fillstyle=vlines, hatchcolor=red, hatchwidth=0.5pt, hatchsep=4pt, hatchangle=60](-0.5,1.74)(0.5,1.74)(0.3,1.74)(0.1,1.54)(-0.2,1.54)(-0.32,1.34)(-0.4,1.2)(-0.81,1.2)
\pspolygon[linestyle=solid, linewidth=2pt, fillstyle=vlines, hatchcolor=black, hatchwidth=0.5pt, hatchsep=4pt, hatchangle=-60](-1,0.87)(0,0.87)(0.5,1.74)(0.3,1.74)(0.1,1.54)(-0.2,1.54)(-0.32,1.34)(-0.4,1.2)(-0.81,1.2)
\psline[linewidth=2pt,linestyle=solid,linecolor=white](-1,0.87)(-0.81,1.2)
\pspolygon[linestyle=solid, linewidth=0pt, fillstyle=vlines, hatchcolor=red, hatchwidth=0.5pt, hatchsep=4pt, hatchangle=60](-0.5,0)(0,0.87)(1,0.87)(0.8,0.87)(0.6,0.5)(0.4,0.5)(0.12,0)
\pspolygon[linestyle=solid, linewidth=2pt, fillstyle=vlines, hatchcolor=black, hatchwidth=0.5pt, hatchsep=4pt, hatchangle=-60](0.5,0)(1,0.87)(0.8,0.87)(0.6,0.5)(0.4,0.5)(0.12,0)
\psline[linewidth=2pt,linestyle=solid,linecolor=white](0.5,0)(0.12,0)
\psline[linewidth=2pt,linestyle=solid,linecolor=blue](1,0.87)(-1,0.87)
\psline[linewidth=2pt,linestyle=solid,linecolor=blue](0.5,1.74)(-0.5,1.74)
\psline[linewidth=2pt,linestyle=solid,linecolor=blue](1,0.87)(0.5,0)
\psline[linewidth=2pt,linestyle=solid,linecolor=blue](0.5,1.74)(-0.5,0)
\pspolygon[linestyle=dashed](-1.5,0)(0,2.6)(1.5,0)
\put(1.5,1.3){\psline[doubleline=true,linewidth=1pt, doublesep=1pt, linecolor=black]{->}(-0.3,0)(0.3,0)}
\put(0.6,1.8){\makebox(0,0)[bl]{\hbox{{$i$}}}}
\put(1.1,.9){\makebox(0,0)[bl]{\hbox{{$j$}}}}
\put(0.5,-1){
\pspolygon[linestyle=solid, linewidth=1pt, fillstyle=vlines, hatchcolor=red, hatchwidth=0.5pt, hatchsep=4pt, hatchangle=60](-0.5,-0.25)(0,-0.25)(0,0.25)(-0.5,0.25)
\put(0.2,0){\makebox(0,0)[cl]{\hbox{{$\sim Z^{1/2}$}}}}
 }
\put(0.5,-2){
\pspolygon[linestyle=solid, linewidth=1pt, fillstyle=vlines, hatchcolor=black, hatchwidth=0.5pt, hatchsep=4pt, hatchangle=-60](-0.5,-0.25)(0,-0.25)(0,0.25)(-0.5,0.25)
\put(0.2,0){\makebox(0,0)[cl]{\hbox{{$\sim Z^{-1/2}$}}}}
 }
}
\end{pspicture}
\begin{pspicture}(-2.5,-4)(2.5,4)
{\psset{unit=1.5}
\rput{60}(0.75,-1.3){
\pspolygon[linestyle=solid, linewidth=0pt, fillstyle=vlines, hatchcolor=red, hatchwidth=0.5pt, hatchsep=4pt, hatchangle=0](0.5,1.74)(-0.5,1.74)(-0.3,1.74)(-0.1,1.54)(0.2,1.54)(0.32,1.34)(0.4,1.2)(0.81,1.2)
\pspolygon[linestyle=solid, linewidth=2pt, fillstyle=vlines, hatchcolor=black, hatchwidth=0.5pt, hatchsep=4pt, hatchangle=60](1,0.87)(0,0.87)(-0.5,1.74)(-0.3,1.74)(-0.1,1.54)(0.2,1.54)(0.32,1.34)(0.4,1.2)(0.81,1.2)
\psline[linewidth=2pt,linestyle=solid,linecolor=white](1,0.87)(0.81,1.2)
\psline[linewidth=2pt,linestyle=solid,linecolor=blue](0,0.87)(1,0.87)
\psline[linewidth=2pt,linestyle=solid,linecolor=blue](-0.5,1.74)(0.5,1.74)
\psline[linewidth=2pt,linestyle=solid,linecolor=blue](-0.5,1.74)(0,0.87)
\pspolygon[linestyle=solid, linewidth=0pt, fillstyle=vlines, hatchcolor=red, hatchwidth=0.5pt, hatchsep=4pt, hatchangle=0](0.5,0)(0,0.87)(-1,0.87)(-0.8,0.87)(-0.6,0.5)(-0.4,0.5)(-0.12,0)
\pspolygon[linestyle=solid, linewidth=0.5pt, fillstyle=vlines, hatchcolor=black, hatchwidth=0.5pt, hatchsep=4pt, hatchangle=60](-0.5,0)(-1,0.87)(-0.8,0.87)(-0.6,0.5)(-0.4,0.5)(-0.12,0)
\pspolygon[linestyle=dashed](-1.5,0)(0,2.6)(1.5,0)
}
\pspolygon[linestyle=solid, linewidth=0pt, fillstyle=vlines, hatchcolor=red, hatchwidth=0.5pt, hatchsep=4pt, hatchangle=60](-0.5,1.74)(0.5,1.74)(0.3,1.74)(0.1,1.54)(-0.2,1.54)(-0.32,1.34)(-0.4,1.2)(-0.81,1.2)
\pspolygon[linestyle=solid, linewidth=0.5pt, fillstyle=vlines, hatchcolor=black, hatchwidth=0.5pt, hatchsep=4pt, hatchangle=-60](-1,0.87)(0,0.87)(0.5,1.74)(0.3,1.74)(0.1,1.54)(-0.2,1.54)(-0.32,1.34)(-0.4,1.2)(-0.81,1.2)
\psline[linewidth=2pt,linestyle=solid,linecolor=white](-1,0.87)(-0.81,1.2)
\pspolygon[linestyle=solid, linewidth=0pt, fillstyle=vlines, hatchcolor=red, hatchwidth=0.5pt, hatchsep=4pt, hatchangle=60](-0.5,0)(0,0.87)(1,0.87)(0.8,0.87)(0.6,0.5)(0.4,0.5)(0.12,0)
\pspolygon[linestyle=solid, linewidth=2pt, fillstyle=vlines, hatchcolor=black, hatchwidth=0.5pt, hatchsep=4pt, hatchangle=-60](0.5,0)(1,0.87)(0.8,0.87)(0.6,0.5)(0.4,0.5)(0.12,0)
\psline[linewidth=2pt,linestyle=solid,linecolor=white](0.47,0)(0.14,0)
\psline[linewidth=2pt,linestyle=solid,linecolor=blue](1,0.87)(0,0.87)
\psline[linewidth=2pt,linestyle=solid,linecolor=blue](1,0.87)(0.5,0)
\psline[linewidth=2pt,linestyle=solid,linecolor=blue](0,0.87)(-0.5,0)
\pspolygon[linestyle=dashed](-1.5,0)(0,2.6)(1.5,0)
\put(-1.1,-0.9){\makebox(0,0)[tr]{\hbox{{$i$}}}}
\put(1.1,.9){\makebox(0,0)[bl]{\hbox{{$j$}}}}
}
\end{pspicture}
\caption{Paths contributing to $a_{i,j}$ in the glued  $\mathcal A_n$-quivers. In the left picture we indicate contributions of cluster variables of $SL_n$ quiver into a normalized element $a_{i,j}$: we schematically draw two paths: the upper path corresponds to $[M_1]_{k,i}$ and the lower path corresponds to $[M_2]_{k,j}$. All cluster variables above the first path enter in power two and all variables between two path enter with power one into a nonnormalized expression; we then normalize it by the products of cluster variables entering the product of Casimirs $\prod_{k=0}^i K_k\prod_{l=i+1}^j K^{1/2}_l$. The resulting pattern is presented in the middle picture: cluster variables enter with powers $1/2$ or $-1/2$ and variables from empty areas do not contribute. In the rightmost figure we take another copy of the $\mathcal A_n$-quiver and attach it as in Fig.~\ref{fi:covering}; two domains in the original triangle then constitute a parallelogram with a continuous path joining its opposite vertices.}
\label{fi:gluing}
\end{figure}

We then glue two copies of the $SL_n$ triangle in the rightmost part of Fig.~\ref{fi:gluing}: the union of two domains containing cluster variables contributing into the normalized element $a_{i,j}$ is then a parallelogram with sides of positive lengths $j-i$ and $n+i-j$, and in order to obtain the element $a_{i,j}$ we have to take a sum over all paths inside this parallelogram starting at the vertex of the dual lattice located ``beyond'' NE vertex $j$ and terminating at the vertex of the dual lattice located ``beyond'' SW vertex $i$ (a standard exercise in combinatorics is that we have exactly $\Bigl[ {n\atop j-i}\Bigr]$ such paths, cf. toy Example~\ref{ex:toy}).

{We present a more detailed picture below. In this picture, we indicate a part of the directed network inside which we take a sum over paths from $j$ to $i$  contributing to $a_{i,j}$ {(an example of such path is shown in light color in the Figure)}; all contributing cluster variables are confined inside the corresponding parallelogram. All cluster variables inside the parallelogram and  above a path enter with the power $1/2$ and all cluster variables inside the parallelogram and below the path enter with the power $-1/2$. All variables outside the parallelogram do not contribute.}
$$
{\psset{unit=1}
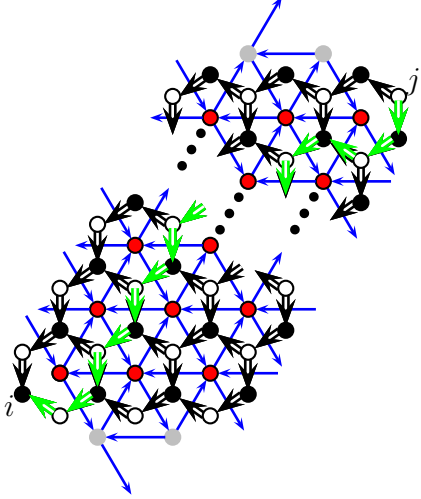
\begin{pspicture}(-5,-3)(5,3)
\newcommand{\PATGEN}{%
{\psset{unit=1}
\rput(0,0){\psline[linecolor=blue,linewidth=1pt]{->}(0,0)(.45,.765)}
\rput(0,0){\psline[linecolor=blue,linewidth=1pt]{->}(0.9,0)(0.1,0)}
\rput(0,0){\psline[linecolor=blue,linewidth=1pt]{->}(0,0)(.45,-.765)}
\put(0,0){\pscircle[fillstyle=solid,fillcolor=red]{.1}}
\psline[doubleline=true,linewidth=1pt, doublesep=1pt, linecolor=black]{->}(0.5,0.23)(0.5,-0.23)
\psline[doubleline=true,linewidth=1pt, doublesep=1pt, linecolor=black]{->}(0.415,0.285)(0.085,0.52)
\psline[doubleline=true,linewidth=1pt, doublesep=1pt, linecolor=black]{->}(0.415,-0.285)(0.085,-0.52)
\put(0.5,0.28){\pscircle[fillstyle=solid,fillcolor=white]{.1}}
\put(0.5,-0.28){\pscircle[fillstyle=solid,fillcolor=black]{.1}}
}}
\newcommand{\PATBOT}{%
{\psset{unit=1}
\rput(0,0){\psline[linecolor=blue,linewidth=1pt]{->}(0,0)(.45,.765)}
\rput(0,0){\psline[linecolor=blue,linewidth=1pt]{->}(0.9,0)(0.1,0)}
\rput(0,0){\psline[linecolor=blue,linewidth=1pt]{->}(0,0)(.45,-.765)}
\put(0,0){\pscircle[fillstyle=solid,linecolor=lightgray,fillcolor=lightgray]{.1}}
\psline[doubleline=true,linewidth=1pt, doublesep=1pt, linecolor=black]{->}(0.415,0.285)(0.085,0.52)
\put(0.5,0.28){\pscircle[fillstyle=solid,fillcolor=white]{.1}}
}}
\newcommand{\PATUP}{%
{\psset{unit=1}
\rput(0,0){\psline[linecolor=blue,linewidth=1pt]{->}(0,0)(.45,.765)}
\rput(0,0){\psline[linecolor=blue,linewidth=1pt]{->}(0.9,0)(0.1,0)}
\rput(0,0){\psline[linecolor=blue,linewidth=1pt]{->}(0,0)(.45,-.765)}
\put(0,0){\pscircle[fillstyle=solid,linecolor=lightgray,fillcolor=lightgray]{.1}}
\psline[doubleline=true,linewidth=1pt, doublesep=1pt, linecolor=black]{->}(0.415,-0.285)(0.085,-0.52)
\put(0.5,-0.28){\pscircle[fillstyle=solid,fillcolor=black]{.1}}
}}
\newcommand{\PATCENTER}{%
{\psset{unit=1}
\rput(0,0){\psline[linecolor=blue,linewidth=1pt]{->}(0,0)(.45,.765)}
\rput(0,0){\psline[linecolor=blue,linewidth=1pt]{->}(0.9,0)(0.1,0)}
\put(0,0){\pscircle[fillstyle=solid,fillcolor=red]{.1}}
\psline[doubleline=true,linewidth=1pt, doublesep=1pt, linecolor=black]{->}(0.5,0.23)(0.5,-0.23)
\psline[doubleline=true,linewidth=1pt, doublesep=1pt, linecolor=black]{->}(0.415,0.285)(0.085,0.52)
\put(0.5,0.28){\pscircle[fillstyle=solid,fillcolor=white]{.1}}
}}
\newcommand{\PATRIGHT}{%
{\psset{unit=1}
\rput(0,0){\psline[linecolor=blue,linewidth=1pt]{->}(0,0)(.45,-.765)}
\put(0,0){\pscircle[fillstyle=solid,fillcolor=red]{.1}}
\psline[doubleline=true,linewidth=1pt, doublesep=1pt, linecolor=black]{<-}(-0.415,0.285)(-0.085,0.52)
\psline[doubleline=true,linewidth=1pt, doublesep=1pt, linecolor=black]{->}(0.415,-0.285)(0.085,-0.52)
}}
\newcommand{\PATBOTTOM}{%
{\psset{unit=1}
\rput(0,0){\psline[linecolor=blue,linewidth=1pt]{->}(0,0)(.45,.765)}
\rput(0,0){\psline[linecolor=blue,linewidth=1pt]{->}(0.9,0)(0.1,0)}
\rput(0,0){\psline[linecolor=blue,linewidth=1pt]{->}(0,0)(.45,-.765)}
\put(0,0){\pscircle[fillstyle=solid,fillcolor=red]{.1}}
\psline[doubleline=true,linewidth=1pt, doublesep=1pt, linecolor=black]{->}(0.5,0.23)(0.5,-0.23)
\psline[doubleline=true,linewidth=1pt, doublesep=1pt, linecolor=black]{->}(0.415,0.285)(0.085,0.52)
\psline[doubleline=true,linewidth=1pt, doublesep=1pt, linecolor=black]{->}(0.415,-0.285)(0.085,-0.52)
\put(0.5,0.28){\pscircle[fillstyle=solid,fillcolor=white]{.1}}
\put(0.5,-0.28){\pscircle[fillstyle=solid,fillcolor=black]{.1}}
}}
\newcommand{\PATBOTBOT}{%
{\psset{unit=1}
\rput(0,0){\psline[linecolor=blue,linewidth=1pt]{->}(0,0)(.45,.765)}
\put(0,0){\pscircle[fillstyle=solid,linecolor=lightgray,fillcolor=lightgray]{.1}}
\psline[doubleline=true,linewidth=1pt, doublesep=1pt, linecolor=black]{->}(0.415,0.285)(0.085,0.52)
\put(0.5,0.28){\pscircle[fillstyle=solid,fillcolor=white]{.1}}
}}
\newcommand{\NEANGLE}{%
{\psset{unit=1}
\rput(0,0){\psline[linecolor=blue,linewidth=1pt]{->}(0,0)(.45,-.765)}
\rput(0,0){\psline[linecolor=blue,linewidth=1pt]{->}(-0.45,0.765)(-0.05,0.085)}
\put(0,0){\pscircle[fillstyle=solid,fillcolor=red]{.1}}
\put(-0.5,0.85){\pscircle[fillstyle=solid,linecolor=lightgray,fillcolor=lightgray]{.1}}
\psline[doubleline=true,linewidth=1pt, doublesep=1pt, linecolor=black]{->}(0.5,0.23)(0.5,-0.23)
\psline[doubleline=true,linewidth=1pt, doublesep=1pt, linecolor=black]{->}(0.415,0.285)(0.085,0.52)
\psline[doubleline=true,linewidth=1pt, doublesep=1pt, linecolor=black]{->}(0.415,-0.285)(0.085,-0.52)
\psline[doubleline=true,linewidth=1pt, doublesep=1pt, linecolor=black]{<-}(-0.415,0.285)(-0.085,0.52)
\put(0.5,0.28){\pscircle[fillstyle=solid,fillcolor=white]{.1}}
\put(0.5,-0.28){\pscircle[fillstyle=solid,fillcolor=black]{.1}}
\put(0,0.57){\pscircle[fillstyle=solid,fillcolor=black]{.1}}
}}
\newcommand{\ARRD}{%
{\psset{unit=1}
\psline[doubleline=true,linewidth=1pt, doublesep=1pt, linecolor=green]{->}(-0.5,0.23)(-0.5,-0.23)
}}
\newcommand{\ARRLD}{%
{\psset{unit=1}
\psline[doubleline=true,linewidth=1pt, doublesep=1pt, linecolor=green]{<-}(-0.415,-0.285)(-0.085,-0.52)
}}
\newcommand{\ARRLU}{%
{\psset{unit=1}
\psline[doubleline=true,linewidth=1pt, doublesep=1pt, linecolor=green]{<-}(-0.415,0.285)(-0.085,0.52)
}}
\newcommand{\SWANGLE}{%
{\psset{unit=1}
\rput(0,0){\psline[linecolor=blue,linewidth=1pt]{->}(-0.45,0.765)(-0.05,0.085)}
\put(0,0){\pscircle[fillstyle=solid,fillcolor=red]{.1}}
\psline[doubleline=true,linewidth=1pt, doublesep=1pt, linecolor=black]{->}(-0.5,0.23)(-0.5,-0.23)
\psline[doubleline=true,linewidth=1pt, doublesep=1pt, linecolor=black]{<-}(-0.415,-0.285)(-0.085,-0.52)
\psline[doubleline=true,linewidth=1pt, doublesep=1pt, linecolor=black]{<-}(-0.415,0.285)(-0.085,0.52)
\put(-0.5,0.28){\pscircle[fillstyle=solid,fillcolor=white]{.1}}
\put(-0.5,-0.28){\pscircle[fillstyle=solid,fillcolor=black]{.1}}
\put(0,0.57){\pscircle[fillstyle=solid,fillcolor=black]{.1}}
\put(0,-0.57){\pscircle[fillstyle=solid,fillcolor=white]{.1}}
}}
\newcommand{\PATUPPER}{%
{\psset{unit=1}
\rput(0,0){\psline[linecolor=blue,linewidth=1pt]{->}(-0.1,0)(-0.8,0)}
\rput(0,0){\psline[linecolor=blue,linewidth=1pt]{->}(-0.45,0.765)(-0.05,0.085)}
\psline[doubleline=true,linewidth=1pt, doublesep=1pt, linecolor=black]{->}(-0.5,0.23)(-0.5,-0.23)
\psline[doubleline=true,linewidth=1pt, doublesep=1pt, linecolor=black]{<-}(-0.415,0.285)(-0.085,0.52)
\put(-0.5,0.28){\pscircle[fillstyle=solid,fillcolor=white]{.1}}
\put(0,0.57){\pscircle[fillstyle=solid,fillcolor=black]{.1}}
}}
\newcommand{\DOTS}{%
{\psset{unit=1}
\put(0,0){\pscircle[fillstyle=solid,fillcolor=black]{0.05}}
\put(-0.25,0){\pscircle[fillstyle=solid,fillcolor=black]{0.05}}
\put(0.25,0){\pscircle[fillstyle=solid,fillcolor=black]{0.05}}
}}
%
\multiput(0,0)(0.5,-0.85){1}{\PATRIGHT}
\multiput(0,-1.7)(0.5,0.85){2}{\PATBOTTOM}
\multiput(-0.5,-2.55)(0.5,0.85){1}{\PATBOTBOT}
\multiput(1.5,0.85)(0.5,0.85){1}{\PATBOTTOM}
\multiput(-2,-1.7)(0.5,0.85){3}{\PATGEN}
\multiput(0,1.7)(0.5,0.85){1}{\PATGEN}
\multiput(0.5,2.55)(0.5,0.85){1}{\PATUP}
\multiput(-1,-1.7)(0.5,0.85){2}{\PATGEN}
\multiput(-1.5,-2.55)(0.5,0.85){1}{\PATBOT}
\multiput(1,1.7)(0.5,0.85){1}{\PATGEN}
\multiput(0.5,0.85)(0.5,0.85){1}{\PATCENTER}
\multiput(-0.25,1.275)(0.5,-0.85){2}{\rput{60}(0,0){\DOTS}}
\multiput(2,1.7)(0.5,0.85){1}{\NEANGLE}
\multiput(-2,-1.7)(0.5,0.85){1}{\SWANGLE}
\multiput(-1.5,-0.85)(0.5,0.85){2}{\PATUPPER}
\multiput(0,1.7)(0.5,0.85){1}{\PATUPPER}
\rput{60}(1.25,0.425){\DOTS}
\put(3.0,1.7){\ARRD}
\put(2.5,0.9){\ARRLU}
\put(2.,1.7){\ARRLD}
\put(1.5,0.9){\ARRLU}
\put(1.5,0.9){\ARRD}
\put(0.,0.){\ARRLU}
\put(0.,0.){\ARRD}
\put(-0.5,-0.8){\ARRLU}
\put(-0.5,-0.8){\ARRD}
\put(-1.0,-1.65){\ARRLU}
\put(-1.0,-1.65){\ARRD}
\put(-1.5,-2.5){\ARRLU}
\put(-2.,-1.7){\ARRLD}
%
%
\put(-2.6,-2){\makebox(0,0)[tr]{\hbox{{$i$}}}}
\put(2.6,2){\makebox(0,0)[bl]{\hbox{{$j$}}}}
\end{pspicture}
}
$$

We now explore how cluster variables transform under chains of mutations $\beta_{i,i+1}$. Recall that mutation $\mu_Z$ transforms any variable  $Y$ at the head of  an outgoing arrow   $Z\rightarrow Y$ as $Y\mapsto Y(1+Z)$, whereas a variable $X$ joined to $Z$ by an incoming arrow, $Z\leftarrow X$ transforms as $X\to X(1+Z^{-1})^{-1}$. Finally, $\mu_Z(Z)=Z^{-1}$ and quiver mutation is standard (\cite{FZ}). 

Till the end of this section we let $r=n-3$ for brevity.
Consider the following sequence of mutations $\beta=\mu_{B_1}\dots\mu_{B_r}\mu_{S_2}\mu_{B_r}\dots\mu_{B_2}\mu_{B_1}$ (see Fig.~\ref{fi:mutation}). 
The net result of this chain of mutations is shown on the right hand side of Fig.~\ref{fi:mutation}. Note that the resulting quiver is isomorphic to the original one when all the mutated variables except $S_1$ and $S_2$ retain their positions, while the boundary variables $S_1$ and $S_2$ are permuted. 

The following lemma is proved by a direct calculation.

\begin{lemma}\label{lm:mutation}
In the notation of Fig.~\ref{fi:mutation}, cluster variables transform as follows (recall that for brevity we set $r:=n-3$): 
\begin{align*}
&B'_k=B_k\frac{\eta_{k+2}}{\eta_k},\ k=1,\dots,r;\quad A'_k=A_k\frac{\eta_{k+1}}{\eta_{k+2}},\ k=0,\dots,r-1;\quad C'_k=C_k\frac{\eta_{k+1}}{\eta_{k+2}},\ k=1,\dots,r;\\
&A'_r=C'_0=A_r \frac{\eta_{n+1}\eta_1}{\eta_2}; \quad S'_1=\frac{S_1S^2_2 B_1\cdots B_r}{\eta_{r+1}\eta_1};\quad S'_2=\frac{\eta_2}{S_2B_1\cdots B_r},
\end{align*}
where
\begin{align}
&\eta_{r+2}=1,\quad \eta_{r+1}=1+S_2,\quad \eta_r=1+S_2+S_2B_r,\nonumber \\ 
&\eta_{r-1}=1+S_2+S_2B_r+S_2B_rB_{r-1},\ \dots, \ 
\eta_1=1+S_2+\dots+S_2B_r\cdots B_1. 
\label{eta}
\end{align}
\end{lemma} 

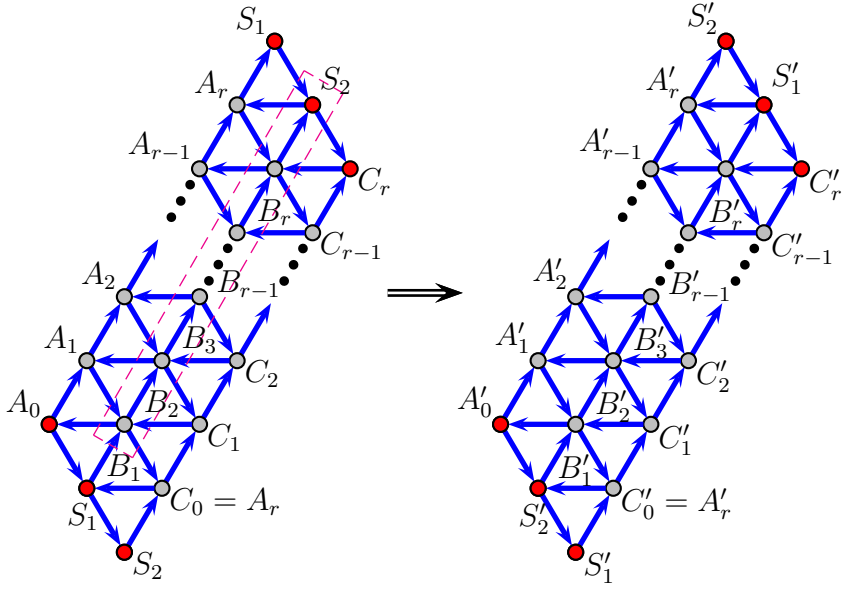
\begin{figure}[tb]
{\psset{unit=1}
\begin{pspicture}(-5,-4)(5,4)
\newcommand{\PATGEN}{%
{\psset{unit=1}
\rput(0,0){\psline[linecolor=blue,linewidth=2pt]{->}(0,0)(.45,.765)}
\rput(0,0){\psline[linecolor=blue,linewidth=2pt]{->}(0.9,0)(0.1,0)}
\rput(0,0){\psline[linecolor=blue,linewidth=2pt]{->}(0,0)(.45,-.765)}
\put(0,0){\pscircle[fillstyle=solid,fillcolor=lightgray]{.1}}
}}
\newcommand{\PATCENTER}{%
{\psset{unit=1}
\rput(0,0){\psline[linecolor=blue,linewidth=2pt]{->}(0,0)(.45,.765)}
\rput(0,0){\psline[linecolor=blue,linewidth=2pt]{->}(0.9,0)(0.1,0)}
\put(0,0){\pscircle[fillstyle=solid,fillcolor=lightgray]{.1}}
}}
\newcommand{\PATRIGHT}{%
{\psset{unit=1}
\rput(0,0){\psline[linecolor=blue,linewidth=2pt]{->}(0,0)(.45,-.765)}
\put(0,0){\pscircle[fillstyle=solid,fillcolor=lightgray]{.1}}
}}
\newcommand{\PATBOTTOM}{%
{\psset{unit=1}
\rput(0,0){\psline[linecolor=blue,linewidth=2pt]{->}(0,0)(.45,.765)}
\put(0,0){\pscircle[fillstyle=solid,fillcolor=lightgray]{.1}}
}}
\newcommand{\DASHEDBOX}{%
{\psset{unit=1}
\psline[linecolor=magenta,linewidth=2pt,linestyle=dashed](-1.3,0.3)(1.3,0.3)
\psline[linecolor=magenta,linewidth=2pt,linestyle=dashed](-1.3,-0.3)(1.3,-0.3)
\psline[linecolor=magenta,linewidth=2pt,linestyle=dashed](-1.3,0.3)(-1.3,-0.3)
\psline[linecolor=magenta,linewidth=2pt,linestyle=dashed](1.3,0.3)(1.3,-0.3)
}}
\newcommand{\DASHEDBOXLONG}{%
{\psset{unit=1}
\psline[linecolor=magenta,linewidth=0.5pt,linestyle=dashed](-2.3,0.3)(3.3,0.3)
\psline[linecolor=magenta,linewidth=0.5pt,linestyle=dashed](-2.3,-0.3)(3.3,-0.3)
\psline[linecolor=magenta,linewidth=0.5pt,linestyle=dashed](-2.3,0.3)(-2.3,-0.3)
\psline[linecolor=magenta,linewidth=0.5pt,linestyle=dashed](3.3,0.3)(3.3,-0.3)
}}
\newcommand{\DOTS}{%
{\psset{unit=1}
\put(0,0){\pscircle[fillstyle=solid,fillcolor=black]{0.05}}
\put(-0.25,0){\pscircle[fillstyle=solid,fillcolor=black]{0.05}}
\put(0.25,0){\pscircle[fillstyle=solid,fillcolor=black]{0.05}}
}}
\rput(-3,0){
\multiput(1,3.4)(0.5,-0.85){2}{\PATRIGHT}
\multiput(0,0)(0.5,-0.85){1}{\PATRIGHT}
\multiput(-1,-3.4)(0.5,0.85){4}{\PATBOTTOM}
\multiput(1.5,0.85)(0.5,0.85){1}{\PATBOTTOM}
\multiput(-2,-1.7)(0.5,0.85){3}{\PATGEN}
\multiput(0,1.7)(0.5,0.85){2}{\PATGEN}
\multiput(-1.5,-2.55)(0.5,0.85){3}{\PATGEN}
\multiput(1,1.7)(0.5,0.85){1}{\PATGEN}
\multiput(0.5,0.85)(0.5,0.85){1}{\PATCENTER}
\multiput(-0.25,1.275)(0.5,-0.85){2}{\rput{60}(0,0){\DOTS}}
\rput{60}(1.25,0.425){\DOTS}
\multiput(-2,-1.7)(0.5,-0.85){3}{\pscircle[fillstyle=solid,fillcolor=red](0,0){.1}}
\multiput(1,3.4)(0.5,-0.85){3}{\pscircle[fillstyle=solid,fillcolor=red](0,0){.1}}
\rput{60}(0,0){\DASHEDBOXLONG}
\put(-2.1,-1.6){\makebox(0,0)[br]{\hbox{{$A_0$}}}}
\put(-1.6,-0.75){\makebox(0,0)[br]{\hbox{{$A_1$}}}}
\put(-1.1,0.1){\makebox(0,0)[br]{\hbox{{$A_2$}}}}
\put(-0.1,1.8){\makebox(0,0)[br]{\hbox{{$A_{r-1}$}}}}
\put(0.4,2.65){\makebox(0,0)[br]{\hbox{{$A_r$}}}}
\put(0.9,3.5){\makebox(0,0)[br]{\hbox{{$S_1$}}}}
\put(-1,-2.1){\makebox(0,0)[tc]{\hbox{{$B_1$}}}}
\put(-0.5,-1.25){\makebox(0,0)[tc]{\hbox{{$B_2$}}}}
\put(0,-0.4){\makebox(0,0)[tc]{\hbox{{$B_3$}}}}
\put(0.65,0.35){\makebox(0,0)[tc]{\hbox{{$B_{r{-}1}$}}}}
\put(1,1.3){\makebox(0,0)[tc]{\hbox{{$B_r$}}}}
\put(1.6,2.7){\makebox(0,0)[bl]{\hbox{{$S_2$}}}}
\put(-1.55,-2.75){\makebox(0,0)[tc]{\hbox{{$S_1$}}}}
\put(-0.9,-3.4){\makebox(0,0)[tl]{\hbox{{$S_2$}}}}
\put(-0.4,-2.55){\makebox(0,0)[tl]{\hbox{{$C_0=A_r$}}}}
\put(0.1,-1.7){\makebox(0,0)[tl]{\hbox{{$C_1$}}}}
\put(0.6,-0.85){\makebox(0,0)[tl]{\hbox{{$C_2$}}}}
\put(1.6,0.85){\makebox(0,0)[tl]{\hbox{{$C_{r-1}$}}}}
\put(2.1,1.7){\makebox(0,0)[tl]{\hbox{{$C_r$}}}}
}
\put(0,0){\psline[doubleline=true,linewidth=1pt, doublesep=1pt, linecolor=black]{->}(-0.5,0)(0.5,0)}
\rput(3,0){
\multiput(1,3.4)(0.5,-0.85){2}{\PATRIGHT}
\multiput(0,0)(0.5,-0.85){1}{\PATRIGHT}
\multiput(-1,-3.4)(0.5,0.85){4}{\PATBOTTOM}
\multiput(1.5,0.85)(0.5,0.85){1}{\PATBOTTOM}
\multiput(-2,-1.7)(0.5,0.85){3}{\PATGEN}
\multiput(0,1.7)(0.5,0.85){2}{\PATGEN}
\multiput(-1.5,-2.55)(0.5,0.85){3}{\PATGEN}
\multiput(1,1.7)(0.5,0.85){1}{\PATGEN}
\multiput(0.5,0.85)(0.5,0.85){1}{\PATCENTER}
\multiput(-0.25,1.275)(0.5,-0.85){2}{\rput{60}(0,0){\DOTS}}
\rput{60}(1.25,0.425){\DOTS}
\multiput(-2,-1.7)(0.5,-0.85){3}{\pscircle[fillstyle=solid,fillcolor=red](0,0){.1}}
\multiput(1,3.4)(0.5,-0.85){3}{\pscircle[fillstyle=solid,fillcolor=red](0,0){.1}}
\put(-2.1,-1.6){\makebox(0,0)[br]{\hbox{{$A'_0$}}}}
\put(-1.6,-0.75){\makebox(0,0)[br]{\hbox{{$A'_1$}}}}
\put(-1.1,0.1){\makebox(0,0)[br]{\hbox{{$A'_2$}}}}
\put(-0.1,1.8){\makebox(0,0)[br]{\hbox{{$A'_{r-1}$}}}}
\put(0.4,2.65){\makebox(0,0)[br]{\hbox{{$A'_r$}}}}
\put(0.9,3.5){\makebox(0,0)[br]{\hbox{{$S'_2$}}}}
\put(-1,-2.1){\makebox(0,0)[tc]{\hbox{{$B'_1$}}}}
\put(-0.5,-1.25){\makebox(0,0)[tc]{\hbox{{$B'_2$}}}}
\put(0,-0.4){\makebox(0,0)[tc]{\hbox{{$B'_3$}}}}
\put(0.65,0.35){\makebox(0,0)[tc]{\hbox{{$B'_{r{-}1}$}}}}
\put(1,1.3){\makebox(0,0)[tc]{\hbox{{$B'_r$}}}}
\put(1.6,2.7){\makebox(0,0)[bl]{\hbox{{$S'_1$}}}}
\put(-1.55,-2.75){\makebox(0,0)[tc]{\hbox{{$S'_2$}}}}
\put(-0.9,-3.4){\makebox(0,0)[tl]{\hbox{{$S'_1$}}}}
\put(-0.4,-2.55){\makebox(0,0)[tl]{\hbox{{$C'_0=A'_r$}}}}
\put(0.1,-1.7){\makebox(0,0)[tl]{\hbox{{$C'_1$}}}}
\put(0.6,-0.85){\makebox(0,0)[tl]{\hbox{{$C'_2$}}}}
\put(1.6,0.85){\makebox(0,0)[tl]{\hbox{{$C'_{r-1}$}}}}
\put(2.1,1.7){\makebox(0,0)[tl]{\hbox{{$C'_r$}}}}
}
%
\end{pspicture}
}
\caption{The transformation of variables under a braid group transformation $\beta_{i,i+1}$. We indicate only variables that are transformed under the corresponding chain of mutations. Note that the cluster variables $C_0$ and $A_r$ (and therefore $C'_0$ and $A'_r$) are identified.} 
\label{fi:mutation}
\end{figure}

Note first that Casimirs (\ref{Cas-in}) of the $\mathcal A_n$-quiver are invariant under the transformation in Lemma~\ref{lm:mutation}. This immediately follows from the equalities
\begin{align}
&A'_kB'_kC'_{k-1}=A_kB_kC_{k-1},\ k=2,\dots,r-1,\nonumber\\
&A'_1B'_1A'_rB'_rC'_r=A_1B_1A_rB_rC_r,\nonumber \\
 &\qquad\qquad\hbox{and}\\
&S'_1S'_2A'_0C'_r=S_1S_2A_0C_r.\nonumber
\end{align}

We now formulate the main statement
\begin{theorem}\label{th:braid}
The cluster transformations in Lemma~\ref{lm:mutation} generate the braid-group transformations for entries $a_{i,j}$ of the (classical) matrix $\mathbb A$. 
\end{theorem}
{\bf Proof}. The pivotal calculation is the transformation of quantities $\eta_k$ defined in (\ref{eta}) under transformations in Lemma~\ref{lm:mutation}. 
Let us denote by ${\mathcal G}$ the non-normalized transport matrix corresponding to matrix $\mathbb A$.
First, let us compute the (non-normalized) element $G_{i,i+1}$ of $\mathcal G$:
\be\label{GII+1}
G_{i,i+1}:=\eta_1+S_2B_r\cdots B_1S_1=1+S_2+S_2B_r+\cdots + S_2B_r\cdots B_1+S_2B_r\cdots B_1S_1.
\ee
Then, after mutation sequence $\beta_{i,i+1}$, we obtain
\begin{align*}
\eta'_k&=1+S'_1+S'_1B'_r+S'_1B'_rB'_{r-1}+\cdots+S'_1B'_rB'_{r-1}\cdots B'_{k}\\
&=1+\frac{S_1S_2^2B_r\cdots B_1}{\eta_{r+1}\eta_1}\Bigl(1+\frac{B_r}{\eta_r}+B_rB_{r-1}\frac{\eta_{r+1}}{\eta_r \eta_{r-1}}+\cdots+B_rB_{r-1}\cdots B_k\frac{\eta_{r+1}}{\eta_{k+1}\eta_k} \Bigr)\\
&(\hbox{note that } 1+{B_r}/{\eta_r}=(1+S_2)(1+B_r)/\eta_r=\eta_{r+1}(1+B_r)/\eta_r)\\
&=1+\frac{S_1S_2^2B_r\cdots B_1}{\eta_1}\Bigl(\frac{1+B_r}{\eta_r}+\frac{B_rB_{r-1}}{\eta_r \eta_{r-1}}+\cdots+\frac{B_rB_{r-1}\cdots B_k}{\eta_{k+1}\eta_k} \Bigr)\\
&\left(\hbox{note that } \frac{1+B_r}{\eta_r}+\frac{B_rB_{r-1}}{\eta_r \eta_{r-1}}=\frac{1}{\eta_r\eta_{r-1}}\bigl((1+B_r)(1+S_2+S_2B_r+S_2B_rB_{r-1})+B_rB_{r-1} \bigr)\right.\\
&=\left.\frac{1}{\eta_r\eta_{r-1}}(1+B_r+B_rB_{r-1})\eta_r=\frac{1+B_r+B_rB_{r-1}}{\eta_{r-1}}\right)\\
&=\cdots =1+\frac{S_1S_2^2B_r\cdots B_1}{\eta_1}\frac{1+B_r+B_rB_{r-1}+\dots+B_rB_{r-1}\cdots B_k}{\eta_k}\\
&=1+\frac{S_1S_2B_r\cdots B_1(\eta_k-1)}{\eta_1\eta_k}=\frac{G_{i,i+1}}{\eta_1}-\frac{S_1S_2B_r\cdots B_1}{\eta_1\eta_k}.
\end{align*} 
We therefore obtain that
\be
\eta'_k=\frac{G_{i,i+1}}{\eta_1}-\frac{S_1S_2B_r\cdots B_1}{\eta_1\eta_k},\ k=1,\dots,r+2,
\ee
and
\be
G'_{i,i+1}=\eta'_r+S'_1B'_rB'_{r-1}\cdots B'_{1}S'_2=\frac{G_{i,i+1}}{\eta_1}.
\ee
We consider several cases of matrix entries $a_{ij}$; the rest we leave for the reader. Note, first, the relations for triples of the cluster variables:
\be\label{triples}
A'_{k-1}B'_kC'_k=A_{k-1}B_kC_k,\ k=1,\dots,r,\ \hbox{and}\ S'_1S'_2A'_r=S_1S_2A_r.
\ee
In particular, these relations imply that the total product of all cluster variables is conserved.

For all elements $a_{i,j}$ we take into account their normalization by taking the sum over paths weighted by products of cluster variables (in power one) inside the corresponding parallelogram and above the path and dividing this sum by the product of all cluster variables inside the parallelogram taken with power $1/2$.

We begin with the element 
$$
a_{i,i+1}=\bigl(S_2B_r\cdots B_1S_1 \bigr)^{-1/2}G_{i,i+1}.
$$
Since $S'_1B'_r\cdots B'_1S'_2=S_2B_r\cdots B_1S_1 \eta_1^{-2}$ we have that
$$
a'_{i,i+1}=\bigl(S'_1B'_r\cdots B'_1S'_2 \bigr)^{-1/2}G'_{i,i+1}=a_{i,i+1},
$$
so, as expected, this element is preserved by the braid-group transformation $\beta_{i,i+1}$.

We next consider an arbitrary element $a_{l,m}$ with $l,m\ne i,i+1$. Note first that the normalizing factor for any such element is a product of triples of cluster variables (\ref{triples}) taken either in powers $1/2$ or zero; since all these triples are preserved by the transformation, all such factors are invariant under the transformation. It suffices therefore to consider a nonnormalized sum over paths contributing to $a_{l,m}$.  Consider a contribution of cluster variables to  paths that enter the pattern in Fig.~\ref{fi:mutation} from the right between elements $C_{p-1}$ and $C_p$ and exit from the left between elements $A_{k-1}$ and $A_{k-2}$ (with $k\le p$). This path may cross the ``$B$-line'' anywhere between $B_{k-1}$ and $B_p$ and we have to take a sum over all possible variants. The corresponding contribution therefore has the form
\begin{align*}
\Pi'_{k,p}&=C'_pC'_{p+1}\cdots C'_m\times  \bigl[\eta'_{k}-\eta'_{p+1} \bigr]\times S'_2A'_mA'_{m-1}\cdots A'_{k-1}\\
&=\eta_{p+1}C_pC_{p+1}\cdots C_m\times \frac{S_1S_2B_r\cdots B_1}{\eta_1\eta_k\eta_{p+1}}(\eta_k-\eta_{p+1})\times \frac{\eta_1\eta_k}{S_2B_r\cdots B_1}A_m\cdots A_{k-1}\\
&=C_pC_{p+1}\cdots C_m(\eta_k-\eta_{p+1})S_1A_m\cdots A_{k-1}=\Pi_{k,p},
\end{align*}
so all these elements are preserved, as well as all normalizing factors, and $a'_{l,m}=a_{l,m}$.

Consider now
\begin{align*}
a'_{i,i+2}=&\bigl(A'_r\cdots A'_0S'_1B'_r\cdots B'_1 \bigr)^{-1/2}\bigl[ \eta'_{r+2}+A'_r\eta'_{r+1}+A'_rA'_{r-1}\eta'_{r}+\dots +A'_rA'_{r-1}\cdots A'_0\eta'_1\bigr]\\
=&\bigl(S_1A_r\cdots A_0S^2_2B^2_r\cdots B^2_1 \eta_2^{-2} \bigr)^{-1/2}\Bigl[ 1+A_r\frac{\eta_{r+1}\eta_1}{\eta_2}\Bigl(\frac{G_{i,i+1}}{\eta_1}-\frac{S_1S_2B_r\cdots B_1}{\eta_1\eta_{r+1}}\Bigr)+\dots\Bigr.\\
&+\Bigl. (A_r\cdots A_k)\frac{\eta_{k+1}\eta_1}{\eta_2}\Bigl(\frac{G_{i,i+1}}{\eta_1}-\frac{S_1S_2B_r\cdots B_1}{\eta_1\eta_{k+1}}\Bigr)+\dots + (A_r\cdots A_0)\frac{\eta_{1}\eta_1}{\eta_2}\Bigl(\frac{G_{i,i+1}}{\eta_1}-\frac{S_1S_2B_r\cdots B_1}{\eta_1\eta_{1}}\Bigr)\Bigr] \\
=&\Bigl[ \frac{1+A_r\eta_{r+1}+\dots+A_rA_{r-1}\cdots A_0\eta_1}{(A_r\cdots A_0S_2B_r\cdots B_1)^{1/2}}\cdot \frac{G_{i,i+1}}{(S_2B_r\cdots B_1 S_1)^{1/2}}-
\frac{1+S_1+S_1A_r+\dots+S_1A_r\cdots A_0}{(S_1A_r\cdots A_0)^{1/2}}\Bigr]\\
=&a_{i,i+2}a_{i,i+1}-a_{i+1,i+2}.
\end{align*}
Next,
\begin{align*}
a'_{i-1,i+1}=& \bigl(B'_r\cdots B'_1S'_2C'_r\cdots C'_0 \bigr)^{-1/2}\bigl(1\cdot(G'_{i,i+1}-\eta'_{r+2})+C'_r(G'_{i,i+1}-\eta'_{r+1})\bigr.\\
&+C'_rC'_{r-1}(G'_{i,i+1}-\eta'_{r}) +\dots+\bigl. C'_r\cdots C'_0(G'_{i,i+1}-\eta'_{1})\bigr)\\
=&\frac{1}{(C_r\cdots C_0S_2)^{1/2}}\bigl(1+C_r+C_rC_{r-1}+\dots+C_rC_{r-1}\cdots C_1+ C_rC_{r-1}\cdots C_1C_0\eta_{r+1}\bigr)\\
=&\frac{1}{(C_r\cdots C_0S_2)^{1/2}}\bigl(1+C_r+C_rC_{r-1}+\dots+C_rC_{r-1}\cdots C_1+ C_rC_{r-1}\cdots C_1C_0(1+S_2)\bigr)=a_{i-1,i}.
\end{align*}
Proving the rest of relations we leave to the reader.$\square$

%

\section{Casimirs}\label{s:Casimirs}

In this section, we derive complete sets of Casimirs for all relevant (sub)varieties of cluster variables related to regular quivers associated to $SL_n$ systems. All proofs are direct calculations: {it is easy to check that indicated products of cluster variables are Casimirs, whereas their completeness follows from the known answers for dimensions of symplectic leaves}.

\subsection{The full-rank $SL_n$-quiver}

\begin{lemma}\label{lem:Casimir-slN}
The complete set of Casimir operators for the full-rank $SL_n$-quiver are $\bigl[\frac{n}{2}\bigr]$ products of cluster variables depicted in the figure below for the example of $SL_6$: numbers at vertices indicate the power with which the corresponding variable comes into the product; all nonnumbered variables have power zero. All Casimirs correspond to closed broken-line paths in the $SL_n$-quiver with reflections at the boundaries (the ``frozen'' variables at boundaries enter the product with powers two, powers of non-frozen variables can be 0,1,2, and 3, and they count how many times the path goes through the corresponding variable. The total Poisson dimension of the full-rank quiver is therefore $\frac{(n+2)(n+1)}{2}-3-\bigl[\frac{n}{2}\bigr]$.
\end{lemma}

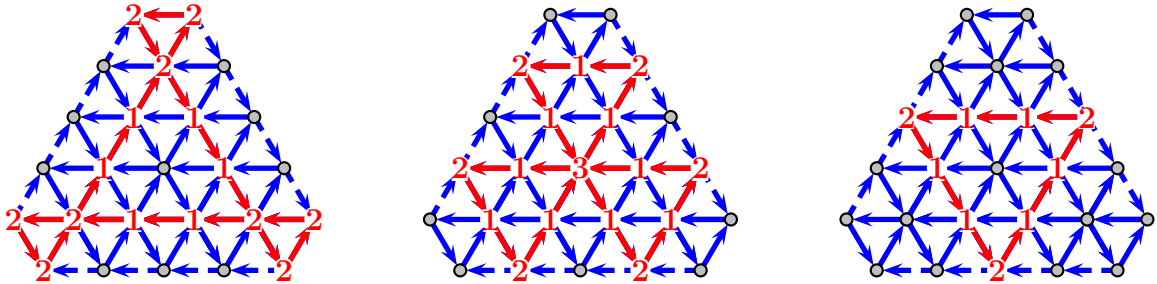
\begin{figure}[H]
\begin{pspicture}(-2.7,-2)(2.7,3){\psset{unit=0.8}
\newcommand{\PATGEN}{%
{\psset{unit=1}
\rput(0,0){\psline[linecolor=blue,linewidth=2pt]{->}(0,0)(.45,.765)}
\rput(0,0){\psline[linecolor=blue,linewidth=2pt]{->}(1,0)(0.1,0)}
\rput(0,0){\psline[linecolor=blue,linewidth=2pt]{->}(0,0)(.45,-.765)}
\put(0,0){\pscircle[fillstyle=solid,fillcolor=lightgray]{.1}}
}}
\newcommand{\PATLEFT}{%
{\psset{unit=1}
\rput(0,0){\psline[linecolor=blue,linewidth=2pt,linestyle=dashed]{->}(0,0)(.45,.765)}
\rput(0,0){\psline[linecolor=blue,linewidth=2pt]{->}(1,0)(0.1,0)}
\rput(0,0){\psline[linecolor=blue,linewidth=2pt]{->}(0,0)(.45,-.765)}
\put(0,0){\pscircle[fillstyle=solid,fillcolor=lightgray]{.1}}
}}
\newcommand{\PATRIGHT}{%
{\psset{unit=1}
\rput(0,0){\psline[linecolor=blue,linewidth=2pt,linestyle=dashed]{->}(0,0)(.45,-.765)}
\put(0,0){\pscircle[fillstyle=solid,fillcolor=lightgray]{.1}}
}}
\newcommand{\PATBOTTOM}{%
{\psset{unit=1}
\rput(0,0){\psline[linecolor=blue,linewidth=2pt]{->}(0,0)(.45,.765)}
\rput(0,0){\psline[linecolor=blue,linewidth=2pt,linestyle=dashed]{->}(1,0)(0.1,0)}
\put(0,0){\pscircle[fillstyle=solid,fillcolor=lightgray]{.1}}
}}
\newcommand{\PATTOP}{%
{\psset{unit=1}
\rput(0,0){\psline[linecolor=blue,linewidth=2pt]{->}(1,0)(0.1,0)}
\rput(0,0){\psline[linecolor=blue,linewidth=2pt]{->}(0,0)(.45,-.765)}
\put(0,0){\pscircle[fillstyle=solid,fillcolor=lightgray]{.1}}
}}
\newcommand{\PATBOTRIGHT}{%
{\psset{unit=1}
\rput(0,0){\psline[linecolor=blue,linewidth=2pt]{->}(0,0)(.45,.765)}
\put(0,0){\pscircle[fillstyle=solid,fillcolor=lightgray]{.1}}
\put(.5,0.85){\pscircle[fillstyle=solid,fillcolor=lightgray]{.1}}
}}
\newcommand{\ODIN}{%
{\psset{unit=1}
\put(0,0){\pscircle[fillstyle=solid,fillcolor=white,linecolor=white]{.15}}
\put(0,0){\makebox(0,0)[cc]{\hbox{\tcw{\large$\mathbf 1$}}}}
\put(0,0){\makebox(0,0)[cc]{\hbox{\tcr{$\mathbf 1$}}}}
}}
\newcommand{\DVA}{%
{\psset{unit=1}
\put(0,0){\pscircle[fillstyle=solid,fillcolor=white,linecolor=white]{.15}}
\put(0,0){\makebox(0,0)[cc]{\hbox{\tcw{\large$\mathbf 2$}}}}
\put(0,0){\makebox(0,0)[cc]{\hbox{\tcr{$\mathbf 2$}}}}
}}
\newcommand{\TRI}{%
{\psset{unit=1}
\put(0,0){\pscircle[fillstyle=solid,fillcolor=white,linecolor=white]{.15}}
\put(0,0){\makebox(0,0)[cc]{\hbox{\tcw{\large$\mathbf 3$}}}}
\put(0,0){\makebox(0,0)[cc]{\hbox{\tcr{$\mathbf 3$}}}}
}}
\multiput(-2.5,-0.85)(0.5,0.85){4}{\PATLEFT}
\multiput(-2,-1.7)(1,0){4}{\PATBOTTOM}
\put(-0.5,2.55){\PATTOP}
\multiput(-1.5,-0.85)(1,0){4}{\PATGEN}
\multiput(-1,0)(1,0){3}{\PATGEN}
\multiput(-.5,0.85)(1,0){2}{\PATGEN}
\put(0,1.7){\PATGEN}
\multiput(-1.5,-0.85)(1,0){4}{\PATGEN}
\multiput(0.5,2.55)(0.5,-0.85){4}{\PATRIGHT}
\put(2,-1.7){\PATBOTRIGHT}
\multiput(-2,-1.7)(0.5,0.85){5}{\psline[linecolor=red,linewidth=2pt]{->}(0,0)(.45,.765)}
\multiput(2,-1.7)(0.5,0.85){1}{\psline[linecolor=red,linewidth=2pt]{->}(0,0)(.45,.765)}
\multiput(-0.5,2.55)(0.5,-0.85){5}{\psline[linecolor=red,linewidth=2pt]{->}(0,0)(.45,-.765)}
\multiput(-2.5,-0.85)(0.5,-0.85){1}{\psline[linecolor=red,linewidth=2pt]{->}(0,0)(.45,-.765)}
\multiput(-2.5,-0.85)(1,0){5}{\psline[linecolor=red,linewidth=2pt]{->}(1,0)(.1,0)}
\multiput(-0.5,2.55)(1,0){1}{\psline[linecolor=red,linewidth=2pt]{->}(1,0)(.1,0)}
\multiput(-.5,-0.85)(1,0){2}{\ODIN}
\multiput(-1,0)(2,0){2}{\ODIN}
\multiput(-.5,0.85)(1,0){2}{\ODIN}
\multiput(-2.5,-0.85)(1,0){2}{\DVA}
\multiput(2.5,-0.85)(-1,0){2}{\DVA}
\multiput(-2,-1.7)(4,0){2}{\DVA}
\multiput(-0.5,2.55)(1,0){2}{\DVA}
\put(0,1.7){\DVA}
%
}
\end{pspicture}
\begin{pspicture}(-2.7,-2)(2.7,3){\psset{unit=0.8}
\newcommand{\PATGEN}{%
{\psset{unit=1}
\rput(0,0){\psline[linecolor=blue,linewidth=2pt]{->}(0,0)(.45,.765)}
\rput(0,0){\psline[linecolor=blue,linewidth=2pt]{->}(1,0)(0.1,0)}
\rput(0,0){\psline[linecolor=blue,linewidth=2pt]{->}(0,0)(.45,-.765)}
\put(0,0){\pscircle[fillstyle=solid,fillcolor=lightgray]{.1}}
}}
\newcommand{\PATLEFT}{%
{\psset{unit=1}
\rput(0,0){\psline[linecolor=blue,linewidth=2pt,linestyle=dashed]{->}(0,0)(.45,.765)}
\rput(0,0){\psline[linecolor=blue,linewidth=2pt]{->}(1,0)(0.1,0)}
\rput(0,0){\psline[linecolor=blue,linewidth=2pt]{->}(0,0)(.45,-.765)}
\put(0,0){\pscircle[fillstyle=solid,fillcolor=lightgray]{.1}}
}}
\newcommand{\PATRIGHT}{%
{\psset{unit=1}
\rput(0,0){\psline[linecolor=blue,linewidth=2pt,linestyle=dashed]{->}(0,0)(.45,-.765)}
\put(0,0){\pscircle[fillstyle=solid,fillcolor=lightgray]{.1}}
}}
\newcommand{\PATBOTTOM}{%
{\psset{unit=1}
\rput(0,0){\psline[linecolor=blue,linewidth=2pt]{->}(0,0)(.45,.765)}
\rput(0,0){\psline[linecolor=blue,linewidth=2pt,linestyle=dashed]{->}(1,0)(0.1,0)}
\put(0,0){\pscircle[fillstyle=solid,fillcolor=lightgray]{.1}}
}}
\newcommand{\PATTOP}{%
{\psset{unit=1}
\rput(0,0){\psline[linecolor=blue,linewidth=2pt]{->}(1,0)(0.1,0)}
\rput(0,0){\psline[linecolor=blue,linewidth=2pt]{->}(0,0)(.45,-.765)}
\put(0,0){\pscircle[fillstyle=solid,fillcolor=lightgray]{.1}}
}}
\newcommand{\PATBOTRIGHT}{%
{\psset{unit=1}
\rput(0,0){\psline[linecolor=blue,linewidth=2pt]{->}(0,0)(.45,.765)}
\put(0,0){\pscircle[fillstyle=solid,fillcolor=lightgray]{.1}}
\put(.5,0.85){\pscircle[fillstyle=solid,fillcolor=lightgray]{.1}}
}}
\newcommand{\ODIN}{%
{\psset{unit=1}
\put(0,0){\pscircle[fillstyle=solid,fillcolor=white,linecolor=white]{.15}}
\put(0,0){\makebox(0,0)[cc]{\hbox{\tcw{\large$\mathbf 1$}}}}
\put(0,0){\makebox(0,0)[cc]{\hbox{\tcr{$\mathbf 1$}}}}
}}
\newcommand{\DVA}{%
{\psset{unit=1}
\put(0,0){\pscircle[fillstyle=solid,fillcolor=white,linecolor=white]{.15}}
\put(0,0){\makebox(0,0)[cc]{\hbox{\tcw{\large$\mathbf 2$}}}}
\put(0,0){\makebox(0,0)[cc]{\hbox{\tcr{$\mathbf 2$}}}}
}}
\newcommand{\TRI}{%
{\psset{unit=1}
\put(0,0){\pscircle[fillstyle=solid,fillcolor=white,linecolor=white]{.15}}
\put(0,0){\makebox(0,0)[cc]{\hbox{\tcw{\large$\mathbf 3$}}}}
\put(0,0){\makebox(0,0)[cc]{\hbox{\tcr{$\mathbf 3$}}}}
}}
\multiput(-2.5,-0.85)(0.5,0.85){4}{\PATLEFT}
\multiput(-2,-1.7)(1,0){4}{\PATBOTTOM}
\put(-0.5,2.55){\PATTOP}
\multiput(-1.5,-0.85)(1,0){4}{\PATGEN}
\multiput(-1,0)(1,0){3}{\PATGEN}
\multiput(-.5,0.85)(1,0){2}{\PATGEN}
\put(0,1.7){\PATGEN}
\multiput(-1.5,-0.85)(1,0){4}{\PATGEN}
\multiput(0.5,2.55)(0.5,-0.85){4}{\PATRIGHT}
\put(2,-1.7){\PATBOTRIGHT}
\multiput(-1,-1.7)(0.5,0.85){4}{\psline[linecolor=red,linewidth=2pt]{->}(0,0)(.45,.765)}
\multiput(1,-1.7)(0.5,0.85){2}{\psline[linecolor=red,linewidth=2pt]{->}(0,0)(.45,.765)}
\multiput(-1,1.7)(0.5,-0.85){4}{\psline[linecolor=red,linewidth=2pt]{->}(0,0)(.45,-.765)}
\multiput(-2,0)(0.5,-0.85){2}{\psline[linecolor=red,linewidth=2pt]{->}(0,0)(.45,-.765)}
\multiput(-2,0)(1,0){4}{\psline[linecolor=red,linewidth=2pt]{->}(1,0)(.1,0)}
\multiput(-1,1.7)(1,0){2}{\psline[linecolor=red,linewidth=2pt]{->}(1,0)(.1,0)}
\put(0,0){\TRI}
\multiput(-1.5,-0.85)(1,0){3}{\ODIN}
\multiput(-1,0)(0.5,0.85){3}{\ODIN}
\multiput(.5,0.85)(.5,-0.85){3}{\ODIN}
\multiput(-2,0)(1,1.7){2}{\DVA}
\multiput(2,0)(-1,1.7){2}{\DVA}
\multiput(-1,-1.7)(2,0){2}{\DVA}
%
}
\end{pspicture}
\begin{pspicture}(-2.7,-2)(2.7,3){\psset{unit=0.8}
\newcommand{\PATGEN}{%
{\psset{unit=1}
\rput(0,0){\psline[linecolor=blue,linewidth=2pt]{->}(0,0)(.45,.765)}
\rput(0,0){\psline[linecolor=blue,linewidth=2pt]{->}(1,0)(0.1,0)}
\rput(0,0){\psline[linecolor=blue,linewidth=2pt]{->}(0,0)(.45,-.765)}
\put(0,0){\pscircle[fillstyle=solid,fillcolor=lightgray]{.1}}
}}
\newcommand{\PATLEFT}{%
{\psset{unit=1}
\rput(0,0){\psline[linecolor=blue,linewidth=2pt,linestyle=dashed]{->}(0,0)(.45,.765)}
\rput(0,0){\psline[linecolor=blue,linewidth=2pt]{->}(1,0)(0.1,0)}
\rput(0,0){\psline[linecolor=blue,linewidth=2pt]{->}(0,0)(.45,-.765)}
\put(0,0){\pscircle[fillstyle=solid,fillcolor=lightgray]{.1}}
}}
\newcommand{\PATRIGHT}{%
{\psset{unit=1}
\rput(0,0){\psline[linecolor=blue,linewidth=2pt,linestyle=dashed]{->}(0,0)(.45,-.765)}
\put(0,0){\pscircle[fillstyle=solid,fillcolor=lightgray]{.1}}
}}
\newcommand{\PATBOTTOM}{%
{\psset{unit=1}
\rput(0,0){\psline[linecolor=blue,linewidth=2pt]{->}(0,0)(.45,.765)}
\rput(0,0){\psline[linecolor=blue,linewidth=2pt,linestyle=dashed]{->}(1,0)(0.1,0)}
\put(0,0){\pscircle[fillstyle=solid,fillcolor=lightgray]{.1}}
}}
\newcommand{\PATTOP}{%
{\psset{unit=1}
\rput(0,0){\psline[linecolor=blue,linewidth=2pt]{->}(1,0)(0.1,0)}
\rput(0,0){\psline[linecolor=blue,linewidth=2pt]{->}(0,0)(.45,-.765)}
\put(0,0){\pscircle[fillstyle=solid,fillcolor=lightgray]{.1}}
}}
\newcommand{\PATBOTRIGHT}{%
{\psset{unit=1}
\rput(0,0){\psline[linecolor=blue,linewidth=2pt]{->}(0,0)(.45,.765)}
\put(0,0){\pscircle[fillstyle=solid,fillcolor=lightgray]{.1}}
\put(.5,0.85){\pscircle[fillstyle=solid,fillcolor=lightgray]{.1}}
}}
\newcommand{\ODIN}{%
{\psset{unit=1}
\put(0,0){\pscircle[fillstyle=solid,fillcolor=white,linecolor=white]{.15}}
\put(0,0){\makebox(0,0)[cc]{\hbox{\tcw{\large$\mathbf 1$}}}}
\put(0,0){\makebox(0,0)[cc]{\hbox{\tcr{$\mathbf 1$}}}}
}}
\newcommand{\DVA}{%
{\psset{unit=1}
\put(0,0){\pscircle[fillstyle=solid,fillcolor=white,linecolor=white]{.15}}
\put(0,0){\makebox(0,0)[cc]{\hbox{\tcw{\large$\mathbf 2$}}}}
\put(0,0){\makebox(0,0)[cc]{\hbox{\tcr{$\mathbf 2$}}}}
}}
\newcommand{\TRI}{%
{\psset{unit=1}
\put(0,0){\pscircle[fillstyle=solid,fillcolor=white,linecolor=white]{.15}}
\put(0,0){\makebox(0,0)[cc]{\hbox{\tcw{\large$\mathbf 3$}}}}
\put(0,0){\makebox(0,0)[cc]{\hbox{\tcr{$\mathbf 3$}}}}
}}
\multiput(-2.5,-0.85)(0.5,0.85){4}{\PATLEFT}
\multiput(-2,-1.7)(1,0){4}{\PATBOTTOM}
\put(-0.5,2.55){\PATTOP}
\multiput(-1.5,-0.85)(1,0){4}{\PATGEN}
\multiput(-1,0)(1,0){3}{\PATGEN}
\multiput(-.5,0.85)(1,0){2}{\PATGEN}
\put(0,1.7){\PATGEN}
\multiput(-1.5,-0.85)(1,0){4}{\PATGEN}
\multiput(0.5,2.55)(0.5,-0.85){4}{\PATRIGHT}
\put(2,-1.7){\PATBOTRIGHT}
\multiput(0,-1.7)(0.5,0.85){3}{\psline[linecolor=red,linewidth=2pt]{->}(0,0)(.45,.765)}
\multiput(-1.5,0.85)(0.5,-0.85){3}{\psline[linecolor=red,linewidth=2pt]{->}(0,0)(.45,-.765)}
\multiput(-1.5,0.85)(1,0){3}{\psline[linecolor=red,linewidth=2pt]{->}(1,0)(.1,0)}
\multiput(-.5,-0.85)(1,0){2}{\ODIN}
\multiput(-1,0)(2,0){2}{\ODIN}
\multiput(-.5,0.85)(1,0){2}{\ODIN}
\multiput(-1.5,0.85)(3,0){2}{\DVA}
\put(0,-1.7){\DVA}
%
}
\end{pspicture}
\caption{\small
Three central elements of the full-rank quiver for $SL_6$.
}
\label{fi:Casimirs}
\end{figure}

\begin{remark}
{All Casimir operators from Lemma \ref{lem:Casimir-slN} remain Casimirs for the full-rank $GL_n$-quiver obtained by adding three more cluster variables at the corners of the triangle (the variables $Z_{6,0,0}$, $Z_{0,6,0}$, and $Z_{0,0,6}$ in Fig.~\ref{fi:plabic_weights}). If we include these three corner variables into the quiver, we have to add one more Casimir operator which is the product of all frozen (non-corner) variables along all three boundaries of the $SL_n$-quiver taken in power one and the product of three corner variables taken in power three. }
\end{remark}

For completeness, we also present Casimirs for a reduced quiver in which we eliminate one of the three sets of frozen variables. The remaining $n(n+1)/2-1$ variables are those parameterizing, say, the transport matrix $M_1$ (for $M_2$ we have to remove another set of frozen variables). In this case, every Casimir of the full-rank quiver has its counterpart in the reduced quiver except the element that is represented by a triangle-shaped path in the full-rank quiver (such an element exists only for even $n$), which has no counterpart.

\begin{lemma}\label{lem:Casimir-slN-reduced}
The complete set of Casimir operators for the reduced $SL_n$-quiver are $\bigl[\frac{n-1}{2}\bigr]$ products of cluster variables depicted in Fig.~\ref{fi:Cas-incomplete} for the example of $SL_6$: numbers at vertices indicate the power with which the corresponding variable comes into the product; all nonnumbered variables have power zero. All Casimirs correspond, as in Fig.~\ref{fi:Casimirs}, to closed broken-line paths in the corresponding full-rank quiver with reflections at the boundaries (the ``frozen'' variables at boundaries enter the product with powers two), but now the path is split into two parts separated  by two reflections at the side of the triangle that corresponds to the erased frozen variables; these two parts enter with opposite signs; the corresponding Casimir therefore contains cluster variables in both positive and negative powers. As in the case of full-rank quiver, these powers count (with signs) how many times the path goes through the corresponding variable). The total Poisson dimension of the reduced $SL_n$-quiver is therefore $\frac{n(n+1)}{2}-1-\bigl[\frac{n-1}{2}\bigr]$.
\end{lemma}

\begin{figure}[tb]
\begin{pspicture}(-2.7,-1.5)(2.7,2.5){\psset{unit=0.8}
\newcommand{\PATGEN}{%
{\psset{unit=1}
\rput(0,0){\psline[linecolor=blue,linewidth=2pt]{->}(0,0)(.45,.765)}
\rput(0,0){\psline[linecolor=blue,linewidth=2pt]{->}(1,0)(0.1,0)}
\rput(0,0){\psline[linecolor=blue,linewidth=2pt]{->}(0,0)(.45,-.765)}
\put(0,0){\pscircle[fillstyle=solid,fillcolor=lightgray]{.1}}
}}
\newcommand{\PATLEFT}{%
{\psset{unit=1}
\rput(0,0){\psline[linecolor=blue,linewidth=2pt,linestyle=dashed]{->}(0,0)(.45,.765)}
\rput(0,0){\psline[linecolor=blue,linewidth=2pt]{->}(1,0)(0.1,0)}
\rput(0,0){\psline[linecolor=blue,linewidth=2pt]{->}(0,0)(.45,-.765)}
\put(0,0){\pscircle[fillstyle=solid,fillcolor=lightgray]{.1}}
}}
\newcommand{\PATRIGHT}{%
{\psset{unit=1}
\rput(0,0){\psline[linecolor=blue,linewidth=2pt,linestyle=dashed]{->}(0,0)(.45,-.765)}
\put(0,0){\pscircle[fillstyle=solid,fillcolor=lightgray]{.1}}
}}
\newcommand{\PATBOTTOM}{%
{\psset{unit=1}
\rput(0,0){\psline[linecolor=blue,linewidth=2pt]{->}(0,0)(.45,.765)}
\rput(0,0){\psline[linecolor=blue,linewidth=2pt,linestyle=dashed]{->}(1,0)(0.1,0)}
\put(0,0){\pscircle[fillstyle=solid,fillcolor=lightgray]{.1}}
}}
\newcommand{\PATTOP}{%
{\psset{unit=1}
\rput(0,0){\psline[linecolor=blue,linewidth=2pt]{->}(1,0)(0.1,0)}
\rput(0,0){\psline[linecolor=blue,linewidth=2pt]{->}(0,0)(.45,-.765)}
\put(0,0){\pscircle[fillstyle=solid,fillcolor=lightgray]{.1}}
}}
\newcommand{\PATBOTRIGHT}{%
{\psset{unit=1}
\rput(0,0){\psline[linecolor=blue,linewidth=2pt]{->}(0,0)(.45,.765)}
\put(0,0){\pscircle[fillstyle=solid,fillcolor=lightgray]{.1}}
\put(.5,0.85){\pscircle[fillstyle=solid,fillcolor=lightgray]{.1}}
}}
\newcommand{\ODIN}{%
{\psset{unit=1}
\put(0,0){\pscircle[fillstyle=solid,fillcolor=white,linecolor=white]{.15}}
\put(0,0){\makebox(0,0)[cc]{\hbox{\tcw{\large$\mathbf 1$}}}}
\put(0,0){\makebox(0,0)[cc]{\hbox{\tcr{$\mathbf 1$}}}}
}}
\newcommand{\MINUSODIN}{%
{\psset{unit=1}
\put(-0.05,-0.03){\pscircle[fillstyle=solid,fillcolor=white,linecolor=white]{.2}}
\put(0,0){\makebox(0,0)[cc]{\hbox{\tcw{\large-$\mathbf 1$}}}}
\put(0,0){\makebox(0,0)[cc]{\hbox{\textcolor{black}{-$\mathbf 1$}}}}
}}
\newcommand{\DVA}{%
{\psset{unit=1}
\put(0,0){\pscircle[fillstyle=solid,fillcolor=white,linecolor=white]{.15}}
\put(0,0){\makebox(0,0)[cc]{\hbox{\tcw{\large$\mathbf 2$}}}}
\put(0,0){\makebox(0,0)[cc]{\hbox{\tcr{$\mathbf 2$}}}}
}}
\newcommand{\MINUSDVA}{%
{\psset{unit=1}
\put(-0.05,-0.03){\pscircle[fillstyle=solid,fillcolor=white,linecolor=white]{.2}}
\put(0,0){\makebox(0,0)[cc]{\hbox{\tcw{\large -$\mathbf 2$}}}}
\put(0,0){\makebox(0,0)[cc]{\hbox{\textcolor{black}{-$\mathbf 2$}}}}
}}
\newcommand{\TRI}{%
{\psset{unit=1}
\put(0,0){\pscircle[fillstyle=solid,fillcolor=white,linecolor=white]{.15}}
\put(0,0){\makebox(0,0)[cc]{\hbox{\tcw{\large$\mathbf 3$}}}}
\put(0,0){\makebox(0,0)[cc]{\hbox{\tcr{$\mathbf 3$}}}}
}}
\newcommand{\ZERO}{%
{\psset{unit=1}
\put(0,0){\pscircle[fillstyle=solid,fillcolor=white,linecolor=white]{.15}}
\put(0,0){\makebox(0,0)[cc]{\hbox{\tcw{\large$\mathbf 0$}}}}
\put(0,0){\makebox(0,0)[cc]{\hbox{{$\mathbf 0$}}}}
}}
\multiput(-2.5,-0.85)(0.5,0.85){4}{\PATLEFT}
\multiput(-2,-1.7)(1,0){4}{\PATBOTTOM}
\put(-0.5,2.55){\PATTOP}
\multiput(-1.5,-0.85)(1,0){4}{\PATGEN}
\multiput(-1,0)(1,0){3}{\PATGEN}
\multiput(-.5,0.85)(1,0){2}{\PATGEN}
\put(0,1.7){\PATGEN}
\multiput(-1.5,-0.85)(1,0){4}{\PATGEN}
\multiput(0.5,2.55)(0.5,-0.85){4}{\PATRIGHT}
\put(2,-1.7){\PATBOTRIGHT}
\psframe[linecolor=white,fillstyle=solid, fillcolor=white](-3,-2.2)(3,-.9)
\multiput(-1.5,-0.85)(0.5,0.85){4}{\psline[linecolor=red,linewidth=2pt]{->}(0,0)(.45,.765)}
\multiput(2,-1.7)(-4,0){2}{\psline[linecolor=black,linewidth=3pt]{->}(0,0)(.45,.765)}
\multiput(1.95,-1.785)(-4,0){2}{\psline[linecolor=white,linewidth=2pt]{->}(0,0)(.45,.765)}
\multiput(-0.5,2.55)(0.5,-0.85){4}{\psline[linecolor=red,linewidth=2pt]{->}(0,0)(.45,-.765)}
\multiput(-2.45,-0.935)(4,0){2}{\psline[linecolor=black,linewidth=3pt]{->}(0,0)(.45,-.765)}
\multiput(-2.5,-0.85)(4,0){2}{\psline[linecolor=white,linewidth=2pt]{->}(0,0)(.45,-.765)}
\multiput(-2.5,-0.85)(1,0){5}{\psline[linecolor=magenta,linewidth=2pt]{->}(1,0)(.1,0)}
\multiput(-0.5,2.55)(1,0){1}{\psline[linecolor=red,linewidth=2pt]{->}(1,0)(.1,0)}
\multiput(-.5,0.85)(1,0){2}{\ODIN}
\multiput(-1,0)(2,0){2}{\ODIN}
\multiput(-.5,-0.85)(1,0){2}{\MINUSODIN}
\multiput(-2.5,-0.85)(5,0){2}{\MINUSDVA}
\multiput(-1.5,-0.85)(3,0){2}{\ZERO}
\multiput(-0.5,2.55)(1,0){2}{\DVA}
\put(0,1.7){\DVA}
%
}
\end{pspicture}
\begin{pspicture}(-2.7,-1.5)(2.7,2.5){\psset{unit=0.8}
\newcommand{\PATGEN}{%
{\psset{unit=1}
\rput(0,0){\psline[linecolor=blue,linewidth=2pt]{->}(0,0)(.45,.765)}
\rput(0,0){\psline[linecolor=blue,linewidth=2pt]{->}(1,0)(0.1,0)}
\rput(0,0){\psline[linecolor=blue,linewidth=2pt]{->}(0,0)(.45,-.765)}
\put(0,0){\pscircle[fillstyle=solid,fillcolor=lightgray]{.1}}
}}
\newcommand{\PATLEFT}{%
{\psset{unit=1}
\rput(0,0){\psline[linecolor=blue,linewidth=2pt,linestyle=dashed]{->}(0,0)(.45,.765)}
\rput(0,0){\psline[linecolor=blue,linewidth=2pt]{->}(1,0)(0.1,0)}
\rput(0,0){\psline[linecolor=blue,linewidth=2pt]{->}(0,0)(.45,-.765)}
\put(0,0){\pscircle[fillstyle=solid,fillcolor=lightgray]{.1}}
}}
\newcommand{\PATRIGHT}{%
{\psset{unit=1}
\rput(0,0){\psline[linecolor=blue,linewidth=2pt,linestyle=dashed]{->}(0,0)(.45,-.765)}
\put(0,0){\pscircle[fillstyle=solid,fillcolor=lightgray]{.1}}
}}
\newcommand{\PATBOTTOM}{%
{\psset{unit=1}
\rput(0,0){\psline[linecolor=blue,linewidth=2pt]{->}(0,0)(.45,.765)}
\rput(0,0){\psline[linecolor=blue,linewidth=2pt,linestyle=dashed]{->}(1,0)(0.1,0)}
\put(0,0){\pscircle[fillstyle=solid,fillcolor=lightgray]{.1}}
}}
\newcommand{\PATTOP}{%
{\psset{unit=1}
\rput(0,0){\psline[linecolor=blue,linewidth=2pt]{->}(1,0)(0.1,0)}
\rput(0,0){\psline[linecolor=blue,linewidth=2pt]{->}(0,0)(.45,-.765)}
\put(0,0){\pscircle[fillstyle=solid,fillcolor=lightgray]{.1}}
}}
\newcommand{\PATBOTRIGHT}{%
{\psset{unit=1}
\rput(0,0){\psline[linecolor=blue,linewidth=2pt]{->}(0,0)(.45,.765)}
\put(0,0){\pscircle[fillstyle=solid,fillcolor=lightgray]{.1}}
\put(.5,0.85){\pscircle[fillstyle=solid,fillcolor=lightgray]{.1}}
}}
\newcommand{\ODIN}{%
{\psset{unit=1}
\put(0,0){\pscircle[fillstyle=solid,fillcolor=white,linecolor=white]{.15}}
\put(0,0){\makebox(0,0)[cc]{\hbox{\tcw{\large$\mathbf 1$}}}}
\put(0,0){\makebox(0,0)[cc]{\hbox{\tcr{$\mathbf 1$}}}}
}}
\newcommand{\MINUSODIN}{%
{\psset{unit=1}
\put(-0.05,-0.03){\pscircle[fillstyle=solid,fillcolor=white,linecolor=white]{.2}}
\put(0,0){\makebox(0,0)[cc]{\hbox{\tcw{\large-$\mathbf 1$}}}}
\put(0,0){\makebox(0,0)[cc]{\hbox{\textcolor{black}{-$\mathbf 1$}}}}
}}
\newcommand{\DVA}{%
{\psset{unit=1}
\put(0,0){\pscircle[fillstyle=solid,fillcolor=white,linecolor=white]{.15}}
\put(0,0){\makebox(0,0)[cc]{\hbox{\tcw{\large$\mathbf 2$}}}}
\put(0,0){\makebox(0,0)[cc]{\hbox{\tcr{$\mathbf 2$}}}}
}}
\newcommand{\MINUSDVA}{%
{\psset{unit=1}
\put(-0.05,-0.03){\pscircle[fillstyle=solid,fillcolor=white,linecolor=white]{.2}}
\put(0,0){\makebox(0,0)[cc]{\hbox{\tcw{\large -$\mathbf 2$}}}}
\put(0,0){\makebox(0,0)[cc]{\hbox{\textcolor{black}{-$\mathbf 2$}}}}
}}
\newcommand{\TRI}{%
{\psset{unit=1}
\put(0,0){\pscircle[fillstyle=solid,fillcolor=white,linecolor=white]{.15}}
\put(0,0){\makebox(0,0)[cc]{\hbox{\tcw{\large$\mathbf 3$}}}}
\put(0,0){\makebox(0,0)[cc]{\hbox{\tcr{$\mathbf 3$}}}}
}}
\newcommand{\ZERO}{%
{\psset{unit=1}
\put(0,0){\pscircle[fillstyle=solid,fillcolor=white,linecolor=white]{.15}}
\put(0,0){\makebox(0,0)[cc]{\hbox{\tcw{\large$\mathbf 0$}}}}
\put(0,0){\makebox(0,0)[cc]{\hbox{{$\mathbf 0$}}}}
}}
\multiput(-2.5,-0.85)(0.5,0.85){4}{\PATLEFT}
\multiput(-2,-1.7)(1,0){4}{\PATBOTTOM}
\put(-0.5,2.55){\PATTOP}
\multiput(-1.5,-0.85)(1,0){4}{\PATGEN}
\multiput(-1,0)(1,0){3}{\PATGEN}
\multiput(-.5,0.85)(1,0){2}{\PATGEN}
\put(0,1.7){\PATGEN}
\multiput(-1.5,-0.85)(1,0){4}{\PATGEN}
\multiput(0.5,2.55)(0.5,-0.85){4}{\PATRIGHT}
\put(2,-1.7){\PATBOTRIGHT}
\psframe[linecolor=white,fillstyle=solid, fillcolor=white](-3,-2.2)(3,-.95)
\multiput(-0.5,-.85)(0.5,0.85){3}{\psline[linecolor=red,linewidth=2pt]{->}(0,0)(.45,.765)}
\multiput(1.5,-.85)(0.5,0.85){1}{\psline[linecolor=magenta,linewidth=2pt]{->}(0,0)(.45,.765)}
\multiput(-1,1.7)(0.5,-0.85){3}{\psline[linecolor=red,linewidth=2pt]{->}(0,0)(.45,-.765)}
\multiput(-2,0)(0.5,-0.85){1}{\psline[linecolor=magenta,linewidth=2pt]{->}(0,0)(.45,-.765)}
\multiput(-2,0)(1,0){4}{\psline[linecolor=magenta,linewidth=2pt]{->}(1,0)(.1,0)}
\multiput(-1,1.7)(1,0){2}{\psline[linecolor=red,linewidth=2pt]{->}(1,0)(.1,0)}
\multiput(1,-1.7)(-2,0){2}{\psline[linecolor=black,linewidth=3pt]{->}(0,0)(.45,.765)}
\multiput(0.95,-1.785)(-2,0){2}{\psline[linecolor=white,linewidth=2pt]{->}(0,0)(.45,.765)}
\multiput(-1.45,-0.935)(2,0){2}{\psline[linecolor=black,linewidth=3pt]{->}(0,0)(.45,-.765)}
\multiput(-1.5,-0.85)(2,0){2}{\psline[linecolor=white,linewidth=2pt]{->}(0,0)(.45,-.765)}
%
\multiput(0,1.7)(0,-1.7){2}{\ODIN}
\multiput(-0.5,0.85)(0,-1.7){2}{\ODIN}
\multiput(0.5,0.85)(0,-1.7){2}{\ODIN}
\multiput(-1.5,-0.85)(3,0){2}{\MINUSODIN}
\multiput(-1,0)(2,0){2}{\MINUSODIN}
\multiput(-2,0)(4,0){2}{\MINUSDVA}
\multiput(-1,1.7)(2,0){2}{\DVA}
%
}
\end{pspicture}
\caption{\small
Two central elements of the reduced quiver for $SL_6$. Every element of the complete quiver in Fig.~\ref{fi:Casimirs} has its counterpart in the reduced quiver except the third element.
}
\label{fi:Cas-incomplete}
\end{figure}
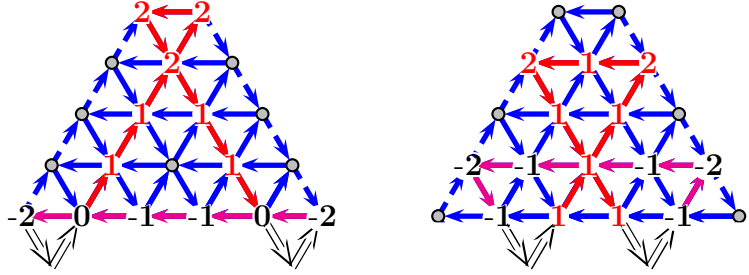

In our construction below, an important role is played by the additional Casimir that appears if we add the variable $(0,0,6)$ at the summit of the triangle corresponding to a reduced quiver. In this case, besides the Casimirs in Lemma~\ref{lem:Casimir-slN-reduced}, we have one more central element $D$ described in the following statement.

\begin{lemma}\label{lem:Casimir-slN-reduced-D}
The complete set of Casimir operators for the reduced $SL_n$-quiver with the (``frozen'') cluster variable $(0,0,n)$ added comprises all Casimirs described in Lemma~\ref{lem:Casimir-slN-reduced} plus the element $D_1$ given by the following formula. Let us enumerate the plabic weights $Z_{i,j,k}$ as in Fig.~\ref{fi:plabic_weights} by three nonnegative integers $(i,j,k)$ with $i+j+k=n$. Then the element
\begin{equation}\label{D-Casimir}
D_1=\col{\prod_{k=1}^n \Bigl[\prod_{i+j=N-k}\bigl[Z_{i,j,k}\bigr]^{k/n}\Bigr]}
\end{equation}
is central for the subset of $Z_{i,j,k}$ with $k>0$. Moreover, the only elements that have nonzero homogeneous commutation relations with $D_1$ are $Z_{n,0,0}$ and $Z_{0,n,0}$.
\end{lemma}

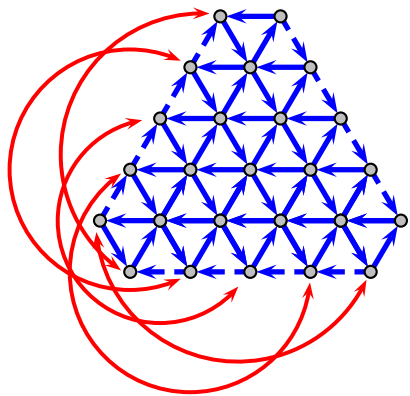
\begin{figure}[tb]
\begin{pspicture}(-3,-3)(3,2){\psset{unit=0.8}
\newcommand{\PATGEN}{%
{\psset{unit=1}
\rput(0,0){\psline[linecolor=blue,linewidth=2pt]{->}(0,0)(.45,.765)}
\rput(0,0){\psline[linecolor=blue,linewidth=2pt]{->}(1,0)(0.1,0)}
\rput(0,0){\psline[linecolor=blue,linewidth=2pt]{->}(0,0)(.45,-.765)}
\put(0,0){\pscircle[fillstyle=solid,fillcolor=lightgray]{.1}}
}}
\newcommand{\PATLEFT}{%
{\psset{unit=1}
\rput(0,0){\psline[linecolor=blue,linewidth=2pt,linestyle=dashed]{->}(0,0)(.45,.765)}
\rput(0,0){\psline[linecolor=blue,linewidth=2pt]{->}(1,0)(0.1,0)}
\rput(0,0){\psline[linecolor=blue,linewidth=2pt]{->}(0,0)(.45,-.765)}
\put(0,0){\pscircle[fillstyle=solid,fillcolor=lightgray]{.1}}
}}
\newcommand{\PATRIGHT}{%
{\psset{unit=1}
\rput(0,0){\psline[linecolor=blue,linewidth=2pt,linestyle=dashed]{->}(0,0)(.45,-.765)}
\put(0,0){\pscircle[fillstyle=solid,fillcolor=lightgray]{.1}}
}}
\newcommand{\PATBOTTOM}{%
{\psset{unit=1}
\rput(0,0){\psline[linecolor=blue,linewidth=2pt]{->}(0,0)(.45,.765)}
\rput(0,0){\psline[linecolor=blue,linewidth=2pt,linestyle=dashed]{->}(1,0)(0.1,0)}
\put(0,0){\pscircle[fillstyle=solid,fillcolor=lightgray]{.1}}
}}
\newcommand{\PATTOP}{%
{\psset{unit=1}
\rput(0,0){\psline[linecolor=blue,linewidth=2pt]{->}(1,0)(0.1,0)}
\rput(0,0){\psline[linecolor=blue,linewidth=2pt]{->}(0,0)(.45,-.765)}
\put(0,0){\pscircle[fillstyle=solid,fillcolor=lightgray]{.1}}
}}
\newcommand{\PATBOTRIGHT}{%
{\psset{unit=1}
\rput(0,0){\psline[linecolor=blue,linewidth=2pt]{->}(0,0)(.45,.765)}
\put(0,0){\pscircle[fillstyle=solid,fillcolor=lightgray]{.1}}
\put(.5,0.85){\pscircle[fillstyle=solid,fillcolor=lightgray]{.1}}
}}
\multiput(-2.5,-0.85)(0.5,0.85){4}{\PATLEFT}
\multiput(-2,-1.7)(1,0){4}{\PATBOTTOM}
\put(-0.5,2.55){\PATTOP}
\multiput(-1.5,-0.85)(1,0){4}{\PATGEN}
\multiput(-1,0)(1,0){3}{\PATGEN}
\multiput(-.5,0.85)(1,0){2}{\PATGEN}
\put(0,1.7){\PATGEN}
\multiput(-1.5,-0.85)(1,0){4}{\PATGEN}
\multiput(0.5,2.55)(0.5,-0.85){4}{\PATRIGHT}
\put(2,-1.7){\PATBOTRIGHT}
\psarc[linecolor=red, linewidth=1.5pt]{<->}(-1.5,-0.85){1.7}{100}{320}
\psarc[linecolor=red, linewidth=1.5pt]{<->}(-2,0){2}{65}{295}
\psarc[linecolor=red, linewidth=1.5pt]{<->}(-1,-1.7){2}{125}{355}
\psarc[linecolor=red, linewidth=1.5pt]{<->}(-0.8,0.25){2.35}{87}{235}
\rput{100}(-0.1,0){\psarc[linecolor=red, linewidth=1.5pt]{<->}(-0.8,0.25){2.36}{85}{235}}
%
}
\end{pspicture}
\caption{\small
The amalgamation of the quiver corresponding to the triangle $\Sigma_{0,1,3}$ (The example in the figure corresponds to $SL_6$).
}
\label{fi:amalgamation}
\end{figure}

\subsection{Casimirs for the upper-triangular matrices}

Entries of the matrix $\mathbb A:= M_1^{\text{T}}M_2$ depend on all variables of the $SL_n$-quiver, but due to the transposition, two sets of the frozen variables become amalgamated, that is, only their products appear in the entries of the matrix  $\mathbb A$. We explicitly show this amalgamation in Fig.~\ref{fi:amalgamation}.

\begin{figure}[tb]
\begin{pspicture}(-2.7,-2.5)(2.7,2.5){\psset{unit=0.7}
\newcommand{\PATGEN}{%
{\psset{unit=1}
\rput(0,0){\psline[linecolor=blue,linewidth=2pt]{->}(0,0)(.45,.765)}
\rput(0,0){\psline[linecolor=blue,linewidth=2pt]{->}(1,0)(0.1,0)}
\rput(0,0){\psline[linecolor=blue,linewidth=2pt]{->}(0,0)(.45,-.765)}
\put(0,0){\pscircle[fillstyle=solid,fillcolor=lightgray]{.1}}
}}
\newcommand{\PATLEFT}{%
{\psset{unit=1}
\rput(0,0){\psline[linecolor=blue,linewidth=2pt,linestyle=dashed]{->}(0,0)(.45,.765)}
\rput(0,0){\psline[linecolor=blue,linewidth=2pt]{->}(1,0)(0.1,0)}
\rput(0,0){\psline[linecolor=blue,linewidth=2pt]{->}(0,0)(.45,-.765)}
\put(0,0){\pscircle[fillstyle=solid,fillcolor=lightgray]{.1}}
}}
\newcommand{\PATRIGHT}{%
{\psset{unit=1}
\rput(0,0){\psline[linecolor=blue,linewidth=2pt,linestyle=dashed]{->}(0,0)(.45,-.765)}
\put(0,0){\pscircle[fillstyle=solid,fillcolor=lightgray]{.1}}
}}
\newcommand{\PATBOTTOM}{%
{\psset{unit=1}
\rput(0,0){\psline[linecolor=blue,linewidth=2pt]{->}(0,0)(.45,.765)}
\rput(0,0){\psline[linecolor=blue,linewidth=2pt,linestyle=dashed]{->}(1,0)(0.1,0)}
\put(0,0){\pscircle[fillstyle=solid,fillcolor=lightgray]{.1}}
}}
\newcommand{\PATTOP}{%
{\psset{unit=1}
\rput(0,0){\psline[linecolor=blue,linewidth=2pt]{->}(1,0)(0.1,0)}
\rput(0,0){\psline[linecolor=blue,linewidth=2pt]{->}(0,0)(.45,-.765)}
\put(0,0){\pscircle[fillstyle=solid,fillcolor=lightgray]{.1}}
}}
\newcommand{\PATBOTRIGHT}{%
{\psset{unit=1}
\rput(0,0){\psline[linecolor=blue,linewidth=2pt]{->}(0,0)(.45,.765)}
\put(0,0){\pscircle[fillstyle=solid,fillcolor=lightgray]{.1}}
\put(.5,0.85){\pscircle[fillstyle=solid,fillcolor=lightgray]{.1}}
}}
\newcommand{\ODIN}{%
{\psset{unit=1}
\put(0,0){\pscircle[fillstyle=solid,fillcolor=white,linecolor=white]{.15}}
\put(0,0){\makebox(0,0)[cc]{\hbox{\tcw{\large$\mathbf 1$}}}}
\put(0,0){\makebox(0,0)[cc]{\hbox{\tcr{$\mathbf 1$}}}}
}}
\newcommand{\DVA}{%
{\psset{unit=1}
\put(0,0){\pscircle[fillstyle=solid,fillcolor=white,linecolor=white]{.15}}
\put(0,0){\makebox(0,0)[cc]{\hbox{\tcw{\large$\mathbf 2$}}}}
\put(0,0){\makebox(0,0)[cc]{\hbox{\tcr{$\mathbf 2$}}}}
}}
\newcommand{\TRI}{%
{\psset{unit=1}
\put(0,0){\pscircle[fillstyle=solid,fillcolor=white,linecolor=white]{.15}}
\put(0,0){\makebox(0,0)[cc]{\hbox{\tcw{\large$\mathbf 3$}}}}
\put(0,0){\makebox(0,0)[cc]{\hbox{\tcr{$\mathbf 3$}}}}
}}
\multiput(-2.5,-0.85)(0.5,0.85){4}{\PATLEFT}
\multiput(-2,-1.7)(1,0){4}{\PATBOTTOM}
\put(-0.5,2.55){\PATTOP}
\multiput(-1.5,-0.85)(1,0){4}{\PATGEN}
\multiput(-1,0)(1,0){3}{\PATGEN}
\multiput(-.5,0.85)(1,0){2}{\PATGEN}
\put(0,1.7){\PATGEN}
\multiput(-1.5,-0.85)(1,0){4}{\PATGEN}
\multiput(0.5,2.55)(0.5,-0.85){4}{\PATRIGHT}
\put(2,-1.7){\PATBOTRIGHT}
\psarc[linecolor=red, linewidth=1.5pt]{<->}(-1.5,-0.85){1.7}{100}{320}
\psarc[linecolor=red, linewidth=1.5pt]{<->}(-2,0){2}{65}{295}
\psarc[linecolor=red, linewidth=1.5pt]{<->}(-1,-1.7){2}{125}{355}
\psarc[linecolor=red, linewidth=1.5pt]{<->}(-0.8,0.25){2.35}{87}{235}
\rput{100}(-0.1,0){\psarc[linecolor=red, linewidth=1.5pt]{<->}(-0.8,0.25){2.36}{85}{235}}
\multiput(-2,-1.7)(0.5,0.85){5}{\psline[linecolor=red,linewidth=2pt]{->}(0,0)(.45,.765)}
\multiput(-0.5,2.55)(1,0){1}{\psline[linecolor=red,linewidth=2pt]{->}(1,0)(.1,0)}
\multiput(-2,-1.7)(0.5,0.85){5}{\ODIN}
\put(0.5,2.55){\DVA}
\put(-0.5,2.55){\ODIN}
%
}
\end{pspicture}
\begin{pspicture}(-2.7,-2.5)(2.7,2.5){\psset{unit=0.7}
\newcommand{\PATGEN}{%
{\psset{unit=1}
\rput(0,0){\psline[linecolor=blue,linewidth=2pt]{->}(0,0)(.45,.765)}
\rput(0,0){\psline[linecolor=blue,linewidth=2pt]{->}(1,0)(0.1,0)}
\rput(0,0){\psline[linecolor=blue,linewidth=2pt]{->}(0,0)(.45,-.765)}
\put(0,0){\pscircle[fillstyle=solid,fillcolor=lightgray]{.1}}
}}
\newcommand{\PATLEFT}{%
{\psset{unit=1}
\rput(0,0){\psline[linecolor=blue,linewidth=2pt,linestyle=dashed]{->}(0,0)(.45,.765)}
\rput(0,0){\psline[linecolor=blue,linewidth=2pt]{->}(1,0)(0.1,0)}
\rput(0,0){\psline[linecolor=blue,linewidth=2pt]{->}(0,0)(.45,-.765)}
\put(0,0){\pscircle[fillstyle=solid,fillcolor=lightgray]{.1}}
}}
\newcommand{\PATRIGHT}{%
{\psset{unit=1}
\rput(0,0){\psline[linecolor=blue,linewidth=2pt,linestyle=dashed]{->}(0,0)(.45,-.765)}
\put(0,0){\pscircle[fillstyle=solid,fillcolor=lightgray]{.1}}
}}
\newcommand{\PATBOTTOM}{%
{\psset{unit=1}
\rput(0,0){\psline[linecolor=blue,linewidth=2pt]{->}(0,0)(.45,.765)}
\rput(0,0){\psline[linecolor=blue,linewidth=2pt,linestyle=dashed]{->}(1,0)(0.1,0)}
\put(0,0){\pscircle[fillstyle=solid,fillcolor=lightgray]{.1}}
}}
\newcommand{\PATTOP}{%
{\psset{unit=1}
\rput(0,0){\psline[linecolor=blue,linewidth=2pt]{->}(1,0)(0.1,0)}
\rput(0,0){\psline[linecolor=blue,linewidth=2pt]{->}(0,0)(.45,-.765)}
\put(0,0){\pscircle[fillstyle=solid,fillcolor=lightgray]{.1}}
}}
\newcommand{\PATBOTRIGHT}{%
{\psset{unit=1}
\rput(0,0){\psline[linecolor=blue,linewidth=2pt]{->}(0,0)(.45,.765)}
\put(0,0){\pscircle[fillstyle=solid,fillcolor=lightgray]{.1}}
\put(.5,0.85){\pscircle[fillstyle=solid,fillcolor=lightgray]{.1}}
}}
\newcommand{\ODIN}{%
{\psset{unit=1}
\put(0,0){\pscircle[fillstyle=solid,fillcolor=white,linecolor=white]{.15}}
\put(0,0){\makebox(0,0)[cc]{\hbox{\tcw{\large$\mathbf 1$}}}}
\put(0,0){\makebox(0,0)[cc]{\hbox{\tcr{$\mathbf 1$}}}}
}}
\newcommand{\DVA}{%
{\psset{unit=1}
\put(0,0){\pscircle[fillstyle=solid,fillcolor=white,linecolor=white]{.15}}
\put(0,0){\makebox(0,0)[cc]{\hbox{\tcw{\large$\mathbf 2$}}}}
\put(0,0){\makebox(0,0)[cc]{\hbox{\tcr{$\mathbf 2$}}}}
}}
\newcommand{\TRI}{%
{\psset{unit=1}
\put(0,0){\pscircle[fillstyle=solid,fillcolor=white,linecolor=white]{.15}}
\put(0,0){\makebox(0,0)[cc]{\hbox{\tcw{\large$\mathbf 3$}}}}
\put(0,0){\makebox(0,0)[cc]{\hbox{\tcr{$\mathbf 3$}}}}
}}
\multiput(-2.5,-0.85)(0.5,0.85){4}{\PATLEFT}
\multiput(-2,-1.7)(1,0){4}{\PATBOTTOM}
\put(-0.5,2.55){\PATTOP}
\multiput(-1.5,-0.85)(1,0){4}{\PATGEN}
\multiput(-1,0)(1,0){3}{\PATGEN}
\multiput(-.5,0.85)(1,0){2}{\PATGEN}
\put(0,1.7){\PATGEN}
\multiput(-1.5,-0.85)(1,0){4}{\PATGEN}
\multiput(0.5,2.55)(0.5,-0.85){4}{\PATRIGHT}
\put(2,-1.7){\PATBOTRIGHT}
\psarc[linecolor=red, linewidth=1.5pt]{<->}(-1.5,-0.85){1.7}{100}{320}
\psarc[linecolor=red, linewidth=1.5pt]{<->}(-2,0){2}{65}{295}
\psarc[linecolor=red, linewidth=1.5pt]{<->}(-1,-1.7){2}{125}{355}
\psarc[linecolor=red, linewidth=1.5pt]{<->}(-0.8,0.25){2.35}{87}{235}
\rput{100}(-0.1,0){\psarc[linecolor=red, linewidth=1.5pt]{<->}(-0.8,0.25){2.36}{85}{235}}
\multiput(-1,-1.7)(0.5,0.85){4}{\psline[linecolor=red,linewidth=2pt]{->}(0,0)(.45,.765)}
\multiput(-1,1.7)(1,0){2}{\psline[linecolor=red,linewidth=2pt]{->}(1,0)(.1,0)}
\multiput(-1,-1.7)(0.5,0.85){4}{\ODIN}
\put(1,1.7){\DVA}
\multiput(-1,1.7)(1,0){2}{\ODIN}
%
}
\end{pspicture}
\begin{pspicture}(-2.7,-2.5)(2.7,2.5){\psset{unit=0.7}
\newcommand{\PATGEN}{%
{\psset{unit=1}
\rput(0,0){\psline[linecolor=blue,linewidth=2pt]{->}(0,0)(.45,.765)}
\rput(0,0){\psline[linecolor=blue,linewidth=2pt]{->}(1,0)(0.1,0)}
\rput(0,0){\psline[linecolor=blue,linewidth=2pt]{->}(0,0)(.45,-.765)}
\put(0,0){\pscircle[fillstyle=solid,fillcolor=lightgray]{.1}}
}}
\newcommand{\PATLEFT}{%
{\psset{unit=1}
\rput(0,0){\psline[linecolor=blue,linewidth=2pt,linestyle=dashed]{->}(0,0)(.45,.765)}
\rput(0,0){\psline[linecolor=blue,linewidth=2pt]{->}(1,0)(0.1,0)}
\rput(0,0){\psline[linecolor=blue,linewidth=2pt]{->}(0,0)(.45,-.765)}
\put(0,0){\pscircle[fillstyle=solid,fillcolor=lightgray]{.1}}
}}
\newcommand{\PATRIGHT}{%
{\psset{unit=1}
\rput(0,0){\psline[linecolor=blue,linewidth=2pt,linestyle=dashed]{->}(0,0)(.45,-.765)}
\put(0,0){\pscircle[fillstyle=solid,fillcolor=lightgray]{.1}}
}}
\newcommand{\PATBOTTOM}{%
{\psset{unit=1}
\rput(0,0){\psline[linecolor=blue,linewidth=2pt]{->}(0,0)(.45,.765)}
\rput(0,0){\psline[linecolor=blue,linewidth=2pt,linestyle=dashed]{->}(1,0)(0.1,0)}
\put(0,0){\pscircle[fillstyle=solid,fillcolor=lightgray]{.1}}
}}
\newcommand{\PATTOP}{%
{\psset{unit=1}
\rput(0,0){\psline[linecolor=blue,linewidth=2pt]{->}(1,0)(0.1,0)}
\rput(0,0){\psline[linecolor=blue,linewidth=2pt]{->}(0,0)(.45,-.765)}
\put(0,0){\pscircle[fillstyle=solid,fillcolor=lightgray]{.1}}
}}
\newcommand{\PATBOTRIGHT}{%
{\psset{unit=1}
\rput(0,0){\psline[linecolor=blue,linewidth=2pt]{->}(0,0)(.45,.765)}
\put(0,0){\pscircle[fillstyle=solid,fillcolor=lightgray]{.1}}
\put(.5,0.85){\pscircle[fillstyle=solid,fillcolor=lightgray]{.1}}
}}
\newcommand{\ODIN}{%
{\psset{unit=1}
\put(0,0){\pscircle[fillstyle=solid,fillcolor=white,linecolor=white]{.15}}
\put(0,0){\makebox(0,0)[cc]{\hbox{\tcw{\large$\mathbf 1$}}}}
\put(0,0){\makebox(0,0)[cc]{\hbox{\tcr{$\mathbf 1$}}}}
}}
\newcommand{\DVA}{%
{\psset{unit=1}
\put(0,0){\pscircle[fillstyle=solid,fillcolor=white,linecolor=white]{.15}}
\put(0,0){\makebox(0,0)[cc]{\hbox{\tcw{\large$\mathbf 2$}}}}
\put(0,0){\makebox(0,0)[cc]{\hbox{\tcr{$\mathbf 2$}}}}
}}
\newcommand{\TRI}{%
{\psset{unit=1}
\put(0,0){\pscircle[fillstyle=solid,fillcolor=white,linecolor=white]{.15}}
\put(0,0){\makebox(0,0)[cc]{\hbox{\tcw{\large$\mathbf 3$}}}}
\put(0,0){\makebox(0,0)[cc]{\hbox{\tcr{$\mathbf 3$}}}}
}}
\multiput(-2.5,-0.85)(0.5,0.85){4}{\PATLEFT}
\multiput(-2,-1.7)(1,0){4}{\PATBOTTOM}
\put(-0.5,2.55){\PATTOP}
\multiput(-1.5,-0.85)(1,0){4}{\PATGEN}
\multiput(-1,0)(1,0){3}{\PATGEN}
\multiput(-.5,0.85)(1,0){2}{\PATGEN}
\put(0,1.7){\PATGEN}
\multiput(-1.5,-0.85)(1,0){4}{\PATGEN}
\multiput(0.5,2.55)(0.5,-0.85){4}{\PATRIGHT}
\put(2,-1.7){\PATBOTRIGHT}
\psarc[linecolor=red, linewidth=1.5pt]{<->}(-1.5,-0.85){1.7}{100}{320}
\psarc[linecolor=red, linewidth=1.5pt]{<->}(-2,0){2}{65}{295}
\psarc[linecolor=red, linewidth=1.5pt]{<->}(-1,-1.7){2}{125}{355}
\psarc[linecolor=red, linewidth=1.5pt]{<->}(-0.8,0.25){2.35}{87}{235}
\rput{100}(-0.1,0){\psarc[linecolor=red, linewidth=1.5pt]{<->}(-0.8,0.25){2.36}{85}{235}}
\multiput(0,-1.7)(0.5,0.85){3}{\psline[linecolor=red,linewidth=2pt]{->}(0,0)(.45,.765)}
\multiput(-1.5,0.85)(1,0){3}{\psline[linecolor=red,linewidth=2pt]{->}(1,0)(.1,0)}
\multiput(0,-1.7)(0.5,0.85){3}{\ODIN}
\put(1.5,.85){\DVA}
\multiput(-1.5,.85)(1,0){3}{\ODIN}
}
\end{pspicture}
\begin{pspicture}(-2.7,-2.5)(2.7,2.5){\psset{unit=0.7}
%
\newcommand{\PATGEN}{%
{\psset{unit=1}
\rput(0,0){\psline[linecolor=blue,linewidth=2pt]{->}(0,0)(.45,.765)}
\rput(0,0){\psline[linecolor=blue,linewidth=2pt]{->}(1,0)(0.1,0)}
\rput(0,0){\psline[linecolor=blue,linewidth=2pt]{->}(0,0)(.45,-.765)}
\put(0,0){\pscircle[fillstyle=solid,fillcolor=lightgray]{.1}}
}}
\newcommand{\PATLEFT}{%
{\psset{unit=1}
\rput(0,0){\psline[linecolor=blue,linewidth=2pt,linestyle=dashed]{->}(0,0)(.45,.765)}
\rput(0,0){\psline[linecolor=blue,linewidth=2pt]{->}(1,0)(0.1,0)}
\rput(0,0){\psline[linecolor=blue,linewidth=2pt]{->}(0,0)(.45,-.765)}
\put(0,0){\pscircle[fillstyle=solid,fillcolor=lightgray]{.1}}
}}
\newcommand{\PATRIGHT}{%
{\psset{unit=1}
\rput(0,0){\psline[linecolor=blue,linewidth=2pt,linestyle=dashed]{->}(0,0)(.45,-.765)}
\put(0,0){\pscircle[fillstyle=solid,fillcolor=lightgray]{.1}}
}}
\newcommand{\PATBOTTOM}{%
{\psset{unit=1}
\rput(0,0){\psline[linecolor=blue,linewidth=2pt]{->}(0,0)(.45,.765)}
\rput(0,0){\psline[linecolor=blue,linewidth=2pt,linestyle=dashed]{->}(1,0)(0.1,0)}
\put(0,0){\pscircle[fillstyle=solid,fillcolor=lightgray]{.1}}
}}
\newcommand{\PATTOP}{%
{\psset{unit=1}
\rput(0,0){\psline[linecolor=blue,linewidth=2pt]{->}(1,0)(0.1,0)}
\rput(0,0){\psline[linecolor=blue,linewidth=2pt]{->}(0,0)(.45,-.765)}
\put(0,0){\pscircle[fillstyle=solid,fillcolor=lightgray]{.1}}
}}
\newcommand{\PATBOTRIGHT}{%
{\psset{unit=1}
\rput(0,0){\psline[linecolor=blue,linewidth=2pt]{->}(0,0)(.45,.765)}
\put(0,0){\pscircle[fillstyle=solid,fillcolor=lightgray]{.1}}
\put(.5,0.85){\pscircle[fillstyle=solid,fillcolor=lightgray]{.1}}
}}
\newcommand{\ODIN}{%
{\psset{unit=1}
\put(0,0){\pscircle[fillstyle=solid,fillcolor=white,linecolor=white]{.15}}
\put(0,0){\makebox(0,0)[cc]{\hbox{\tcw{\large$\mathbf 1$}}}}
\put(0,0){\makebox(0,0)[cc]{\hbox{\tcr{$\mathbf 1$}}}}
}}
\newcommand{\DVA}{%
{\psset{unit=1}
\put(0,0){\pscircle[fillstyle=solid,fillcolor=white,linecolor=white]{.15}}
\put(0,0){\makebox(0,0)[cc]{\hbox{\tcw{\large$\mathbf 2$}}}}
\put(0,0){\makebox(0,0)[cc]{\hbox{\tcr{$\mathbf 2$}}}}
}}
\newcommand{\TRI}{%
{\psset{unit=1}
\put(0,0){\pscircle[fillstyle=solid,fillcolor=white,linecolor=white]{.15}}
\put(0,0){\makebox(0,0)[cc]{\hbox{\tcw{\large$\mathbf 3$}}}}
\put(0,0){\makebox(0,0)[cc]{\hbox{\tcr{$\mathbf 3$}}}}
}}
\multiput(-2.5,-0.85)(0.5,0.85){4}{\PATLEFT}
\multiput(-2,-1.7)(1,0){4}{\PATBOTTOM}
\put(-0.5,2.55){\PATTOP}
\multiput(-1.5,-0.85)(1,0){4}{\PATGEN}
\multiput(-1,0)(1,0){3}{\PATGEN}
\multiput(-.5,0.85)(1,0){2}{\PATGEN}
\put(0,1.7){\PATGEN}
\multiput(-1.5,-0.85)(1,0){4}{\PATGEN}
\multiput(0.5,2.55)(0.5,-0.85){4}{\PATRIGHT}
\put(2,-1.7){\PATBOTRIGHT}
\multiput(1,-1.7)(0.5,0.85){2}{\psline[linecolor=red,linewidth=2pt]{->}(0,0)(.45,.765)}
\multiput(-2,0)(1,0){4}{\psline[linecolor=red,linewidth=2pt]{->}(1,0)(.1,0)}
\psarc[linecolor=red, linewidth=1.5pt]{<->}(-1.5,-0.85){1.7}{100}{320}
\psarc[linecolor=red, linewidth=1.5pt]{<->}(-2,0){2}{65}{295}
\psarc[linecolor=red, linewidth=1.5pt]{<->}(-1,-1.7){2}{125}{355}
\psarc[linecolor=red, linewidth=1.5pt]{<->}(-0.8,0.25){2.35}{87}{235}
\rput{100}(-0.1,0){\psarc[linecolor=red, linewidth=1.5pt]{<->}(-0.8,0.25){2.36}{85}{235}}
\multiput(1,-1.7)(0.5,0.85){2}{\ODIN}
\put(2,0){\DVA}
\multiput(-2,0)(1,0){4}{\ODIN}
%
}
\end{pspicture}
\begin{pspicture}(-2.7,-2.5)(2.7,2.5){\psset{unit=0.7}
%
\newcommand{\PATGEN}{%
{\psset{unit=1}
\rput(0,0){\psline[linecolor=blue,linewidth=2pt]{->}(0,0)(.45,.765)}
\rput(0,0){\psline[linecolor=blue,linewidth=2pt]{->}(1,0)(0.1,0)}
\rput(0,0){\psline[linecolor=blue,linewidth=2pt]{->}(0,0)(.45,-.765)}
\put(0,0){\pscircle[fillstyle=solid,fillcolor=lightgray]{.1}}
}}
\newcommand{\PATLEFT}{%
{\psset{unit=1}
\rput(0,0){\psline[linecolor=blue,linewidth=2pt,linestyle=dashed]{->}(0,0)(.45,.765)}
\rput(0,0){\psline[linecolor=blue,linewidth=2pt]{->}(1,0)(0.1,0)}
\rput(0,0){\psline[linecolor=blue,linewidth=2pt]{->}(0,0)(.45,-.765)}
\put(0,0){\pscircle[fillstyle=solid,fillcolor=lightgray]{.1}}
}}
\newcommand{\PATRIGHT}{%
{\psset{unit=1}
\rput(0,0){\psline[linecolor=blue,linewidth=2pt,linestyle=dashed]{->}(0,0)(.45,-.765)}
\put(0,0){\pscircle[fillstyle=solid,fillcolor=lightgray]{.1}}
}}
\newcommand{\PATBOTTOM}{%
{\psset{unit=1}
\rput(0,0){\psline[linecolor=blue,linewidth=2pt]{->}(0,0)(.45,.765)}
\rput(0,0){\psline[linecolor=blue,linewidth=2pt,linestyle=dashed]{->}(1,0)(0.1,0)}
\put(0,0){\pscircle[fillstyle=solid,fillcolor=lightgray]{.1}}
}}
\newcommand{\PATTOP}{%
{\psset{unit=1}
\rput(0,0){\psline[linecolor=blue,linewidth=2pt]{->}(1,0)(0.1,0)}
\rput(0,0){\psline[linecolor=blue,linewidth=2pt]{->}(0,0)(.45,-.765)}
\put(0,0){\pscircle[fillstyle=solid,fillcolor=lightgray]{.1}}
}}
\newcommand{\PATBOTRIGHT}{%
{\psset{unit=1}
\rput(0,0){\psline[linecolor=blue,linewidth=2pt]{->}(0,0)(.45,.765)}
\put(0,0){\pscircle[fillstyle=solid,fillcolor=lightgray]{.1}}
\put(.5,0.85){\pscircle[fillstyle=solid,fillcolor=lightgray]{.1}}
}}
\newcommand{\ODIN}{%
{\psset{unit=1}
\put(0,0){\pscircle[fillstyle=solid,fillcolor=white,linecolor=white]{.15}}
\put(0,0){\makebox(0,0)[cc]{\hbox{\tcw{\large$\mathbf 1$}}}}
\put(0,0){\makebox(0,0)[cc]{\hbox{\tcr{$\mathbf 1$}}}}
}}
\newcommand{\DVA}{%
{\psset{unit=1}
\put(0,0){\pscircle[fillstyle=solid,fillcolor=white,linecolor=white]{.15}}
\put(0,0){\makebox(0,0)[cc]{\hbox{\tcw{\large$\mathbf 2$}}}}
\put(0,0){\makebox(0,0)[cc]{\hbox{\tcr{$\mathbf 2$}}}}
}}
\newcommand{\TRI}{%
{\psset{unit=1}
\put(0,0){\pscircle[fillstyle=solid,fillcolor=white,linecolor=white]{.15}}
\put(0,0){\makebox(0,0)[cc]{\hbox{\tcw{\large$\mathbf 3$}}}}
\put(0,0){\makebox(0,0)[cc]{\hbox{\tcr{$\mathbf 3$}}}}
}}
\multiput(-2.5,-0.85)(0.5,0.85){4}{\PATLEFT}
\multiput(-2,-1.7)(1,0){4}{\PATBOTTOM}
\put(-0.5,2.55){\PATTOP}
\multiput(-1.5,-0.85)(1,0){4}{\PATGEN}
\multiput(-1,0)(1,0){3}{\PATGEN}
\multiput(-.5,0.85)(1,0){2}{\PATGEN}
\put(0,1.7){\PATGEN}
\multiput(-1.5,-0.85)(1,0){4}{\PATGEN}
\multiput(0.5,2.55)(0.5,-0.85){4}{\PATRIGHT}
\put(2,-1.7){\PATBOTRIGHT}
\psarc[linecolor=red, linewidth=1.5pt]{<->}(-1.5,-0.85){1.7}{100}{320}
\psarc[linecolor=red, linewidth=1.5pt]{<->}(-2,0){2}{65}{295}
\psarc[linecolor=red, linewidth=1.5pt]{<->}(-1,-1.7){2}{125}{355}
\psarc[linecolor=red, linewidth=1.5pt]{<->}(-0.8,0.25){2.35}{87}{235}
\rput{100}(-0.1,0){\psarc[linecolor=red, linewidth=1.5pt]{<->}(-0.8,0.25){2.36}{85}{235}}
\multiput(2,-1.7)(0.5,0.85){1}{\psline[linecolor=red,linewidth=2pt]{->}(0,0)(.45,.765)}
\multiput(-2.5,-0.85)(1,0){5}{\psline[linecolor=red,linewidth=2pt]{->}(1,0)(.1,0)}
\put(2,-1.7){\ODIN}
\put(2.5,-0.85){\DVA}
\multiput(-2.5,-0.85)(1,0){5}{\ODIN}
%
}
\end{pspicture}
\caption{\small
Five new central elements of the main quiver for $SL_6$ due to amalgamation. (We use these central elements to set all diagonal elements of the upper-triangular matrix $\mathbb A=M_1^{\text{T}}M_2$ to be the unities.)
}
\label{fi:diagonal-Casimirs}
\end{figure}
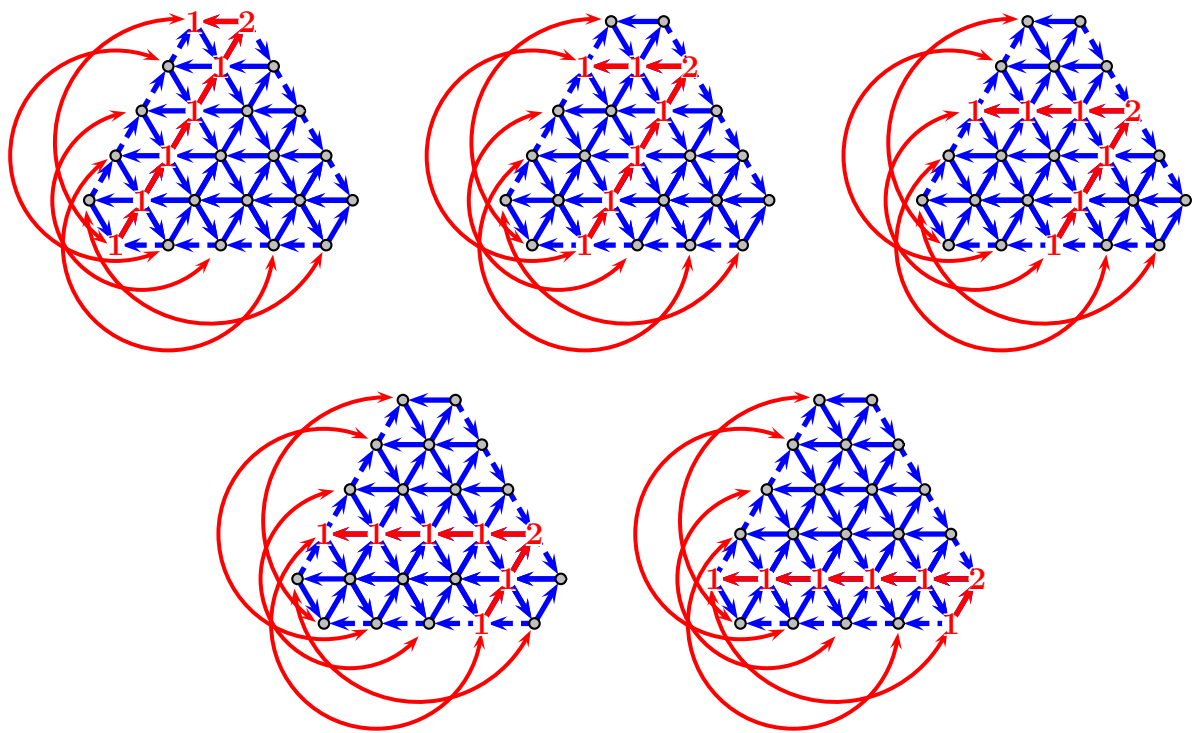

\begin{figure}[H]
\begin{pspicture}(-2.7,-2.5)(2.7,2.5){\psset{unit=0.7}
\newcommand{\PATGEN}{%
{\psset{unit=1}
\rput(0,0){\psline[linecolor=blue,linewidth=2pt]{->}(0,0)(.45,.765)}
\rput(0,0){\psline[linecolor=blue,linewidth=2pt]{->}(1,0)(0.1,0)}
\rput(0,0){\psline[linecolor=blue,linewidth=2pt]{->}(0,0)(.45,-.765)}
\put(0,0){\pscircle[fillstyle=solid,fillcolor=lightgray]{.1}}
}}
\newcommand{\PATLEFT}{%
{\psset{unit=1}
\rput(0,0){\psline[linecolor=blue,linewidth=2pt,linestyle=dashed]{->}(0,0)(.45,.765)}
\rput(0,0){\psline[linecolor=blue,linewidth=2pt]{->}(1,0)(0.1,0)}
\rput(0,0){\psline[linecolor=blue,linewidth=2pt]{->}(0,0)(.45,-.765)}
\put(0,0){\pscircle[fillstyle=solid,fillcolor=lightgray]{.1}}
}}
\newcommand{\PATRIGHT}{%
{\psset{unit=1}
\rput(0,0){\psline[linecolor=blue,linewidth=2pt,linestyle=dashed]{->}(0,0)(.45,-.765)}
\put(0,0){\pscircle[fillstyle=solid,fillcolor=lightgray]{.1}}
}}
\newcommand{\PATBOTTOM}{%
{\psset{unit=1}
\rput(0,0){\psline[linecolor=blue,linewidth=2pt]{->}(0,0)(.45,.765)}
\rput(0,0){\psline[linecolor=blue,linewidth=2pt,linestyle=dashed]{->}(1,0)(0.1,0)}
\put(0,0){\pscircle[fillstyle=solid,fillcolor=lightgray]{.1}}
}}
\newcommand{\PATTOP}{%
{\psset{unit=1}
\rput(0,0){\psline[linecolor=blue,linewidth=2pt]{->}(1,0)(0.1,0)}
\rput(0,0){\psline[linecolor=blue,linewidth=2pt]{->}(0,0)(.45,-.765)}
\put(0,0){\pscircle[fillstyle=solid,fillcolor=lightgray]{.1}}
}}
\newcommand{\PATBOTRIGHT}{%
{\psset{unit=1}
\rput(0,0){\psline[linecolor=blue,linewidth=2pt]{->}(0,0)(.45,.765)}
\put(0,0){\pscircle[fillstyle=solid,fillcolor=lightgray]{.1}}
\put(.5,0.85){\pscircle[fillstyle=solid,fillcolor=lightgray]{.1}}
}}
\newcommand{\ODIN}{%
{\psset{unit=1}
\put(0,0){\pscircle[fillstyle=solid,fillcolor=white,linecolor=white]{.15}}
\put(0,0){\makebox(0,0)[cc]{\hbox{\tcw{\large$\mathbf 1$}}}}
\put(0,0){\makebox(0,0)[cc]{\hbox{\tcr{$\mathbf 1$}}}}
}}
\multiput(-2.5,-0.85)(0.5,0.85){4}{\PATLEFT}
\multiput(-2,-1.7)(1,0){4}{\PATBOTTOM}
\put(-0.5,2.55){\PATTOP}
\multiput(-1.5,-0.85)(1,0){4}{\PATGEN}
\multiput(-1,0)(1,0){3}{\PATGEN}
\multiput(-.5,0.85)(1,0){2}{\PATGEN}
\put(0,1.7){\PATGEN}
\multiput(-1.5,-0.85)(1,0){4}{\PATGEN}
\multiput(0.5,2.55)(0.5,-0.85){4}{\PATRIGHT}
\put(2,-1.7){\PATBOTRIGHT}
\psarc[linecolor=red, linewidth=1.5pt]{<->}(-1.5,-0.85){1.7}{100}{320}
\psarc[linecolor=red, linewidth=1.5pt]{<->}(-2,0){2}{65}{295}
\psarc[linecolor=red, linewidth=1.5pt]{<->}(-1,-1.7){2}{125}{355}
\psarc[linecolor=red, linewidth=1.5pt]{<->}(-0.8,0.25){2.35}{87}{235}
\rput{100}(-0.1,0){\psarc[linecolor=red, linewidth=1.5pt]{<->}(-0.8,0.25){2.36}{85}{235}}
\multiput(-0.5,2.55)(0.5,-0.85){5}{\psline[linecolor=red,linewidth=2pt]{->}(0,0)(.45,-.765)}
\multiput(-2.5,-0.85)(0.5,-0.85){1}{\psline[linecolor=red,linewidth=2pt]{->}(0,0)(.45,-.765)}
\multiput(-.5,2.55)(0.5,-0.85){6}{\ODIN}
\multiput(-2.5,-0.85)(0.5,-0.85){2}{\ODIN}
%
}
\end{pspicture}
\begin{pspicture}(-2.7,-2.5)(2.7,2.5){\psset{unit=0.7}
\newcommand{\PATGEN}{%
{\psset{unit=1}
\rput(0,0){\psline[linecolor=blue,linewidth=2pt]{->}(0,0)(.45,.765)}
\rput(0,0){\psline[linecolor=blue,linewidth=2pt]{->}(1,0)(0.1,0)}
\rput(0,0){\psline[linecolor=blue,linewidth=2pt]{->}(0,0)(.45,-.765)}
\put(0,0){\pscircle[fillstyle=solid,fillcolor=lightgray]{.1}}
}}
\newcommand{\PATLEFT}{%
{\psset{unit=1}
\rput(0,0){\psline[linecolor=blue,linewidth=2pt,linestyle=dashed]{->}(0,0)(.45,.765)}
\rput(0,0){\psline[linecolor=blue,linewidth=2pt]{->}(1,0)(0.1,0)}
\rput(0,0){\psline[linecolor=blue,linewidth=2pt]{->}(0,0)(.45,-.765)}
\put(0,0){\pscircle[fillstyle=solid,fillcolor=lightgray]{.1}}
}}
\newcommand{\PATRIGHT}{%
{\psset{unit=1}
\rput(0,0){\psline[linecolor=blue,linewidth=2pt,linestyle=dashed]{->}(0,0)(.45,-.765)}
\put(0,0){\pscircle[fillstyle=solid,fillcolor=lightgray]{.1}}
}}
\newcommand{\PATBOTTOM}{%
{\psset{unit=1}
\rput(0,0){\psline[linecolor=blue,linewidth=2pt]{->}(0,0)(.45,.765)}
\rput(0,0){\psline[linecolor=blue,linewidth=2pt,linestyle=dashed]{->}(1,0)(0.1,0)}
\put(0,0){\pscircle[fillstyle=solid,fillcolor=lightgray]{.1}}
}}
\newcommand{\PATTOP}{%
{\psset{unit=1}
\rput(0,0){\psline[linecolor=blue,linewidth=2pt]{->}(1,0)(0.1,0)}
\rput(0,0){\psline[linecolor=blue,linewidth=2pt]{->}(0,0)(.45,-.765)}
\put(0,0){\pscircle[fillstyle=solid,fillcolor=lightgray]{.1}}
}}
\newcommand{\PATBOTRIGHT}{%
{\psset{unit=1}
\rput(0,0){\psline[linecolor=blue,linewidth=2pt]{->}(0,0)(.45,.765)}
\put(0,0){\pscircle[fillstyle=solid,fillcolor=lightgray]{.1}}
\put(.5,0.85){\pscircle[fillstyle=solid,fillcolor=lightgray]{.1}}
}}
\newcommand{\ODIN}{%
{\psset{unit=1}
\put(0,0){\pscircle[fillstyle=solid,fillcolor=white,linecolor=white]{.15}}
\put(0,0){\makebox(0,0)[cc]{\hbox{\tcw{\large$\mathbf 1$}}}}
\put(0,0){\makebox(0,0)[cc]{\hbox{\tcr{$\mathbf 1$}}}}
}}
\multiput(-2.5,-0.85)(0.5,0.85){4}{\PATLEFT}
\multiput(-2,-1.7)(1,0){4}{\PATBOTTOM}
\put(-0.5,2.55){\PATTOP}
\multiput(-1.5,-0.85)(1,0){4}{\PATGEN}
\multiput(-1,0)(1,0){3}{\PATGEN}
\multiput(-.5,0.85)(1,0){2}{\PATGEN}
\put(0,1.7){\PATGEN}
\multiput(-1.5,-0.85)(1,0){4}{\PATGEN}
\multiput(0.5,2.55)(0.5,-0.85){4}{\PATRIGHT}
\put(2,-1.7){\PATBOTRIGHT}
\psarc[linecolor=red, linewidth=1.5pt]{<->}(-1.5,-0.85){1.7}{100}{320}
\psarc[linecolor=red, linewidth=1.5pt]{<->}(-2,0){2}{65}{295}
\psarc[linecolor=red, linewidth=1.5pt]{<->}(-1,-1.7){2}{125}{355}
\psarc[linecolor=red, linewidth=1.5pt]{<->}(-0.8,0.25){2.35}{87}{235}
\rput{100}(-0.1,0){\psarc[linecolor=red, linewidth=1.5pt]{<->}(-0.8,0.25){2.36}{85}{235}}
\multiput(-1,1.7)(0.5,-0.85){4}{\psline[linecolor=red,linewidth=2pt]{->}(0,0)(.45,-.765)}
\multiput(-2,0)(0.5,-0.85){2}{\psline[linecolor=red,linewidth=2pt]{->}(0,0)(.45,-.765)}
\multiput(-1,1.7)(0.5,-0.85){5}{\ODIN}
\multiput(-2,0)(0.5,-0.85){3}{\ODIN}
%
}
\end{pspicture}
\begin{pspicture}(-2.3,-2.5)(2.7,2.5){\psset{unit=0.7}
\newcommand{\PATGEN}{%
{\psset{unit=1}
\rput(0,0){\psline[linecolor=blue,linewidth=2pt]{->}(0,0)(.45,.765)}
\rput(0,0){\psline[linecolor=blue,linewidth=2pt]{->}(1,0)(0.1,0)}
\rput(0,0){\psline[linecolor=blue,linewidth=2pt]{->}(0,0)(.45,-.765)}
\put(0,0){\pscircle[fillstyle=solid,fillcolor=lightgray]{.1}}
}}
\newcommand{\PATLEFT}{%
{\psset{unit=1}
\rput(0,0){\psline[linecolor=blue,linewidth=2pt,linestyle=dashed]{->}(0,0)(.45,.765)}
\rput(0,0){\psline[linecolor=blue,linewidth=2pt]{->}(1,0)(0.1,0)}
\rput(0,0){\psline[linecolor=blue,linewidth=2pt]{->}(0,0)(.45,-.765)}
\put(0,0){\pscircle[fillstyle=solid,fillcolor=lightgray]{.1}}
}}
\newcommand{\PATRIGHT}{%
{\psset{unit=1}
\rput(0,0){\psline[linecolor=blue,linewidth=2pt,linestyle=dashed]{->}(0,0)(.45,-.765)}
\put(0,0){\pscircle[fillstyle=solid,fillcolor=lightgray]{.1}}
}}
\newcommand{\PATBOTTOM}{%
{\psset{unit=1}
\rput(0,0){\psline[linecolor=blue,linewidth=2pt]{->}(0,0)(.45,.765)}
\rput(0,0){\psline[linecolor=blue,linewidth=2pt,linestyle=dashed]{->}(1,0)(0.1,0)}
\put(0,0){\pscircle[fillstyle=solid,fillcolor=lightgray]{.1}}
}}
\newcommand{\PATTOP}{%
{\psset{unit=1}
\rput(0,0){\psline[linecolor=blue,linewidth=2pt]{->}(1,0)(0.1,0)}
\rput(0,0){\psline[linecolor=blue,linewidth=2pt]{->}(0,0)(.45,-.765)}
\put(0,0){\pscircle[fillstyle=solid,fillcolor=lightgray]{.1}}
}}
\newcommand{\PATBOTRIGHT}{%
{\psset{unit=1}
\rput(0,0){\psline[linecolor=blue,linewidth=2pt]{->}(0,0)(.45,.765)}
\put(0,0){\pscircle[fillstyle=solid,fillcolor=lightgray]{.1}}
\put(.5,0.85){\pscircle[fillstyle=solid,fillcolor=lightgray]{.1}}
}}
\newcommand{\ODIN}{%
{\psset{unit=1}
\put(0,0){\pscircle[fillstyle=solid,fillcolor=white,linecolor=white]{.15}}
\put(0,0){\makebox(0,0)[cc]{\hbox{\tcw{\large$\mathbf 1$}}}}
\put(0,0){\makebox(0,0)[cc]{\hbox{\tcr{$\mathbf 1$}}}}
}}
\multiput(-2.5,-0.85)(0.5,0.85){4}{\PATLEFT}
\multiput(-2,-1.7)(1,0){4}{\PATBOTTOM}
\put(-0.5,2.55){\PATTOP}
\multiput(-1.5,-0.85)(1,0){4}{\PATGEN}
\multiput(-1,0)(1,0){3}{\PATGEN}
\multiput(-.5,0.85)(1,0){2}{\PATGEN}
\put(0,1.7){\PATGEN}
\multiput(-1.5,-0.85)(1,0){4}{\PATGEN}
\multiput(0.5,2.55)(0.5,-0.85){4}{\PATRIGHT}
\put(2,-1.7){\PATBOTRIGHT}
\psarc[linecolor=red, linewidth=1.5pt]{<->}(-1.5,-0.85){1.7}{100}{320}
\psarc[linecolor=red, linewidth=1.5pt]{<->}(-2,0){2}{65}{295}
\psarc[linecolor=red, linewidth=1.5pt]{<->}(-1,-1.7){2}{125}{355}
\psarc[linecolor=red, linewidth=1.5pt]{<->}(-0.8,0.25){2.35}{87}{235}
\rput{100}(-0.1,0){\psarc[linecolor=red, linewidth=1.5pt]{<->}(-0.8,0.25){2.36}{85}{235}}
\multiput(-1.5,0.85)(0.5,-0.85){3}{\psline[linecolor=red,linewidth=2pt]{->}(0,0)(.45,-.765)}
\multiput(-1.5,0.85)(0.5,-0.85){4}{\ODIN}
%
}
\end{pspicture}
\caption{\small
Three remaining central elements of the full-rank quiver for $SL_6$ after amalgamation and setting the diagonal elements of $\mathbb A$ equal to unities.
}
\label{fi:Casimirs-A}
\end{figure}
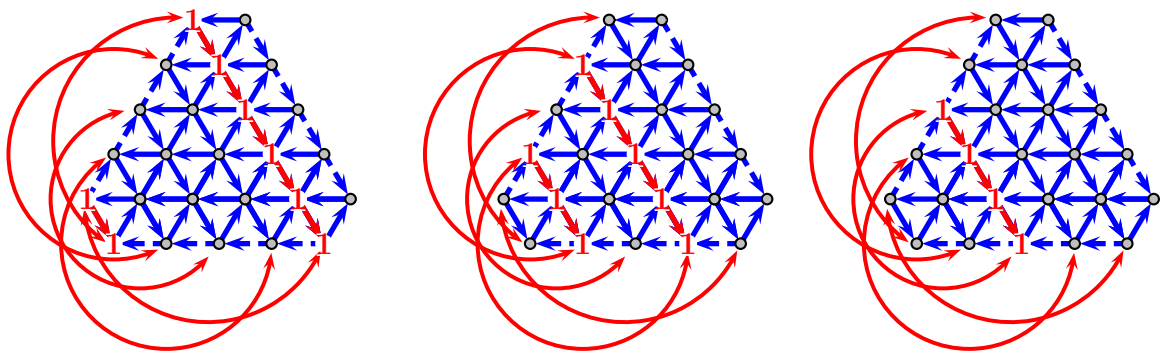

It is easy to see that all Casimirs from Lemma~\ref{lem:Casimir-slN} remain Casimirs in the amalgamated quiver (just four, or two, depending on the Casimir element, frozen variables become pairwise amalgamated). More, this amalgamation results in the appearance of $n-1$ new Casimirs; in Sec.~\ref{s:braid} we have used  these new Casimirs to eliminate the dependence of $\mathbb A$ on remaining $n-1$ frozen variables: diagonal entries of $\mathbb A=M_1^{\text{T}}M_2$ are particular products of these Casimirs, and we adjust the values of these Casimir operators to make all diagonal elements of $\mathbb A$ equal to the unities in the classsical case and $q^{-1/2}$ in the quantum case (recall that all Casimirs are assumed to be self-adjoint operators).

\begin{lemma}
The complete set of central elements for the amalgamated quiver in Fig.~\ref{fi:amalgamation} comprises $n-1$ new Casimirs depicted in Fig.~\ref{fi:diagonal-Casimirs} for the case of $SL_6$ and $\bigl[\frac{n}{2}\bigr]$ central elements (products of old Casimirs with the new ones) depicted in Fig.~\ref{fi:Casimirs-A}.
\end{lemma}

\section{Quantum Grassmannian and measurement maps}\label{sec:QuantumMeasurements}

\subsection{Non-normalized quantum transport matrices}\label{sect:proof:th:MM}

We add additional vertices labelled $(n,0,0), (0,n,0)$ and $(0,0,n)$ to the quiver of Fock-Goncharov parameters $Z_{abc}$  and construct dual planar bicolored (plabic) graph $G$ ( Figure~\ref{fi:plab}).
Then,  we define non-normalized quantum transport matrices $\M_1$ and $\M_2$ as quantization of 
boundary measurement matrices of graph $G$ introduced by Postnikov in \cite{Po}. 
Namely, we assign to every path $P$ connecting a source of $G$ to a sink a quantum weight $w(P)$  that is element of the quantum torus $\Upsilon$.
We define the boundary measurement between source $p$ and sink $q$ as $\M_{pq}=\displaystyle{\sum_{\text{path }P:p\leadsto q} w(P)}$.
Finally, note that $G$ has $n$ sources and $2n$ sinks, we organize boundary measurements $\M_{pq}$ into $2n\times n$ matrix that we divide into two $n\times n$ matrices $\M_1$ and $\M_2$.

Vertices of $G$ are colored into black and white color as follows: a black vertex has two incoming arrows and one outgoing, while a white vertex has two outgoing and one incoming arrows.

We equip faces of Figure~\ref{fi:plabic_weights} with weights $Z_\alpha$ associated with the corresponding vertices of graph Figure~\ref{fi:triangle}. We define the quantum weight of a maximal oriented path in $G$ by formula~\ref{eq:pathweight} (see Fig.\ref{fi:plabic_weights}).

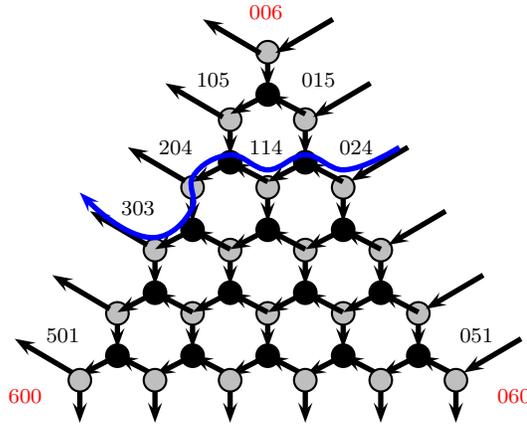
\begin{figure}[H]
	\begin{pspicture}(-3,-2)(4,2.8){
		\newcommand{\LEFTDOWNARROW}{%
			{\psset{unit=1}
				\rput(0,0){\psline[linecolor=black,linewidth=2pt]{<-}(0,0)(.765,.45)}
		}}
		\newcommand{\DOWNARROW}{%
			{\psset{unit=1}
				\rput(0,0){\psline[linecolor=black,linewidth=2pt]{->}(0,0.1)(0,-0.566)}
				\put(0,0){\pscircle[fillstyle=solid,fillcolor=lightgray]{.15}}
		}}
		\newcommand{\LEFTUPARROW}{%
			{\psset{unit=1}
				\rput(0,0){\psline[linecolor=black,linewidth=2pt]{->}(0,0)(-.765,.45)}
		}}
		\newcommand{\STARUP}{
			{\psset{unit=1}
				\rput(0,0){\psline[linecolor=black,linewidth=2pt]{<-}(0,0)(.5,-.26)}
				\rput(0,0){\psline[linecolor=black,linewidth=2pt]{<-}(0,0.1)(0,.466)}
				\rput(0,0){\psline[linecolor=black,linewidth=2pt]{->}(0,0)(-.5,-.26)}
				\put(0,0){\pscircle[fillstyle=solid,fillcolor=black]{.15}}
				\put(0,.566){\pscircle[fillstyle=solid,fillcolor=lightgray]{.15}}
		}}
		\newcommand{\PATGEN}{%
			{\psset{unit=1}
				\rput(0,0){\psline[linecolor=blue,linewidth=2pt]{->}(0,0)(.45,.765)}
				\rput(0,0){\psline[linecolor=blue,linewidth=2pt]{->}(1,0)(0.1,0)}
				\rput(0,0){\psline[linecolor=blue,linewidth=2pt]{->}(0,0)(.45,-.765)}
				\put(0,0){\pscircle[fillstyle=solid,fillcolor=lightgray]{.1}}
		}}
		\newcommand{\PATLEFT}{%
			{\psset{unit=1}
				\rput(0,0){\psline[linecolor=blue,linewidth=2pt,linestyle=dashed]{->}(0,0)(.45,.765)}
				\rput(0,0){\psline[linecolor=blue,linewidth=2pt]{->}(1,0)(0.1,0)}
				\rput(0,0){\psline[linecolor=blue,linewidth=2pt]{->}(0,0)(.45,-.765)}
				\put(0,0){\pscircle[fillstyle=solid,fillcolor=lightgray]{.1}}
		}}
		\newcommand{\PATRIGHT}{%
			{\psset{unit=1}
				\rput(0,0){\psline[linecolor=blue,linewidth=2pt,linestyle=dashed]{->}(0,0)(.45,-.765)}
				\put(0,0){\pscircle[fillstyle=solid,fillcolor=lightgray]{.1}}
		}}
		\newcommand{\PATBOTTOM}{%
			{\psset{unit=1}
				\rput(0,0){\psline[linecolor=blue,linewidth=2pt]{->}(0,0)(.45,.765)}
				\rput(0,0){\psline[linecolor=blue,linewidth=2pt,linestyle=dashed]{->}(1,0)(0.1,0)}
				\put(0,0){\pscircle[fillstyle=solid,fillcolor=lightgray]{.1}}
		}}
		\newcommand{\PATTOP}{%
			{\psset{unit=1}
				\rput(0,0){\psline[linecolor=blue,linewidth=2pt]{->}(1,0)(0.1,0)}
				\rput(0,0){\psline[linecolor=blue,linewidth=2pt]{->}(0,0)(.45,-.765)}
				\put(0,0){\pscircle[fillstyle=solid,fillcolor=lightgray]{.1}}
		}}
		\newcommand{\PATBOTRIGHT}{%
			{\psset{unit=1}
				\rput(0,0){\psline[linecolor=blue,linewidth=2pt]{->}(0,0)(.45,.765)}
				\put(0,0){\pscircle[fillstyle=solid,fillcolor=lightgray]{.1}}
				\put(.5,0.85){\pscircle[fillstyle=solid,fillcolor=lightgray]{.1}}
		}}
		\multiput(-2,-1.176)(1.0,0){5}{\STARUP}
		\multiput(-1.5,-0.335)(1.0,0){4}{\STARUP}
		\multiput(-1.0,0.5)(1.0,0){3}{\STARUP}
		\multiput(-.5,1.4)(1.0,0){2}{\STARUP}
		\put(0,2.3){\STARUP}
		\multiput(2.6,-1.4)(-0.5,.85){6}{\LEFTDOWNARROW}
		\multiput(-2.6,-1.4)(0.5,.85){6}{\LEFTUPARROW}
		\multiput(-2.5,-1.5)(1.0,0){6}{\DOWNARROW}
		\put(-3,-1.8){\makebox(0,0)[br]{\hbox{{\color{red}{\tiny $600$}}}}}
		\put(0.2,3.3){\makebox(0,0)[br]{\hbox{{\color{red}{\tiny $006$}}}}}
		\put(3.5,-1.8){\makebox(0,0)[br]{\hbox{{\color{red}{\tiny $060$}}}}}
		\put(3.,-1){\makebox(0,0)[br]{\hbox{{\tiny $051$}}}}
		\put(-2.5,-1){\makebox(0,0)[br]{\hbox{{\tiny $501$}}}}
		\put(0.9,2.4){\makebox(0,0)[br]{\hbox{{\tiny $015$}}}}
		\put(-0.5,2.4){\makebox(0,0)[br]{\hbox{{\tiny $105$}}}}
		\put(1.4,1.5){\makebox(0,0)[br]{\hbox{{\tiny $024$}}}}
		\put(.2,1.5){\makebox(0,0)[br]{\hbox{{\tiny $114$}}}}
		\put(-1,1.5){\makebox(0,0)[br]{\hbox{{\tiny $204$}}}}
		\put(-1.5,0.7){\makebox(0,0)[br]{\hbox{{\tiny $303$}}}}
       \pscurve[linecolor=blue,linewidth=2pt]{->}(1.75,1.6)(1.,1.3)(0.5,1.5)(0.,1.3)(-0.5,1.5)(-1,1.2)(-1,0.8)(-1.5,0.4)(-2.5,1)
	}
	\end{pspicture}
	\caption{\small Face and path weights of $G$. Faces are labeled by indices $i,j,k\in\mathbb Z$, $i+j+k=6$, the corresponding Fock-Goncharov face weight is denoted by $Z_{ijk}$. The weight $w(P)$ of the blue path $P$ is 
	$w(P)=\col{Z_{024} Z_{015}  Z_{114} {\color{red} Z_{006}} Z_{105} Z_{204} Z_{303}}$ Note that corner faces do not carry Fock-Goncharov variables and don't contribute to the normalized transport matrices $M_1$ and $M_2$. However, they do contribute toward non-normalized transport matrices $\M_1$ and $\M_2$.}
	\label{fi:plabic_weights}
\end{figure}

\begin{example} Consider the triangular network of $SL_3$ (Fig.~\ref{fi:plabic_weightsSL3}) \\

\hskip 6cm

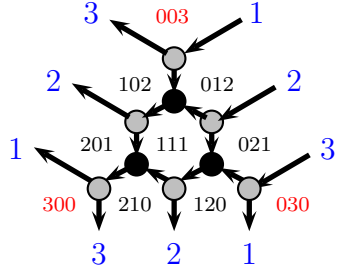
\begin{figure}[h]
	\begin{pspicture}(-3,-3)(2,1){
		\newcommand{\LEFTDOWNARROW}{%
			{\psset{unit=1}
				\rput(0,0){\psline[linecolor=black,linewidth=2pt]{<-}(0,0)(.765,.45)}
		}}
		\newcommand{\DOWNARROW}{%
			{\psset{unit=1}
				\rput(0,0){\psline[linecolor=black,linewidth=2pt]{->}(0,0.1)(0,-0.566)}
				\put(0,0){\pscircle[fillstyle=solid,fillcolor=lightgray]{.15}}
		}}
		\newcommand{\LEFTUPARROW}{%
			{\psset{unit=1}
				\rput(0,0){\psline[linecolor=black,linewidth=2pt]{->}(0,0)(-.765,.45)}
		}}
		\newcommand{\STARUP}{
			{\psset{unit=1}
				\rput(0,0){\psline[linecolor=black,linewidth=2pt]{<-}(0.1,-0.05)(.43,-.22)}
				\rput(0,0){\psline[linecolor=black,linewidth=2pt]{<-}(0,0.1)(0,.466)}
				\rput(0,0){\psline[linecolor=black,linewidth=2pt]{->}(-0.1,-0.05)(-.43,-.22)}
				\put(0,0){\pscircle[fillstyle=solid,fillcolor=black]{.15}}
				\put(0,.566){\pscircle[fillstyle=solid,fillcolor=lightgray]{.15}}
		}}
		\newcommand{\PATGEN}{%
			{\psset{unit=1}
				\rput(0,0){\psline[linecolor=blue,linewidth=2pt]{->}(0,0)(.45,.765)}
				\rput(0,0){\psline[linecolor=blue,linewidth=2pt]{->}(1,0)(0.1,0)}
				\rput(0,0){\psline[linecolor=blue,linewidth=2pt]{->}(0,0)(.45,-.765)}
				\put(0,0){\pscircle[fillstyle=solid,fillcolor=lightgray]{.1}}
		}}
		\newcommand{\PATLEFT}{%
			{\psset{unit=1}
				\rput(0,0){\psline[linecolor=blue,linewidth=2pt,linestyle=dashed]{->}(0,0)(.45,.765)}
				\rput(0,0){\psline[linecolor=blue,linewidth=2pt]{->}(1,0)(0.1,0)}
				\rput(0,0){\psline[linecolor=blue,linewidth=2pt]{->}(0,0)(.45,-.765)}
				\put(0,0){\pscircle[fillstyle=solid,fillcolor=lightgray]{.1}}
		}}
		\newcommand{\PATRIGHT}{%
			{\psset{unit=1}
				\rput(0,0){\psline[linecolor=blue,linewidth=2pt,linestyle=dashed]{->}(0,0)(.45,-.765)}
				\put(0,0){\pscircle[fillstyle=solid,fillcolor=lightgray]{.1}}
		}}
		\newcommand{\PATBOTTOM}{%
			{\psset{unit=1}
				\rput(0,0){\psline[linecolor=blue,linewidth=2pt]{->}(0,0)(.45,.765)}
				\rput(0,0){\psline[linecolor=blue,linewidth=2pt,linestyle=dashed]{->}(1,0)(0.1,0)}
				\put(0,0){\pscircle[fillstyle=solid,fillcolor=lightgray]{.1}}
		}}
		\newcommand{\PATTOP}{%
			{\psset{unit=1}
				\rput(0,0){\psline[linecolor=blue,linewidth=2pt]{->}(1,0)(0.1,0)}
				\rput(0,0){\psline[linecolor=blue,linewidth=2pt]{->}(0,0)(.45,-.765)}
				\put(0,0){\pscircle[fillstyle=solid,fillcolor=lightgray]{.1}}
		}}
		\newcommand{\PATBOTRIGHT}{%
			{\psset{unit=1}
				\rput(0,0){\psline[linecolor=blue,linewidth=2pt]{->}(0,0)(.45,.765)}
				\put(0,0){\pscircle[fillstyle=solid,fillcolor=lightgray]{.1}}
				\put(.5,0.85){\pscircle[fillstyle=solid,fillcolor=lightgray]{.1}}
		}}
		\multiput(-2,-1.176)(1.0,0){2}{\STARUP}
		\multiput(-1.5,-0.335)(1.0,0){1}{\STARUP}
		\multiput(-0.47,-1.5)(-0.45,.9){3}{\LEFTDOWNARROW}
		\multiput(-2.6,-1.4)(0.5,.85){3}{\LEFTUPARROW}
		\multiput(-2.5,-1.5)(1.0,0){3}{\DOWNARROW}
		\put(-2.8,-1.8){\makebox(0,0)[br]{\hbox{{\color{red}{\tiny $300$}}}}}
		\put(-1.8,-1.8){\makebox(0,0)[br]{\hbox{{{\tiny $210$}}}}}
		\put(-0.8,-1.8){\makebox(0,0)[br]{\hbox{{{\tiny $120$}}}}}
		\put(0.3,-1.8){\makebox(0,0)[br]{\hbox{{\color{red}{\tiny $030$}}}}}
		\put(-2.3,-1){\makebox(0,0)[br]{\hbox{{\tiny $201$}}}}
		\put(-1.3,-1){\makebox(0,0)[br]{\hbox{{\tiny $111$}}}}
		\put(-0.2,-1){\makebox(0,0)[br]{\hbox{{\tiny $021$}}}}
		\put(-1.8,-0.2){\makebox(0,0)[br]{\hbox{{\tiny $102$}}}}
		\put(-0.7,-0.2){\makebox(0,0)[br]{\hbox{{\tiny $012$}}}}
		\put(-1.3,0.7){\makebox(0,0)[br]{\hbox{\color{red}{\tiny $003$}}}}
		\put(-.3,0.7){\makebox(0,0)[br]{\hbox{\color{blue}{$1$}}}}
		\put(0.2,-0.2){\makebox(0,0)[br]{\hbox{\color{blue}{$2$}}}}
		\put(0.65,-1.1){\makebox(0,0)[br]{\hbox{\color{blue}{$3$}}}}
		\put(-2.5,0.7){\makebox(0,0)[br]{\hbox{\color{blue}{$3$}}}}
		\put(-3.0,-0.2){\makebox(0,0)[br]{\hbox{\color{blue}{$2$}}}}
		\put(-3.5,-1.1){\makebox(0,0)[br]{\hbox{\color{blue}{$1$}}}}
		\put(-2.4,-2.5){\makebox(0,0)[br]{\hbox{\color{blue}{$3$}}}}
		\put(-1.4,-2.5){\makebox(0,0)[br]{\hbox{\color{blue}{$2$}}}}
		\put(-0.4,-2.5){\makebox(0,0)[br]{\hbox{\color{blue}{$1$}}}}
	}
	\end{pspicture}
	\caption{\small Face and path weigths of $G_{SL_3}$. Triples $i,j,k\in\ZZ$, $i+j+k=3$ label faces.}
	\label{fi:plabic_weightsSL3}
\end{figure}

Quantum transport matrices have the following form:

$$M_1=\begin{pmatrix}
\col{Z_{021}^{-1/3}Z_{102}^{1/3} Z_{111}^{-1/3}Z_{012}^{-2/3}Z_{201}^{2/3}} & \col{Z_{021}^{-1/3}Z_{102}^{1/3}(Z_{111}^{-1/3}+Z_{111}^{2/3}) Z_{102}^{1/3}  Z_{201}^{2/3}} & \col{Z_{021}^{2/3} Z_{102}^{1/3}Z_{111}^{2/3} Z_{012}^{1/3} Z_{201}^{2/3}} \\
 \col{Z_{021}^{-1/3}Z_{102}^{-1/3} Z_{111}^{-1/3}Z_{012}^{-2/3}Z_{201}^{-1/3}} & \col{Z_{021}^{-1/3}Z_{102}^{-1/3} Z_{111}^{-1/3}Z_{012}^{1/3}Z_{201}^{-1/3}} & 0 \\
\col{Z_{021}^{-1/3}Z_{102}^{-2/3} Z_{111}^{-1/3}Z_{012}^{-2/3}Z_{201}^{-1/3}} & 0 & 0
\end{pmatrix}.
$$

Then, $M_1=QSD_1^{-1} \M_1$, where  $D_1=\col{Z_{021}^{\frac{1}{3}} Z_{102}^{\frac{2}{3}} Z_{111}^{\frac{1}{3}} {\color{red}Z_{003}} Z_{012}^{\frac{2}{3}} Z_{201}^{\frac{1}{3}}}$ 
and
$$\M_1=\begin{pmatrix}
\col{{\color{red}Z_{003}}}  & 0 & 0 \\
\col{{\color{red}Z_{003}} Z_{102}} & \col{Z_{012}{\color{red}Z_{003}}Z_{102}} & 0 \\
\col{{\color{red}Z_{003}} Z_{102} Z_{201}} & \col{Z_{012} (1+Z_{111}) {\color{red}Z_{003}} Z_{102}  Z_{201}} & \col{Z_{021} Z_{012} Z_{111}{\color{red}Z_{003}} Z_{102} Z_{201}} 
\end{pmatrix}.
$$

Similarly, 
$$M_2=\begin{pmatrix}
0 & 0 &  \col{Z_{210}^{1/3} Z_{111}^{1/3} Z_{012}^{1/3} Z_{120}^{2/3} Z_{021}^{2/3}}\\
0 &  \col{Z_{210}^{1/3} Z_{111}^{1/3} Z_{012}^{1/3}  Z_{120}^{-1/3} Z_{021}^{-1/3}}&   \col{Z_{210}^{1/3} Z_{111}^{1/3} Z_{012}^{1/3}  Z_{120}^{-1/3} Z_{021}^{2/3}}\\
 \col{Z_{210}^{-2/3} Z_{111}^{-2/3} Z_{012}^{-2/3}  Z_{120}^{-1/3} Z_{021}^{-1/3}} &  \col{Z_{210}^{-2/3} (Z_{111}^{-2/3}+Z_{111}^{1/3}) Z_{012}^{1/3}  Z_{120}^{-1/3} Z_{021}^{-1/3}} &  
 \col{Z_{210}^{-2/3} Z_{111}^{1/3} Z_{012}^{1/3}  Z_{120}^{-1/3} Z_{021}^{2/3}}\\
\end{pmatrix}
$$

Hence, 
$M_2=QS\col{D_1^{-1}D_2^{-1}} \M_2$, where  
$D_2=\col{Z_{300}^{-1}Z_{003}^{-1} Z_{201}^{-1}Z_{102}^{-1}Z_{210}^{-2/3} Z_{111}^{-2/3} Z_{012}^{-2/3}  
Z_{120}^{-1/3} Z_{021}^{-1/3}}$ and 
$$\M_2=\begin{pmatrix}
 \col{{\color{red}Z_{003}Z_{300}}Z_{201}Z_{102}} &  \col{{\color{red}Z_{003}Z_{300}}(1+Z_{111})Z_{012}Z_{201}Z_{102}} &  \col{{\color{red}Z_{003}Z_{300}}Z_{021}Z_{111}Z_{012}Z_{201}Z_{102}} \\
0 &  \col{{\color{red}Z_{003}Z_{300}}Z_{210}Z_{111}Z_{012}Z_{201}Z_{102}} &  \col{{\color{red}Z_{003}Z_{300}} Z_{210}Z_{021}Z_{111}Z_{012}Z_{201}Z_{102}} \\
0 & 0 & \col{{\color{red}Z_{003}Z_{300}}Z_{120}Z_{210}Z_{021}Z_{111}Z_{012}Z_{201}Z_{102}}  \\
\end{pmatrix}
$$

Notice that both $\M_1$ and $\M_2$ are non-normalized quantum transport matrices of network shown on Figure~\ref{fi:plabic_weightsSL3}.

\end{example}

\subsection{Quantum Grassmannian and proofs of Theorems ~\ref{th:MM} and ~\ref{th:MMsquare}}

We now prove Theorems~\ref{th:MM} and ~\ref{th:MMsquare} utilizing the notion of plabic graphs introduced by Postnikov in ~\cite{Po}.

\begin{remark}
The semiclassical statement of Theorem~\ref{th:MM} (see Remark!\ref{rem:MM}) was proved in \cite{GSV}.
\end{remark}

Following Postnikov we call a planar network an oriented graph with faces equipped with weights.
Let $N$ be a network in the disk with neither sources nor sinks inside and separated sources and sinks on the boundary. 
An example of such network $N$ is drawn on Figure~\ref{fig:netw} in rectangle $R$ with sources on the right side and sinks on the left side. 

\begin{figure}[H]
\begin{pspicture}(-3.5,0.)(3.5,3){\psset{unit=0.8}
\newcommand{\NET}{%
{\psset{unit=1}
\rput(0,0){\psline[linecolor=blue,linestyle=dashed,linewidth=2pt]{-}(0,0)(3.,0.)}
\rput(0,0){\psline[linecolor=blue,linestyle=dashed,linewidth=2pt]{-}(3,0)(3.,3.)}
\rput(0,0){\psline[linecolor=blue,linestyle=dashed,linewidth=2pt]{-}(3,3)(0,3)}
\rput(0,0){\psline[linecolor=blue,linestyle=dashed,linewidth=2pt]{-}(0,3)(0,0)}
\rput(0,0){\psline[linecolor=black,linewidth=2pt]{<-}(2.5,1)(3,1)}
\rput(0,0){\psline[linecolor=black,linewidth=2pt]{<-}(0,1)(2.5,1)}
\rput(0,0){\psline[linecolor=black,linewidth=2pt]{<-}(2,2)(3,2)}
\rput(0,0){\psline[linecolor=black,linewidth=2pt]{<-}(1,2)(2,2)}
\rput(0,0){\psline[linecolor=black,linewidth=2pt]{<-}(0,1.5)(1,2)}
\rput(0,0){\psline[linecolor=black,linewidth=2pt]{<-}(0,2.5)(1,2)}
\rput(0,0){\psline[linecolor=black,linewidth=2pt]{->}(2.5,1)(2,2)}
\rput(2,2.5){\makebox(0,0)[cc]{\hbox{\tcr{$\mathbf \alpha$}}}}
\rput(2.7,1.5){\makebox(0,0)[cc]{\hbox{\tcr{$\mathbf \beta$}}}}
\rput(0.3,2.){\makebox(0,0)[cc]{\hbox{\tcr{$\mathbf \gamma$}}}}
\rput(1.2,1.5){\makebox(0,0)[cc]{\hbox{\tcr{$\mathbf \delta$}}}}
\rput(1.7,0.6){\makebox(0,0)[cc]{\hbox{\tcr{$\mathbf \epsilon$}}}}

\rput(3.3,2.){\makebox(0,0)[cc]{\hbox{\tcr{$\mathbf  2$}}}}
\rput(3.3,1){\makebox(0,0)[cc]{\hbox{\tcr{$\mathbf 1$}}}}
\rput(-0.3,1.){\makebox(0,0)[cc]{\hbox{\tcr{$\mathbf 5$}}}}
\rput(-0.3,1.5){\makebox(0,0)[cc]{\hbox{\tcr{$\mathbf 4$}}}}
\rput(-0.3,2.5){\makebox(0,0)[cc]{\hbox{\tcr{$\mathbf 3$}}}}

}
}
\put(0.1,0.1){\NET}
}
\end{pspicture}
\caption{\small
Network $N$ in rectangle $R$.
}
\label{fig:netw}
\end{figure}
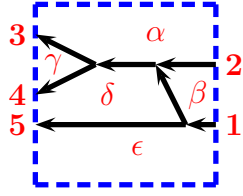

Denote by $Faces(N)$ the set of faces of network $N$.   Figure~\ref{fig:netw} has $Faces(N)=\{\alpha,\beta,\gamma,\delta,\epsilon\}$.
 Consider the integer lattice $\tilde \Lambda$ generated by $Faces(N)$ and vector space $\tilde V={\mathbb Q}\otimes \tilde\Lambda$.
We  equip it with the integer skew-symmetric form 
 $\langle,\rangle$ as follows.  
 
 \begin{definition}\label{def:plabic} \cite{Po} A \emph{planar bicolored} graph, or simply a \emph{plabic} graph is a planar (undirected) graph $G$,  without orientations of edges, such that each boundary vertex $b_i$ is incident to a single edge and all internal vertices are colored either black or white. 
A \emph{perfect orientation} of a plabic graph 
is a choice of orientation of its edges such that each \emph{black} internal vertex $v$ 
is incident to exactly one edge directed \emph{away} from $v$; and each \emph{white} $v$ 
is incident to exactly one edge directed \emph{towards} $v$. A plabic graph 
is called \emph{perfectly orientable} if it has a perfect orientation.
\end{definition}

 Let us transform the oriented graph $G$ of network into plabic graph $G^{pl}$ by coloring inner vertices of $G$ into black and white colors according to the rule: black vertex has two incoming arcs and one outgoing; white vertex has one incoming and two outgoing.  We forget boundary sources and sinks so that any arcs connecting inner vertex to the boundary one becomes a half-arc (see Fig.~\ref{fig:netw1}).
 For a plabic graph $G^{pl}$ we define an oriented dual graph $(G^{pl})^*$ as follows. Vertices of $(G^{pl})^*$ are faces of $G^{pl}$.
 For every black and white vertex $x$ of $G^{pl}$ we define 3 arcs of $(G^{pl})^*$ that cross half-edges attached to $x$ in counterclockwise direction if $x$ is black and clockwise direction if $x$ is white (see Fig.~\ref{fig:form}).
  \begin{figure}[H]
\begin{pspicture}(-3.5,-1)(3.5,1){\psset{unit=1}
\newcommand{\CIW}{%
{\psset{unit=1}
\rput(0,0){\psline[linecolor=black,linewidth=1pt]{-}(-1,0)(0,0)}
\rput(0,0){\psline[linecolor=black,linewidth=1pt]{-}(0,0)(0.5,0.7)}
\rput(0,0){\psline[linecolor=black,linewidth=1pt]{-}(0,0)(0.5,-0.7)}
\rput(0,0){\pscircle[fillstyle=solid,fillcolor=white]{.2}}
\psarc[linewidth=1pt,linecolor=blue,linestyle=dashed]{<-}(0.0,0.){0.7}{20}{100}
\psarc[linewidth=1pt,linecolor=blue,linestyle=dashed]{<-}(0.0,0.){0.7}{-100}{-20}
\psarc[linewidth=1pt,linecolor=blue,linestyle=dashed]{<-}(0.0,0.){0.7}{140}{220}
}
}
\newcommand{\CIB}{%
{\psset{unit=1}
\rput(0,0){\psline[linecolor=black,linewidth=1pt]{-}(-1,0)(0,0)}
\rput(0,0){\psline[linecolor=black,linewidth=1pt]{-}(0,0)(0.5,0.7)}
\rput(0,0){\psline[linecolor=black,linewidth=1pt]{-}(0,0)(0.5,-0.7)}
\rput(0,0){\pscircle[fillstyle=solid,fillcolor=black]{.2}}
\psarc[linewidth=1pt,linecolor=blue,linestyle=dashed]{->}(0.0,0.){0.7}{20}{100}
\psarc[linewidth=1pt,linecolor=blue,linestyle=dashed]{->}(0.0,0.){0.7}{-100}{-20}
\psarc[linewidth=1pt,linecolor=blue,linestyle=dashed]{->}(0.0,0.){0.7}{140}{220}
}
}
\put(-3,0){\CIW}
\put(3,0){\CIB}
}
\end{pspicture}
\caption{\small
Dashed blue arcs are edges of the  dual graph $(G^{pl})^*$ around  black and white vertex of $G^{pl}$.}
\label{fig:form}
\end{figure}
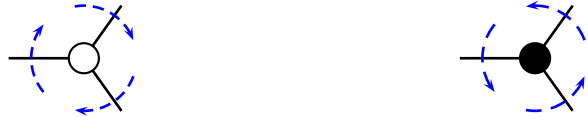

For  $\theta,\phi\in Faces(N)$  let $\#(\theta\to\phi)$ denote the number of arcs from $\theta$ to $\phi$ in $(G^{pl})^*$.

Define the skew-symmetric form $\langle,\rangle$ on $\tilde \Lambda$ by the formula
\begin{equation}\label{eq:Omega}
\langle\theta,\phi\rangle=\frac{1}{2}\left(\#( \theta \to\phi) -\#(\phi\to\theta)\right).
\end{equation}
 
 \begin{example} The plabic graph and its dual for the network Fig.~\ref{fig:netw} are shown on the Fig.~\ref{fig:netw1}.
 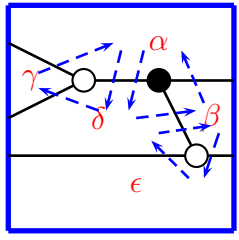
\begin{figure}[H]
\begin{pspicture}(-3.5,0.)(3.5,3){\psset{unit=1}
\newcommand{\NET}{%
{\psset{unit=1}
\rput(0,0){\psline[linecolor=blue,linewidth=2pt]{-}(0,0)(3.,0.)}
\rput(0,0){\psline[linecolor=blue,linewidth=2pt]{-}(3,0)(3.,3.)}
\rput(0,0){\psline[linecolor=blue,linewidth=2pt]{-}(3,3)(0,3)}
\rput(0,0){\psline[linecolor=blue,linewidth=2pt]{-}(0,3)(0,0)}

\rput(0,0){\psline[linecolor=black,linewidth=1pt]{-}(0,1)(3,1)}
\rput(0,0){\psline[linecolor=black,linewidth=1pt]{-}(1,2)(3,2)}
\rput(0,0){\psline[linecolor=black,linewidth=1pt]{-}(0,1.5)(1,2)}
\rput(0,0){\psline[linecolor=black,linewidth=1pt]{-}(0,2.5)(1,2)}
\rput(0,0){\psline[linecolor=black,linewidth=1pt]{-}(2.5,1)(2,2)}

\put(1,2){\pscircle[fillstyle=solid,fillcolor=white]{.15}}
\put(2.5,1){\pscircle[fillstyle=solid,fillcolor=white]{.15}}
\put(2,2){\pscircle[fillstyle=solid,fillcolor=black]{.15}}

\rput(2,2.5){\makebox(0,0)[cc]{\hbox{\tcr{$\mathbf \alpha$}}}}
\rput(2.7,1.5){\makebox(0,0)[cc]{\hbox{\tcr{$\mathbf \beta$}}}}
\rput(0.3,2.){\makebox(0,0)[cc]{\hbox{\tcr{$\mathbf \gamma$}}}}
\rput(1.2,1.5){\makebox(0,0)[cc]{\hbox{\tcr{$\mathbf \delta$}}}}
\rput(1.7,0.6){\makebox(0,0)[cc]{\hbox{\tcr{$\mathbf \epsilon$}}}}

\rput(0,0){\psline[linecolor=blue,linewidth=1pt,linestyle=dashed]{<-}(2.6,0.7)(2.8,1.3)}
\rput(0,0){\psline[linecolor=blue,linewidth=1pt,linestyle=dashed]{->}(2.0,1.3)(2.7,1.4)}
\rput(0,0){\psline[linecolor=blue,linewidth=1pt,linestyle=dashed]{<-}(1.9,1.2)(2.4,0.7)}

\rput(0,0){\psline[linecolor=blue,linewidth=1pt,linestyle=dashed]{->}(2.6,1.7)(2.3,2.4)}
\rput(0,0){\psline[linecolor=blue,linewidth=1pt,linestyle=dashed]{<-}(2.5,1.6)(1.7,1.5)}
\rput(0,0){\psline[linecolor=blue,linewidth=1pt,linestyle=dashed]{<-}(1.6,1.6)(1.8,2.4)}

\rput(0,0){\psline[linecolor=blue,linewidth=1pt,linestyle=dashed]{<-}(0.4,1.9)(1.2,1.6)}
\rput(0,0){\psline[linecolor=blue,linewidth=1pt,linestyle=dashed]{->}(0.4,2.1)(1.4,2.5)}
\rput(0,0){\psline[linecolor=blue,linewidth=1pt,linestyle=dashed]{<-}(1.3,1.6)(1.5,2.4)}


}
}
\put(0.,0.){\NET}
}
\end{pspicture}
\caption{\small
Arcs of plabic graph $G^{pl}$ corresponding to network $N$  on Fig.~\ref{fig:netw} are black solid lines;
arcs of its dual $(G^{pl})^*$ are dashed blue arrows. Then, 
$\langle\alpha,\beta\rangle=-1/2, \langle\alpha,\delta\rangle=1, \langle\alpha,\gamma\rangle=-1/2,
\langle\beta, \epsilon\rangle=1/2, \langle\beta,\delta\rangle=-1, 
\langle\gamma,\delta\rangle=-1/2, \langle\delta,\epsilon\rangle=-1/2$.}
\label{fig:netw1}
\end{figure}
\end{example}

 Let $V$ be the quotient space $V=\tilde V/ \sum_{\theta\in Faces(N)} \theta$ and $\Lambda$ be the induced integer lattice in $V$. Note that $\sum_{\theta\in Faces(N)}\theta$ lies in the kernel of the skew-symmetric form and its push forward to $V$ is well-defined. Abusing notation we will use $\langle,\rangle$ for the induced skew-symmetric form on $V$. Dual lattice $\Lambda^*=\operatorname{Hom}(\Lambda,\mathbb Z)$.

For $\alpha\in Faces(N)$, we denote by $Z_\alpha$ the corresponding quantum face weight. The quantum torus $\Upsilon_N$ is generated by weights $Z_\alpha,\ \alpha\in Faces(N)$ satisfying commutation relations $Z_\alpha Z_\beta=q^{-2\langle\beta,\alpha\rangle} Z_b Z_a$.
For $\aa=\sum_{\alpha\in Faces(N)} c_\alpha \alpha\in \Lambda$, we define $Z_\aa=\col{\prod_{\alpha\in Faces(N)} Z_\alpha^{c_\alpha}}\in\Upsilon_N$.
Note that for $\aa,\bb\in\Lambda$, $Z_{\aa+\bb}=q^{-\langle\aa,\bb\rangle} Z_\aa Z_\bb=q^{-\langle\bb,\aa\rangle} Z_b Z_a$.
As above, for any $\aa\in\Lambda$, $Z_\aa$ is called the normal (Weyl) ordering.
We define weight of any vector in $\Lambda\otimes {\mathbb Q}$ by
 $Z_{\aa+\bb}=\col{Z_\aa Z_\bb}=\col{Z_\bb Z_\aa}$ and $Z_{r \aa}=\col{Z_\aa}^{r}=\col{Z_a^r}$ for any $r\in{\mathbb Q}$. Since every $Z_\alpha$ is a positively defined self-adjoint operator in $\ell^2(\mathbb R)$ having a continuous spectrum $(0,\infty)$, any rational power of it is itself a positively defined self-adjoint operator.  

Let $p$ be the maximal oriented path from a source $i$ on the right to the sink $j$ on the left of a network. Complete $p$ to an oriented loop $\tilde p$  by following path $p$ from $i$ to $j$ first and then closing the loop following  the piece of boundary of the rectangle in the clockwise direction from $j$ to $i$.

The oriented loop $\tilde p$ defines a  covector $\tilde p\in \Lambda^*$ as follows. (We use the same notation for the loop and induced covector.) 
Let  $\rr$  be a half infinite  ray with starting point inside face $\alpha$ and directed towards infinity and $q_1,\dots,q_s$ be intersection points of $\rr$ and loop $\tilde p$, $T_\rr$ be the unit direction vector of $\rr$, $T_{q_j} {\tilde p}$ is the unit tangent vector to $\tilde p$ at $q_j$. We assume that $\rr$ is chosen generic, i.e., for all $q_j$ vectors $T_\rr$ and $T_{q_j} \tilde p$ are linearly independent.

We define the intersection index $\mathop{ind}_{q_j}(\tilde p,\rr)$ of $\tilde p$ and $\rr$ at $q_j$ to be $1$ if orientation of basis 
$(T_{q_j} \tilde p, T_\rr)$ coincides with counterclockwise orientation of the plane and $-1$ otherwise and define $\tilde p(\alpha)=\sum_{j=1}^s ind_{q_j}(\tilde p, \rr)$. Note that $\tilde(\alpha)$  depends neither on exact position of starting point of $\rr$ provided that the starting point varies inside the same connected component of complement to $\tilde p$ nor on the  particular choice of ray $\rr$ with the same starting point. Since any face $\alpha$ lies entirely in some connected  component of $\tilde p$ we conclude that $\tilde p(\alpha)$ is well defined.
Clearly, $\tilde p\in \Lambda^*$.

Assign to any path  $p$  a vector ${\bf v}_p=\sum_{\alpha\in Faces(N)} \tilde p(\alpha) \alpha \in \Lambda$. In the example in Section~\ref{sec:QFG} where any maximal oriented path $p$ is  \emph{non-selfintersecting}  the vector ${\bf v}_p$ is the sum of all faces  to the right from the path. 

Set the weight $w_p$ of the path $p$  as $w_p=Z_{{\bf v}_p}$.

Let $S$ be the set of all sources of $N$, $F$ be the set of all sinks.
Define for any source $a\in S$ and sink $b\in F$
a \emph{quantum boundary measurement} $\mathop{Meas}_q(a,b)=\sum_{p:a\leadsto b} (-1)^{\mathop{cross}(p)}w_p$, where the sum is taken over all oriented paths $p$ from $a$ to $b$ where the \emph{crossing index} $\mathop{cross}(p)$ is the number of self-crossings of the path $p$. Classical boundary measurement is defined by Postnikov in~\cite{Po}, 
 For the network in Section~\ref{sec:QFG},  no path is selfcrossing and $\mathop{Meas}_q(a,b)=\sum_{p:a\leadsto b} w_p$.

Let $n=|S|,\ m=|F|$. Define an $m\times n$ matrix $Q_q$ of quantum boundary measurements as $(Q_q)_{ba}=\left(\mathop{Meas}_q(a, b)\right)_{a\in S,b\in F}$. 
Note that we label rows of $Q_q$ by sinks and columns by sources of $N$.

In Example~\ref{fig:netw}, the matrix 
$Q_q=\begin{pmatrix}
Z_{\alpha+\beta} & Z_\alpha\\
Z_{\alpha+\beta+\gamma} & Z_{\alpha+\gamma}\\
Z_{\alpha+\beta+\gamma+\delta} & 0
\end{pmatrix}
$.

Define $(m+n)\times n$ \emph{quantum grassmannian boundary measurement matrix}  $Q^{gr}_q$ of network $N$. Columns of $Q^{gr}_q$ are labelled by boundary sources of network; rows  are labelled by all boundary vertices. To describe matrix elements of $\tilde Q_q$  we introduce the \emph{order $\mathop{ord}_N(b)$ of boundary vertex $b$}. Enumerate all boundary vertices of $N$ from $1$ to $m+n$ in counterclockwise direction. Let $b\in [1,m+n]$ be the index of boundary vertex. Let $\sigma(b)$ be the number of sources among boundary vertices with indices from $1$ to $b-1$. The order  is defined by the formula $ord_N(b)=\begin{cases} \sigma(b), & \text{ if } b \text{ is not a source}; \\
\sigma(b)+\frac{1}{2}, &\text{ if } b \text{ is a source}.
\end{cases}$. Let  $\mathbb{J}(i)\in [1,m+n]$ be the index of $i$th source, $i\in [1,n]$;  $\mathbb{J}:[1,n]\to [1,m+n]$ is an increasing function.

We define 
$
(Q^{gr}_q)_{ji}=\begin{cases} 
(-1)^{i+\mathop{ord}_N(j)} q^{-\mathop{ord}_N(j)} \mathop{Meas}_q(i,j), & \text{ if } j \text{ is not a source}; \\
q^{-\mathop{ord}_N(j)} \delta(\mathbb{J}(i),j), & \text{ otherwise.}
\end{cases}
$

\begin{example} In Example~\ref{fig:netw} ,
 the matrix 
$Q^{gr}_q=\begin{pmatrix}
q^{-1/2}   & 0 \\
0 &  q^{-3/2} \\
-q^{-2} Z_{\alpha+\beta} & q^{-2} Z_\alpha \\
-q^{-2} Z_{\alpha+\beta+\gamma} & q^{-2} Z_{\alpha+\gamma}\\
-q^{-2} Z_{\alpha+\beta+\gamma+\delta} & 0 
\end{pmatrix}$.
\end{example}

\begin{remark}  In \cite{Po}  a  boundary measurement map is defined as a map $\operatorname{Meas}$ from the space $\operatorname{Net}_{m\times n}$ of networks with $n$ sources, $m$ sinks and commutative weights to $Gr(n,m+n)$. For each $X\in\operatorname{Net}_{m\times n}$, boundary measurements $\operatorname{Meas}(i, j)$ form  an $(m+n) \times n$ matrix $Q^{gr}$ which represents 
$\operatorname{Meas}(X)$. 
The space of $(m+n)\times n$ matrices with elements in $\Upsilon_N$ we denote by $\operatorname{Mat}_{(m+n)\times n}(\Upsilon_N)$.
The group $GL_n(\Upsilon_N)$ of invertible $n\times n$ matrices with entries from $\Upsilon_N$ acts on $\operatorname{Mat}_{(m+n)\times n}(\Upsilon_N)$ by the right multiplication. We define the homogeneous space $Gr_q(n,m+n)$ as the right quotient $Gr_q(n,m+n)=\operatorname{Mat}_{(m+n)\times n}(\Upsilon_N)/GL_n(\Upsilon_N)$.  We denote by $QNet_{m\times n}$ the space of quantum networks with $n$ sources, $m$ sinks and quantum weights from $\Upsilon_N$. We define a quantization $\operatorname{Meas}_q:QNet_{m\times n}\to Gr_q(m,m+n)$ as the composition $QNet_{m\times n}\to \operatorname{Mat}_{(m+n)\times n}(\Upsilon_N)\to Gr_q(n,m+n)$.
\end{remark}

\begin{definition} Two networks are \emph{equivalent}  if they have the same  boundary measurements.
\end{definition}


Simple equivalence relations (M1-M3,R1-R3) on the space of networks (~\cite{Po}) are simple local network transformations preserving boundary measurements. Please, note that in the figures below we draw the plabic graph assuming that it is equipped with a perfect orientation. Different choices of compatible perfect orientation give the same result.

The following claims generalize similar statements for commuting weights (cf~\cite{Po}).

Let ${\bf e}=\{e_i\}$ denote a standard basis in the lattice $\Lambda=\mathbb{Z}^n$, equipped with the standard dot product $\bf\cdot$ with respect to ${\bf e}$,
$Z_i=Z_{e_i}$ be the generators of a quantum torus $\Upsilon$. We say that an infinite linear combination $\sum _{\lambda\in \Lambda} \alpha_\lambda Z_\lambda$ is a Laurent series if there exist integers $b_1,\dots b_n\in \mathbb{Z}$ such that $\alpha_\lambda=0$ unless $\lambda\cdot e_j\ge b_j$ $\forall j\in[1,n]$.
Any Laurent series $u\in \hat R$ where $\hat R= \mathbb{Q}[q,q^{-1}][Z_1^{-1},\dots,Z_N^{-1}][[Z_1,\dots,Z_N]]$. Let $\mathfrak{m}\subset \mathbb{Q}[q,q^{-1}[Z_1,\dots,Z_N]$ be the maximal ideal generated by $Z_i$. Note that for any $x\in \mathfrak{m}$ the expression $(1+x)^{-1}\in \hat R$ and $(1+x^{-1})^{-1}=x(1+x)^{-1}\in \hat R$. 

We say that $x,y\in\hat R$ \emph{$q$-commute} if $xy=q^k yx$ for some $k\in {\mathbb Q}$.

We call a plabic network with weights in  $\hat R$ a \emph{quantum} network.

Define 6 elementary moves (M1-M3), (R1-R3) as shown below.

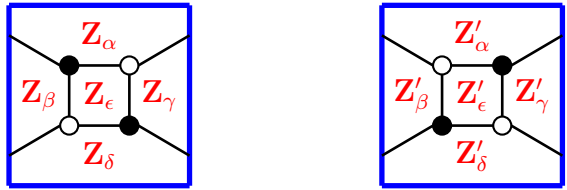
\begin{figure}[H]
\begin{pspicture}(-3.5,0.)(3.5,3){\psset{unit=0.8}
\newcommand{\NET}{%
{\psset{unit=1}
\rput(0,0){\psline[linecolor=blue,linewidth=2pt]{-}(0,0)(3.,0.)}
\rput(0,0){\psline[linecolor=blue,linewidth=2pt]{-}(3,0)(3.,3.)}
\rput(0,0){\psline[linecolor=blue,linewidth=2pt]{-}(3,3)(0,3)}
\rput(0,0){\psline[linecolor=blue,linewidth=2pt]{-}(0,3)(0,0)}
\rput(0,0){\psline[linecolor=black,linewidth=1pt]{-}(0,2.5)(0.85,2)}
\rput(0,0){\psline[linecolor=black,linewidth=1pt]{-}(1.15,2)(1.85,2)}
\rput(0,0){\psline[linecolor=black,linewidth=1pt]{-}(2.15,2)(3,2.5)}
\rput(0,0){\psline[linecolor=black,linewidth=1pt]{-}(0,.5)(0.85,1)}
\rput(0,0){\psline[linecolor=black,linewidth=1pt]{-}(1.15,1)(1.85,1)}
\rput(0,0){\psline[linecolor=black,linewidth=1pt]{-}(2.15,1)(3,.5)}


%
}
}
\newcommand{\NN}{
{\psset{unit=1}
\rput(0,0){\psline[linecolor=black,linewidth=1pt]{-}(1,1.15)(1,1.85)}
\rput(0,0){\psline[linecolor=black,linewidth=1pt]{-}(2.,1.15)(2,1.85)}
\put(1,2){\pscircle[fillstyle=solid,fillcolor=black,linecolor=black]{.15}}
\put(2,2){\pscircle[fillstyle=solid,fillcolor=white,linecolor=black]{.15}}
\put(1,1){\pscircle[fillstyle=solid,fillcolor=white,linecolor=black]{.15}}
\put(2,1){\pscircle[fillstyle=solid,fillcolor=black,linecolor=black]{.15}}
\rput(1.5,2.5){\makebox(0,0)[cc]{\hbox{\tcr{$\mathbf Z_\alpha$}}}}
\rput(0.5,1.5){\makebox(0,0)[cc]{\hbox{\tcr{$\mathbf Z_\beta$}}}}
\rput(2.5,1.5){\makebox(0,0)[cc]{\hbox{\tcr{$\mathbf Z_\gamma$}}}}
\rput(1.5,0.5){\makebox(0,0)[cc]{\hbox{\tcr{$\mathbf Z_\delta$}}}}
\rput(1.5,1.5){\makebox(0,0)[cc]{\hbox{\tcr{$\mathbf Z_\epsilon$}}}}
}
}
\newcommand{\NNN}{
{\psset{unit=1}
\rput(0,0){\psline[linecolor=black,linewidth=1pt]{-}(1,1.15)(1,1.85)}
\rput(0,0){\psline[linecolor=black,linewidth=1pt]{-}(2.,1.15)(2,1.85)}
\put(1,2){\pscircle[fillstyle=solid,fillcolor=white,linecolor=black]{.15}}
\put(2,2){\pscircle[fillstyle=solid,fillcolor=black,linecolor=black]{.15}}
\put(1,1){\pscircle[fillstyle=solid,fillcolor=black,linecolor=black]{.15}}
\put(2,1){\pscircle[fillstyle=solid,fillcolor=white,linecolor=black]{.15}}
\rput(1.5,2.5){\makebox(0,0)[cc]{\hbox{\tcr{$\mathbf Z'_\alpha$}}}}
\rput(0.5,1.5){\makebox(0,0)[cc]{\hbox{\tcr{$\mathbf Z'_\beta$}}}}
\rput(2.5,1.5){\makebox(0,0)[cc]{\hbox{\tcr{$\mathbf Z'_\gamma$}}}}
\rput(1.5,0.5){\makebox(0,0)[cc]{\hbox{\tcr{$\mathbf Z'_\delta$}}}}
\rput(1.5,1.5){\makebox(0,0)[cc]{\hbox{\tcr{$\mathbf Z'_\epsilon$}}}}
}
}

\put(-3.1,0.1){\NET}
\put(-3.1,0.1){\NN}

\put(3.1,0.1){\NET}
\put(3.1,0.1){\NNN}
}
\end{pspicture}
\caption{\small
Elementary move M1: $Z'_\epsilon=Z_{-\epsilon}$, $Z'_\delta=Z_\delta+Z_{\delta+\epsilon}$, $Z'_\alpha=Z_\alpha+Z_{\alpha+\epsilon}$,
$Z'_\beta=\sum_{j=1}^\infty (-1)^{j-1}Z_{\beta+j\epsilon}$,
$Z'_\gamma=\sum_{j=1}^\infty (-1)^{j-1}Z_{\gamma+j\epsilon}$.}
\label{fig:M1}
\end{figure}

\begin{figure}[H]
\begin{pspicture}(-3.5,0.)(3.5,3){\psset{unit=0.8}
\newcommand{\NET}{%
{\psset{unit=1}
\rput(0,0){\psline[linecolor=blue,linewidth=2pt]{-}(0,0)(3.,0.)}
\rput(0,0){\psline[linecolor=blue,linewidth=2pt]{-}(3,0)(3.,3.)}
\rput(0,0){\psline[linecolor=blue,linewidth=2pt]{-}(3,3)(0,3)}
\rput(0,0){\psline[linecolor=blue,linewidth=2pt]{-}(0,3)(0,0)}
\rput(0,0){\psline[linecolor=black,linewidth=1pt]{-}(1,3)(1,1.65)}
\rput(0,0){\psline[linecolor=black,linewidth=1pt]{-}(0,1.5)(0.85,1.5)}
\rput(0,0){\psline[linecolor=black,linewidth=1pt]{-}(1,0.)(1,1.35)}
\rput(0,0){\psline[linecolor=black,linewidth=1pt]{-}(1.15,1.5)(1.85,1.5)}
\rput(0,0){\psline[linecolor=black,linewidth=1pt]{-}(2.15,1.5)(3,3)}
\rput(0,0){\psline[linecolor=black,linewidth=1pt]{-}(2.15,1.5)(3,0)}
\put(1,1.5){\pscircle[fillstyle=solid,fillcolor=white,linecolor=black]{.15}}
\put(2,1.5){\pscircle[fillstyle=solid,fillcolor=white,linecolor=black]{.15}}
\rput(.5,2.){\makebox(0,0)[cc]{\hbox{\tcr{$\mathbf Z_\alpha$}}}}
\rput(0.5,.5){\makebox(0,0)[cc]{\hbox{\tcr{$\mathbf Z_\beta$}}}}
\rput(1.5,2){\makebox(0,0)[cc]{\hbox{\tcr{$\mathbf Z_\gamma$}}}}
\rput(1.5,0.5){\makebox(0,0)[cc]{\hbox{\tcr{$\mathbf Z_\delta$}}}}
\rput(2.7,1.5){\makebox(0,0)[cc]{\hbox{\tcr{$\mathbf Z_\epsilon$}}}}
}
}
\newcommand{\NN}{
{\psset{unit=1}
\rput(0,0){\psline[linecolor=blue,linewidth=2pt]{-}(0,0)(3.,0.)}
\rput(0,0){\psline[linecolor=blue,linewidth=2pt]{-}(3,0)(3.,3.)}
\rput(0,0){\psline[linecolor=blue,linewidth=2pt]{-}(3,3)(0,3)}
\rput(0,0){\psline[linecolor=blue,linewidth=2pt]{-}(0,3)(0,0)}
\rput(0,0){\psline[linecolor=black,linewidth=1pt]{-}(1,3)(1.5,1.65)}
\rput(0,0){\psline[linecolor=black,linewidth=1pt]{-}(0,1.5)(1.35,1.5)}
\rput(0,0){\psline[linecolor=black,linewidth=1pt]{-}(1,0.)(1.5,1.35)}
\rput(0,0){\psline[linecolor=black,linewidth=1pt]{-}(1.65,1.5)(3,3)}
\rput(0,0){\psline[linecolor=black,linewidth=1pt]{-}(1.65,1.5)(3,0)}
\put(1.5,1.5){\pscircle[fillstyle=solid,fillcolor=white,linecolor=black]{.15}}
\rput(.5,2.){\makebox(0,0)[cc]{\hbox{\tcr{$\mathbf Z_\alpha$}}}}
\rput(0.5,.5){\makebox(0,0)[cc]{\hbox{\tcr{$\mathbf Z_\beta$}}}}
\rput(1.8,2.2){\makebox(0,0)[cc]{\hbox{\tcr{$\mathbf Z_\gamma$}}}}
\rput(1.8,0.5){\makebox(0,0)[cc]{\hbox{\tcr{$\mathbf Z_\delta$}}}}
\rput(2.5,1.5){\makebox(0,0)[cc]{\hbox{\tcr{$\mathbf Z_\epsilon$}}}}

}
}

\put(-3.1,0.1){\NET}

\put(3.1,0.1){\NN}
}
\end{pspicture}
\caption{\small
Elementary move M2.}
\label{fig:M2}
\end{figure}
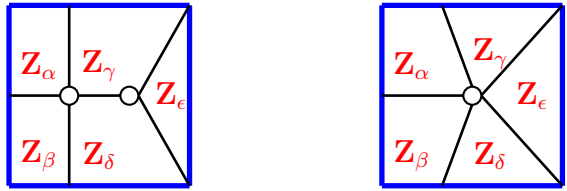

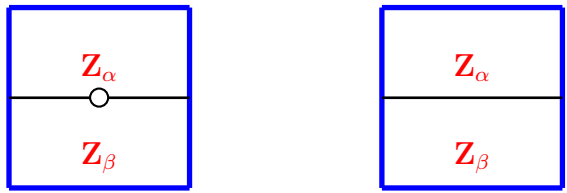
\begin{figure}[H]
\begin{pspicture}(-3.5,0.)(3.5,3){\psset{unit=0.8}
\newcommand{\NET}{%
{\psset{unit=1}
\rput(0,0){\psline[linecolor=blue,linewidth=2pt]{-}(0,0)(3.,0.)}
\rput(0,0){\psline[linecolor=blue,linewidth=2pt]{-}(3,0)(3.,3.)}
\rput(0,0){\psline[linecolor=blue,linewidth=2pt]{-}(3,3)(0,3)}
\rput(0,0){\psline[linecolor=blue,linewidth=2pt]{-}(0,3)(0,0)}
\rput(0,0){\psline[linecolor=black,linewidth=1pt]{-}(0,1.5)(1.35,1.5)}
\rput(0,0){\psline[linecolor=black,linewidth=1pt]{-}(1.65,1.5)(3,1.5)}
\put(1.5,1.5){\pscircle[fillstyle=solid,fillcolor=white,linecolor=black]{.15}}
\rput(1.5,2.){\makebox(0,0)[cc]{\hbox{\tcr{$\mathbf Z_\alpha$}}}}
\rput(1.5,.5){\makebox(0,0)[cc]{\hbox{\tcr{$\mathbf Z_\beta$}}}}
}
}
\newcommand{\NN}{
{\psset{unit=1}
\rput(0,0){\psline[linecolor=blue,linewidth=2pt]{-}(0,0)(3.,0.)}
\rput(0,0){\psline[linecolor=blue,linewidth=2pt]{-}(3,0)(3.,3.)}
\rput(0,0){\psline[linecolor=blue,linewidth=2pt]{-}(3,3)(0,3)}
\rput(0,0){\psline[linecolor=blue,linewidth=2pt]{-}(0,3)(0,0)}
\rput(0,0){\psline[linecolor=black,linewidth=1pt]{-}(0,1.5)(3.,1.5)}
\rput(1.5,2.){\makebox(0,0)[cc]{\hbox{\tcr{$\mathbf Z_\alpha$}}}}
\rput(1.5,.5){\makebox(0,0)[cc]{\hbox{\tcr{$\mathbf Z_\beta$}}}}
}
}

\put(-3.1,0.1){\NET}

\put(3.1,0.1){\NN}
}
\end{pspicture}
\caption{\small
Elementary move M3.}
\label{fig:M3}
\end{figure}

\begin{figure}[H]
\begin{pspicture}(-3.5,0.)(3.5,3){\psset{unit=0.8}
\newcommand{\NET}{%
{\psset{unit=1}
\rput(0,0){\psline[linecolor=blue,linewidth=2pt]{-}(0,0)(3.,0.)}
\rput(0,0){\psline[linecolor=blue,linewidth=2pt]{-}(3,0)(3.,3.)}
\rput(0,0){\psline[linecolor=blue,linewidth=2pt]{-}(3,3)(0,3)}
\rput(0,0){\psline[linecolor=blue,linewidth=2pt]{-}(0,3)(0,0)}
\rput(0,0){\psline[linecolor=black,linewidth=1pt]{-}(0,1.5)(0.85,1.5)}
\psarc[linecolor=black, linewidth=1.pt]{-}(1.5,1.3){.6}{30}{150}
\psarc[linecolor=black, linewidth=1.pt]{-}(1.5,1.7){.6}{210}{330}
\rput(0,0){\psline[linecolor=black,linewidth=1pt]{-}(2.15,1.5)(3,1.5)}
\put(1.,1.5){\pscircle[fillstyle=solid,fillcolor=white,linecolor=black]{.15}}
\put(2.,1.5){\pscircle[fillstyle=solid,fillcolor=black,linecolor=black]{.15}}
\rput(1.5,2.5){\makebox(0,0)[cc]{\hbox{\tcr{$\mathbf Z_\alpha$}}}}
\rput(1.5,1.5){\makebox(0,0)[cc]{\hbox{\tcr{$\mathbf Z_\epsilon$}}}}
\rput(1.5,.5){\makebox(0,0)[cc]{\hbox{\tcr{$\mathbf Z_\beta$}}}}
}
}
\newcommand{\NN}{
{\psset{unit=1}
\rput(0,0){\psline[linecolor=blue,linewidth=2pt]{-}(0,0)(5.,0.)}
\rput(0,0){\psline[linecolor=blue,linewidth=2pt]{-}(5,0)(5.,3.)}
\rput(0,0){\psline[linecolor=blue,linewidth=2pt]{-}(5,3)(0,3)}
\rput(0,0){\psline[linecolor=blue,linewidth=2pt]{-}(0,3)(0,0)}
\rput(0,0){\psline[linecolor=black,linewidth=1pt]{-}(0,1.5)(5.,1.5)}
\rput(2.5,2.){\makebox(0,0)[cc]{\hbox{\tcr{$\mathbf {\sum_{j=1}^\infty (-1)^{j-1}Z_{\alpha+j\epsilon}}$}}}}
\rput(2.5,.5){\makebox(0,0)[cc]{\hbox{\tcr{$\mathbf {Z_\beta+Z_{\beta+\epsilon}}$}}}}
}
}

\put(-4.1,0.1){\NET}

\put(2.1,0.1){\NN}
}
\end{pspicture}
\caption{\small
Elementary move R1.}
\label{fig:R1}
\end{figure}
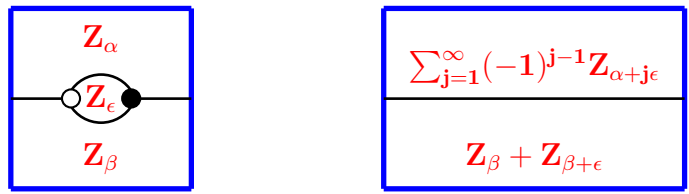

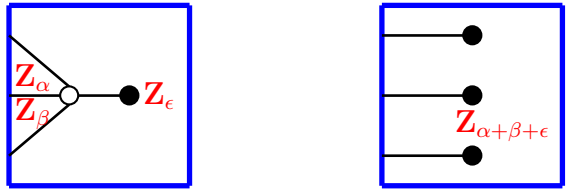
\begin{figure}[H]
\begin{pspicture}(-3.5,0.)(3.5,3){\psset{unit=0.8}
\newcommand{\NET}{%
{\psset{unit=1}
\rput(0,0){\psline[linecolor=blue,linewidth=2pt]{-}(0,0)(3.,0.)}
\rput(0,0){\psline[linecolor=blue,linewidth=2pt]{-}(3,0)(3.,3.)}
\rput(0,0){\psline[linecolor=blue,linewidth=2pt]{-}(3,3)(0,3)}
\rput(0,0){\psline[linecolor=blue,linewidth=2pt]{-}(0,3)(0,0)}
\rput(0,0){\psline[linecolor=black,linewidth=1pt]{-}(0,2.5)(1,1.65)}
\rput(0,0){\psline[linecolor=black,linewidth=1pt]{-}(0,1.5)(0.85,1.5)}
\rput(0,0){\psline[linecolor=black,linewidth=1pt]{-}(0,0.5)(1,1.35)}
\rput(0,0){\psline[linecolor=black,linewidth=1pt]{-}(1.15,1.5)(1.85,1.5)}
\put(1,1.5){\pscircle[fillstyle=solid,fillcolor=white,linecolor=black]{.15}}
\put(2,1.5){\pscircle[fillstyle=solid,fillcolor=black,linecolor=black]{.15}}
\rput(.4,1.8){\makebox(0,0)[cc]{\hbox{\tcr{$\mathbf Z_\alpha$}}}}
\rput(0.4,1.2){\makebox(0,0)[cc]{\hbox{\tcr{$\mathbf Z_\beta$}}}}
\rput(2.5,1.5){\makebox(0,0)[cc]{\hbox{\tcr{$\mathbf Z_\epsilon$}}}}
}
}
\newcommand{\NN}{
{\psset{unit=1}
\rput(0,0){\psline[linecolor=blue,linewidth=2pt]{-}(0,0)(3.,0.)}
\rput(0,0){\psline[linecolor=blue,linewidth=2pt]{-}(3,0)(3.,3.)}
\rput(0,0){\psline[linecolor=blue,linewidth=2pt]{-}(3,3)(0,3)}
\rput(0,0){\psline[linecolor=blue,linewidth=2pt]{-}(0,3)(0,0)}
\rput(0,0){\psline[linecolor=black,linewidth=1pt]{-}(0,2.5)(1.35,2.5)}
\rput(0,0){\psline[linecolor=black,linewidth=1pt]{-}(0,1.5)(1.35,1.5)}
\rput(0,0){\psline[linecolor=black,linewidth=1pt]{-}(0,0.5)(1.35,0.5)}
\put(1.5,1.5){\pscircle[fillstyle=solid,fillcolor=black,linecolor=black]{.15}}
\put(1.5,0.5){\pscircle[fillstyle=solid,fillcolor=black,linecolor=black]{.15}}
\put(1.5,2.5){\pscircle[fillstyle=solid,fillcolor=black,linecolor=black]{.15}}
\rput(2,1){\makebox(0,0)[cc]{\hbox{\tcr{$\mathbf {Z_{\alpha+\beta+\epsilon}}$}}}}
}
}

\put(-3.1,0.1){\NET}

\put(3.1,0.1){\NN}
}
\end{pspicture}
\caption{\small
Elementary move R2.}
\label{fig:R2}
\end{figure}

\begin{figure}[H]
\begin{pspicture}(-3.5,0.)(3.5,3){\psset{unit=0.8}
\newcommand{\NET}{%
{\psset{unit=1}
\rput(0,0){\psline[linecolor=blue,linewidth=2pt]{-}(0,0)(3.,0.)}
\rput(0,0){\psline[linecolor=blue,linewidth=2pt]{-}(3,0)(3.,3.)}
\rput(0,0){\psline[linecolor=blue,linewidth=2pt]{-}(3,3)(0,3)}
\rput(0,0){\psline[linecolor=blue,linewidth=2pt]{-}(0,3)(0,0)}
\rput(0,0){\psline[linecolor=black,linewidth=1pt]{-}(1.15,1.5)(1.85,1.5)}
\put(1,1.5){\pscircle[fillstyle=solid,fillcolor=white,linecolor=black]{.15}}
\put(2,1.5){\pscircle[fillstyle=solid,fillcolor=black,linecolor=black]{.15}}
\rput(1.5,0.5){\makebox(0,0)[cc]{\hbox{\tcr{$\mathbf Z_\epsilon$}}}}
}
}
\newcommand{\NN}{
{\psset{unit=1}
\rput(0,0){\psline[linecolor=blue,linewidth=2pt]{-}(0,0)(3.,0.)}
\rput(0,0){\psline[linecolor=blue,linewidth=2pt]{-}(3,0)(3.,3.)}
\rput(0,0){\psline[linecolor=blue,linewidth=2pt]{-}(3,3)(0,3)}
\rput(0,0){\psline[linecolor=blue,linewidth=2pt]{-}(0,3)(0,0)}
\rput(1.5,1.){\makebox(0,0)[cc]{\hbox{\tcr{$\mathbf {Z_{\epsilon}}$}}}}
}
}

\put(-3.1,0.1){\NET}

\put(3.1,0.1){\NN}
}
\end{pspicture}
\caption{\small
Elementary move R3.}
\label{fig:R3}
\end{figure}
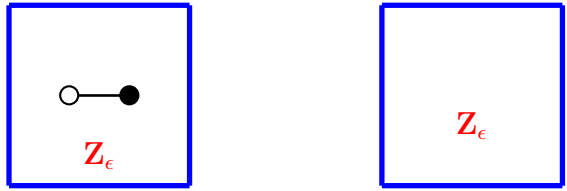

\begin{definition} Two networks are \emph{move equivalent} if they are connected by a sequence of elementary moves.
\end{definition}

The corresponding (move) equivalence is called \emph{quantum}
(move) equivalence.

The following result extends the results of \cite{Po} to quantum networks.

\begin{lemma} Two quantum move equivalent networks are quantum equivalent.
\end{lemma}

\begin{proof}  The proof follows ~\cite{Po}. Compared to the commutative case we just need  to check one additional condition that quantum parameters elementary transformations also  form q-commutative family. The cases R2,R3, M2, and M3 are evident. Let's consider M1 and R1. The case M1 is proved in \cite{FG2}.
We give the proof here for completeness.

We want to show that  $Z'_\alpha, Z'_\beta, Z'_\gamma,Z'_\delta$ and $Z'_\epsilon$ q-commute.

Note that $\langle \alpha,\beta+(k+1)\epsilon\rangle=\langle \alpha, \beta\rangle-(k+1)=\langle\alpha+\epsilon,\beta+k\epsilon\rangle$, therefore $Z'_\alpha Z'_\beta= (Z_{\alpha}+Z_{\alpha+\epsilon}) \sum_{j=1}^\infty (-1)^{j-1} Z_{\beta+j\epsilon}=q^{\langle \alpha,\beta+\epsilon\rangle}Z_{\alpha+\beta+\epsilon}$ while $Z'_\beta Z'_\alpha=q^{-\langle\alpha,\beta+\epsilon\rangle} Z_{\alpha+\beta+\epsilon}$.
Hence, $Z'_\beta Z'_\alpha=q^{-2\langle\alpha,\beta+\epsilon\rangle} Z'_\alpha Z'_\beta$. 
Commutation relations for all other pairs of parameters can be checked similarly.
 
Different perfect orientations are in one-to-one correspondence with the \emph{almost perfect matchings}  (see~\cite{PSW}). Up to evident symmetries, there are only two essentially different almost perfect matchings and, hence, we need to check two  perfect orientations. The straightforward computation shows that the elementary move M1 does not change measurements for any choice of perfect orientation.

The case R1 is similar.
\end{proof}

\begin{definition}(\cite{Po})  We say that a plabic network (or graph) is \emph{reduced} if it has no isolated connected components and there is no network/graph in its move-equivalence class to which we can apply a reduction (R1) or (R2). A leafless reduced network/graph is a reduced network/graph without non-boundary leaves.
\end{definition}

The following statements are proved in~\cite{Po} .

\begin{lemma} \cite{Po} Any network is move equivalent to a \emph{reduced} network.
\end{lemma}

\begin{lemma}\cite{Po} Two reduced equivalent networks are (M1-M3)-move equivalent. 
\end{lemma}

\begin{definition} We call a maximal simple oriented path $P=(p_0,p_1,\dots,p_h)$, $p_i\ne p_j$ for all $i\ne j$ \emph{unequivocal} if there is no oriented path
$(p_k,q_1,\dots,q_t,p_\ell)$ such that 
$q_s\ne p_r$ for all $1\le s\le t$ and $0\le r\le h$.
\end{definition}

Let $P$ be an unequivocal path in a network $N\in\operatorname{Net}_{m,n}$. Reverse the orientation of $P$ keeping face weights we obtain the new network $N'$.

\begin{lemma}\label{lem:rigid}   Reversing the orientation of unequivocal path $P$ does not change quantum grassmannian measurement $\operatorname{Meas}_q(N')=\operatorname{Meas}_q(N)$. 
\end{lemma}

\begin{example} Consider networks on the Fig.~\ref{fig:network}.

\begin{figure}[h]
\begin{pspicture}(-3.5,0.)(3.5,3){\psset{unit=0.8}
\newcommand{\NET}{%
{\psset{unit=1}
\rput(0,0){\psline[linecolor=blue,linestyle=dashed,linewidth=2pt]{-}(0,0)(3.,0.)}
\rput(0,0){\psline[linecolor=blue,linestyle=dashed,linewidth=2pt]{-}(3,0)(3.,3.)}
\rput(0,0){\psline[linecolor=blue,linestyle=dashed,linewidth=2pt]{-}(3,3)(0,3)}
\rput(0,0){\psline[linecolor=blue,linestyle=dashed,linewidth=2pt]{-}(0,3)(0,0)}
\rput(0,0){\psline[linecolor=black,linewidth=2pt]{<-}(2,1)(3,1)}
\rput(0,0){\psline[linecolor=black,linewidth=2pt]{<-}(0,1)(2,1)}
\rput(0,0){\psline[linecolor=black,linewidth=2pt]{<-}(0,2.5)(1,2)}
\rput(0,0){\psline[linecolor=black,linewidth=2pt]{->}(2.,1)(1.5,2)}
\rput(2,2.5){\makebox(0,0)[cc]{\hbox{\tcr{$\mathbf \alpha$}}}}
\rput(2.5,1.5){\makebox(0,0)[cc]{\hbox{\tcr{$\mathbf \beta$}}}}
\rput(0.3,2.){\makebox(0,0)[cc]{\hbox{\tcr{$\mathbf \gamma$}}}}
\rput(0.8,1.5){\makebox(0,0)[cc]{\hbox{\tcr{$\mathbf \delta$}}}}
\rput(1.7,0.6){\makebox(0,0)[cc]{\hbox{\tcr{$\mathbf \epsilon$}}}}

\rput(3.3,2.){\makebox(0,0)[cc]{\hbox{\tcr{$\mathbf  2$}}}}
\rput(3.3,1){\makebox(0,0)[cc]{\hbox{\tcr{$\mathbf 1$}}}}
\rput(-0.3,1.){\makebox(0,0)[cc]{\hbox{\tcr{$\mathbf 5$}}}}
\rput(-0.3,1.5){\makebox(0,0)[cc]{\hbox{\tcr{$\mathbf 4$}}}}
\rput(-0.3,2.5){\makebox(0,0)[cc]{\hbox{\tcr{$\mathbf 3$}}}}

}
}
\put(-3.1,0.1){\NET}
\put(0,0){\psline[linecolor=black,linewidth=2pt]{<-}(-3.1,1.6)(-2.1,2.1)}
\put(0,0){\psline[linecolor=black,linewidth=2pt]{<-}(-2.1,2.1)(-1.6,2.1)}
\put(0,0){\psline[linecolor=black,linewidth=2pt]{<-}(-1.5,2.1)(-0.1,2.1)}

\put(3.1,0.1){\NET}
\put(0,0){\psline[linecolor=black,linewidth=2pt]{->}(3.1,1.6)(4.1,2.1)}
\put(0,0){\psline[linecolor=black,linewidth=2pt]{->}(4.1,2.1)(4.6,2.1)}
\put(0,0){\psline[linecolor=black,linewidth=2pt]{->}(4.6,2.1)(6.1,2.1)}
}
\end{pspicture}
\caption{\small
Changing orientation of the path $P: {\bf 2}\leadsto {\bf 4}$ transforms network $N$ on the left into network $N'$ on the right}
\label{fig:network}
\end{figure}
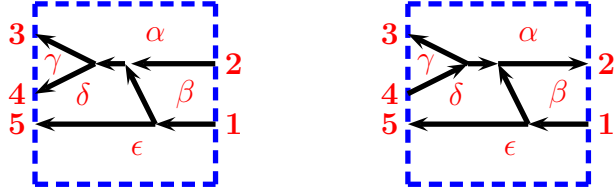

The corresponding quantum grassmannian measurement matrices are \newline
$$Q^{gr}_q=\begin{pmatrix}
q^{-1/2}   & 0 \\
0 &  q^{-3/2} \\
-q^{-2} Z_{\alpha+\beta} & q^{-2} Z_\alpha \\
-q^{-2} Z_{\alpha+\beta+\gamma} & q^{-2} Z_{\alpha+\gamma}\\
-q^{-2} Z_{\alpha+\beta+\gamma+\delta} & 0 
\end{pmatrix} \text{ and } 
\left(Q^{gr}_q\right)'=\begin{pmatrix}
q^{-1/2}   & 0 \\
q^{-1} Z_{\beta} & q^{-1} Z_{\delta+\epsilon+\beta} \\
0 & q^{-1} Z_{-\gamma}\\
0 &  q^{-3/2} \\
-q^{-2} Z_{\alpha+\beta+\gamma+\delta} & 0 
\end{pmatrix}.$$ 
\smallskip

Note that 
$Q^{gr}_q C=\left(Q^{gr}_q\right)'$, where $C=
\begin{pmatrix} 1 & 0 \\
q^{1/2} Z_\beta & q^{1/2} Z_{\delta+\epsilon+\beta}
\end{pmatrix}.$

Indeed, consider for example $(Q^{gr}_q C)_{31}=(Q^{gr}_q)_{31}+q^{1/2}(Q^{gr}_q)_{32} Z_\beta=-q^{-2} Z_{\alpha+\beta}+q^{-3/2} Z_\alpha Z_\beta$.
Recall that $Z_\alpha Z_\beta=q^{-1/2}Z_{\alpha+\beta}$. Therefore, $(Q^{gr}_q C)_{31}=0=\left(Q^{gr}_q\right)' _{31}$. Similarly, we can prove equalities for all the entries of these $5\times 2$ matrices and we observe that  $\operatorname{Meas}_q(N)=\operatorname{Meas}_q(N')\in \operatorname{Gr}_q(2,5)$.
\end{example}

\begin{proof} Let $N'$ be the network obtained as a result of the change of the directions of all arrows of the simple unequivocal path $P$ in $N$ from a boundary vertex  $a$ to a boundary vertex $b$. We will denote by $P^{-1}$ the path in $N'$ obtained from $P$ by orientation reversing. We assume first that the boundary vertices are labelled so that $1\le a < b\le m+n$. Since path $P$ is unequivocal  there is only one path $P$ from $a$ to $b$, and $Q_q(a, b)=w_P$. Moreover, any other path $R$ from $s$ to $t$ where both $s$ and $t$ are distinct from $a$ and $b$ has at most one common interval $[V,W]$  with path $P$. The first point $V$ (counting from $s$) where two paths meet has two incoming arrows and one outgoing and, hence, is colored black, the point $W$ where two paths separate is white (see Figure~\ref{fig:network1}). Similarly, any path from $a$ to a sink different from $b$ separates from $P$ at a white point; any path from a sink different from $a$ to $b$ joins path $P$ at a black point.

Let $Q^{gr}_q$ be the quantum grassmannian bounded measurement matrix of the network $N$, $\left(Q^{gr}_q\right)'$ be the quantum grassmannian boundary measurement matrix\
of  $N'$.  Let  $F_P\subset Faces$ denote the subset of all faces to the right of the path $P$, ${\bf v}_P=\sum_{\alpha\in F_p}\alpha$. Then, 
$w_P=Z_{{\bf v}_P}$, $F_{P^{-1}}=Faces\setminus F_P$, ${\bf v}_{P^{-1}}=-{\bf v}_P$, $w_{P^{-1}}=Z_{{\bf v}_{P^{-1}}}=Z_{-{\bf v}_P}=(w_P)^{-1}$.

Consider first the case $1\le a<b$.

Let $s<a<t<b$ in the cyclic order of the boundary vertices (see Figure~\ref{fig:network1}).
Observe, $w_{b\to W\to t}=Z_{\alpha+\beta+\delta}= \col{Z_{\beta+\delta}Z_\alpha} =\col{w_{a\to V\to W\to t}\cdot w_{a\to V\to W\to b}^{-1}}=
\col{w_{a\to V\to W\to t}\cdot (Q'_q)_{ab}}$  Since the equality holds for any directed path from $b$ to $t$, and 
$\col{(Q'_q)_{ab}(Q_q)_{ta}}$ $=$ $\col{(Q_q)_{ta} (Q_q)_{ba}^{-1}}=
q^{-1/2}(Q_q)_{ta}(Q_q)_{ba}^{-1}$, we conclude that $(Q'_q)_{tb}=\col{(Q'_q)_{ab}(Q_q)_{ta}}$ $=$ $\col{(Q_q)_{ba}^{-1}(Q_q)_{ta}}=
q^{-1/2}(Q_q)_{ta}(Q_q)_{ba}^{-1}$.
Similarly, $(Q'_q)_{as}=
\col{(Q'_q)_{ab}(Q_q)_{bs}}$ $=$ $\col{(Q_q)_{ba}^{-1}(Q_q)_{bs}}=
q^{1/2}(Q_q)_{bs}(Q_q)_{ba}^{-1}$ and $(Q_q)_{ts}$
$=\col{(Q_q)_{bs}(Q'_q)_{tb}} $
$=q^{1/2}(Q_q)_{bs}(Q'_q)_{tb}$. Therefore, $(Q^{gr}_q)'_{ts}=(Q^{gr}_q)_{ts}+(Q^{gr}_q)_{bs}(Q^{gr}_q)'_{tb}=0$.
Here, $\alpha,\beta,\gamma,\delta$ are appropriate subsets of $Faces$.


\begin{figure}[H]
\begin{pspicture}(-3.5,0.)(3.5,3){\psset{unit=0.8}
\newcommand{\NET}{%
{\psset{unit=1}
\rput(0,0){\psline[linecolor=blue,linestyle=dashed,linewidth=2pt]{-}(0,0)(3.,0.)}
\rput(0,0){\psline[linecolor=blue,linestyle=dashed,linewidth=2pt]{-}(3,0)(3.,3.)}
\rput(0,0){\psline[linecolor=blue,linestyle=dashed,linewidth=2pt]{-}(3,3)(0,3)}
\rput(0,0){\psline[linecolor=blue,linestyle=dashed,linewidth=2pt]{-}(0,3)(0,0)}
\rput(0,0){\psline[linecolor=black,linewidth=1pt]{<-}(2.1,1.9)(3,1)}
\rput(0,0){\psline[linecolor=black,linewidth=3pt]{<-}(1.15,2)(1.9,2)}
\rput(0,0){\psline[linecolor=black,linewidth=3pt]{<-}(2.1,2)(3,2)}
\rput(0,0){\psline[linecolor=black,linewidth=3pt]{<-}(0,1.5)(1,2)}
\rput(0,0){\psline[linecolor=black,linewidth=1pt]{<-}(0,2.5)(1,2)}
\rput(2,2.5){\makebox(0,0)[cc]{\hbox{{$\mathbf \alpha$}}}}
\rput(2.7,1.6){\makebox(0,0)[cc]{\hbox{{$\mathbf \beta$}}}}
\rput(0.3,2.){\makebox(0,0)[cc]{\hbox{{$\mathbf \gamma$}}}}
\rput(1.2,1.){\makebox(0,0)[cc]{\hbox{{$\mathbf \delta$}}}}
\put(2,2){\pscircle[fillstyle=solid,fillcolor=black,linecolor=black]{.15}}
\put(1,2){\pscircle[fillstyle=solid,fillcolor=white,linecolor=black]{.15}}

\rput(3.3,2.){\makebox(0,0)[cc]{\hbox{{$\bold a$}}}}
\rput(3.3,1){\makebox(0,0)[cc]{\hbox{{$\mathbf s$}}}}
\rput(-0.3,1.5){\makebox(0,0)[cc]{\hbox{{$\mathbf b$}}}}
\rput(-0.3,2.5){\makebox(0,0)[cc]{\hbox{{$\mathbf t$}}}}
\rput(2,1.5){\makebox(0,0)[cc]{\hbox{{$\mathbf V$}}}}
\rput(1,2.5){\makebox(0,0)[cc]{\hbox{{$\mathbf W$}}}}

}
}

\newcommand{\NNET}{%
{\psset{unit=1}
\rput(0,0){\psline[linecolor=blue,linestyle=dashed,linewidth=2pt]{-}(0,0)(3.,0.)}
\rput(0,0){\psline[linecolor=blue,linestyle=dashed,linewidth=2pt]{-}(3,0)(3.,3.)}
\rput(0,0){\psline[linecolor=blue,linestyle=dashed,linewidth=2pt]{-}(3,3)(0,3)}
\rput(0,0){\psline[linecolor=blue,linestyle=dashed,linewidth=2pt]{-}(0,3)(0,0)}
\rput(0,0){\psline[linecolor=black,linewidth=1pt]{<-}(2.1,1.9)(3,1)}
\rput(0,0){\psline[linecolor=black,linewidth=3pt]{->}(1.15,2)(1.9,2)}
\rput(0,0){\psline[linecolor=black,linewidth=3pt]{->}(2.1,2)(3,2)}
\rput(0,0){\psline[linecolor=black,linewidth=3pt]{->}(0,1.5)(0.9,1.9)}
\rput(0,0){\psline[linecolor=black,linewidth=1pt]{<-}(0,2.5)(1,2)}
\rput(2,2.5){\makebox(0,0)[cc]{\hbox{{$\mathbf \alpha$}}}}
\rput(2.7,1.6){\makebox(0,0)[cc]{\hbox{{$\mathbf \beta$}}}}
\rput(0.3,2.){\makebox(0,0)[cc]{\hbox{{$\mathbf \gamma$}}}}
\rput(1.2,1.){\makebox(0,0)[cc]{\hbox{{$\mathbf \delta$}}}}
\put(2,2){\pscircle[fillstyle=solid,fillcolor=black,linecolor=black]{.15}}
\put(1,2){\pscircle[fillstyle=solid,fillcolor=white,linecolor=black]{.15}}

\rput(3.3,2.){\makebox(0,0)[cc]{\hbox{{$\bold a$}}}}
\rput(3.3,1){\makebox(0,0)[cc]{\hbox{{$\mathbf s$}}}}
\rput(-0.3,1.5){\makebox(0,0)[cc]{\hbox{{$\mathbf b$}}}}
\rput(-0.3,2.5){\makebox(0,0)[cc]{\hbox{{$\mathbf t$}}}}
\rput(-0.3,2.5){\makebox(0,0)[cc]{\hbox{{$\mathbf t$}}}}
\rput(2,1.5){\makebox(0,0)[cc]{\hbox{{$\mathbf V$}}}}
\rput(1,2.5){\makebox(0,0)[cc]{\hbox{{$\mathbf W$}}}}

}
}

\put(-3.1,0.1){\NET}
\put(3.1,0.1){\NNET}

}
\end{pspicture}
\caption{\small
Change of the orientation of the path $P: {\bf a}\leadsto {\bf b}$, $s<a<t<b$.}
\label{fig:network1}
\end{figure}
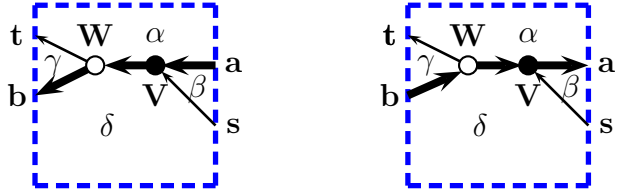

In the same way we investigate all the remaining mutual positions of $s,t$ and $1\le a<b$ which leads to the following matrix identity. 
Let $a_0={\mathbb J}^{-1}(a)$, $f_0$ be the index of source of $X$ such that $b$ lies between the source $f_0$ and $f_0+1$
(equivalently, ${\mathbb J}(f_0)<b<{\mathbb J}(f_0+1)$.  Note that $a_0\le f_0$ since $a<b$. Define $n\times n$ matrix $C$ for $1<a<b$ as follows
$$C_{ij}=\begin{cases}
\delta_{ij}, &\text{ if } i<a_0 \text{ or } i> f_0;\\
(-1)^{\big\lfloor{|j-a_0+1/2|}\big\rfloor}q^{1/2}(Q_q)_{ja}(Q_q)_{ba}^{-1} &\text{ if } i=a_0;\\
q\delta_{i-1,j}, &\text{ if } a_0<i\le f_0.
\end{cases}
$$

Then,  $\left(Q^{gr}_q\right)'=Q^{gr}_q C$.

To study $1\le b<a$, note that $\left(Q^{gr}_q\right)'=Q^{gr}\cdot C$ implies $Q^{gr}_q=\left(Q^{gr}_q\right)'\cdot C^{-1}$ where $C^{-1}$ is obtained 
from $C$ by changing the signs of the off-diagonal elements and adjusting the powers of $q$. More exactly,  
define $n\times n$ matrix $\widetilde C$ for $1<b<a$ as follows
$$\widetilde C_{ij}=\begin{cases}
\delta_{ij}, &\text{ if } i\le f_0 \text{ or } i> a_0;\\
(-1)^{\big\lfloor{|j-a_0-1/2|}\big\rfloor}q^{-1/2}(Q_q)_{ja}(Q_q)_{ba}^{-1} &\text{ if } i=a_0;\\
q^{-1}\delta_{i+1,j}, &\text{ if } f_0< i< a_0.
\end{cases}
$$

This observation proves Lemma~\ref{lem:rigid}  $1\le b<a$.
\end{proof}


\begin{example} Consider the networks in Figure~\ref{fig:network3}.

\begin{figure}[h]
\begin{pspicture}(-10,0.)(10,4.5){\psset{unit=0.8}
\newcommand{\NET}{%
{\psset{unit=1}
\rput(0,0){\psline[linecolor=blue,linewidth=2pt]{-}(0,0)(8.,0.)}
\rput(0,0){\psline[linecolor=blue,linewidth=2pt]{-}(8,0)(8.,4.)}
\rput(0,0){\psline[linecolor=blue,linewidth=2pt]{-}(8,4)(0,4)}
\rput(0,0){\psline[linecolor=blue,linewidth=2pt]{-}(0,4)(0,0)}
\rput(0,0){\psline[linecolor=black,linewidth=3pt]{<-}(0,2)(8,2)}
\rput(0,0){\psline[linecolor=black,linewidth=1pt]{<-}(1,0)(1,2)}
\rput(0,0){\psline[linecolor=black,linewidth=1pt]{->}(2,0)(2,1.9)}
\rput(0,0){\psline[linecolor=black,linewidth=1pt]{<-}(3,0)(3,2)}
\rput(0,0){\psline[linecolor=black,linewidth=1pt]{->}(4,0)(4,1.9)}
\rput(0,0){\psline[linecolor=black,linewidth=1pt]{<-}(5,0)(5,2)}
\rput(0,0){\psline[linecolor=black,linewidth=1pt]{->}(6,0)(6,1.9)}
\rput(0,0){\psline[linecolor=black,linewidth=1pt]{<-}(7,0)(7,2)}

\rput(0,0){\psline[linecolor=black,linewidth=1pt]{<-}(1.4,4)(1.4,2)}
\rput(0,0){\psline[linecolor=black,linewidth=1pt]{->}(2.4,4)(2.4,2.1)}
\rput(0,0){\psline[linecolor=black,linewidth=1pt]{<-}(3.4,4)(3.4,2)}
\rput(0,0){\psline[linecolor=black,linewidth=1pt]{->}(4.4,4)(4.4,2.1)}
\rput(0,0){\psline[linecolor=black,linewidth=1pt]{<-}(5.4,4)(5.4,2)}
\rput(0,0){\psline[linecolor=black,linewidth=1pt]{->}(6.4,4)(6.4,2.1)}
\rput(0,0){\psline[linecolor=black,linewidth=1pt]{<-}(7.4,4)(7.4,2)}
\put(2,2){\pscircle[fillstyle=solid,fillcolor=black,linecolor=black]{.15}}
\put(1,2){\pscircle[fillstyle=solid,fillcolor=white,linecolor=black]{.15}}
\put(3,2){\pscircle[fillstyle=solid,fillcolor=white,linecolor=black]{.15}}
\put(4,2){\pscircle[fillstyle=solid,fillcolor=black,linecolor=black]{.15}}
\put(5,2){\pscircle[fillstyle=solid,fillcolor=white,linecolor=black]{.15}}
\put(6,2){\pscircle[fillstyle=solid,fillcolor=black,linecolor=black]{.15}}
\put(7,2){\pscircle[fillstyle=solid,fillcolor=white,linecolor=black]{.15}}
\put(1.4,2){\pscircle[fillstyle=solid,fillcolor=white,linecolor=black]{.15}}
\put(2.4,2){\pscircle[fillstyle=solid,fillcolor=black,linecolor=black]{.15}}
\put(3.4,2){\pscircle[fillstyle=solid,fillcolor=white,linecolor=black]{.15}}
\put(4.4,2){\pscircle[fillstyle=solid,fillcolor=black,linecolor=black]{.15}}
\put(5.4,2){\pscircle[fillstyle=solid,fillcolor=white,linecolor=black]{.15}}
\put(6.4,2){\pscircle[fillstyle=solid,fillcolor=black,linecolor=black]{.15}}
\put(7.4,2){\pscircle[fillstyle=solid,fillcolor=white,linecolor=black]{.15}}

\rput(4,-.5){\makebox(0,0)[cc]{\hbox{ {$\bold 1$}}}}
\rput(5,-.5){\makebox(0,0)[cc]{\hbox{ {$\mathbf 2$}}}}
\rput(6,-.5){\makebox(0,0)[cc]{\hbox{ {$\mathbf 3$}}}}
\rput(7,-.5){\makebox(0,0)[cc]{\hbox{ {$\mathbf 4$}}}}
\rput(8.3,2.){\makebox(0,0)[cc]{\hbox{ {$\mathbf 5$}}}}
\rput(7.4,4.4){\makebox(0,0)[cc]{\hbox{ {$\mathbf 6$}}}}
\rput(6.4,4.4){\makebox(0,0)[cc]{\hbox{ {$\mathbf 7$}}}}
\rput(5.4,4.4){\makebox(0,0)[cc]{\hbox{ {$\mathbf 8$}}}}
\rput(4.4,4.4){\makebox(0,0)[cc]{\hbox{ {$\mathbf 9$}}}}
\rput(3.4,4.4){\makebox(0,0)[cc]{\hbox{ {${\mathbf 10}$}}}}
\rput(2.4,4.4){\makebox(0,0)[cc]{\hbox{ {${\mathbf 11}$}}}}
\rput(1.4,4.4){\makebox(0,0)[cc]{\hbox{ {${\mathbf 12}$}}}}
\rput(-.4,2.){\makebox(0,0)[cc]{\hbox{ {${\mathbf 13}$}}}}
\rput(1,-.5){\makebox(0,0)[cc]{\hbox{ {${\mathbf 14}$}}}}
\rput(2,-.5){\makebox(0,0)[cc]{\hbox{ {${\mathbf 15}$}}}}
\rput(3,-.5){\makebox(0,0)[cc]{\hbox{ {${\mathbf 16}$}}}}

\rput(4.5,.5){\makebox(0,0)[cc]{\hbox{ {$\bold \alpha_1$}}}}
\rput(5.5,.5){\makebox(0,0)[cc]{\hbox{ {$\mathbf \alpha_2$}}}}
\rput(6.5,.5){\makebox(0,0)[cc]{\hbox{ {$\mathbf \alpha_3$}}}}
\rput(7.5,.5){\makebox(0,0)[cc]{\hbox{ {$\mathbf \alpha_4$}}}}
\rput(7.7,2.4){\makebox(0,0)[cc]{\hbox{ {$\mathbf \alpha_5$}}}}
\rput(7,2.4){\makebox(0,0)[cc]{\hbox{ {$\mathbf \alpha_6$}}}}
\rput(6,2.4){\makebox(0,0)[cc]{\hbox{ {$\mathbf \alpha_7$}}}}
\rput(5,2.4){\makebox(0,0)[cc]{\hbox{ {$\mathbf \alpha_8$}}}}
\rput(4,2.4){\makebox(0,0)[cc]{\hbox{ {$\mathbf \alpha_9$}}}}
\rput(3,2.4){\makebox(0,0)[cc]{\hbox{ {$\mathbf \alpha_{10}$}}}}
\rput(2,2.4){\makebox(0,0)[cc]{\hbox{ {$\mathbf \alpha_{11}$}}}}
\rput(0.7,2.4){\makebox(0,0)[cc]{\hbox{ {$\mathbf \alpha_{12}$}}}}
\rput(0.5,.5){\makebox(0,0)[cc]{\hbox{ {$\mathbf \alpha_{13}$}}}}
\rput(1.5,.5){\makebox(0,0)[cc]{\hbox{ {$\mathbf \alpha_{14}$}}}}
\rput(2.5,.5){\makebox(0,0)[cc]{\hbox{ {$\mathbf \alpha_{15}$}}}}
\rput(3.5,.5){\makebox(0,0)[cc]{\hbox{ {$\mathbf \alpha_{16}$}}}}

}
}

\newcommand{\NNET}{%
{\psset{unit=1}
\rput(0,0){\psline[linecolor=blue,linewidth=2pt]{-}(0,0)(8.,0.)}
\rput(0,0){\psline[linecolor=blue,linewidth=2pt]{-}(8,0)(8.,4.)}
\rput(0,0){\psline[linecolor=blue,linewidth=2pt]{-}(8,4)(0,4)}
\rput(0,0){\psline[linecolor=blue,linewidth=2pt]{-}(0,4)(0,0)}
\rput(0,0){\psline[linecolor=black,linewidth=3pt]{->}(0,2)(8,2)}
\rput(0,0){\psline[linecolor=black,linewidth=1pt]{<-}(1,0)(1,2)}
\rput(0,0){\psline[linecolor=black,linewidth=1pt]{->}(2,0)(2,1.9)}
\rput(0,0){\psline[linecolor=black,linewidth=1pt]{<-}(3,0)(3,2)}
\rput(0,0){\psline[linecolor=black,linewidth=1pt]{->}(4,0)(4,1.9)}
\rput(0,0){\psline[linecolor=black,linewidth=1pt]{<-}(5,0)(5,2)}
\rput(0,0){\psline[linecolor=black,linewidth=1pt]{->}(6,0)(6,1.9)}
\rput(0,0){\psline[linecolor=black,linewidth=1pt]{<-}(7,0)(7,2)}

\rput(0,0){\psline[linecolor=black,linewidth=1pt]{<-}(1.4,4)(1.4,2)}
\rput(0,0){\psline[linecolor=black,linewidth=1pt]{->}(2.4,4)(2.4,2.1)}
\rput(0,0){\psline[linecolor=black,linewidth=1pt]{<-}(3.4,4)(3.4,2)}
\rput(0,0){\psline[linecolor=black,linewidth=1pt]{->}(4.4,4)(4.4,2.1)}
\rput(0,0){\psline[linecolor=black,linewidth=1pt]{<-}(5.4,4)(5.4,2)}
\rput(0,0){\psline[linecolor=black,linewidth=1pt]{->}(6.4,4)(6.4,2.1)}
\rput(0,0){\psline[linecolor=black,linewidth=1pt]{<-}(7.4,4)(7.4,2)}
\put(2,2){\pscircle[fillstyle=solid,fillcolor=black,linecolor=black]{.15}}
\put(1,2){\pscircle[fillstyle=solid,fillcolor=white,linecolor=black]{.15}}
\put(3,2){\pscircle[fillstyle=solid,fillcolor=white,linecolor=black]{.15}}
\put(4,2){\pscircle[fillstyle=solid,fillcolor=black,linecolor=black]{.15}}
\put(5,2){\pscircle[fillstyle=solid,fillcolor=white,linecolor=black]{.15}}
\put(6,2){\pscircle[fillstyle=solid,fillcolor=black,linecolor=black]{.15}}
\put(7,2){\pscircle[fillstyle=solid,fillcolor=white,linecolor=black]{.15}}
\put(1.4,2){\pscircle[fillstyle=solid,fillcolor=white,linecolor=black]{.15}}
\put(2.4,2){\pscircle[fillstyle=solid,fillcolor=black,linecolor=black]{.15}}
\put(3.4,2){\pscircle[fillstyle=solid,fillcolor=white,linecolor=black]{.15}}
\put(4.4,2){\pscircle[fillstyle=solid,fillcolor=black,linecolor=black]{.15}}
\put(5.4,2){\pscircle[fillstyle=solid,fillcolor=white,linecolor=black]{.15}}
\put(6.4,2){\pscircle[fillstyle=solid,fillcolor=black,linecolor=black]{.15}}
\put(7.4,2){\pscircle[fillstyle=solid,fillcolor=white,linecolor=black]{.15}}

\rput(4,-.5){\makebox(0,0)[cc]{\hbox{ {$\bold 1$}}}}
\rput(5,-.5){\makebox(0,0)[cc]{\hbox{ {$\mathbf 2$}}}}
\rput(6,-.5){\makebox(0,0)[cc]{\hbox{ {$\mathbf 3$}}}}
\rput(7,-.5){\makebox(0,0)[cc]{\hbox{ {$\mathbf 4$}}}}
\rput(8.3,2.){\makebox(0,0)[cc]{\hbox{ {$\mathbf 5$}}}}
\rput(7.4,4.4){\makebox(0,0)[cc]{\hbox{ {$\mathbf 6$}}}}
\rput(6.4,4.4){\makebox(0,0)[cc]{\hbox{ {$\mathbf 7$}}}}
\rput(5.4,4.4){\makebox(0,0)[cc]{\hbox{ {$\mathbf 8$}}}}
\rput(4.4,4.4){\makebox(0,0)[cc]{\hbox{ {$\mathbf 9$}}}}
\rput(3.4,4.4){\makebox(0,0)[cc]{\hbox{ {${\mathbf 10}$}}}}
\rput(2.4,4.4){\makebox(0,0)[cc]{\hbox{ {${\mathbf 11}$}}}}
\rput(1.4,4.4){\makebox(0,0)[cc]{\hbox{ {${\mathbf 12}$}}}}
\rput(-.4,2.){\makebox(0,0)[cc]{\hbox{ {${\mathbf 13}$}}}}
\rput(1,-.5){\makebox(0,0)[cc]{\hbox{ {${\mathbf 14}$}}}}
\rput(2,-.5){\makebox(0,0)[cc]{\hbox{ {${\mathbf 15}$}}}}
\rput(3,-.5){\makebox(0,0)[cc]{\hbox{ {${\mathbf 16}$}}}}

\rput(4.5,.5){\makebox(0,0)[cc]{\hbox{ {$\bold \alpha_1$}}}}
\rput(5.5,.5){\makebox(0,0)[cc]{\hbox{ {$\mathbf \alpha_2$}}}}
\rput(6.5,.5){\makebox(0,0)[cc]{\hbox{ {$\mathbf \alpha_3$}}}}
\rput(7.5,.5){\makebox(0,0)[cc]{\hbox{ {$\mathbf \alpha_4$}}}}
\rput(7.7,2.4){\makebox(0,0)[cc]{\hbox{ {$\mathbf \alpha_5$}}}}
\rput(7,2.4){\makebox(0,0)[cc]{\hbox{ {$\mathbf \alpha_6$}}}}
\rput(6,2.4){\makebox(0,0)[cc]{\hbox{ {$\mathbf \alpha_7$}}}}
\rput(5,2.4){\makebox(0,0)[cc]{\hbox{ {$\mathbf \alpha_8$}}}}
\rput(4,2.4){\makebox(0,0)[cc]{\hbox{ {$\mathbf \alpha_9$}}}}
\rput(3,2.4){\makebox(0,0)[cc]{\hbox{ {$\mathbf \alpha_{10}$}}}}
\rput(2,2.4){\makebox(0,0)[cc]{\hbox{ {$\mathbf \alpha_{11}$}}}}
\rput(0.7,2.4){\makebox(0,0)[cc]{\hbox{ {$\mathbf \alpha_{12}$}}}}
\rput(0.5,.5){\makebox(0,0)[cc]{\hbox{ {$\mathbf \alpha_{13}$}}}}
\rput(1.5,.5){\makebox(0,0)[cc]{\hbox{ {$\mathbf \alpha_{14}$}}}}
\rput(2.5,.5){\makebox(0,0)[cc]{\hbox{ {$\mathbf \alpha_{15}$}}}}
\rput(3.5,.5){\makebox(0,0)[cc]{\hbox{ {$\mathbf \alpha_{16}$}}}}

}
}

\put(-10,1){\NET}
\put(0,1){\NNET}

}
\end{pspicture}
\caption{\small
Change of the orientation of the unequivocal path $P: {\bf 5}\leadsto {\bf 13}$, $1<5<13$.}
\label{fig:network3}
\end{figure}
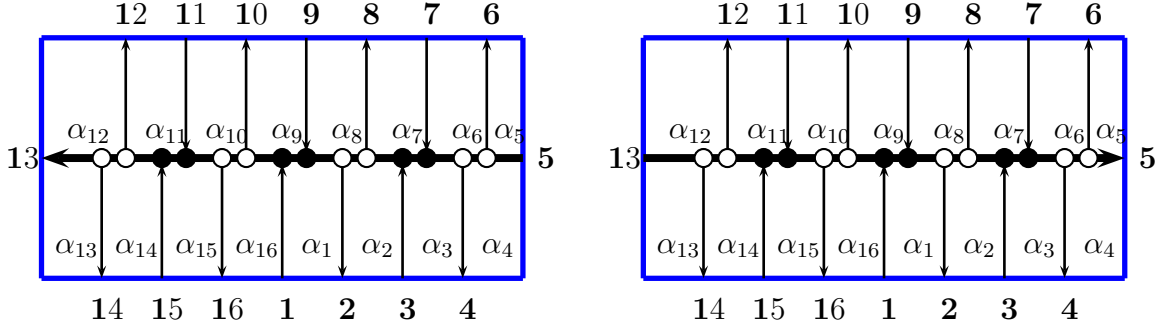

Matrices $Q^{gr}_q$ and $\left(Q^{gr}_q\right)'$ have the following form.

\begin{align*}
Q^{gr}_q &{=}
{\tiny 
\begin{pmatrix}
q^{-1/2} & 0 & 0 & 0 & 0 & 0 & 0 \\
0 & q^{-1}Z_{-\alpha_2} & -q^{-1}Z_{-\sum_{j=2}^4\alpha_j} 
& q^{-1}Z_{-\sum_{j=2}^6\alpha_j} 
& 0 & 0 & 0\\
0 & q^{-3/2} & 0 & 0 & 0 & 0 & 0 \\
0 & 0 & q^{-2} Z_{-\alpha_4} & 0 & 0 & 0 & 0 \\
0 & 0 & q^{-5/2} & 0 & 0 & 0 & 0 \\
0 & 0 & q^{-3} Z_{\alpha_6} & 0 & 0 & 0 & 0 \\
0 & 0 & 0 & q^{-7/2} & 0 & 0 & 0 \\
0 & q^{-4}Z_{\sum_{j=3}^7\alpha_j} 
& -q^{-4} Z_{\sum_{j=5}^7\alpha_j} 
& q^{-4} Z_{\alpha_7} & 0 & 0 & 0 \\
0 & 0 & 0 & 0 & q^{-9/2}  & 0 & 0 \\
q^{-5} Z_{\sum_{j=1}^9\alpha_j}& -q^{-5}Z_{\sum_{j=3}^9\alpha_j} & q^{-5} Z_{\sum_{j=5}^9\alpha_j}  & -q^{-5} Z_{\sum_{j=7}^9\alpha_j} & q^{-5} Z_{\alpha_9} & 0 & 0 \\
0 & 0 & 0 & 0 & 0 & q^{-11/2}  & 0 \\
-q^{-6} Z_{\sum_{j=1}^{11}\alpha_j}& q^{-6} Z_{\sum_{j=3}^{11}\alpha_j} & -q^{-6} Z_{\sum_{j=5}^{11}\alpha_j}  & q^{-6} Z_{\sum_{j=7}^{11}\alpha_j} & -q^{-6} Z_{\sum_{j=9}^{11}\alpha_j} & q^{-6} Z_{\alpha_{11}}  & -q^{-6} Z_{-\sum_{j=12}^{14}\alpha_j}  \\
-q^{-6} Z_{\sum_{j=1}^{12}\alpha_j}& q^{-6} Z_{\sum_{j=3}^{12}\alpha_j} & -q^{-6} Z_{\sum_{j=5}^{12}\alpha_j}  & q^{-6} Z_{\sum_{j=7}^{12}\alpha_j} & -q^{-6} Z_{\sum_{j=9}^{12}\alpha_j} & q^{-6} Z_{\sum_{j=11}^{12}\alpha_j}  & -q^{-6} Z_{-\sum_{j=13}^{14}\alpha_j}  \\
-q^{-6} Z_{\sum_{j=1}^{13}\alpha_j}& q^{-6} Z_{\sum_{j=3}^{13}\alpha_j} & -q^{-6} Z_{\sum_{j=5}^{13}\alpha_j}  & q^{-6} Z_{\sum_{j=7}^{13}\alpha_j} & -q^{-6} Z_{\sum_{j=9}^{13}\alpha_j} & q^{-6} Z_{\sum_{j=11}^{13}\alpha_j}  & -q^{-6} Z_{-\alpha_{14}}  \\
0 & 0 & 0 & 0 & 0 & 0 & q^{-13/2}  \\
q^{-7} Z_{\sum_{j=1}^{15}\alpha_j}& -q^{-7} Z_{\sum_{j=3}^{15}\alpha_j} & q^{-7} Z_{\sum_{j=5}^{15}\alpha_j}  & -q^{-7} Z_{\sum_{j=7}^{15}\alpha_j} & q^{-7} Z_{\sum_{j=9}^{15}\alpha_j} & 0  & 0  \\
\end{pmatrix}
}
\\
\left(Q^{gr}_q\right)' &{=}
{\tiny
\begin{pmatrix}
q^{\frac{-1}{2}}  & 0 & 0 & 0 & 0 & 0 & 0 \\
q^{-1}Z_{\alpha_1} & 0 & 0 &-q^{-1}Z_{-\sum_{j=2}^8\alpha_j} 
& q^{-1}Z_{-\sum_{j=2}^{10}\alpha_j} 
& -q^{-1}Z_{-\sum_{j=2}^{12}\alpha_j} & q^{-1}Z_{-\sum_{j=2}^{14}\alpha_j} \\
0 & q^{-\frac{3}{2}} & 0 & 0 & 0 & 0 & 0 \\
-q^{-2}Z_{\sum_{j=1}^3\alpha_j} & q^{-2} Z_{\alpha_3} & q^{-2} Z_{-\sum_{j=4}^6\alpha_j} 
& -q^{-2} Z_{-\sum_{j=4}^{8}\alpha_j} 
& q^{-2}Z_{-\sum_{j=4}^{10}\alpha_j} & -q^{-2}Z_{-\sum_{j=4}^{12}\alpha_j} & q^{-2}Z_{-\sum_{j=4}^{14}\alpha_j} \\
-q^{-2}Z_{\sum_{j=1}^4\alpha_j} & q^{-2} Z_{\sum_{j=3}^4\alpha_j} & q^{-2} Z_{-\sum_{j=5}^6\alpha_j} 
& -q^{-2} Z_{-\sum_{j=5}^{8}\alpha_j} 
& q^{-2}Z_{-\sum_{j=5}^{10}\alpha_j} & -q^{-2}Z_{-\sum_{j=5}^{12}\alpha_j} & q^{-2}Z_{-\sum_{j=5}^{14}\alpha_j} \\
-q^{-2}Z_{\sum_{j=1}^5\alpha_j} & q^{-2} Z_{\sum_{j=3}^5\alpha_j} & q^{-2} Z_{-\alpha_6} 
& -q^{-2} Z_{-\sum_{j=6}^{8}\alpha_j} 
& q^{-2}Z_{-\sum_{j=6}^{10}\alpha_j} & -q^{-2}Z_{-\sum_{j=6}^{12}\alpha_j} & q^{-2}Z_{-\sum_{j=6}^{14}\alpha_j} \\
0 & 0 & q^{-5/2} & 0 & 0 & 0 & 0 \\
q^{-3}Z_{\sum_{j=1}^{7}\alpha_j}  & 0 & 0 & q^{-3} Z_{-\alpha_{8}} & -q^{-3} Z_{-\sum_{j=8}^{10}\alpha_{j}} & q^{-3}Z_{-\sum_{j=8}^{12}\alpha_j}  & -q^{-3}Z_{-\sum_{j=8}^{14}\alpha_j}  \\
0 & 0 & 0 & q^{-7/2} & 0 & 0 & 0 \\
0  & 0 & 0 & 0 & q^{-4} Z_{-\alpha_{10}} & -q^{-4} Z_{-\sum_{j=10}^{12}\alpha_j}  & q^{-4}Z_{-\sum_{j=10}^{14}\alpha_j}  \\
0 & 0 & 0 & 0 & q^{-9/2}  & 0 & 0 \\
0 & 0 & 0 & 0 & 0 & q^{-5}Z_{-\alpha_{12}} & 0 \\
0 & 0 & 0 & 0 & 0 & q^{-11/2}   & 0 \\
0 & 0 & 0 & 0 & 0 & q^{-6} Z_{\alpha_{13}} & 0 \\
0 & 0 & 0 & 0 & 0 & 0 & q^{-13/2} \\
0 & 0 & 0 & 0 & q^{-7} Z_{\sum_{j=11}^{15} \alpha_j} &- q^{-7} Z_{\sum_{j=13}^{15} \alpha_j } & q^{-7} Z_{\alpha_{15}} \\
\end{pmatrix}
}
\\
C &{=}
{\tiny
\begin{pmatrix}
1 & 0 & 0 & 0 & 0 & 0 & 0 \\
0 & 1 & 0 & 0 & 0 & 0 & 0 \\
-q^{\frac{1}{2}}Z_{\sum_{j=1}^4\alpha_j} & q^{\frac{1}{2}} Z_{\sum_{j=3}^4\alpha_j} & q^{\frac{1}{2}}  Z_{-\sum_{j=5}^6\alpha_j} 
& -q^{\frac{1}{2}}  Z_{-\sum_{j=5}^{8}\alpha_j} 
& q^{\frac{1}{2}} Z_{-\sum_{j=5}^{10}\alpha_j} & -q^{\frac{1}{2}} Z_{-\sum_{j=5}^{12}\alpha_j} & q^{\frac{1}{2}} Z_{-\sum_{j=5}^{14}\alpha_j} \\
0 & 0 & q & 0 & 0 & 0 & 0 \\
0 & 0 & 0 & q & 0 & 0 & 0 \\
0 & 0 & 0 & 0 & q  & 0 & 0 \\
0 & 0 & 0 & 0 & 0 & 0 & 1 \\
\end{pmatrix}
}
\\
\end{align*}

Straightforward checking proves $\left( Q^{gr}_q\right)'=Q^{gr}_q C$ showing that path orientation reversion does not change the quantum grassmannian measurement. 

\end{example}

Consider the network for $\X_{SL_n,\Sigma}$ shown for $n=6$ in Figure~\ref{fi:plabic_weights}. 
The boundary measurement matrix $Q^{gr}_q$ has size
$3n\times n$. The top $n\times n$ part $U=(Q^{gr}_q)_{[1,n]}$ is the diagonal matrix with $j$th diagonal elements $q^{-j+\frac{1}{2}}$;
the middle part of the quantum grassmannian matrix $(Q^{gr}_q)_{[n+1,2n]}=q^{-n} \M_1 S$; and the bottom part $(Q^{gr}_q)_{[2n+1,3n]}=q^{-n}  \M_2 S$.

Let's change the orientation of all snakelike right to left horizontal paths of the network Fig.~\ref{fi:plabic_weights}.  The result is shown on the Fig.~\ref{fig:n6}.

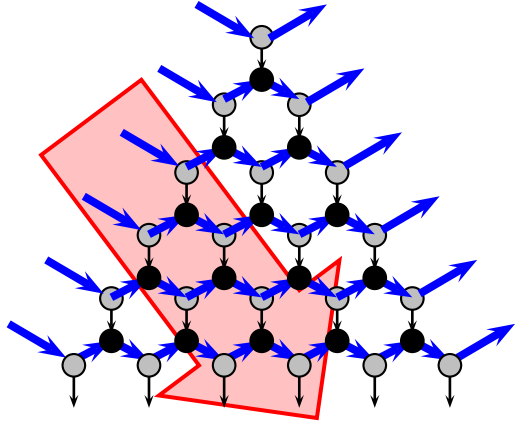
\begin{figure}[H]
	\begin{pspicture}(-3,-2)(4,4){
		%
		\psBigArrow[fillstyle=solid,opacity=0.8,fillcolor=red!30,linecolor=red,linewidth=1.5pt](-2.0,2)(1.,-2)

		\newcommand{\LEFTDOWNARROW}{%
			{\psset{unit=1}
				\rput(0,0){\psline[linecolor=blue,linewidth=3pt]{->}(0,0)(.765,.45)}
		}}
		\newcommand{\DOWNARROW}{%
	{\psset{unit=1}
					\rput(0,0){\psline[linecolor=black,linewidth=1pt]{->}(0,0.1)(0,-0.566)}
		\put(0,0){\pscircle[fillstyle=solid,fillcolor=lightgray]{.15}}
}}
		\newcommand{\LEFTUPARROW}{%
	{\psset{unit=1}
		\rput(0,0){\psline[linecolor=blue,linewidth=3pt]{<-}(0,0)(-.765,.45)}
}}
	\newcommand{\STARUP}{
			{\psset{unit=1}
	\rput(0,0){\psline[linecolor=blue,linewidth=3pt]{->}(0,0)(.5,-.26)}
	\rput(0,0){\psline[linecolor=black,linewidth=1pt]{<-}(0,0.1)(0,.466)}
	\rput(0,0){\psline[linecolor=blue,linewidth=3pt]{<-}(0,0)(-.5,-.26)}
	\put(0,0){\pscircle[fillstyle=solid,fillcolor=black]{.15}}
	\put(0,.566){\pscircle[fillstyle=solid,fillcolor=lightgray]{.15}}
}}
		\newcommand{\PATGEN}{%
			{\psset{unit=1}
				\rput(0,0){\psline[linecolor=blue,linewidth=2pt]{->}(0,0)(.45,.765)}
				\rput(0,0){\psline[linecolor=blue,linewidth=2pt]{->}(1,0)(0.1,0)}
				\rput(0,0){\psline[linecolor=blue,linewidth=2pt]{->}(0,0)(.45,-.765)}
				\put(0,0){\pscircle[fillstyle=solid,fillcolor=lightgray]{.1}}
		}}
		\newcommand{\PATLEFT}{%
			{\psset{unit=1}
				\rput(0,0){\psline[linecolor=blue,linewidth=2pt,linestyle=dashed]{->}(0,0)(.45,.765)}
				\rput(0,0){\psline[linecolor=blue,linewidth=2pt]{->}(1,0)(0.1,0)}
				\rput(0,0){\psline[linecolor=blue,linewidth=2pt]{->}(0,0)(.45,-.765)}
				\put(0,0){\pscircle[fillstyle=solid,fillcolor=lightgray]{.1}}
		}}
		\newcommand{\PATRIGHT}{%
			{\psset{unit=1}
				\rput(0,0){\psline[linecolor=blue,linewidth=2pt,linestyle=dashed]{->}(0,0)(.45,-.765)}
				\put(0,0){\pscircle[fillstyle=solid,fillcolor=lightgray]{.1}}
		}}
		\newcommand{\PATBOTTOM}{%
			{\psset{unit=1}
				\rput(0,0){\psline[linecolor=blue,linewidth=2pt]{->}(0,0)(.45,.765)}
				\rput(0,0){\psline[linecolor=blue,linewidth=2pt,linestyle=dashed]{->}(1,0)(0.1,0)}
				\put(0,0){\pscircle[fillstyle=solid,fillcolor=lightgray]{.1}}
		}}
		\newcommand{\PATTOP}{%
			{\psset{unit=1}
				\rput(0,0){\psline[linecolor=blue,linewidth=2pt]{->}(1,0)(0.1,0)}
				\rput(0,0){\psline[linecolor=blue,linewidth=2pt]{->}(0,0)(.45,-.765)}
				\put(0,0){\pscircle[fillstyle=solid,fillcolor=lightgray]{.1}}
		}}
		\newcommand{\PATBOTRIGHT}{%
			{\psset{unit=1}
				\rput(0,0){\psline[linecolor=blue,linewidth=2pt]{->}(0,0)(.45,.765)}
				\put(0,0){\pscircle[fillstyle=solid,fillcolor=lightgray]{.1}}
				\put(.5,0.85){\pscircle[fillstyle=solid,fillcolor=lightgray]{.1}}
		}}
		\multiput(-2,-1.176)(1.0,0){5}{\STARUP}
		\multiput(-1.5,-0.335)(1.0,0){4}{\STARUP}
		\multiput(-1.0,0.5)(1.0,0){3}{\STARUP}
		\multiput(-.5,1.4)(1.0,0){2}{\STARUP}
		\put(0,2.3){\STARUP}
		\multiput(2.6,-1.4)(-0.5,.85){6}{\LEFTDOWNARROW}
		\multiput(-2.6,-1.4)(0.5,.85){6}{\LEFTUPARROW}
		\multiput(-2.5,-1.5)(1.0,0){6}{\DOWNARROW}
		%
		
	}
	\end{pspicture}
	\caption{\small
		This network is obtained by the simultaneous change of orientations of the snakelike horizontal bold paths (colored blue). The big arrow shows the direction of non-normalized transport matrix $M_3$.
	}
	\label{fig:n6}
\end{figure}

The orientation change of all bold paths (see Figure~\ref{fig:n6}) leads to the new quantum grassmannian measurement $\left(Q^{gr}_q\right)'$.
Its middle part is the submatrix $(\left(Q^{gr}_q\right)'_{[n+1,2n]}=U$, the bottom $n\times n$ part $(\left(Q^{gr}_q\right)'_{[2n+1,3n]}=q^{-n}\M_3 S$. 
Using the fact that $Q^{gr}_q$ and $\left(Q^{gr}_q\right)'$ represent the same quantum grassmann element, we obtain
$q^{-n}\M_1 S C=U$, $q^{-n}\M_2 S C=q^{-n} \M_3 S$. Find $C$ from the first equation: $C=S^{-1}\M_1^{-1}q^n U$. Substituting the expression for $C$ into second equation we obtain $\M_2=\M_3SU^{-1}q^{-n}\M_1$. Note that $US=SU^{-1}q^{-n}$, then $US\M_2=(US\M_3)(US\M_1)$.
We conclude that $M_2=M_3 M_1$.

This is clearly equivalent to the second part of Theorem~\ref{th:MMsquare}
 $$T_1 T_2 T_3 =1.$$

%
%
%
%
%

Commutation relations between face weights induce $R$-matrix commutation relations between entries of $Q_q$.

Namely, the next lemma describes the commutation relation between elements of $Q_q$.
\begin{lemma}\label{lem:r-matrix}
$R_m \sheet{1}{Q_q}\otimes\sheet{2}{Q_q}=\sheet{2}{Q_q}\otimes\sheet{1}{Q_q} R_n$, where $R_m, R_n$ are given by formula ~\ref{R-matrix1}.
\end{lemma}  

\begin{proof} We will prove this statement using factorization of matrix $Q_q$ into a product of elementary matrices. Let $N_1$ be a network in rectangle with $m$ sinks on the left and $n$ sources on the right. $N_1$ can be presented as a concatenation of elementary networks of two special kinds.

\begin{figure}[H]
\begin{pspicture}(0,0)(6,3){\psset{unit=0.8}
\newcommand{\NETA}{%
{\psset{unit=1}
{\psline[linecolor=blue,linestyle=dashed,linewidth=2pt]{-}(0,0)(2.,0.)}
{\psline[linecolor=blue,linestyle=dashed,linewidth=2pt]{-}(2,0)(2.,3.)}
{\psline[linecolor=blue,linestyle=dashed,linewidth=2pt]{-}(2,3)(0,3)}
{\psline[linecolor=blue,linestyle=dashed,linewidth=2pt]{-}(0,3)(0,0)}
{\psline[linecolor=black,linewidth=2pt]{<-}(0,1)(2,1)}
{\psline[linecolor=black,linewidth=2pt]{<-}(0,1.5)(1,1.5)}
{\psline[linecolor=black,linewidth=2pt]{<-}(1,1.5)(2,1.5)}
{\psline[linecolor=black,linewidth=2pt]{<-}(0,2)(1,1.5)}
{\psline[linecolor=black,linewidth=2pt]{<-}(0,2.5)(2,2.5)}
}
}
\newcommand{\NETB}{%
{\psset{unit=1}
{\psline[linecolor=blue,linestyle=dashed,linewidth=2pt]{-}(0,0)(2.,0.)}
{\psline[linecolor=blue,linestyle=dashed,linewidth=2pt]{-}(2,0)(2,3.)}
{\psline[linecolor=blue,linestyle=dashed,linewidth=2pt]{-}(2,3)(0,3)}
{\psline[linecolor=blue,linestyle=dashed,linewidth=2pt]{-}(0,3)(0,0)}
{\psline[linecolor=black,linewidth=2pt]{<-}(0,1)(2,1)}
{\psline[linecolor=black,linewidth=2pt]{<-}(0,1.5)(1,1.5)}
{\psline[linecolor=black,linewidth=2pt]{<-}(1.1,1.5)(2,1.5)}
{\psline[linecolor=black,linewidth=2pt]{<-}(1.05,1.55)(2,2)}
{\psline[linecolor=black,linewidth=2pt]{<-}(0,2.5)(2,2.5)}
}
}
\put(0,0){\NETA}
\put(4,0){\NETB}
}
\end{pspicture}
\caption{\small Elementary forks}
\label{fig:elementaryfork}
\end{figure}
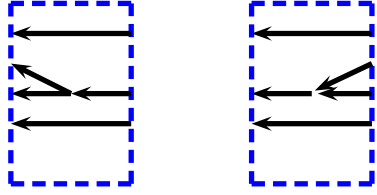

Let $N$ be a network obtained by adding the fork to $j$th sink, $j\in [1,m]$ (see Figure~\ref{fig:addfork}).

\begin{figure}[H]
\begin{pspicture}(-5,-1)(10,4){\psset{unit=0.8}

{\psline[linecolor=blue,linestyle=dashed,linewidth=2pt]{-}(3,0)(6.,0.)}
{\psline[linecolor=blue,linestyle=dashed,linewidth=2pt]{-}(6,0)(6.,3.2)}
{\psline[linecolor=blue,linestyle=dashed,linewidth=2pt]{-}(6,3.2)(3,3.2)}
{\psline[linecolor=blue,linestyle=dashed,linewidth=2pt]{-}(3,3.2)(3,0)}
{\psline[linecolor=black,linewidth=2pt]{<-}(3,0.7)(5,0.7)}
{\psline[linecolor=black,linewidth=2pt]{<-}(5,0.7)(6,0.7)}
{\psline[linecolor=black,linewidth=2pt]{<-}(4,2)(4.5,2)}
{\psline[linecolor=black,linewidth=2pt]{<-}(4.5,2)(6,2)}
{\psline[linecolor=black,linewidth=2pt]{<-}(3,1.5)(4,2)}
{\psline[linecolor=black,linewidth=2pt]{<-}(3,2.5)(4,2)}
{\psline[linecolor=black,linewidth=2pt]{->}(5,0.7)(4.5,2)}

{\psline[linecolor=blue,linestyle=dashed,linewidth=2pt]{-}(0,0)(2.,0.)}
{\psline[linecolor=blue,linestyle=dashed,linewidth=2pt]{-}(2,0)(2.,3.2)}
{\psline[linecolor=blue,linestyle=dashed,linewidth=2pt]{-}(2,3.2)(0,3.2)}
{\psline[linecolor=blue,linestyle=dashed,linewidth=2pt]{-}(0,3.2)(0,0)}
{\psline[linecolor=black,linewidth=2pt]{<-}(0,0.7)(2,0.7)}
{\psline[linecolor=black,linewidth=2pt]{<-}(1,1.5)(2,1.5)}
{\psline[linecolor=black,linewidth=2pt]{<-}(0,1.5)(1,1.5)}
{\psline[linecolor=black,linewidth=2pt]{<-}(0,2.2)(1,1.5)}
{\psline[linecolor=black,linewidth=2pt]{<-}(0,2.5)(2,2.5)}
\rput(1.5,2.8){\makebox(0,0)[cc]{\hbox{{$\mathbf  f_1$}}}}
\rput(1.5,2.3){\makebox(0,0)[cc]{\hbox{{$\mathbf  \dots $}}}}
\rput(1.5,1.9){\makebox(0,0)[cc]{\hbox{{$\mathbf  f_2 $}}}}
\rput(0.3,1.75){\makebox(0,0)[cc]{\hbox{{$\mathbf  a_2 $}}}}
\rput(1.5,1.2){\makebox(0,0)[cc]{\hbox{{$\mathbf  f_3 $}}}}
\rput(1.5,0.4){\makebox(0,0)[cc]{\hbox{{$\mathbf  f_4 $}}}}

}
\end{pspicture}
\caption{\small Adding fork to the second sink. Note, $\col{a_2 f_1 f_2 \dots f_4}=1$.}
\label{fig:addfork}
\end{figure}
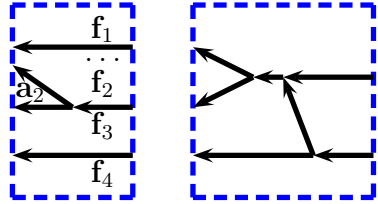

Using the construction above, we write that, $Q_q=\prod_i X_i$, where $X_i=L_i$ or $X_i=U_i$ is an  $m_i\times m_{i+1}$ matrix  where ($m_{i+1}=m_i-1$ in the first case or $m_{i+1}=m_i+1$ in the second.) 
$$L_i=\begin{pmatrix}
t_1 & 0 & \dots & \dots &  0 & 0 & \dots\\
0 & t_2 & \dots & \dots  &   0 & 0 & \dots\\
\vdots & & \ddots &     &  0 & 0 & \dots \\
\hline 
0& &     \dots & t_i & 0 & 0 & \dots \\
0 & &  \dots   & \col{t_i Z_{a_i}} & 0 & 0 & \dots \\
\hline 
0 &  & \dots  & 0  & t_{i+1} & 0 & \dots \\
 &  & \dots  &   &  &  & \ddots \\
\end{pmatrix}, \quad t_1=\col{Z_{f_1}}, t_2=\col{Z_{f_1} Z_{f_2}}, \text{etc}
$$
and 
$$U_i=\begin{pmatrix}
t_1 & 0 & \dots & \dots &  0 & 0 & \dots\\
0 & t_2 & \dots & \dots  &   0 & 0 & \dots\\
\vdots & & \ddots &     &  0 & 0 & \dots \\
\hline 
0& &     \dots & \col{t_i} & \col{t_i Z_{a_i}} & 0 & \dots \\
\hline 
0 & &  \dots   & 0 & 0 & t_{i+1} & \dots \\
 &  & \dots  &   &  &  & \ddots \\
\end{pmatrix}, \quad t_1=\col{Z_{f_1}}, t_2=\col{Z_{f_1} Z_{f_2}}, \text{etc}.
$$

Variables $t_j $ commute with each other and commute with $Z_{a_i}$ unless $j=i$. Otherwise,  $Z_{a_i} t_i=q t_i Z_{a_i}$, $\col{Z_{a_i} t_i}=q^{-1/2}Z_{a_i} t_i$. All entries of $X_i$  commute with all entries of $X_j$ for $i\ne j$.

Clearly, it is enough to check the relations for $2\times 1$ and $1\times 2$ matrices. 

For $m_i=1$, we have $R_{m_i}\sheet{1}{L_i}\otimes\sheet{2}{L_i}=\sheet{2}{L_i}\otimes\sheet{1}{L_i}R_{m_{i+1}}$, 
$R_{m_i}\sheet{1}{U_i}\otimes\sheet{2}{U_i}=\sheet{2}{U_i}\otimes\sheet{1}{U_i}R_{m_{i+1}}$.
Indeed,
$$\begin{pmatrix} q & 0 & 0 & 0\\ 0 & 1 & 0 & 0\\ 0 & q-q^{-1} & 1 & 0 \\ 0 & 0 & 0 & q\end{pmatrix}\left[
\sheet{1}{\begin{pmatrix} a \\ b\end{pmatrix}}\otimes \sheet{2}{\begin{pmatrix} a\\ b\end{pmatrix}}\right]=
\begin{pmatrix} q & 0 & 0 & 0\\ 0 & 1 & 0 & 0\\ 0 & q-q^{-1} & 1 & 0 \\ 0 & 0 & 0 & q\end{pmatrix}\begin{pmatrix} a^2\\a b\\b a\\b^2\end{pmatrix}=$$
$$=\begin{pmatrix} qa^2\\qba\\qab\\qb^2\end{pmatrix}=\left[\sheet{2}{\begin{pmatrix} a\\b\end{pmatrix}}\otimes\sheet{1}{\begin{pmatrix} a\\ b\end{pmatrix}}\right]\cdot\begin{pmatrix}q\end{pmatrix}
$$

Then, $R_{m_1}\sheet{1}{Q_q}\otimes\sheet{2}{Q_q}=R_{m_1} \prod^n_{i=1}\sheet{1}{X_i}\otimes \prod_{i=1}^n \sheet{2}{X_i}=R_{m_1} \sheet{1} {X_1}\prod_{i=2}^n\sheet{1}{X_i}\otimes \sheet{2}{X_1}\prod_{i=2}^n \sheet{2}{X_i}=\sheet{2}{X_1}\otimes\sheet{1}{X_1}R_{m_2}\prod_{i=2}^n \sheet{1}{X_i}\otimes\prod_{i+2}^n \sheet{2}{X_i}=\dots=\prod_{i=1}^n \sheet{2}{X_i}\otimes\prod_{i=1}^n \sheet{1}{X_i}R_{m_n}=\sheet{2}{Q_q}\otimes\sheet{1}{Q_q}R_{m_n}$.
\end{proof}

\begin{corollary} Theorem~\ref{th:MM}  and Remark~\ref{rem:M-M} follow from Lemma~\ref{lem:r-matrix}.
\end{corollary}

\begin{proof} Indeed, it is enough to consider $2n\times n$ matrix $B$ of boundary measurements of the network shown on Figure~\ref{fi:plabic_weights}.
Denote by $\M_1$ the top $n\times n$ block of $Q_q$, $\M_2$ is the bottom $n\times n$ block. $Q_q$ satisfies $R$-matrix relation \ref{lem:r-matrix}.
Choose subset of these relations between tensor products of elements of $\M_1$ and $\M_2$ we obtain Theorem~\ref{th:MM}.
Relations between tensor products of elements of each matrix $\M_i$ give Remark~\ref{rem:M-M}.
\end{proof}

%
%
%
%
%
%
%
%
%

\section{Directed networks with cycles}\label{s:cycles}

In this section we generalize Lemma~\ref{lem:r-matrix} to the case of planar networks containing oriented cycles and interlacing sources and sinks.  $R$-matrix formulation of commutation relation of elements of transport matrices is not valid for more general networks with interlacing sources and sinks. The corresponding statement is formulated in Theorem~\ref{thm:qnetwork}.  For the case of networks with separated sources and sinks
commutation relation of Theorem~\ref{thm:qnetwork}   coincide with those of Lemma~\ref{lem:r-matrix}.

\begin{definition}
We assign to every oriented path $P:j\leadsto i$ from a source $j$  to a sink $i$  the {\sl quantum weight}
$$
w(P)=\col{\mathop{\prod_{\text{ faces }\alpha\text{ lie to the right}}}_{\text{ of the path } P} Z_\alpha},
$$
where the product is taken with repetitions, 
\end{definition}

\begin{definition}
For {\bf any} planar directed network $\mathcal N$, define {\sl transport elements}
$$
(\alpha,a):=\sum_{\text{all paths }\alpha\leadsto a }(-1)^{\#\text{\,self-intersections}} w(P_{\alpha\leadsto a})
$$
where the sum ranges all paths from the source $j$ to the sink $i$. This sum is finite for acyclic networks and can be infinite for networks containing cycles. 
In this section, we let Greek letters denote sources and Latin letters denote sinks. We draw these transport elements as simple directed paths $\alpha\to a$.
\end{definition}
\begin{theorem}\label{thm:qnetwork}
For any planar network, we have the algebra of transport elements:
\begin{align*}
&
\raisebox{-15pt}{\mbox{
{\psset{unit=0.2}
\begin{pspicture}(-4,-3)(4,3)
\pscircle[linecolor=red,linestyle=dashed](0,0){3}
\psarc[linewidth=2pt,linecolor=blue]{->}(0,-6){5.2}{60}{120}
\psarc[linewidth=2pt,linecolor=blue]{<-}(0,6){5.2}{240}{300}
\put(3,2){\makebox(0,0)[lc]{\hbox{{$\alpha$}}}}
\put(3,-2){\makebox(0,0)[lc]{\hbox{{$\beta$}}}}
\put(-3,2){\makebox(0,0)[rc]{\hbox{{$a$}}}}
\put(-3,-2){\makebox(0,0)[rc]{\hbox{{$b$}}}}
\end{pspicture}
}}}
\quad 
[(\alpha,a),(\beta,b)]=(q-q^{-1})(\alpha,b)(\beta,a);
\\
&\raisebox{-15pt}{\mbox{
{\psset{unit=0.2}
\begin{pspicture}(-4,-3)(4,3)
\pscircle[linecolor=red,linestyle=dashed](0,0){3}
\psline[linewidth=2pt,linecolor=blue]{->}(2.8,1)(-2.8,-1)
\psline[linewidth=4pt,linecolor=white](2.8,-1)(-2.8,1)
\psline[linewidth=2pt,linecolor=blue]{->}(2.8,-1)(-2.8,1)
\put(3,2){\makebox(0,0)[lc]{\hbox{{$\alpha$}}}}
\put(3,-2){\makebox(0,0)[lc]{\hbox{{$\beta$}}}}
\put(-3,2){\makebox(0,0)[rc]{\hbox{{$a$}}}}
\put(-3,-2){\makebox(0,0)[rc]{\hbox{{$b$}}}}
\end{pspicture}
}}}
\quad
[(\alpha,b),(\beta,a)]=0,
\qquad
\raisebox{-15pt}{\mbox{
{\psset{unit=0.2}
\begin{pspicture}(-4,-3)(4,3)
\pscircle[linecolor=red,linestyle=dashed](0,0){3}
\psarc[linewidth=2pt,linecolor=blue]{->}(0,-6){5.2}{60}{120}
\psarc[linewidth=2pt,linecolor=blue]{->}(0,6){5.2}{240}{300}
\put(3,2){\makebox(0,0)[lc]{\hbox{{$a$}}}}
\put(3,-2){\makebox(0,0)[lc]{\hbox{{$\beta$}}}}
\put(-3,2){\makebox(0,0)[rc]{\hbox{{$\alpha$}}}}
\put(-3,-2){\makebox(0,0)[rc]{\hbox{{$b$}}}}
\end{pspicture}
}}}
\quad
[(\alpha,a),(\beta,b)]=0,
\\
&\raisebox{-15pt}{\mbox{
{\psset{unit=0.2}
\begin{pspicture}(-4,-3)(4,3)
\pscircle[linecolor=red,linestyle=dashed](0,0){3}
\rput{30}(0,0){\psarc[linewidth=2pt,linecolor=blue]{->}(0,-6){5.2}{60}{120}}
\rput{-30}(0,0){\psarc[linewidth=2pt,linecolor=blue]{<-}(0,6){5.2}{240}{300}}
\put(3.2,0){\makebox(0,0)[lc]{\hbox{{$\alpha$}}}}
\put(-2,2.5){\makebox(0,0)[rb]{\hbox{{$a$}}}}
\put(-2,-2.5){\makebox(0,0)[rt]{\hbox{{$b$}}}}
\end{pspicture}
}}}
\quad 
(\alpha,a)(\alpha,b)=q^{-1}(\alpha,b)(\alpha,a),
\qquad
\raisebox{-15pt}{\mbox{
{\psset{unit=0.2}
\begin{pspicture}(-4,-3)(4,3)
\pscircle[linecolor=red,linestyle=dashed](0,0){3}
\rput{-30}(0,0){\psarc[linewidth=2pt,linecolor=blue]{->}(0,-6){5.2}{60}{120}}
\rput{30}(0,0){\psarc[linewidth=2pt,linecolor=blue]{<-}(0,6){5.2}{240}{300}}
\put(-3.2,0){\makebox(0,0)[rc]{\hbox{{$a$}}}}
\put(2,2.5){\makebox(0,0)[lb]{\hbox{{$\alpha$}}}}
\put(2,-2.5){\makebox(0,0)[lt]{\hbox{{$\beta$}}}}
\end{pspicture}
}}}
\quad
(\beta,a)(\alpha,a)=q^{-1}(\alpha,a)(\beta,a).
\end{align*}
\end{theorem}

For acyclic networks these theorem was proven above; we now consider the case of network with cycles. The proof will be by induction. We treat in details only cases with four distinct sources and sinks (the first three cases in the theorem). 

We consider all possible cases corresponding to the situation in which we close the sink $a$ and the (neighbour) source $\alpha$. For a path $(\beta,b)$ we have two possibilities: 
$$
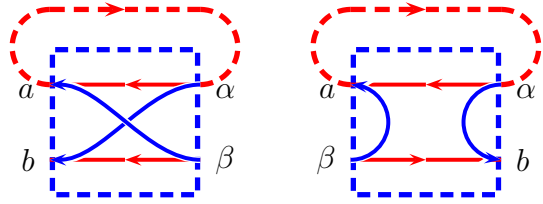
\begin{pspicture}(-2,-1)(2,2.5)
\psframe[linewidth=2pt,linecolor=blue,linestyle=dashed](-1,-0.5)(1,1.5)
\psarc[linewidth=2pt,linecolor=red,linestyle=dashed](-1,1.5){0.5}{90}{270}
\psarc[linewidth=2pt,linecolor=red,linestyle=dashed](1,1.5){0.5}{-90}{90}
\psline[linewidth=2pt,linecolor=red,linestyle=dashed]{->}(-1,2)(0,2)
\psline[linewidth=2pt,linecolor=red,linestyle=dashed](0,2)(1,2)
\psline[linewidth=1.5pt,linecolor=red,linestyle=solid](-1,1)(0,1)
\psline[linewidth=1.5pt,linecolor=red,linestyle=solid]{<-}(0,1)(1,1)
\psline[linewidth=1.5pt,linecolor=red,linestyle=solid]{<-}(0,0)(1,0)
\psline[linewidth=1.5pt,linecolor=red,linestyle=solid](-1,0)(0,0)
\psbezier[linecolor=white,linewidth=2.5pt]{<-}(-1,0)(-0.3,0)(0.3,1)(1,1)
\psbezier[linecolor=blue,linewidth=1.5pt]{<-}(-1,0)(-0.3,0)(0.3,1)(1,1)
\psbezier[linecolor=white,linewidth=2.5pt]{<-}(-1,1)(-0.3,1)(0.3,0)(1,0)
\psbezier[linecolor=blue,linewidth=1.5pt]{<-}(-1,1)(-0.3,1)(0.3,0)(1,0)
\put(1.2,1){\makebox(0,0)[lt]{\hbox{{$\alpha$}}}}
\put(1.2,0){\makebox(0,0)[lc]{\hbox{{$\beta$}}}}
\put(-1.2,1){\makebox(0,0)[rt]{\hbox{{$a$}}}}
\put(-1.2,0){\makebox(0,0)[rc]{\hbox{{$b$}}}}
\end{pspicture}
\begin{pspicture}(-2,-1)(2,2.5)
\psframe[linewidth=2pt,linecolor=blue,linestyle=dashed](-1,-0.5)(1,1.5)
\psarc[linewidth=2pt,linecolor=red,linestyle=dashed](-1,1.5){0.5}{90}{270}
\psarc[linewidth=2pt,linecolor=red,linestyle=dashed](1,1.5){0.5}{-90}{90}
\psline[linewidth=2pt,linecolor=red,linestyle=dashed]{->}(-1,2)(0,2)
\psline[linewidth=2pt,linecolor=red,linestyle=dashed](0,2)(1,2)
\psline[linewidth=1.5pt,linecolor=red,linestyle=solid](-1,1)(0,1)
\psline[linewidth=1.5pt,linecolor=red,linestyle=solid]{<-}(0,1)(1,1)
\psline[linewidth=1.5pt,linecolor=red,linestyle=solid](0,0)(1,0)
\psline[linewidth=1.5pt,linecolor=red,linestyle=solid]{->}(-1,0)(0,0)
\psarc[linecolor=white,linewidth=2.5pt]{->}(-1,0.5){0.5}{-90}{90}
\psarc[linecolor=blue,linewidth=1.5pt]{->}(-1,0.5){0.5}{-90}{90}
\psarc[linecolor=white,linewidth=2.5pt]{->}(1,0.5){0.5}{90}{270}
\psarc[linecolor=blue,linewidth=1.5pt]{->}(1,0.5){0.5}{90}{270}
\put(1.2,1){\makebox(0,0)[lt]{\hbox{{$\alpha$}}}}
\put(1.2,0){\makebox(0,0)[lc]{\hbox{{$b$}}}}
\put(-1.2,1){\makebox(0,0)[rt]{\hbox{{$a$}}}}
\put(-1.2,0){\makebox(0,0)[rc]{\hbox{{$\beta$}}}}
\end{pspicture}
$$

We begin with observation that in both these cases,
$$
\col{(\alpha,b)(\alpha,a)^n(\beta,a)}=(\alpha,b)(\alpha,a)^n(\beta,a)=(\beta,a)(\alpha,a)^n(\alpha,b),
$$
where on the right we assume the natural order of the product of operators. The effect of closing the line between $a$ and $\alpha$ changes the transport element from $\beta$ to $b$:
in the respective cases, we have
$$
\overline{(\beta,b)}=(\beta,b)-(\beta,a)\frac{1}{1+(\alpha,a)}(\alpha,b),\ \hbox{and}\ \overline{(\beta,b)}=(\beta,b)+(\beta,a)\frac{1}{1+(\alpha,a)}(\alpha,b).
$$
Here and hereafter, we understand rational expressions as geometrical-progression expansions in powers of the corresponding operator. We also use the standard commutation relation formulas
$$
\Bigl[ A,\frac{1}{1+B}\Bigr]=-\frac{1}{1+B}[A,B]\frac{1}{1+B}\quad \forall A,B,
$$
and use the color graphics to indicate permutations of operators in formulas of this section: a pair of operators painted red produces the factor $q$ upon permuting these operators in the operatorial product, and a pair of operators painted blue produces a factor $q^{-1}$ upon the corresponding permutation; pairs of operators painted magenta commute.

 Below we have six cases of mutual distribution of sources $\{\alpha,\beta,\gamma\}$ and sinks $\{a,b,c\}$  (Note, that planarity condition requires $\alpha$ and $a$ always to be neighbour), and in each such case we have two choices of transport elements: $\{\overline{(\beta,b)},\overline{(\gamma,c)}\}$ and $\{\overline{(\beta,c)},\overline{(\gamma,b)}\}$, so, altogether, we have 12 variants to be checked.

{\bf Case 1}. 
\raisebox{-25pt}{\mbox{
{\psset{unit=0.7}
\begin{pspicture}(-2,-1.5)(2,2.5)
\psframe[linewidth=2pt,linecolor=blue,linestyle=dashed](-1,-1.5)(1,1.5)
\psarc[linewidth=2pt,linecolor=red,linestyle=dashed](-1,1.5){0.5}{90}{270}
\psarc[linewidth=2pt,linecolor=red,linestyle=dashed](1,1.5){0.5}{-90}{90}
\psline[linewidth=2pt,linecolor=red,linestyle=dashed]{->}(-1,2)(0,2)
\psline[linewidth=2pt,linecolor=red,linestyle=dashed](0,2)(1,2)
\psline[linewidth=1.5pt,linecolor=red,linestyle=solid]{<-}(-1,1)(-0.5,1)
\psline[linewidth=1.5pt,linecolor=red,linestyle=solid]{<-}(-1,0)(-0.5,0)
\psline[linewidth=1.5pt,linecolor=red,linestyle=solid]{<-}(-1,-1)(-0.5,-1)
\psline[linewidth=1.5pt,linecolor=red,linestyle=solid]{<-}(0.5,1)(1,1)
\psline[linewidth=1.5pt,linecolor=red,linestyle=solid]{<-}(0.5,0)(1,0)
\psline[linewidth=1.5pt,linecolor=red,linestyle=solid]{<-}(0.5,-1)(1,-1)
\put(1.2,1){\makebox(0,0)[lt]{\hbox{{$\alpha$}}}}
\put(1.2,0){\makebox(0,0)[lc]{\hbox{{$\beta$}}}}
\put(1.2,-1){\makebox(0,0)[lc]{\hbox{{$\gamma$}}}}
\put(-1.2,1){\makebox(0,0)[rt]{\hbox{{$a$}}}}
\put(-1.2,0){\makebox(0,0)[rc]{\hbox{{$b$}}}}
\put(-1.2,-1){\makebox(0,0)[rc]{\hbox{{$c$}}}}
\end{pspicture}
}}}
Variant (a):  $\{\overline{(\beta,b)},\overline{(\gamma,c)}\}$.
\begin{align*}
&[\overline{(\beta,b)},\overline{(\gamma,c)}]=\left[ \Bigl((\beta,b)-(\beta,a)\frac{1}{1+(\alpha,a)}(\alpha,b) \Bigr), \Bigl((\gamma,c)-(\gamma,a)\frac{1}{1+(\alpha,a)}(\alpha,c) \Bigr) \right]\\
=&(q-q^{-1})(\beta,c)(\gamma,b)+\tcb{(\gamma,a)}\frac{-1}{1+(\alpha,a)}(q-q^{-1})\tcb{(\beta,a)}\tcr{(\alpha,b)}\frac{1}{1+(\alpha,a)}\tcr{(\alpha,c)}\\
&-(q-q^{-1})(\beta,c)(\gamma,a)\frac{1}{1+(\alpha,a)}(\alpha,b)+(\beta,a)\frac{1}{1+(\alpha,a)}(q-q^{-1})(\alpha,c)(\gamma,a)\frac{1}{1+(\alpha,a)}(\alpha,b)\\
&-(q-q^{-1})(\beta,a)\frac{1}{1+(\alpha,a)}(\alpha,c)(\gamma,b)+(\beta,a)\frac{1}{1+(\alpha,a)}\tcr{(\alpha,b)}(\gamma,a)\frac{1}{1+(\alpha,a)}\tcr{(\alpha,c)}\\
&-\tcb{(\gamma,a)}\frac{1}{1+(\alpha,a)}(\alpha,c)\tcb{(\beta,a)}\frac{1}{1+(\alpha,a)}(\alpha,b)\\
=&(q-q^{-1})\Bigl[(\beta,c)(\gamma,b)- (\beta,c)(\gamma,a)\frac{1}{1+(\alpha,a)}(\alpha,b) -(\beta,a)\frac{1}{1+(\alpha,a)}(\alpha,c)(\gamma,b) \Bigr. \\
&+\Bigl.  (\beta,a)\frac{1}{1+(\alpha,a)}(\alpha,c)(\gamma,a)\frac{1}{1+(\alpha,a)}(\alpha,b)\Bigr]\\
&+\bigl((q-q^{-1})-q+q^{-1}\bigr)(\beta,a)\frac{1}{1+(\alpha,a)}(\alpha,c)(\gamma,a)\frac{1}{1+(\alpha,a)}(\alpha,b)\\
=&(q-q^{-1})\overline{(\beta,c)}\overline{(\gamma,b)}.
\end{align*}
Variant (b):  $\{\overline{(\beta,c)},\overline{(\gamma,b)}\}$.
\begin{align*}
&[\overline{(\beta,c)},\overline{(\gamma,b)}]=\left[ \Bigl((\beta,c)-(\beta,a)\frac{1}{1+(\alpha,a)}(\alpha,c) \Bigr), \Bigl((\gamma,b)-(\gamma,a)\frac{1}{1+(\alpha,a)}(\alpha,b) \Bigr) \right]\\
=&(\gamma,a)\frac{-1}{1+(\alpha,a)}(q-q^{-1})(\beta,a)(\alpha,c)\frac{1}{1+(\alpha,a)}(\alpha,b)  +(q-q^{-1})(\gamma,a)\frac{1}{1+(\alpha,a)}(\beta,b)(\alpha,c)\\
&-(q-q^{-1}) (\beta,b)(\gamma,a)\frac{1}{1+(\alpha,a)}(\alpha,c)- (q-q^{-1})  (\beta,b)(\gamma,a)\frac{1}{1+(\alpha,a)}(\alpha,c)\\ 
&-(q-q^{-1})(\beta,a)\frac{-1}{1+(\alpha,a)}\tcm{(\alpha,b)(\gamma,a)}\frac{1}{1+(\alpha,a)}(\alpha,c)\\
&+(\beta,a)\frac{1}{1+(\alpha,c)}(\alpha,c)(\gamma,a)\frac{1}{1+(\alpha,a)}(\alpha,b) -\tcb{(\gamma,a)}\frac{1}{1+(\alpha,a)}\tcr{(\alpha,b)}\tcb{(\beta,a)}\frac{1}{1+(\alpha,a)}\tcr{(\alpha,c)}\\
&\quad\hbox{(two last terms mutually cancelled)}\\
=&(q-q^{-1})\left[ -(\gamma,a)\frac{1}{1+(\alpha,a)}(\beta,a)\tcb{(\alpha,c)}\frac{1}{1+(\alpha,a)}\tcb{(\alpha,b)}- (\gamma,a)\Bigl[(\beta,b), \frac{1}{1+(\alpha,a)}\Bigr](\alpha,c)\right.\\
&\left. +\tcr{(\beta,a)}\frac{1}{1+(\alpha,a)}\tcr{(\gamma,a)}(\alpha,b)\frac{1}{1+(\alpha,a)}(\alpha,c)\right]\\
=&-(q-q^{-1})\bigl(q^{-1}+(q-q^{-1})-q\bigr) (\gamma,a)\frac{1}{1+(\alpha,a)}(\beta,a)(\alpha,b)\frac{1}{1+(\alpha,a)}(\alpha,c)=0.
\end{align*}

{\bf Case 2}. 
\raisebox{-25pt}{\mbox{
{\psset{unit=0.7}
\begin{pspicture}(-2,-1.5)(2,2.5)
\psframe[linewidth=2pt,linecolor=blue,linestyle=dashed](-1,-1.5)(1,1.5)
\psarc[linewidth=2pt,linecolor=red,linestyle=dashed](-1,1.5){0.5}{90}{270}
\psarc[linewidth=2pt,linecolor=red,linestyle=dashed](1,1.5){0.5}{-90}{90}
\psline[linewidth=2pt,linecolor=red,linestyle=dashed]{->}(-1,2)(0,2)
\psline[linewidth=2pt,linecolor=red,linestyle=dashed](0,2)(1,2)
\psline[linewidth=1.5pt,linecolor=red,linestyle=solid]{<-}(-1,1)(-0.5,1)
\psline[linewidth=1.5pt,linecolor=red,linestyle=solid]{<-}(-1,0)(-0.5,0)
\psline[linewidth=1.5pt,linecolor=red,linestyle=solid]{->}(-1,-1)(-0.5,-1)
\psline[linewidth=1.5pt,linecolor=red,linestyle=solid]{<-}(0.5,1)(1,1)
\psline[linewidth=1.5pt,linecolor=red,linestyle=solid]{<-}(0.5,0)(1,0)
\psline[linewidth=1.5pt,linecolor=red,linestyle=solid]{->}(0.5,-1)(1,-1)
\put(1.2,1){\makebox(0,0)[lt]{\hbox{{$\alpha$}}}}
\put(1.2,0){\makebox(0,0)[lc]{\hbox{{$\beta$}}}}
\put(1.2,-1){\makebox(0,0)[lc]{\hbox{{$c$}}}}
\put(-1.2,1){\makebox(0,0)[rt]{\hbox{{$a$}}}}
\put(-1.2,0){\makebox(0,0)[rc]{\hbox{{$b$}}}}
\put(-1.2,-1){\makebox(0,0)[rc]{\hbox{{$\gamma$}}}}
\end{pspicture}
}}}
Variant (a):  $\{\overline{(\beta,b)},\overline{(\gamma,c)}\}$.
\begin{align*}
&[\overline{(\gamma,c)},\overline{(\beta,b)}]=\left[ \Bigl((\gamma,c)+(\gamma,a)\frac{1}{1+(\alpha,a)}(\alpha,c) \Bigr), \Bigl((\beta,b)-(\beta,a)\frac{1}{1+(\alpha,a)}(\alpha,b) \Bigr) \right]\\
=&(\gamma,a)\frac{-1}{1+(\alpha,a)}(q-q^{-1})(\alpha,b)(\beta,a)\frac{1}{1+(\alpha,a)}(\alpha,c)\\
&-(\gamma,a)\frac{1}{1+(\alpha,a)}\tcb{(\alpha,c)}(\beta,a)\frac{1}{1+(\alpha,a)}\tcb{(\alpha,b)}+\tcr{(\beta,a)}\frac{1}{1+(\alpha,a)}(\alpha,b)\tcr{(\gamma,a)}\frac{1}{1+(\alpha,a)}(\alpha,c)=0
\end{align*}
Variant (b):  $\{\overline{(\beta,c)},\overline{(\gamma,b)}\}$.
\begin{align*}
&[\overline{(\gamma,b)},\overline{(\beta,c)}]=\left[ \Bigl((\gamma,b)-(\gamma,a)\frac{1}{1+(\alpha,a)}(\alpha,b) \Bigr) , \Bigl((\beta,c)-(\beta,a)\frac{1}{1+(\alpha,a)}(\alpha,c) \Bigr) \right]\\
=&(q-q^{-1}) (\gamma,a)(\beta,b)\frac{1}{1+(\alpha,a)}(\alpha,c)+(\beta,a)\frac{-1}{1+(\alpha,a)}(q-q^{-1})(\gamma,a)(\alpha,b)\frac{1}{1+(\alpha,a)}(\alpha,c)\\
&+(\gamma,a)\frac{1}{1+(\alpha,a)}(q-q^{-1})(\alpha,c)(\beta,a)\frac{1}{1+(\alpha,a)}(\alpha,b)+(q-q^{-1})(\gamma,a)\frac{1}{1+(\alpha,a)}(\alpha,c)(\beta,b)\\
&+\tcb{(\gamma,a)}\frac{1}{1+(\alpha,c)}\tcr{(\alpha,b)}\tcb{(\beta,a)}\frac{1}{1+(\alpha,a)}\tcr{(\alpha,c)} -(\beta,a)\frac{1}{1+(\alpha,a)}(\alpha,c)(\gamma,a)\frac{1}{1+(\alpha,a)}(\alpha,b)\\
&\quad\hbox{(two last terms mutually cancelled)}\\
=&-(q-q^{-1})\left[ - (\gamma,a)\Bigl[(\beta,b), \frac{1}{1+(\alpha,a)}\Bigr](\alpha,c) +\tcr{(\beta,a)}\frac{1}{1+(\alpha,a)}\tcr{(\gamma,a)}(\alpha,b)\frac{1}{1+(\alpha,a)}(\alpha,c)\right. \\
&\left. -(\gamma,a)\frac{1}{1+(\alpha,a)}(\beta,a)\tcb{(\alpha,c)}\frac{1}{1+(\alpha,a)}\tcb{(\alpha,b)} \right]\\
=&(q-q^{-1})\bigl((q-q^{-1})+q^{-1}-q\bigr) (\gamma,a)\frac{1}{1+(\alpha,a)}(\beta,a)(\alpha,b)\frac{1}{1+(\alpha,a)}(\alpha,c)=0.
\end{align*}

{\bf Case 3}. 
\raisebox{-25pt}{\mbox{
{\psset{unit=0.7}
\begin{pspicture}(-2,-1.5)(2,2.5)
\psframe[linewidth=2pt,linecolor=blue,linestyle=dashed](-1,-1.5)(1,1.5)
\psarc[linewidth=2pt,linecolor=red,linestyle=dashed](-1,1.5){0.5}{90}{270}
\psarc[linewidth=2pt,linecolor=red,linestyle=dashed](1,1.5){0.5}{-90}{90}
\psline[linewidth=2pt,linecolor=red,linestyle=dashed]{->}(-1,2)(0,2)
\psline[linewidth=2pt,linecolor=red,linestyle=dashed](0,2)(1,2)
\psline[linewidth=1.5pt,linecolor=red,linestyle=solid]{<-}(-1,1)(-0.5,1)
\psline[linewidth=1.5pt,linecolor=red,linestyle=solid]{<-}(-1,0)(-0.5,0)
\psline[linewidth=1.5pt,linecolor=red,linestyle=solid]{->}(-1,-1)(-0.5,-1)
\psline[linewidth=1.5pt,linecolor=red,linestyle=solid]{<-}(0.5,1)(1,1)
\psline[linewidth=1.5pt,linecolor=red,linestyle=solid]{->}(0.5,0)(1,0)
\psline[linewidth=1.5pt,linecolor=red,linestyle=solid]{<-}(0.5,-1)(1,-1)
\put(1.2,1){\makebox(0,0)[lt]{\hbox{{$\alpha$}}}}
\put(1.2,0){\makebox(0,0)[lc]{\hbox{{$c$}}}}
\put(1.2,-1){\makebox(0,0)[lc]{\hbox{{$\beta$}}}}
\put(-1.2,1){\makebox(0,0)[rt]{\hbox{{$a$}}}}
\put(-1.2,0){\makebox(0,0)[rc]{\hbox{{$b$}}}}
\put(-1.2,-1){\makebox(0,0)[rc]{\hbox{{$\gamma$}}}}
\end{pspicture}
}}}
Variant (a):  $\{\overline{(\gamma,b)},\overline{(\beta,c)}\}$.
\begin{align*}
&[\overline{(\gamma,b)},\overline{(\beta,c)}]=\left[ \Bigl((\gamma,b)-(\gamma,a)\frac{1}{1+(\alpha,a)}(\alpha,b) \Bigr), \Bigl((\beta,c)+(\beta,a)\frac{1}{1+(\alpha,a)}(\alpha,c) \Bigr) \right]\\
=&-(q-q^{-1})(\gamma,c)(\beta,b)-(q-q^{-1})(\gamma,a)(\beta,a)\frac{1}{1+(\alpha,a)}(\alpha,c)\\
&-(q-q^{-1})(\beta,a)\frac{-1}{1+(\alpha,a)}(\gamma,a)(\alpha,b)\frac{1}{1+(\alpha,a)}(\alpha,c)
+(q-q^{-1})(\gamma,c)(\beta,a)\frac{1}{1+(\alpha,a)}(\alpha,b)\\
&-(\gamma,a)\frac{1}{1+(\alpha,a)}\tcr{(\alpha,b)}(\beta,a)\frac{1}{1+(\alpha,a)}\tcr{(\alpha,c)}
+\tcr{(\beta,a)}\frac{1}{1+(\alpha,a)}(\alpha,c)\tcr{(\gamma,a)}\frac{1}{1+(\alpha,a)}(\alpha,b)\\
&\quad\hbox{(two last terms mutually cancelled)}\\
=&-(q-q^{-1})(\gamma,c)\left[(\beta,b) -(\beta,a)\frac{1}{1+(\alpha,a)}(\alpha,b)\right]- (q-q^{-1})(\gamma,a)\frac{1}{1+(\alpha,a)}(\alpha,c)(\beta.b)\\
&- (q-q^{-1}) (\gamma,a)\Bigl[(\beta,b), \frac{1}{1+(\alpha,a)}\Bigr](\alpha,c)+(q-q^{-1})\tcr{(\beta,a)}\frac{1}{1+(\alpha,a)}\tcr{(\gamma,a)}(\alpha,b)\frac{1}{1+(\alpha,a)}(\alpha,c)\\
=&-(q-q^{-1})(\gamma,c)\left[(\beta,b) -(\beta,a)\frac{1}{1+(\alpha,a)}(\alpha,b)\right]- (q-q^{-1})(\gamma,a)\frac{1}{1+(\alpha,a)}(\alpha,c)(\beta.b)\\
&-(q-q^{-1})\bigl((q-q^{-1})-q\bigr) (\gamma,a)\frac{1}{1+(\alpha,a)}(\beta,a)\tcr{(\alpha,b)}\frac{1}{1+(\alpha,a)}\tcr{(\alpha,c)}\\
=&-(q-q^{-1})\left[(\gamma,c)+(\gamma,a)\frac{1}{1+(\alpha,a)}(\alpha,c) \right]\left[(\beta,b) -(\beta,a)\frac{1}{1+(\alpha,a)}(\alpha,b)\right]=
- (q-q^{-1})\overline{(\gamma,c)}\,\overline{(\beta,b)}
\end{align*}
Variant (b):  $\{\overline{(\gamma,c)},\overline{(\beta,b)}\}$.
\begin{align*}
&[\overline{(\gamma,c)},\overline{(\beta,b)}]=\left[ \Bigl((\gamma,c)+(\gamma,a)\frac{1}{1+(\alpha,a)}(\alpha,c) \Bigr), \Bigl((\beta,b)-(\beta,a)\frac{1}{1+(\alpha,a)}(\alpha,b) \Bigr) \right]\\
=&-(\gamma,a)\frac{1}{1+(\alpha,a)}(q-q^{-1})(\alpha,b)(\beta,a)\frac{1}{1+(\alpha,a)}(\alpha,c)\\
&-(\gamma,a)\frac{1}{1+(\alpha,a)}\tcb{(\alpha,c)}(\beta,a)\frac{1}{1+(\alpha,a)}\tcb{(\alpha,b)}+\tcr{(\beta,a)}\frac{1}{1+(\alpha,c)}(\alpha,b)\tcr{(\gamma,a)}\frac{1}{1+(\alpha,a)}(\alpha,c)=0.
\end{align*}

{\bf Case 4}. 
\raisebox{-25pt}{\mbox{
{\psset{unit=0.7}
\begin{pspicture}(-2,-1.5)(2,2.5)
\psframe[linewidth=2pt,linecolor=blue,linestyle=dashed](-1,-1.5)(1,1.5)
\psarc[linewidth=2pt,linecolor=red,linestyle=dashed](-1,1.5){0.5}{90}{270}
\psarc[linewidth=2pt,linecolor=red,linestyle=dashed](1,1.5){0.5}{-90}{90}
\psline[linewidth=2pt,linecolor=red,linestyle=dashed]{->}(-1,2)(0,2)
\psline[linewidth=2pt,linecolor=red,linestyle=dashed](0,2)(1,2)
\psline[linewidth=1.5pt,linecolor=red,linestyle=solid]{<-}(-1,1)(-0.5,1)
\psline[linewidth=1.5pt,linecolor=red,linestyle=solid]{->}(-1,0)(-0.5,0)
\psline[linewidth=1.5pt,linecolor=red,linestyle=solid]{->}(-1,-1)(-0.5,-1)
\psline[linewidth=1.5pt,linecolor=red,linestyle=solid]{<-}(0.5,1)(1,1)
\psline[linewidth=1.5pt,linecolor=red,linestyle=solid]{->}(0.5,0)(1,0)
\psline[linewidth=1.5pt,linecolor=red,linestyle=solid]{->}(0.5,-1)(1,-1)
\put(1.2,1){\makebox(0,0)[lt]{\hbox{{$\alpha$}}}}
\put(1.2,0){\makebox(0,0)[lc]{\hbox{{$c$}}}}
\put(1.2,-1){\makebox(0,0)[lc]{\hbox{{$b$}}}}
\put(-1.2,1){\makebox(0,0)[rt]{\hbox{{$a$}}}}
\put(-1.2,0){\makebox(0,0)[rc]{\hbox{{$\gamma$}}}}
\put(-1.2,-1){\makebox(0,0)[rc]{\hbox{{$\beta$}}}}
\end{pspicture}
}}}
Variant (a):  $\{\overline{(\gamma,b)},\overline{(\beta,c)}\}$.
\begin{align*}
&[\overline{(\gamma,b)},\overline{(\beta,c)}]=\left[ \Bigl((\gamma,b)+(\gamma,a)\frac{1}{1+(\alpha,a)}(\alpha,b) \Bigr), \Bigl((\beta,c)+(\beta,a)\frac{1}{1+(\alpha,a)}(\alpha,c) \Bigr) \right]\\
=&(q-q^{-1})\tcm{(\beta,a)\frac{1}{1+(\alpha,a)}(\gamma,c)}(\alpha,b)
-(q-q^{-1})(\gamma,c)(\beta,a)\frac{1}{1+(\alpha,a)}(\alpha,b)\\
&+(\gamma,a)\frac{1}{1+(\alpha,a)}(\alpha,b)(\beta,a)\frac{1}{1+(\alpha,a)}(\alpha,c)
-\tcr{(\beta,a)}\frac{1}{1+(\alpha,a)}\tcb{(\alpha,c)}\tcr{(\gamma,a)}\frac{1}{1+(\alpha,a)}\tcb{(\alpha,b)}=0.
\end{align*}
Variant (b):  $\{\overline{(\gamma,c)},\overline{(\beta,b)}\}$.
\begin{align*}
&[\overline{(\gamma,c)},\overline{(\beta,b)}]=\left[ \Bigl((\gamma,c)+(\gamma,a)\frac{1}{1+(\alpha,a)}(\alpha,c) \Bigr), \Bigl((\beta,b)-(\beta,a)\frac{1}{1+(\alpha,a)}(\alpha,b) \Bigr) \right]\\
=&-(q-q^{-1})(\gamma,b)(\beta,c)-(q-q^{-1})(\gamma,b)(\beta,a)\frac{1}{1+(\alpha,a)}(\alpha,c)\\
&-(q-q^{-1})(\gamma,a)\frac{1}{1+(\alpha,a)}(\alpha,b)(\beta,c)\\
&+(\gamma,a)\frac{1}{1+(\alpha,a)}\tcb{(\alpha,c)}(\beta,a)\frac{1}{1+(\alpha,a)}\tcb{(\alpha,b)}-\tcr{(\beta,a)}\frac{1}{1+(\alpha,a)}(\alpha,b)\tcr{(\gamma,a)}\frac{1}{1+(\alpha,a)}(\alpha,c)\\
=&-(q-q^{-1})\left[(\gamma,b)+(\gamma,a)\frac{1}{1+(\alpha,a)}(\alpha,b) \right]\left[(\beta,c) -(\beta,a)\frac{1}{1+(\alpha,a)}(\alpha,c)\right]=
- (q-q^{-1})\overline{(\gamma,b)}\,\overline{(\beta,c)}.
\end{align*}

{\bf Case 5}. 
\raisebox{-25pt}{\mbox{
{\psset{unit=0.7}
\begin{pspicture}(-2,-1.5)(2,2.5)
\psframe[linewidth=2pt,linecolor=blue,linestyle=dashed](-1,-1.5)(1,1.5)
\psarc[linewidth=2pt,linecolor=red,linestyle=dashed](-1,1.5){0.5}{90}{270}
\psarc[linewidth=2pt,linecolor=red,linestyle=dashed](1,1.5){0.5}{-90}{90}
\psline[linewidth=2pt,linecolor=red,linestyle=dashed]{->}(-1,2)(0,2)
\psline[linewidth=2pt,linecolor=red,linestyle=dashed](0,2)(1,2)
\psline[linewidth=1.5pt,linecolor=red,linestyle=solid]{<-}(-1,1)(-0.5,1)
\psline[linewidth=1.5pt,linecolor=red,linestyle=solid]{->}(-1,0)(-0.5,0)
\psline[linewidth=1.5pt,linecolor=red,linestyle=solid]{<-}(-1,-1)(-0.5,-1)
\psline[linewidth=1.5pt,linecolor=red,linestyle=solid]{<-}(0.5,1)(1,1)
\psline[linewidth=1.5pt,linecolor=red,linestyle=solid]{->}(0.5,0)(1,0)
\psline[linewidth=1.5pt,linecolor=red,linestyle=solid]{<-}(0.5,-1)(1,-1)
\put(1.2,1){\makebox(0,0)[lt]{\hbox{{$\alpha$}}}}
\put(1.2,0){\makebox(0,0)[lc]{\hbox{{$c$}}}}
\put(1.2,-1){\makebox(0,0)[lc]{\hbox{{$\beta$}}}}
\put(-1.2,1){\makebox(0,0)[rt]{\hbox{{$a$}}}}
\put(-1.2,0){\makebox(0,0)[rc]{\hbox{{$\gamma$}}}}
\put(-1.2,-1){\makebox(0,0)[rc]{\hbox{{$b$}}}}
\end{pspicture}
}}}
Variant (a):  $\{\overline{(\gamma,b)},\overline{(\beta,c)}\}$.
\begin{align*}
&[\overline{(\gamma,b)},\overline{(\beta,c)}]=\left[ \Bigl((\gamma,b)+(\gamma,a)\frac{1}{1+(\alpha,a)}(\alpha,b) \Bigr), \Bigl((\beta,c)+(\beta,a)\frac{1}{1+(\alpha,a)}(\alpha,c) \Bigr) \right]\\
=&(q-q^{-1})\tcm{(\beta,a)\frac{1}{1+(\alpha,a)}(\gamma,c)}(\alpha,b)
-(q-q^{-1})(\gamma,c)(\beta,a)\frac{1}{1+(\alpha,a)}(\alpha,b)\\
&+(\gamma,a)\frac{1}{1+(\alpha,a)}(\alpha,b)(\beta,a)\frac{1}{1+(\alpha,a)}(\alpha,c)
-\tcr{(\beta,a)}\frac{1}{1+(\alpha,a)}\tcb{(\alpha,c)}\tcr{(\gamma,a)}\frac{1}{1+(\alpha,a)}\tcb{(\alpha,b)}=0.
\end{align*}
Variant (b):  $\{\overline{(\gamma,c)},\overline{(\beta,b)}\}$.
\begin{align*}
&[\overline{(\gamma,c)},\overline{(\beta,b)}]=\left[ \Bigl((\gamma,c)+(\gamma,a)\frac{1}{1+(\alpha,a)}(\alpha,c) \Bigr), \Bigl((\beta,b)-(\beta,a)\frac{1}{1+(\alpha,a)}(\alpha,b) \Bigr) \right]\\
=&-(q-q^{-1})(\gamma,a)\frac{1}{1+(\alpha,a)}(\alpha,b)(\beta,a)\frac{1}{1+(\alpha,a)}(\alpha,c)\\
&-(\gamma,a)\frac{1}{1+(\alpha,a)}\tcb{(\alpha,c)}(\beta,a)\frac{1}{1+(\alpha,a)}\tcb{(\alpha,b)}+\tcr{(\beta,a)}\frac{1}{1+(\alpha,a)}(\alpha,b)\tcr{(\gamma,a)}\frac{1}{1+(\alpha,a)}(\alpha,c)=0.
\end{align*}

{\bf Case 6}. 
\raisebox{-25pt}{\mbox{
{\psset{unit=0.7}
\begin{pspicture}(-2,-1.5)(2,2.5)
\psframe[linewidth=2pt,linecolor=blue,linestyle=dashed](-1,-1.5)(1,1.5)
\psarc[linewidth=2pt,linecolor=red,linestyle=dashed](-1,1.5){0.5}{90}{270}
\psarc[linewidth=2pt,linecolor=red,linestyle=dashed](1,1.5){0.5}{-90}{90}
\psline[linewidth=2pt,linecolor=red,linestyle=dashed]{->}(-1,2)(0,2)
\psline[linewidth=2pt,linecolor=red,linestyle=dashed](0,2)(1,2)
\psline[linewidth=1.5pt,linecolor=red,linestyle=solid]{<-}(-1,1)(-0.5,1)
\psline[linewidth=1.5pt,linecolor=red,linestyle=solid]{->}(-1,0)(-0.5,0)
\psline[linewidth=1.5pt,linecolor=red,linestyle=solid]{<-}(-1,-1)(-0.5,-1)
\psline[linewidth=1.5pt,linecolor=red,linestyle=solid]{<-}(0.5,1)(1,1)
\psline[linewidth=1.5pt,linecolor=red,linestyle=solid]{<-}(0.5,0)(1,0)
\psline[linewidth=1.5pt,linecolor=red,linestyle=solid]{->}(0.5,-1)(1,-1)
\put(1.2,1){\makebox(0,0)[lt]{\hbox{{$\alpha$}}}}
\put(1.2,0){\makebox(0,0)[lc]{\hbox{{$\beta$}}}}
\put(1.2,-1){\makebox(0,0)[lc]{\hbox{{$c$}}}}
\put(-1.2,1){\makebox(0,0)[rt]{\hbox{{$a$}}}}
\put(-1.2,0){\makebox(0,0)[rc]{\hbox{{$\gamma$}}}}
\put(-1.2,-1){\makebox(0,0)[rc]{\hbox{{$b$}}}}
\end{pspicture}
}}}
Variant (a):  $\{\overline{(\gamma,b)},\overline{(\beta,c)}\}$.
\begin{align*}
&[\overline{(\gamma,b)},\overline{(\beta,c)}]=\left[ \Bigl((\gamma,b)+(\gamma,a)\frac{1}{1+(\alpha,a)}(\alpha,b) \Bigr), \Bigl((\beta,c)-(\beta,a)\frac{1}{1+(\alpha,a)}(\alpha,c) \Bigr) \right]\\
=&(q-q^{-1})(\gamma,c)(\beta,b)-(q-q^{-1})\tcm{(\beta,a)\frac{1}{1+(\alpha,a)}(\gamma,c)}(\alpha,b)\\
&-(q-q^{-1})(\gamma,a)\frac{1}{1+(\alpha,a)}(\alpha,c)(\beta,a)\frac{1}{1+(\alpha,a)}(\alpha,b)+ (q-q^{-1})  (\gamma,a)\frac{1}{1+(\alpha,a)}(\alpha,c)(\beta,b)\\
&-(\gamma,a)\frac{1}{1+(\alpha,a)}(\alpha,b)(\beta,a)\frac{1}{1+(\alpha,a)}(\alpha,c)
-\tcr{(\beta,a)}\frac{1}{1+(\alpha,a)}\tcb{(\alpha,c)}\tcr{(\gamma,a)}\frac{1}{1+(\alpha,a)}\tcb{(\alpha,b)}\\
&\quad\hbox{(two last terms mutually cancelled)}\\
=&(q-q^{-1})\left[(\gamma,c)+(\gamma,a)\frac{1}{1+(\alpha,a)}(\alpha,c) \right]\left[(\beta,b) -(\beta,a)\frac{1}{1+(\alpha,a)}(\alpha,b)\right]=
 (q-q^{-1})\overline{(\gamma,c)}\,\overline{(\beta,b)}.
\end{align*}
Variant (b):  $\{\overline{(\gamma,c)},\overline{(\beta,b)}\}$.
\begin{align*}
&[\overline{(\gamma,c)},\overline{(\beta,b)}]=\left[ \Bigl((\gamma,c)+(\gamma,a)\frac{1}{1+(\alpha,a)}(\alpha,c) \Bigr), \Bigl((\beta,b)-(\beta,a)\frac{1}{1+(\alpha,a)}(\alpha,b) \Bigr) \right]\\
=&-(q-q^{-1})(\gamma,a)\frac{1}{1+(\alpha,a)}(\alpha,b)(\beta,a)\frac{1}{1+(\alpha,a)}(\alpha,c)\\
&-(\gamma,a)\frac{1}{1+(\alpha,a)}\tcb{(\alpha,c)}(\beta,a)\frac{1}{1+(\alpha,a)}\tcb{(\alpha,b)}+\tcr{(\beta,a)}\frac{1}{1+(\alpha,a)}(\alpha,b)\tcr{(\gamma,a)}\frac{1}{1+(\alpha,a)}(\alpha,c)=0.
\end{align*}

\section{Concluding remarks}\label{s:conclusion}
\setcounter{equation}{0}

In this paper, we have found the Darboux coordinate representation for matrices $\mathbb A$ enjoying the quantum reflection equation. We have also solved the problem of representing the braid-group action for the upper-triangular $\mathbb A$ in terms of mutations of cluster variables associated with the corresponding quiver. 

In conclusion, we indicate some directions of development. The first interesting problem is to construct mutation realizations for braid-group and Serre element actions that are Poisson automorphisms in the case of block-upper triangular matrices $\mathbb A$ (the corresponding action in terms of entries of a block-upper-triangular $\mathbb A$ was constructed in \cite{ChM3}). It is not difficult to construct planar networks producing block-triangular transport matrices $M_1$ and $M_2$ enjoying the standard Lie Poisson algebra, then $\mathbb A=M_1^{\text{T}}M_2$ will satisfy the semiclassical reflection equation.

{We have shown that particular sequences of mutations lead to transformations $\beta_{i,i+1}$ for $a_{ij}$ satisfying braid-group relations.
We conjecture that the sequences of mutations itself provide birational transformations of $Z_{abc}$ that also satisfy braid group relations and we checked this conjecture for small examples.}


Finally, the third direction of development is based on the semiclassical groupoid construction; explicit calculations in Sec.12 of \cite{ChM4} show that, given that $B$ is a general $SL_n$-matrix endowed with the standard semiclassical Lie Poisson bracket, solving the matrix equation $B\mathbb A B^{\text{T}}=\mathbb A'$, where $\mathbb A$ and $\mathbb A'$ are unipotent upper-triangular matrices, we obtain that entries of $\mathbb A$ are uniquely determined (provided all upper-right and lower-left minors of $B$ are nonzero); the Lie Poisson bracket on $B$ produces the reflection equation bracket on $\mathbb A$, the mapping $\mathbb A\to \mathbb A'$ is a Poisson anti-automorphism, and finally, entries of $\mathbb A$ and $\mathbb A'$ are mutually Poisson commute. Therefore, having a set $X_\alpha$ of Darboux coordinates for entries of $B$ we automatically obtain that $a_{i,j}\in \mathbb Z[[e^{X_\alpha}]]$ (the determinant of the linear system determining $a_{i,j}$ is the product of central elements of the Lie Poisson algebra for $B$). We therefore obtain a hypothetically alternative Darboux coordinate representation, this time for admissible pairs $(B,\mathbb A)$.

\section*{Acknowledgements}
The authors are grateful to Alexander Shapiro for the useful discussion. 
Both L.Ch. and M.S. are supported by International Laboratory of Cluster Geometry NRU HSE, RF Government grant, ag. \textnumero 2020-220-08-7077;
M.S. was supported by NSF grant DMS-1702115.

\end{document}